\documentclass[12pt]{amsart}
\usepackage[all]{xy}
\usepackage{amssymb,amsmath,amsfonts,verbatim, float, color}
\usepackage{enumerate, multicol}
\usepackage{graphicx, color}
\usepackage[linktocpage=true,colorlinks=true, linkcolor=blue, citecolor=red, urlcolor=green]{hyperref}
\newtheorem{theorem}{Theorem}
\newtheorem{lemma}[theorem]{Lemma}
\newtheorem{proposition}[theorem]{Proposition}
\newtheorem{corollary}[theorem]{Corollary}

\theoremstyle{definition}
\newtheorem{definition}[theorem]{Definition}
\newtheorem{example}[theorem]{Example}

\theoremstyle{remark}
\newtheorem{remark}[theorem]{Remark}

\numberwithin{equation}{section}
\numberwithin{theorem}{section}

\newcommand\thref{Theorem \ref}
\newcommand\leref{Lemma \ref}
\newcommand\prref{Proposition \ref}
\newcommand\coref{Corollary \ref}

\newcommand\reref{Remark \ref}
\newcommand\seref{Section \ref}
\def\CC{\mathbb{C}}
\def\QQ{\mathbb{Q}}
\def\RR{\mathbb{R}}
\def\ZZ{\mathbb{Z}}
\def\h{{\mathfrak{h}}}  
\DeclareMathOperator\End{End}

\DeclareMathOperator\tr{tr}
\DeclareMathOperator\rank{rank}
\DeclareMathOperator\asdim{asdim}
\DeclareMathOperator\qdim{qdim}
\DeclareMathOperator\Ker{Ker}

\DeclareMathOperator\Ind{Ind}
\DeclareMathOperator\Span{span}
\def\vac{{\boldsymbol{1}}}  %{|0\rangle} % vacuum vector
\def\tw{\mathrm{tw}}
\def\ii{\mathrm{i}} %{\sqrt{-1}}
\def\al{\alpha}                         %%% some abbreviations
\def\be{\beta}
\def\ga{\gamma}
\def\Ga{\Gamma}
\def\de{\delta}

\def\ep{\varepsilon}

\def\la{\lambda}
\def\La{\Lambda}

\def\om{\omega}
\def\Om{\Omega}
\def\ph{\varphi}
\def\si{\sigma}

\def\th{\theta}

\def\ze{\zeta}
\def\SL{\mathrm{SL}}
\def\d{\partial}
\def\<{\left\langle}
\def\>{\right\rangle}

\def\lieh{{\mathfrak{h}}}

\def\F{\mathcal{F}}
\def\B{\mathcal{B}}
\def\C{\mathcal{C}}
\def\L{\mathcal{L}}
\def\M{\mathcal{M}}
\def\O{\mathcal{O}}
\def\W{\mathcal{W}}
\def\Z{\mathcal{Z}}

\begin{document}
\title[Orbifolds of Lattice Vertex Algebras]{Orbifolds of Lattice Vertex Algebras}
\author[B.~Bakalov]{Bojko Bakalov}
\address{Department of Mathematics\\
North Carolina State University\\
Raleigh, NC 27695, USA}
\email{bojko\_bakalov@ncsu.edu}

\author[J.~Elsinger]{Jason Elsinger}
\address{Department of Mathematics\\
Florida Southern College\\
Lakeland, FL 33801, USA}
\email{jelsinger@flsouthern.edu}

\author[V.~G.~Kac]{Victor G.~Kac}

\address{Department of Mathematics, MIT, Cambridge, MA 02139, USA}
\email{kac@math.mit.edu}

\author[I.~Todorov]{Ivan Todorov}

\address{INRNE,
%Institute for Nuclear Research and Nuclear Energy, 
Bulgarian Academy of Sciences,
Tsarigradsko Chaussee 72, BG-1784 Sofia, Bulgaria}
\email{ivbortodorov@gmail.com}

%\thanks{The first author was supported in part by a Simons Foundation grant}

\date{October 21, 2022}



\begin{abstract}
To a positive-definite even lattice $Q$, one can associate the lattice vertex algebra $V_Q$, and any automorphism $\sigma$ of $Q$ lifts to an automorphism of $V_Q$.
In this paper, we investigate the orbifold vertex algebra $V_Q^\sigma$, which consists of the elements of $V_Q$ fixed under $\sigma$, in the case when $\sigma$ has prime order.
We describe explicitly the irreducible $V_Q^\sigma$-modules, compute their characters, and determine the modular transformations of characters.
As an application, we find the asymptotic and quantum dimensions of all irreducible $V_Q^\sigma$-modules.
We consider in detail the cases when the order of $\sigma$ is $2$ or $3$, as well as the case of permutation orbifolds.
\end{abstract}

\maketitle

\tableofcontents

%%%%%%%%%%%%%%%%%%%%%%%%%%%%%%%
\section{Introduction}\label{sintro}

In the past $25$ years there has been significant progress in the theory of regular vertex algebras \cite{DLM2, MC, DRX}. Recall that a vertex algebra $V$ with a conformal vector $L$ \cite{K2} is called \emph{regular} \cite{DLM2} if all $V$-modules are completely reducible, and all eigenvalues of $L_0$ are non-negative integers with finite multiplicities. It is often required, in addition, that the $0$-eigenvalue of $L_0$ has multiplicity $1$, that $V$ is simple, and that the $V$-module $V$ is self-dual, which we shall do in the present paper.

Examples of regular vertex algebras include the simple affine vertex algebras at positive integer levels, Virasoro minimal series, and lattice vertex algebras $V_Q$ associated to a positive-definite even lattice $Q$ (see e.g.\ \cite{K2} for their construction). 

Regular vertex algebras were introduced in \cite{DLM2}, where it was shown that they are in particular rational. Hence for them the Zhu Theorem \cite{Z} holds, i.e., a regular vertex algebra has finitely many irreducible modules and the span of their characters is $\mathrm{SL}_2(\ZZ)$-invariant.

A major advance of the theory has been made by Carnahan and Miyamoto \cite{MC}, who proved that for any finite-order automorphism $\si$ of a regular vertex algebra $V$, its fixed-point subalgebra $V^\si$ is regular as well. Another important result, by Dong, Ren and Xu \cite{DRX}, implies that for a regular vertex algebra $V$ all irreducible modules over $V^\si$ occur as submodules of irreducible $\si^k$-twisted $V$-modules for some $k$.

In the present paper, we use these results to describe explicitly the irreducible $V_Q^\si$-modules for a prime order $p$ automorphism $\si$ of the positive-definite even lattice $Q$ (Theorem \ref{class}), to compute their characters (Section \ref{charVQ}), and to obtain explicit modular transformation formulas for these characters (Theorems \ref{chiorbtran}, \ref{chiorbtran2}, \ref{chiorbtran3}). For this we use the explicit description of twisted $V_Q$-modules obtained in \cite{BK}.

As an application, we find the asymptotic dimensions \cite{KP2} of all irreducible $V_Q^\si$-modules, where $\si$ is a prime order $p$ automorphism of $Q$ (Corollary \ref{corqdim}).
In Sections \ref{ex2} and \ref{ex3}, we consider the cases $p=2$ and $p=3$, 
and in Section \ref{permorb} the case of a permutation orbifold with $Q=Q_0^{\oplus p}$ for any positive-definite even lattice $Q_0$ and the cyclic permutation $\si$ of the summands 
(cf.\ \cite{DXY1,DXY2,DXY3,DXY4}).

%%%%%%%%%%%%%%%%%%%%%%%%%%%%%%%
\section{Vertex Algebras and Their Twisted Modules}\label{sva}
%%%%%%%%%%%%%%%%%%%%%%%%%%%%%%%

In this section, we briefly recall the notions of a vertex algebra and of a twisted module over a vertex algebra.
Then we review several important theorems about regular vertex algebras.
Good general references on vertex algebras are \cite{FLM, FHL, K2, FB, LL, KRR}. 

%%%%%%%%%%%%%%%%%%%%%%%%%%%%%
\subsection{Conformal vertex algebras}\label{vert}
%%%%%%%%%%%%%%%%%%%%%%%%%%%%%

Recall that a \emph{vertex algebra} 
is a vector space of states $V$ with a distinguished vector $\vac\in V$ 
(vacuum vector), together with a linear map 
(state-field correspondence)
\begin{equation}\label{vert2}
Y(\cdot,z)\cdot \colon V \otimes V \to V(\!(z)\!) = V[[z]][z^{-1}] \,,
\end{equation}
satisfying axioms \eqref{vert6} and \eqref{vert5} below.
Thus, for every $a\in V$, we have the \emph{field}
$Y(a,z) \colon V \to V(\!(z)\!)$. This field can be viewed as
a formal power series from $(\End V)[[z,z^{-1}]]$, which 
involves only finitely many negative powers of $z$ when
applied to any vector.

The coefficients in front of powers of $z$ in this expansion are known as the
\emph{modes} of $a$:
\begin{equation}\label{vert4}
Y(a,z) = \sum_{n\in\ZZ} a_{(n)} \, z^{-n-1} \,, \qquad
a_{(n)} \in \End V \,.
\end{equation}
%The \emph{formal residue} $\Res_z$ of a formal power series is defined as the coefficient of $z^{-1}$, and
%\begin{equation}\label{fres}
%a_{(n)} = \Res_z z^n Y(a,z) \,.
%\end{equation}
The vacuum vector $\vac$ plays the role of an identity in the sense that
\begin{equation}\label{vert6}
a_{(-1)}\vac = \vac_{(-1)} a = a \,, \quad 
a_{(n)}\vac = \vac_{(m)} a = 0 \,, \quad n\geq 0 \,, \; m\ne-1 \,.
\end{equation}
This means that $Y(\vac,z)$ is the identity operator, $Y(a,z)\vac \in V[[z]]$ is regular at $z=0$, and
$Y(a,z)\vac|_{z=0} = a$.

The main axiom for a vertex algebra is the \emph{Borcherds identity}
(also called Jacobi identity \cite{FLM})
satisfied by the modes:
\begin{equation}\label{vert5}
\begin{split}
\sum_{j=0}^\infty & \binom{m}{j} (a_{(n+j)}b)_{(k+m-j)}c
= \sum_{j=0}^\infty \binom{n}{j} (-1)^j 
a_{(m+n-j)}(b_{(k+j)}c)
\\
&- \sum_{j=0}^\infty \binom{n}{j} (-1)^{j+n} \, b_{(k+n-j)}(a_{(m+j)}c) \,,
\end{split}
\end{equation}
where $a,b,c \in V$ and $m,n,k\in\ZZ$. %and $p(a) \in\ZZ/2\ZZ$ denotes the parity of $a$.
Note that the above sums are finite, because
$a_{(n)}b = 0$ for sufficiently large~$n$.

We say that a vertex algebra $V$ is (strongly) \emph{generated} by a subset $S\subset V$ if $V$
is linearly spanned by the vacuum $\vac$ and all elements of the form
${a_1}_{(n_1)} \cdots {a_k}_{(n_k)} \vac$, where $k\geq1$, $a_i \in S$, $n_i<0$.
An \emph{ideal} of a vertex algebra $V$ is a subspace $W$ such that $a_{(n)}w \in W$ for all $a\in V$, $w\in W$, $n\in\ZZ$. The vertex algebra $V$ is \emph{simple} if it contains no nonzero proper ideals. 

Essential to our setting is the notion of a conformal vertex algebra and a vertex operator algebra, which we define below.
\begin{definition}\label{CVA}
A vertex algebra $V$ is called \emph{conformal} of central charge $c\in\CC$ if
there exists a Virasoro vector $L\in V$ such that the corresponding field $Y(L,z)=\sum_{n\in\ZZ}L_nz^{-n-2}$ satisfies:
\begin{enumerate}[$(i)$]
%\begin{enumerate}
\item $[L_{-1},Y(a,z)]=\partial_zY(a,z), \quad a\in V$;
\item $[L_m,L_n]=(m-n)L_{m+n}+\delta_{m,-n}\frac{m^3-m}{12}c$;\label{CVA2}
\item $L_0$ is diagonalizable.
%with nonzero integral eigenvalues and the $0$-th eigenspace $\CC\vac$;
\end{enumerate}
\end{definition}
%The scalar $c\in\CC$ in Definition \ref{CVA}(\ref{CVA2}) is called the \emph{central charge} of $V$. 
The eigenvalues of $L_0$ are called \emph{conformal weights} or \emph{conformal dimensions}. The \emph{Virasoro algebra} is the Lie algebra with basis $\{c,L_n\,|\,n\in\ZZ\}$ 
equipped with the bracket relations in Definition \ref{CVA}(\textit{\ref{CVA2}}) together with the condition that $c$ is central.
A \emph{vertex operator algebra} is a conformal vertex algebra, in which all eigenvalues of $L_0$ are integers and all eigenspaces of $L_0$ are finite dimensional. 

A \emph{representation} of a vertex algebra $V$, or a $V$-\emph{module}, is a vector space $M$ endowed with a
linear map $Y(\cdot,z)\cdot \colon V \otimes M \to M(\!(z)\!)$
(cf.\ \eqref{vert2}, \eqref{vert4}) such that the Borcherds identity
\eqref{vert5} holds for $a,b\in V$, $c\in M$ (see \cite{FB, LL, KRR}).
For a vertex operator algebra $V=\bigoplus_{n\in\ZZ}V_n$, where $V_n$ is the $L_0$-eigenspace with eigenvalue $n$,
its \emph{dual} or \emph{contragradient} $V'$ is defined as
\begin{align}\label{Vdual}
V'=\bigoplus_{n\in\ZZ}V_n^* \,.
\end{align}
It is shown in \cite{FHL} that this duality can be defined more generally for $V$-modules $M$ with analogous grading, and in this case the dual $M'$ is also a $V$-module.
The vertex operator algebra $V$ is called \emph{self-dual} if $V\cong V'$ as $V$-modules. 

%%%%%%%%%%%%%%%%%%%%%%%%%%%%%%%%%%
\subsection{Twisted representations of vertex algebras}\label{twrep}
%%%%%%%%%%%%%%%%%%%%%%%%%%%%%%%%%%

Let $\si$ be an automorphism of a vertex algebra $V$ of a finite order $N$. Then $\si$ is diagonalizable.
In the definition of a \emph{$\si$-twisted representation} $M$ of $V$ \cite{FFR, D2}, 
the image of the linear map $Y$ is allowed to have nonintegral rational powers of $z$, so that
\begin{equation}\label{twlat1}
Y(a,z) = \sum_{n\in m+\ZZ} a_{(n)} \, z^{-n-1} \,, \qquad
\text{if} \quad \si a = e^{-2\pi\ii m} a \,, \;\; m\in\frac1N\ZZ \,,
\end{equation}
where $a_{(n)} \in \End M$.
%Equivalently, the monodromy around $z=0$ is given by the action of $\si$:
%\begin{equation}\label{twrep2}
%Y(\si a,z) = Y(a, e^{2\pi\ii}z) \,, \qquad a\in V \,.
%\end{equation}
The Borcherds identity \eqref{vert5} satisfied by the modes remains the
same in the twisted case, provided that $a$ is an eigenvector of $\si$, where $a,b\in V$, $c\in M$, $n\in\ZZ$, and $k,m\in\frac1N\ZZ$.

An important consequence of the Borcherds identity is the
\emph{locality} property \cite{DL1,Li,K2}:
\begin{equation}\label{locpr}
(z-w)^{N} [Y(a,z), Y(b,w)] = 0 
\end{equation}
for sufficiently large $N$ depending on $a,b$ (one can take $N$ to be such that
$a_{(n)}b=0$ for $n\ge N$).

\begin{proposition}[\cite{BM}]\label{pnprod}
Let\/ $V$ be a vertex algebra, $\si$ an automorphism of\/ $V$, and\/ $M$ a $\si$-twisted representation of\/ $V$.
Then
\begin{equation}\label{locpr3}
\frac1{k!} \d_{z}^k \Bigl( (z-w)^{N} \, Y(a,z) Y(b,w) v \Bigr)\Big|_{z=w}
= Y(a_{(N-1-k)} b, w) v
\end{equation}
for all\/ $a,b\in V$, $v\in M$, $k\geq0$, and sufficiently large\/ $N$.
%We can take\/ $N$ such that\/ \eqref{locpr} holds.
Conversely, \eqref{locpr} and \eqref{locpr3} imply the Borcherds identity \eqref{vert5}.
\end{proposition}

\begin{comment}
Recall from \cite{FHL} that if $V_1$ and $V_2$ are vertex algebras, their tensor product is again a vertex algebra with 
\begin{equation*}%\label{tensprod}
Y(v_1\otimes v_2,z) = Y(v_1,z) \otimes Y(v_2,z) \,, \qquad v_i \in V_i \,.
\end{equation*}
Furthermore, if $M_i$ is a $V_i$-module, then the above formula defines the structure of a $(V_1\otimes V_2)$-module
on $M_1\otimes M_2$ (see \cite{FHL}). This is also true for twisted modules \cite{BE}.

\begin{lemma}\label{ltensprod1}
For\/ $i=1,2$, let\/ $V_i$ be a vertex algebra, $\si_i$ an automorphism of\/ $V_i$, and\/ $M_i$ a $\si_i$-twisted representation of\/ $V_i$.
Then\/ $M_1\otimes M_2$ is a\/ $(\si_1\otimes \si_2)$-twisted module over\/ $V_1\otimes V_2$.
\end{lemma}
\begin{proof}
By \prref{pnprod}, it is enough to check \eqref{locpr} and \eqref{locpr3} for $a=a_1\otimes a_2$ and $b=b_1\otimes b_2$, given that they hold
for the pairs $a_1,b_1\in V_1$ and $a_2,b_2\in V_2$. This is done by a straightforward calculation.
\end{proof}
\end{comment}

%%%%%%%%%%%%%%%%%%%%%%%%%%%%%%%%%%%%%
\subsection{Regular vertex algebras}
%%%%%%%%%%%%%%%%%%%%%%%%%%%%%%%%%%%%%

In this subsection, we give a definition of regular vertex algebras and state a series of remarkable theorems about them, which were proved by several authors. 
%Recall Definition \ref{CVA} and the definitions that follow it. 
%For our setting in this work, all of these assumptions will hold for lattice vertex algebras.

\begin{definition}\label{reg}
Let $V$ be a conformal vertex algebra with conformal vector $L\in V$. Then $V$ is called \emph{regular} if the following additional conditions hold:
\begin{enumerate}[$(i)$]
%\item There exists a Virasoro vector $L\in V$, namely, the corresponding field $Y(L,z)=\sum_{n\in\ZZ}L_nz^{-n-2}$ satisfies
%\begin{enumerate}
%\item $[L_{-1},Y(a,z)]=\partial_zY(a,z), \quad a\in V$;
%\item $[L_m,L_n]=(m-n)L_{m+n}+\delta_{m,-n}\frac{m^3-m}{12}c, \quad c\in\CC$;
\item $L_0$ has non-negative integral eigenvalues;
\item the eigenspaces of $L_0$ are finite dimensional; %the $0$-th eigenspace $\CC\vac$;
%\end{enumerate}
%\item $V$ is simple;
%\item $V$ is self-dual, i.e., that the $V$-modules $V$ and its dual $V'$ are isomorphic;
\item all $V$-modules are completely reducible; \label{reg3}
\item $V$ is \emph{simple}, i.e.,  $V$ contains no nontrivial ideals; \label{reg4}
\item $V$ is \emph{self-dual}, i.e.,  the contragredient module $V'$ (cf.\ \eqref{Vdual}) is isomorphic to $V$;  \label{reg5}
\item %$V$ is of \emph{CFT-type}, i.e.,  
the $0$-eigenspace of $L_0$ is $\CC\vac{}$. \label{reg6}
\end{enumerate}
\end{definition}

%\begin{definition}
%Let $V$ be a vertex operator algebra with conformal vector $L\in V$. Then $V$ is \emph{regular} if the additional conditions hold:
%\begin{enumerate}[$(i)$]
%%\item There exists a Virasoro vector $L\in V$, namely, the corresponding field $Y(L,z)=\sum_{n\in\ZZ}L_nz^{-n-2}$ satisfies
%%\begin{enumerate}
%%\item $[L_{-1},Y(a,z)]=\partial_zY(a,z), \quad a\in V$;
%%\item $[L_m,L_n]=(m-n)L_{m+n}+\delta_{m,-n}\frac{m^3-m}{12}c, \quad c\in\CC$;
%\item $L_0$ has non-negative integral eigenvalues,
%%\item the eigenspaces of $L_0$ are finite dimensional,%the $0$-th eigenspace $\CC\vac$;
%%\end{enumerate}
%%\item $V$ is simple;
%%\item $V$ is self-dual, i.e., that the $V$-modules $V$ and its dual $V'$ are isomorphic;
%\item all $V$-modules are completely reducible.
%\end{enumerate}
%\end{definition}

\begin{remark}
The $V$-modules that we consider in Definition \ref{reg}(\textit{\ref{reg3}}) are the same as the weak $V$-modules as in \cite{MC} and \cite{ABD}. In general, there are three types of $V$-modules, 
labeled as weak, admissible, and ordinary. The \emph{weak} modules are not necessarily graded. %only assumed to have a $\CC$-grading. 
The \emph{admissible} modules have a $\ZZ_+$-grading, 
%$\frac1N\ZZ_+$-grading, where $N$ is the order of the twisting automorphism, 
which is compatible with the action of $V$. 
The strongest notion of module is the \emph{ordinary} $V$-module, which is graded by the eigenvalues of $L_0$, all eigenvalues are in $\ZZ_+$, and all eigenspaces are finite dimensional.
%where the grading has the same grading as $V$.
In the literature, rationality refers to complete reducibility of admissible modules, while regularity refers to complete reducibility of weak modules, which {\it a priori} is more general than rationality. For regular vertex algebras, all these notions of $V$-module coincide. 
%This is true, in particular, for lattice vertex algebras.
%As a consequence of regularity, every weak module is actually an ordinary module. 
%Another way to state regularity is that every weak module decomposes as a direct some of 
\end{remark}

The original definition of regularity from \cite{DLM2} does not include assumptions (\textit{\ref{reg4}}), (\textit{\ref{reg5}}), (\textit{\ref{reg6}}). %$(iv), (v),$ and $(vi)$. 
While in general one may consider dropping several of these assumptions, we include them here in order to state the following theorems in a more concise form.

%Some of the above conditions are omitted by some authors in the definition of regularity, but we included all of them in order to state a series of remarkable theorems, proved by many authors.

\begin{theorem}[\cite{MC}]\label{MC}
Let\/ $V$ be a regular vertex algebra, and\/ $\Gamma$ be a cyclic group of automorphisms of\/ $V$. Then the fixed point subalgebra $V^\Gamma$ (called the orbifold) is regular as well.
\end{theorem}

\begin{theorem}[\cite{Z}, \cite{ABD}]\label{ZABD}
Let\/ $V$ be a regular vertex algebra of central charge $c$. Then
\begin{enumerate}[$(i)$]
\item $V$ has, up to isomorphism, a finite number of irreducible modules, $M_0=V$, $M_1,\ldots,M_m$.
\item The characters
\begin{align*}
\chi_j(\tau):=\tr_{M_j}q^{L_0-c/24},
\end{align*}
where $q=e^{2\pi\ii\tau}$, are convergent series to holomorphic functions for\/ $\mathrm{Im}\,\tau>0$.
\item The\/ $\CC$-span of the functions\/ $\chi_0,\ldots,\chi_m$ is\/ $\mathrm{SL}_2(\ZZ)$-invariant under the modular transformation \label{ZABD3}
\begin{align*}
f(\tau)\mapsto f\left(\frac{a\tau+b}{c\tau+d}\right),
\end{align*}
where $\begin{pmatrix}
a&b\\c&d
\end{pmatrix}\in \mathrm{SL}_2(\ZZ)$.
Equivalently, we have
\begin{align*}
\chi_j(\tau+1)&=e^{2\pi\ii(\Delta_j-c/24)}\chi_j(\tau),\quad
\chi_j\left(-\frac{1}{\tau}\right)=\sum_{k=0}^mS_{j,k}\chi_k(\tau),
\end{align*}
where $S_{j,k}\in\CC$, and\/ $\Delta_j$ is the conformal weight of\/ $M_j$, i.e., the minimal eigenvalue of\/ $L_0$ in $M_j$.
\end{enumerate}
\end{theorem}

\begin{theorem}[\cite{H}]\label{H}
Let\/ $V$ be a regular vertex algebra. Then, in the notation of Theorem \ref{ZABD}, we have:
\begin{enumerate}[$(i)$]
\item $S_{j,k}=S_{k,j}$ \, for all \, $j, k=0,\ldots,m$,
\item $S_{j,0}\neq0$ \, for all \, $j=0,\ldots,m$.\label{H2}
\item Let 
\begin{align*}
M_i\boxtimes M_j=\bigoplus_{k=0}^mN_{i,j}^{k}M_k
\end{align*}
be the fusion product. Then Verlinde's formula holds:
\begin{equation}\label{Verlinde}
N_{i,j}^{k}=\sum_{l=0}^m\frac{S_{i,l}S_{j,l}\bar{S}_{k,l}}{S_{l,0}} \,.
\end{equation}
%\emph{(}known as the Verlinde formula\emph{)}.
\end{enumerate}
\end{theorem}

Theorems \ref{ZABD} and \ref{H} have the following simple corollary (cf.\ \cite{KP2}).

\begin{corollary}\label{KacCor}
Assume that\/ $V$ is a regular vertex algebra of central charge\/ $c$, such that\/ $\Delta_j>0$ for all\/ $j>0$. Then 
\begin{enumerate}[$(i)$]
\item As\/ $\tau\to0^+$ we have
\begin{align*}
\chi_j(\tau) \sim S_{j,0}e^{\pi\ii c/(12\tau)},\quad j=0,\ldots,m,
\end{align*}
\item $S_{j,0}>0 \quad\text{for all}\quad j=0,\ldots,m$.
\end{enumerate}
\end{corollary} 

\begin{proof}
Replacing $\tau$ by $-\frac{1}{\tau}$, we obtain from Theorem \ref{ZABD}(\textit{\ref{ZABD3}}):
$$\chi_j(\tau)=\sum_{k=0}^mS_{j,k} \,\chi_k\Bigl(-\frac{1}{\tau}\Bigr).$$
But, by the definition of $\chi_k$, we have:
\begin{align*}
\chi_k\Bigl(-\frac{1}{\tau}\Bigr)=\exp\Bigl(\frac{\pi\ii c}{12\tau}-\frac{2\pi\ii\Delta_j}{\tau}\Bigr)(a_k+O(\tau)),
\end{align*}
where $a_k$ is the multiplicity of the $L_0$-eigenvalue $\Delta_k$ in $M_k$, so that $a_0=1$. Since, by the assumption, $\Delta_k>0$ for $k\geq1$, we obtain the first part.

Letting $\tau=\ii\beta$, where $\beta$ is a positive real number, we obtain from the first part:
\begin{align*}
\chi_j(\ii\beta)\sim S_{j,0}e^{\pi c/(12\beta)}>0
\end{align*}
since $S_{j,0}\neq0$. But the category of $V$-modules is a modular tensor category \cite{H2}, and
$c\in\QQ$ in any modular tensor category (see e.g.\ \cite{BKr}); hence $S_{j,0}>0$.
\end{proof}

\begin{definition}
Provided that the condition of Corollary \ref{KacCor} holds, the positive real number $S_{j,0}$ is called the \emph{asymptotic dimension} of the $V$-module $M_j$ and denoted by $\asdim M_j$.  The number $S_{j,0}/S_{0,0}$ is called the {\it quantum dimension} of $M_j$ and is denoted $\qdim M_j$.
\end{definition}

\begin{remark}
Obviously, the asymptotic dimension is additive. Also, the quantum dimension is a ring homomorphism from the fusion ring of modules to $\CC$ (see \cite{H, V}).
\end{remark}

\begin{remark}\label{Sprops}
Two other well-known properties of the matrix $S=(S_{j,k})_{j,k=0}^m$ for arbitrary regular vertex algebras are (cf.\ \cite{KP2} for affine $V$): $S^2=C$, where $C$ is a permutation matrix of square $1$ (cf.\ \cite{H}),  the matrix $S$ is unitary,  and in particular, 
\begin{align*}
\sum_{j=0}^m|S_{0,j}|^2=1.
\end{align*}
As pointed out to us by Y.-Z. Huang, the unitarity of $S$ follows from his theorem that the category of modules over a regular vertex algebra is modular \cite{H2}, and a theorem of Etingof--Nikshich--Ostrik \cite{ENO}, that $S$ is unitary for any modular tensor category.
\end{remark}

\subsection{Theta functions and transformation laws}\label{theta}
%%%%%%%%%%%%%%%%%%%%%%%%%%%%%%%%%

In this subsection, we review the definition and transformation laws of the classical theta functions (see e.g.\ \cite{KP2, K1}). 

Let $\mathcal{L}$ be a positive-definite even integral \emph{lattice},  i.e., a free abelian group of finite rank equipped with a symmetric
bilinear form $(\cdot|\cdot) \colon Q\times Q\to\ZZ$ such that $|\al|^2=(\al|\al)$ is a positive even integer for all nonzero $\al\in Q$.
Set $r=\text{rank} \,\mathcal{L}$ and $\mathfrak{h}_\mathcal{L}=\CC\otimes_{\ZZ}\mathcal{L}$. %As usual we set $q=e^{2\pi\ii\tau}$, where $\tau\in\CC$ with Im$\tau>0$.
%Let $\mathfrak{h}$ be the $(r+2)$-dimensional vector space 
%\[
%\mathfrak{h}=\mathfrak{h}_\mathcal{L}\oplus\CC\de\oplus\CC\Lambda_0,
%\]
%where $\de,\Lambda_0\in\RR\otimes_\ZZ \mathcal{L}$ are such that $|\Lambda_0|^2=|\de|^2=0$, $(\Lambda_0|\de)=1$, and both are orthogonal to $\mathcal{L}$.
%%$(\de|M)=0=(\Lambda_0|M)$.
%For $z\in\mathfrak{h}_\mathcal{L}$ and complex numbers $\tau, u$ with Im$\tau>0$, we coordinatize $\mathfrak{h}$ by
%\begin{equation}\label{coor}
%(\tau,z,u)\rightarrow2\pi\ii(z-\tau\Lambda_0+u\de).
%\end{equation}
%In additon, set $q=e^{2\pi\ii\tau}$ and
%\[
%P=\{\la\in\mathfrak{h}\,|\,(\la|\de)=1, \bar{\la}\in \mathcal{L}^*\},
%\]
%where $\la\to\bar{\la}$ is the projection of $\mathfrak{h}$ onto $\mathfrak{h}_\mathcal{L}$, and $\mathcal{L}^*$ is the dual lattice of $\mathcal{L}$.
For $\la\in \mathcal{L}^*, u\in\CC$, and $z\in \mathfrak{h}_\mathcal{L}$, the {\it classical theta function} (of degree 1) is defined as

\begin{equation}\label{th}
\theta_{\la+\mathcal{L}}(\tau,z,u)=e^{2\pi\ii u}\sum_{\ga\in {\la}+\mathcal{L}}e^{2\pi\ii(\ga|z)}q^{|\ga|^2/2} \,,
\end{equation} 
where $q=e^{2\pi\ii\tau}$.
%using the coordinatization \eqref{coor}. 
%For our purposes, we will mainly need theta functions with $u=0$. 
The parameter $u$ is introduced in order to simplify the transformation formulas in \thref{thtlaws} below.
%Set notation of theta(tau)
\begin{remark}
For any scalar $c$, we have
\begin{align}\label{thc}
\begin{split}
\theta_{c\la+c\mathcal{L}}(\tau,z,u)&=e^{2\pi\ii u}\sum_{\ga\in{\la}+\mathcal{L}}e^{2\pi\ii(c\ga|z)}e^{\pi\ii\tau c^2|\ga|^2}\\
&=\theta_{\la+\mathcal{L}}(c^2\tau, cz, u).
\end{split}
\end{align} 
This formula is particularly useful when the lattice $\mathcal{L}$ is not integral. If we can choose a suitable value of $c$ which makes $c\mathcal{L}$ an integral lattice, then the known transformation laws of theta functions can be applied using the lattice $c\mathcal{L}$ and \eqref{thc}.
\end{remark}

The transformation laws of theta functions go back to Jacobi and are given in \cite[Theorem 13.5]{K1} among many other works.

\begin{theorem}\label{thtlaws}
Let $\mathcal{L}$ be a positive-definite even integral lattice of rank $r$, and let $\la\in \mathcal{L}^*$. Then the transformation laws of \eqref{th} are
\begin{align}
\begin{split}
\theta_{\la+\mathcal{L}}\left(-\frac{1}{\tau},\frac{z}{\tau},u-\frac{|z|^2}{2\tau}\right)&=(-\ii\tau)^{r/2}|\mathcal{L}^*/\mathcal{L}|^{-1/2}\\
&\times\sum_{\mu+\mathcal{L}\in \mathcal{L}^*/\mathcal{L}}e^{-2\pi\ii(\la|\mu)}\theta_{\mu+\mathcal{L}}(\tau,z,u),\label{thtlaws1}
\end{split}
\end{align}
\begin{equation}
\theta_{\la+\mathcal{L}}(\tau+1,z,u)=e^{\pi\ii|\la|^2}\theta_{\la+\mathcal{L}}(\tau,z,u).\label{thtlaws2}
\end{equation}
\end{theorem}

For our purposes, it will often be convenient to set $u=0$. In this case, we set $\theta_{\la+\mathcal{L}}(\tau,z)=\theta_{\la+\mathcal{L}}(\tau,z,0)$, and Theorem \ref{thtlaws} becomes:
\begin{align}
\begin{split}
\theta_{\la+\mathcal{L}}\left(-\frac{1}{\tau},\frac{z}{\tau}\right)&=(-\ii\tau)^{r/2}|\mathcal{L}^*/\mathcal{L}|^{-1/2}\\
&\times e^{\pi\ii|z|^2/\tau}\sum_{\mu+\mathcal{L}\in \mathcal{L}^*/\mathcal{L}}e^{-2\pi\ii(\la|\mu)}\theta_{\mu+\mathcal{L}}(\tau,z),\label{thtlaws3}
\end{split}
\end{align}
\begin{equation}
\theta_{\la+\mathcal{L}}(\tau+1,z)=e^{\pi\ii|\la|^2}\theta_{\la+\mathcal{L}}(\tau,z).\label{thtlaws4}
\end{equation}
In many of the examples, we also set $z=0$. In this case we set $\theta_{\la+\mathcal{L}}(\tau)=\theta_{\la+\mathcal{L}}(\tau,0,0)$.

The transformation corresponding to a general element of $\mathrm{SL}(2,\RR)$ is given in \cite[Corollary 3.9]{KP2}. The following is a special case of their general result, which we will need later.

\begin{theorem}[\cite{KP2}, Corollary 3.9]\label{KP}
Let\/ $\mathcal{L}$ be a positive-definite even lattice with rank $r$. 
Consider the action of\/ $A=\begin{pmatrix}
a&b\\c&d
\end{pmatrix}\in \mathrm{SL}(2,\RR)$ on triples $(\tau,z,u)$ given by
\begin{align*}
A\cdot(\tau,z,u)=\left(\frac{a\tau+b}{c\tau+d} \,, \frac{z}{c\tau+d} \,,u-\frac{c}{2}\frac{|z|^2}{c\tau+d}\right),
\end{align*}
where $\tau, u\in\CC$, {\upshape{Im}}$\,\tau>0$, and $z\in \mathcal{L}$. There exists\/ $\be_0\in\CC\otimes_\ZZ\mathcal{L}$ such that
\[
ac|\nu|^2\equiv 2(\nu|\be_0)\mod2\ZZ\quad\text{for all}\quad\nu\in\mathcal{L}^*\;\;\text{with}\;\; c\nu\in\mathcal{L}.
\]
Then, for any $\la\in\mathcal{L}^*$, we have
%If $c\mathcal{L}^*$ is not contained in $\mathcal{L}$, then
\begin{align*}
\begin{split}
&\th_{\la+\mathcal{L}}\left(A\cdot(\tau,z,u)\right)
%&=(c\tau+d)^{r/2}v(A)\sum_{\la+\frac1c\mathcal{L}\in \mathcal{L}^*/\frac1c\mathcal{L}}e^{\pi\ii(cd|\la|^2+2bc(\la|\mu)+ab|\mu|^2)}\th_{a\mu+c\la+\mathcal{L}}(\tau, z, u)
%\end{split}\\
%\begin{split}
=(c\tau+d)^{r/2}v(A)\\
&\quad\times\sum_{\substack{\mu: \;c\mu+\mathcal{L}\in \mathcal{L}^*/\mathcal{L}\\c\mu\in c\mathcal{L}^*}}e^{\pi\ii(cd|\mu|^2+2bc(\mu|\la)+ab|\la|^2+2b(\la|\be_0)+2d(\mu|\be_0))}\th_{a\la+c\mu+\mathcal{L}}(\tau, z, u),
%e^{\pi\ii(\frac{d}{c}|\mu'|^2+2b(\mu'|\la)+ab|\la|^2)}\th_{a\la+\mu'+\mathcal{L}}(\tau, z, u),
\end{split}
\end{align*}
where $v(A)\in\CC$ depends only on $A$.
%In particular, if $c=0$, then the summation collapses to one term, corresponding to $\mu'=0$.

In the case when $c\mathcal{L}^*\subset\mathcal{L}$, which includes the case when $c=0$,  the summation collapses to one term, corresponding to $\mu=0$\emph{:}
\begin{align*}
\begin{split}
\th_{\la+\mathcal{L}}&\left(A\cdot(\tau,z,u)\right)
%&=(c\tau+d)^{r/2}v(A)\sum_{\la+\frac1c\mathcal{L}\in \mathcal{L}^*/\frac1c\mathcal{L}}e^{\pi\ii(cd|\la|^2+2bc(\la|\mu)+ab|\mu|^2)}\th_{a\mu+c\la+\mathcal{L}}(\tau, z, u)
%\end{split}\\
%\begin{split}
=(c\tau+d)^{r/2}v(A)e^{\pi\ii b (a|\la|^2+2(\la|\be_0))}\th_{a\la+\mathcal{L}}(\tau, z, u).
%e^{\pi\ii(\frac{d}{c}|\mu'|^2+2b(\mu'|\la)+ab|\la|^2)}\th_{a\la+\mu'+\mathcal{L}}(\tau, z, u).
\end{split}
\end{align*}
%Here $v(A)$ is given in terms of the Jacobi symbol and an eighth root of unity which depends on $r=\rank\mathcal{L}${\upshape{:}}
%\begin{equation}\label{vk}
%v(A)=
%\begin{cases}
%\displaystyle\left(\frac{(-1)^{r/2}|\mathcal{L}^*/\mathcal{L}|}{d}\right),& r\;\text{is even}\\
%\displaystyle\left(\frac{c|\mathcal{L}^*/\mathcal{L}|}{d}\right)e^{2\pi\ii(d-1)r/8}(c\tau+d)^{r/2}j(\tau)^{-r},& r\;\text{is odd}
%\end{cases},
%\end{equation}
%where $j(\tau)$ is a holomorphic function $j$ on the upper half complex plane into the complex numbers such that $j(\tau)^2=c\tau+d$. 
\end{theorem}

%Later we shall use Theorem \ref{KP} for the lattice $\mathcal{L}=\sqrt{p}\pi_0(Q)$, and show that this lattitice is indeed a positive-definite even lattice for any prime $p$. We will also later write $v_k=v(A)$, where $A$ will be a matrix that depends on a parameter $k$.

\begin{remark}\label{beta0}
In general, Corollary 3.9 in \cite{KP2} requires two vectors $\al_0, \be_0\in \mathcal{L}$ and a chosen scalar $t_0\in \RR$. However, in the case when the lattice is even and $c$ is an odd integer, each of the parameters $\al_0, \be_0, t_0$ can be set to zero and their result simplifies to the result shown in Theorem \ref{KP} (using that $c|\mu|^2\equiv c^2|\mu|^2\equiv|c\mu|^2\equiv0\mod2$ when $c$ is odd). In the case when $c$ is even, the parameter $\beta_0$ may be nonzero. 
For example, if $a=d=1, b=0$, and $c=2$, then $\beta_0=\frac{\al}{2}$ is nonzero for the $A_1$ root lattice generated by the vector $\al$ with $|\al|^2=2$.
%The condition is that $\beta_0\in Q$ such that
%\[
%ac(\al|\al)\equiv2(\al|\beta_0)\mod2\ZZ,\qquad \text{for all} \quad \al\in Q^* \;\;\text{ with } \;\;c\al\in Q.
%\]
\end{remark}

\begin{remark}
In Corollary 3.9 in \cite{KP2}, the summation is indexed by the set of $\mu\in\mathcal{L}^*, c\mu\hspace*{-0.05in}\mod\mathcal{L}$, and this set coincides with the index set of cosets used throughout this paper using the cosets $\mu'+\mathcal{L}\in \mathcal{L}^*/\mathcal{L}$ with $\mu'=c\mu\in c\mathcal{L}^*$. In particular, for $c=0$ both sums reduce to one term corresponding to $\mu=\mu'=0$.
%$\sum_{\substack{\mu\in\mathcal{L}^*\\c\mu\hspace*{-0.05in}\mod\mathcal{L}}}=\sum_{\substack{\mu'+\mathcal{L}\in \mathcal{L}^*/\mathcal{L}\\\mu'\in c\mathcal{L}^*}}$
\end{remark}

\begin{remark}
It is important to point out that in Theorem \ref{KP}, a choice of square root of $c\tau+d$ is made.  This means that Theorem \ref{KP} describes the action of the metaplectic group, which contains the modular group as a subgroup, due to the choice of square root. In this paper, we will not concern ourselves too much with this choice because the characters and trace functions later in Section \ref{charVQ} will be the ratio of a theta function divided by another function (see Theorem \ref{chiF}) whose transformation involves the same choice of square root (see Corollary \ref{chiFtran2} in the case $A=S$).
\end{remark}

%\begin{remark}
In general, $v(A)$ is a complex valued function such that 
\begin{equation}\label{vprops}
|v(A)|=|(\mathcal{L}+c\mathcal{L}^*)/\mathcal{L}|^{-\frac12},\qquad v(A^{-1})=\overline{v(A)}
\end{equation}
(cf.\ \cite[Proposition 3.8]{KP2}). We also calculate in general that 
\begin{equation}\label{vprops2}
v(-A)=\ii^rv(A).
\end{equation}
%In particular, if $c\mathcal{L}^*\subset\mathcal{L}$ implies that $|v(A)|=1$. 
%Evidently there is a choice involved in \eqref{vk} when the rank of $\mathcal{L}$ is odd. 
While the theorem proves the existence of the complex number $v(A)$, it does not provide what they are explicitly. In practice, it may be more convenient to write the matrix $A$ first in terms of the generators $S$ and $T$, then repeatedly use Theorem \ref{thtlaws}.
%\end{remark}

%\begin{remark}
By comparing the coefficients in Theorem \ref{thtlaws} with Theorem \ref{KP} when $A=S$ and when $A=T$, we can obtain the constants $v(S)$ and $v(T)$:
\begin{equation}\label{vSTremark}
v(S)=\displaystyle\frac{(-\ii)^{r/2}}{|\mathcal{L}^*/\mathcal{L}|^{1/2}},\qquad v(T)=1.
\end{equation}
It follows that $v(T^k)=1$ for all integers $k$ and $v(-S)=\displaystyle\frac{\ii^{r/2}}{|\mathcal{L}^*/\mathcal{L}|^{1/2}}$, using  \eqref{vprops2}.
%\end{remark}
In addition to these facts, the following proposition will be useful later.

\begin{proposition}\label{vTSprop}
Let\/ $\mathcal{L}$, $A=\begin{pmatrix}
a&b\\c&d
\end{pmatrix}\in \mathrm{SL}(2,\ZZ)$, and $v(A)$ be as in Theorem \ref{KP}.
Then the following hold for the generators $S$ and $T:$
\begin{align}
v(AT^k)&=v(T^kA)=v(A)\qquad\text{for all}\quad k\in\ZZ,\label{vTA}\\
v(AS)&=v(A)v(S)\sum_{\mu+\mathcal{L}\in\mathcal{L}^*/\mathcal{L}}e^{\pi\ii (cd|\mu|^2+2d(\mu|\be_0))}.\label{vSA}
\end{align}
\end{proposition}

\begin{proof}
To prove the first identity, we use Theorem \ref{KP} to calculate $\theta_\mathcal{L}((AT^k)\cdot\tau)$, where $k$ in an integer, in two different ways. Since 
\[
AT^k=\begin{pmatrix}
a&ak+b\\c&ck+d
\end{pmatrix},
\]
we have from setting $\la, z, u$ equal to zero in Theorem \ref{KP} that 
\begin{align*}
\begin{split}
\th_{\mathcal{L}}\left((AT^k)\cdot\tau\right)&=(c\tau+ck+d)^{r/2}v(AT^k)\sum_{\substack{\mu'+\mathcal{L}\in \mathcal{L}^*/\mathcal{L}\\\mu'\in c\mathcal{L}^*}}e^{\pi\ii c(ck+d)|\mu|^2}\th_{c\mu+\mathcal{L}}(\tau),
\end{split}
\end{align*}
where $\mu'=c\mu$. On the other hand, we compute this transformation using two steps and \eqref{thtlaws2}:
\begin{align*}
\begin{split}
\th_{\mathcal{L}}\left(A\cdot(T^k\cdot\tau)\right)&=(c(\tau+k)+d)^{r/2}v(A)\sum_{\substack{\mu'+\mathcal{L}\in \mathcal{L}^*/\mathcal{L}\\\mu'\in c\mathcal{L}^*}}e^{\pi\ii cd|\mu|^2}\th_{c\mu+\mathcal{L}}(T^k\cdot\tau)\\
%&=(c\tau+ck+d)^{r/2}v(A)\sum_{\substack{\mu'+\mathcal{L}\in \mathcal{L}^*/\mathcal{L}\\\mu'\in c\mathcal{L}^*}}e^{\pi\ii cd|\mu|^2}\th_{c\mu+\mathcal{L}}(T^k\cdot\tau)\\
&=(c\tau+ck+d)^{r/2}v(A)\sum_{\substack{\mu'+\mathcal{L}\in \mathcal{L}^*/\mathcal{L}\\\mu'\in c\mathcal{L}^*}}e^{\pi\ii c(d+ck)|\mu|^2}\th_{c\mu+\mathcal{L}}(\tau).
\end{split}
\end{align*}
By comparing both equations and noting that $(AT^k)\cdot\tau=A\cdot(T^k\cdot\tau)$, we arrive at the relation $v(AT^k)=v(A)$. In a similar way, we can show that $v(T^kA)=v(A)$.

We now compute in a similar way $\theta_\mathcal{L}((AS)\cdot\tau)$ in two different ways. First we assume that $c, d\neq0$.
Since 
\[
AS=\begin{pmatrix}
b&-a\\d&-c
\end{pmatrix},
\]
we have from setting $\la, z, u$ equal to zero in Theorem \ref{KP} that 
\begin{align}\label{kp1}
\begin{split}
\th_{\mathcal{L}}\left((AS)\cdot\tau\right)&=\left(d\tau-c\right)^{r/2}v(AS)\sum_{\substack{\mu'+\mathcal{L}\in \mathcal{L}^*/\mathcal{L}\\\mu'\in d\mathcal{L}^*}}e^{-\pi\ii (cd|\mu|^2+2c(\mu|\be_0))}\th_{d\mu+\mathcal{L}}(\tau),
\end{split}
\end{align}
where $\mu'=d\mu$. On the other hand, we compute this transformation using two steps and \eqref{thtlaws1}:
\begin{equation}\label{kp2}
\begin{split}
\th_{\mathcal{L}}&\left(A\cdot(S\cdot\tau)\right)\\
&=\left(-\frac{c}{\tau}+d\right)^{r/2}v(A)\sum_{\substack{\nu'+\mathcal{L}\in \mathcal{L}^*/\mathcal{L}\\\nu'\in c\mathcal{L}^*}}e^{\pi\ii (cd|\nu|^2+2d(\nu|\be_0))}\th_{c\nu+\mathcal{L}}\left(-\frac{1}{\tau}\right)\\
&=(d\tau-c)^{r/2}\tau^{r/2}v(A)v(S)\\
&\times\sum_{\substack{\nu'+\mathcal{L}\in \mathcal{L}^*/\mathcal{L}\\\nu'\in c\mathcal{L}^*}} \; \sum_{\al+\mathcal{L}\in\mathcal{L}^*/\mathcal{L}}e^{\pi\ii (cd|\nu|^2+2d(\nu|\be_0)-2c(\nu|\al))}\th_{\al+\mathcal{L}}(\tau)\,,
%&=(d\tau-c)^{r/2}v(A)v(S)\sum_{\al+\mathcal{L}\in\mathcal{L}^*/\mathcal{L}}\sum_{\substack{\nu'+\mathcal{L}\in \mathcal{L}^*/\mathcal{L}\\\nu'\in c\mathcal{L}^*}}e^{\pi\ii (cd|\nu|^2+2d(\nu|\be_0)-2c(\nu|\al))}\th_{\al+\mathcal{L}}(\tau).
\end{split}
\end{equation}
where $\nu'=c\nu$.
Since $(AS)\cdot\tau=A\cdot(S\cdot\tau)$, it follows from Theorem \ref{KP} that the terms in the sum in \eqref{kp1} and \eqref{kp2} must have equal coefficients for each corresponding theta function. In particular, if we set $\mu'=0$ in \eqref{kp1} and $\al=0$ in \eqref{kp2},  
we obtain \eqref{vSA}.

Next we consider the cases $c=0$ or $d=0$. 
When $c=0$, the summation over $\nu'$ in \eqref{kp2} turns into one term corresponding to $\nu'=0$. Hence a comparison of \eqref{kp1} and \eqref{kp2} in this case yields that $v(AS)=v(A)v(S)$. In fact, when $c=0$ and $\det A=1$, the matrix is actually of the form $A=\pm T^k$, for some integer $k$, and this case reduces to \eqref{vTA}.
When $d=0$, the summation over $\mu'$ in \eqref{kp1} turns into one term corresponding to $\mu'=0$. A comparison of \eqref{kp1} and \eqref{kp2} then yields \eqref{vSA}.
%Finally if $c, d\neq0$, the summations above can be replaced with a summation over the cosets $\mathcal{L}^*/\mathcal{L}$. The relation \eqref{vSA} then follows by comparison of the results for $\th_{\mathcal{L}}\left(A\cdot(S\cdot\tau)\right)$ and $\th_{\mathcal{L}}\left((AS)\cdot\tau\right)$.
\end{proof}

We also recall the Dedekind $\eta$-function
\begin{equation}\label{etadef}
\eta(\tau)=e^{{\pi\ii\tau}/{12}}\prod_{n=1}^\infty (1-q^n),\qquad q=e^{2\pi\ii\tau}, \quad\mathrm{Im}\,\tau>0.
\end{equation}
The transformation laws for $\eta(\tau)$ are well known:
%\begin{proposition}\label{eta}
%The transformation laws for the Dedekind $\eta$-function are
\begin{equation}\label{eta}
\eta(\tau+1)=e^{\pi\ii/12}\eta(\tau),\qquad
\eta\Bigl(-\frac{1}{\tau}\Bigr)=(-\ii\tau)^{1/2}\eta(\tau).
\end{equation}
%\end{proposition}
More generally, it is well known that for any $A=\begin{pmatrix}
a&b\\c&d
\end{pmatrix}\in \mathrm{SL}(2,\ZZ)$,
\begin{align}\label{etaA}
\eta\left(A\cdot\tau\right)=\epsilon(A)(c\tau+d)^{1/2}\eta(\tau),
\end{align}
where $\epsilon(A)$ is a 24-th root of unity that depends on $A$.

Consider the one-dimensional lattice $\ZZ\al$, with $|\al|^2=m$. Set $z=\frac{\zeta}{2}\al$ so that $(z|\al)=\frac{\zeta}{2}m$. Then the theta function \eqref{th} takes the form
\begin{equation}\label{1Dth}
\theta_{\frac lm\al+\ZZ\al}\Bigl(\tau,\frac{\zeta}{2}\al,u\Bigr)=e^{2\pi\ii u}\sum_{n\in\frac lm+\ZZ}e^{\pi\ii nm\zeta}q^{n^2m/2},
\end{equation}
where $l\in\ZZ/m\ZZ$. For later use we set
\begin{equation}\label{K}
K_l(\tau,\zeta;m)=\frac1{\eta(\tau)}\theta_{\frac lm\al+\ZZ\al} \Bigl(\tau,\frac{\zeta}{2}\al,0\Bigr).
\end{equation}
The following is then immediate using Theorem \ref{thtlaws} and \eqref{eta}:
\begin{proposition}\label{Ktlaws}
The transformation laws for $K_l(\tau,\zeta;m)$ are
\begin{align}
K_l\Bigl(-\frac{1}{\tau},\frac{\zeta}{\tau};m\Bigr)&=\frac{1}{\sqrt{m}}e^{\pi\ii m\zeta^2/4\tau}\sum_{l'\in\ZZ/m\ZZ}e^{-2\pi\ii ll'/m}K_{l'}(\tau,\zeta;m),\label{Ktlaws1}\\
K_l(\tau+1, \zeta;m)&=e^{\pi\ii\left(\frac{l^2}{m}-\frac{1}{12}\right)}K_l(\tau, \zeta;m).\label{Ktlaws2}
\end{align}
\end{proposition}
%In the case when the lattice $\mathcal{L}$ has rank one, it will be more explicit in practice to use these transformation laws of the quotient $K_l(\tau,\zeta;m)$ instead of .
Note that since the function $K_l$ is a ratio of a theta function and the Dedekind $\eta$-function, the choice of square root $(-\ii\tau)^{1/2}$ cancels. % using \eqref{eta} and Theorem \cite{KP} with $A=S$. 

%%%%%%%%%%%%%%%%%%%%%%%%%%%%%%%%%%%%%
\section{Twisted Representations of Lattice Vertex Algebras}\label{twlat}
%%%%%%%%%%%%%%%%%%%%%%%%%%%%%%%%%%%%%
In this section, we review the construction of irreducible $\si$-twisted modules over a lattice vertex algebra associated to a positive-definite even lattice (see \cite{KP,Le,D2,BK}). 
We do so in more details than available in the literature, and in particular, we calculate the action of the twisted energy operator $L_0^{\tw}$.

%%%%%%%%%%%%%%%%%%%%%%%%%%%%
\subsection{Lattice vertex algebras}\label{lat}
%%%%%%%%%%%%%%%%%%%%%%%%%%%%

Let $Q$ be a positive-definite even integral lattice with the bilinear form $(\cdot|\cdot)$.
%i.e., a free abelian group of finite rank equipped with a symmetric
%bilinear form $(\cdot|\cdot) \colon Q\times Q\to\ZZ$ such that $|\al|^2=(\al|\al)$ is a positive even integer for all nonzero $\al\in Q$.
We denote by   
$\lieh = \CC\otimes_\ZZ Q$ the corresponding complex 
vector space considered as an abelian Lie algebra, and extend the bilinear form $(\cdot|\cdot)$ to it by linearity.

The \emph{Heisenberg algebra}
$\hat\lieh = \lieh[t,t^{-1}] \oplus \CC K$
is the Lie algebra with brackets
\begin{equation}\label{heis1}
[a_m,b_n] = m \delta_{m,-n} (a|b) K \,, \qquad
a_m=at^m \,,
\end{equation}
where $K$ is central, i.e., $[a_m, K]=0$ for all $m\in\ZZ$.
Its irreducible highest-weight representation 
\begin{equation}\label{F}
\F = \Ind^{\hat\lieh}_{\lieh[t]\oplus\CC K} \CC \cong S(\lieh[t^{-1}]t^{-1}),
\end{equation}
where $\lieh[t]$ acts trivially on $\CC$ and $K=1$, is known as the (bosonic) \emph{Fock space}.

Following \cite{FK, B}, we consider a
$2$-cocycle $\ep\colon Q \times Q \to \{\pm1\}$ 
such that
\begin{equation}\label{lat2}
\ep(\al,\al) = (-1)^{|\al|^2/2}  \,,
\qquad \al\in Q \,,
\end{equation}
and the associative algebra $\CC_\ep[Q]$ with basis
$\{ e^\al \}_{\al\in Q}$ and multiplication
\begin{equation}\label{lat1}
%\CC_\ep[Q]=\Span_\CC \{ e^\al | \al\in Q \} \,, \qquad
e^\al e^\be = \ep(\al,\be) e^{\al+\be} \,.
\end{equation}
Such a $2$-cocycle $\ep$ is unique up to equivalence and can be chosen
to be bimultiplicative, which we shall assume. In this case, by \eqref{lat2}, we also have
\begin{equation}\label{lat22}
\ep(\al,\be) \ep(\be,\al) = (-1)^{(\al|\be)}  \,, 
\qquad \al,\be\in Q \,.
\end{equation}

The space of states of the \emph{lattice vertex algebra} %\cite{B, FLM, K2, FB, LL}
associated to $Q$ is defined as $V_Q=\F\otimes\CC_\ep[Q]$,
where the vacuum vector is $\vac=1\otimes e^0$.
We extend the action of the Heisenberg algebra on $\F$ to $V_Q$ by
\begin{equation}\label{lat3}
a_n e^\be = \delta_{n,0} (a|\be) e^\be \,,\qquad
a\in\lieh, \; n\in\ZZ \,.
\end{equation}
%In other words, as a module over $\hat\lieh$, it is a direct sum of highest-weight representations with highest weights $\al\in Q$.
%
The state-field correspondence on $V_Q$ is uniquely determined by the
generating fields (recall that $z^{\al_0} e^\be = z^{(\al|\be)} e^\be$):
\begin{align}\label{lat4}
Y(a_{(-1)}\vac,z) &= \sum_{n\in\ZZ} a_{n} \, z^{-n-1} \,, \qquad a\in\lieh \,,
\\ \label{lat5}
Y(e^\al,z) &= e^\al z^{\al_0} 
\exp\Bigl( \sum_{n<0} \al_{n} \frac{z^{-n}}{-n} \Bigr) 
\exp\Bigl( \sum_{n>0} \al_{n} \frac{z^{-n}}{-n} \Bigr) \,.
\end{align}

Notice that $\F\subset V_Q$ is a vertex subalgebra, which we call the \emph{Heisenberg vertex algebra}.
The map $\lieh\to\F$ given by $a\mapsto a_{-1}\vac$ is injective.
{}From now on, we will slightly abuse the notation and identify
$a\in\lieh$ with $a_{-1}\vac \in\F$; then $a_{(n)}=a_n$ for all $n\in\ZZ$. 
%When the lattice $Q$ is non-integral, the above formulas still make sense but
%involve non-integral powers of $z$. In this case, $V_Q$ is a 
%\emph{generalized} vertex algebra \cite{FFR, DL1, BK2}. 
Let $\{a^i\}$ and $\{b^i\}$, $i=1,\ldots,r$, be dual bases of $\mathfrak{h}$ with respect to the bilinear form, so that
\begin{equation*}%\label{aibj}
(a^i|b^j) = \de_{i,j} \,.
\end{equation*}
Then lattice vertex algebras are conformal with central charge $r$ and Virasoro vector given by
\begin{equation}\label{om}
L=\frac12\sum_{i=1}^ra^i_{(-1)}b^i\in V_Q.
\end{equation}

We will denote by $Q^*$ the \emph{dual lattice} to $Q$, which is defined by
\begin{equation}\label{duallat}
Q^* = \{\la\in\lieh \,|\, (\la|\al) \in\ZZ \;\;\forall\, \al\in Q \} \,.
\end{equation}
This is a free abelian group of the same rank as $Q$; however, $Q^*$ is not integral in general. Notice that $Q\subset Q^*$ because $Q$ is integral.
It is well known (see Theorem 2.7 in \cite{D1}) that the irreducible $V_Q$-modules are classified by the finite abelian group $Q^*/Q$, and are given explicitly by
\begin{equation}\label{VFW}
V_{\la+Q}=\F\otimes  \CC_\ep[Q]e^\la \,, \qquad \la\in Q^* \,.
\end{equation}
%where $\F$ is the Fock space (cf.\ \eqref{F}), $\la\in Q^*$, and $\CC_\ep[Q]$ is the group algebra of $Q$ (cf.\ \eqref{lat2}--\eqref{lat22}). 
The contragradient module of $V_{\la+Q}$ is 
$$V'_{\la+Q}=V_{-\la+Q} \,,$$ 
so that the duality corresponds to sending $\la$ to $-\la$. In particular, we see that $V_Q$ is self-dual.

\begin{theorem}[\cite{DLM2}]\label{DLM2}
For any positive-definite even lattice $Q$, the lattice vertex algebra $V_Q$ is regular.
\end{theorem}

%%%%%%%%%%%%%%%%%%%%%%%%%%%%%%%%%
\subsection{Twisted Heisenberg algebra}\label{twheis}
%%%%%%%%%%%%%%%%%%%%%%%%%%%%%%%%%

Every automorphism $\si$ of $\lieh$ preserving the bilinear form induces automorphisms
of $\hat\lieh$ and $\F$, which will be denoted again as $\si$, by setting $\si(a t^m) = \si(a) t^m$,
$\si(K)=K$ and $\si(\vac)=\vac$.
As before, we assume that $\si$ has a finite order $N$.

The \emph{$\si$-twisted Heisenberg algebra} 
$\hat\lieh_\si$ is spanned over $\CC$ by $K$ and
the elements $a_m = at^m$, where $m\in\frac1N\ZZ$ is such that $\si a = e^{-2\pi\ii m} a$.
This is a Lie algebra with bracket  (cf.\ \eqref{heis1}):
\begin{equation*}%\label{twheis5}
[a_m,b_n] = m \delta_{m,-n} (a|b) K \,, \qquad a,b\in\lieh \,, \;\; m,n\in \frac1N\ZZ \,.
\end{equation*}
Let $\hat\lieh_\si^\ge$ (respectively, $\hat\lieh_\si^<$) be the abelian subalgebra of 
$\hat\lieh_\si$ spanned by all elements $a_m$ with $m\geq0$ 
(respectively, $m<0$). 
%Elements of $\hat\lieh_\si^\ge$ are called \emph{annihilation operators}, 
%while elements of $\hat\lieh_\si^<$ \emph{creation operators}.

The \emph{$\si$-twisted Fock space} is defined as
\begin{equation}\label{twheis2}
\F_\si = \Ind^{\hat\lieh_\si}_{\hat\lieh_\si^\ge \oplus\CC K} \CC \cong S(\hat\lieh_\si^<) \,,
\end{equation}
where $\hat\lieh_\si^\ge$ acts on $\CC$ trivially and $K$ acts as the identity operator.
Then $\F_\si$ is an irreducible highest-weight representation of $\hat\lieh_\si$, and
has the structure of a $\si$-twisted representation of the Heisenberg vertex algebra $\F$
(see e.g.\ \cite{FLM,KRR}). This structure can be described as follows. 
We let $Y(\vac,z)$ be the identity operator, let
\begin{equation}\label{twheis3}
Y(a,z) = \sum_{n\in m+\ZZ} a_{n} \, z^{-n-1} \,, 
%\qquad a\in\lieh \,, \;\; \si a = e^{-2\pi\ii m} a \,, \;\;  m\in\frac1N\ZZ \,,
\end{equation}
for $a\in\lieh$ and $m\in\frac1N\ZZ$ such that $\si a = e^{-2\pi\ii m} a$,
and we extend $Y$ to all $a\in\lieh$ by linearity.
%These satisfy the locality property \eqref{locpr} because
%\begin{equation}\label{twheis4}
%[Y(a,z_1), Y(b,z_2)] = (a|b) \, \d_{z_2} \bigl( z_1^{-p} z_2^{p} \de(z_1,z_2) \bigr) \,.
%\end{equation}
The action of $Y$ on other elements of $\F$ is then determined by applying several times the product
formula \eqref{locpr3}.
More explicitly, $\F$ is spanned by elements of the form $a^1_{m_1}\cdots a^k_{m_k} \vac$
where $a^j\in\lieh$, $m_j\in\ZZ$, and we have:
\begin{align*}
Y& ( a^1_{m_1}\cdots a^k_{m_k} \vac, z ) v
\\
&= \prod_{j=1}^k \d_{z_j}^{(N-1-m_j)} 
\Bigl( \prod_{j=1}^k (z_{j}-z)^N \; Y(a^1,z_1) \cdots Y(a^k,z_k) v \Bigr)\Big|_{z_1=\cdots=z_k=z}
\end{align*}
for all $v\in\F_\si$ and sufficiently large $N$.
In the above formula, we use the divided-power notation $\d^{(n)} = \d^n / n!$.
%%%%%%%%%%%%%%%%%%%%%%%%%%%%%%%%%%%%%
\subsection{The groups $G_\si$ and $G_\si^\perp$}\label{gsigsiperp}
%%%%%%%%%%%%%%%%%%%%%%%%%%%%%%%%%%%%%
%Now let $Q$ be a positive-definite \emph{even} lattice, which means that
%$|\al|^2$ is a positive even integer for every nonzero $\al\in Q$.
Let $\si$ be an isometry of the even lattice $Q$ of finite order $N$, so that
\begin{equation}\label{twlat1a}
(\si\al | \si\be) = (\al | \be) \,, \qquad \al,\be\in Q \,.
\end{equation}
The uniqueness of the cocycle $\ep$, \eqref{lat22} and \eqref{twlat1a} imply that
there exists a function $\eta\colon Q\to\{\pm1\}$ such that
\begin{equation}\label{twlat3}
\eta(\al+\be) \ep(\si\al,\si\be) = \eta(\al)\eta(\be) \ep(\al,\be)
\end{equation}
for all $\al,\be\in Q$
(the distinction from the Dedekind $\eta$-function should be clear from the context).
If $L$ is a sublattice of $Q$ with the property $\ep(\si\al,\si\be) = \ep(\al,\be)$
for $\al,\be\in L$, then $\eta$ can be chosen to satisfy $\eta(\al)=1$ for all\/ $\al\in L$
\cite[Lemma 2.3]{BE}.
In particular, we can choose $\eta$ so that
\begin{equation}\label{twlat2}
\eta(\al)=1 \,,\qquad \al\in Q\cap\lieh_0 \,,
\end{equation}
where 
\begin{equation}\label{twlat-h0}
\lieh_0 = \{ h\in\lieh \,|\, \si h=h\}
\end{equation}
is the subspace of $\lieh$ consisting of vectors fixed under $\si$.

There is a natural lifting of $\si$ to an automorphism of 
the lattice vertex algebra $V_Q$ by setting
\begin{equation}\label{twlat4}
\si(a_n)=\si(a)_n \,, \quad \si(e^\al)=\eta(\al) e^{\si\al} \,,
\qquad a\in\lieh \,, \; \al\in Q \,.
\end{equation}
Note that the order of $\si$ is either $N$ or $2N$ when acting on $V_Q$.

We recall the following useful fact concerning dual spaces taken in a subspace of $\mathfrak{h}$.
\begin{lemma}[\cite{BK}, Lemma 4.6]\label{tdual}
Let\/ $\mathfrak{t}$ be a subspace of\/ $\mathfrak{h}$ on which the bilinear form $(\cdot|\cdot)$
is nondegenerate. Denote by $\pi_\mathfrak{t}\colon \mathfrak{h}\to\mathfrak{t}$ the orthogonal projection of\/ $\mathfrak{h}$ onto $\mathfrak{t}$. Then for any lattice $L\subset\mathfrak{h}$, we have 
\begin{equation}%\label{twlat8}
\pi_\mathfrak{t}(L^*)=(L\cap\mathfrak{t})^{*_\mathfrak{t}},
\end{equation}
where $*_\mathfrak{t}$ denotes taking dual in $\mathfrak{t}$.  Equivalently
\begin{equation}
(\pi_\mathfrak{t}(L))^{*_\mathfrak{t}}=L^*\cap\mathfrak{t}.
\end{equation}
\end{lemma}

Next, we introduce the group $G = \CC^\times \times \exp\lieh_0 \times Q$ consisting of elements $c \, e^h U_\al$ 
($c\in\CC^\times$, $h\in\lieh_0$, $\al\in Q$) with multiplication
\begin{align}
\label{tig1}
e^h e^{h'} &= e^{h+h'} \,,
\\
\label{tig2}
e^h U_\al e^{-h} &= e^{(h|\al)} U_\al \,,
\\
\label{tig3}
U_\al U_\be &= \ep(\al,\be) B_{\al,\be}^{-1} \, U_{\al+\be} \,,
\end{align}
where
\begin{equation}\label{Balbe}
B_{\al,\be} = 
N^{ -(\al|\be) } \prod_{k=1}^{N-1} \bigl(1 - e^{2\pi\ii k/N} \bigr)^{ (\si^k\al|\be) } \,.
\end{equation}
We set the following notation for the eigenspaces of $\si$:
\begin{align}
\mathfrak{h}_{j/N}&=\bigl\{h\in\mathfrak{h}\,\big|\,\si h=e^{-2\pi\ii j/N}h\bigr\} \,, \qquad 0\leq j<N\,,\label{j/N}\\
\mathfrak{h}_\perp&=(\mathfrak{h}_0)^\perp=\bigoplus_{j=1}^{N-1}\mathfrak{h}_{j/N}\label{hperp} \,.
\end{align}

From \eqref{tig3}, we get the commutator
\begin{align}\label{Calbe}
C_{\al,\be}&=U_\al U_\be U_\al^{-1}U_\be^{-1}=e^{\pi\ii(\pi_0\al|\be)}e^{2\pi\ii(\al_*|\be)},
\end{align}
for $\al=\pi_0\al+(1-\si)\al_*,$ where $\al_*\in\mathfrak{h}_\perp$ and 
$\pi_0$ is the orthogonal projection of $\lieh$ onto $\lieh_0$ (see \cite[(4.44)]{BK}).
We will also use the notation $\pi_\perp=1-\pi_0$ for the orthogonal projection of $\lieh$ onto $\lieh_\perp$.

By \cite[Lemma 4.4]{BK}, the center $Z(G)$ of $G$ consists of all elements of the form 
\begin{equation}\label{ZG}
c \, e^{2\pi\ii \,\pi_0(\la)}U_{(1-\si)\la} \,,
\end{equation} 
where $c\in\CC^\times$ and $\la\in Q^*$ is such that $(1-\si)\la\in Q$. 
For $\al\in Q$, we set (cf.\ \cite[(4.46)]{BK}):
\begin{equation}\label{Cal}
C_\al = \eta(\al) U_{\si\al}^{-1} U_\al e^{ 2\pi\ii(b_\al+\pi_0\al) } \,, 
\end{equation}
where 
\begin{equation}\label{bal}
b_\al=\frac12 \bigl( |\pi_0\al|^2-|\al|^2 \bigr).
\end{equation}
Applying \eqref{tig3}, we see that $C_\al$ has the form \eqref{ZG} for $\la=\al$ and a suitable scalar $c$; hence
$C_\al\in Z(G)$ for all $\al\in Q$.

The next lemma is contained in the proof of \cite[Proposition 5.5]{BS}, but is provided here for completeness.

\begin{lemma}\label{CalCbe}
We have\/
$C_\al C_\be = C_{\al+\be}$ 
for all\/ $\al,\be\in Q$.
\end{lemma}
\begin{proof}
Using that $C_\be\in Z(G)$, we find:
\begin{align*}
C_\al C_\be &=\eta(\al) U_{\si\al}^{-1} U_\al e^{ 2\pi\ii(b_\al+\pi_0\al) } C_\be \\
&=\eta(\al) U_{\si\al}^{-1} C_\be U_\al e^{ 2\pi\ii(b_\al+\pi_0\al) } \\
&=\eta(\al) \eta(\be) e^{ 2\pi\ii(b_\al+b_\be)} U_{\si\al}^{-1} U_{\si\be}^{-1}  U_\be e^{ 2\pi\ii \, \pi_0(\be) } U_\al e^{ 2\pi\ii \, \pi_0(\al) } \\
&=\eta(\al) \eta(\be) e^{ 2\pi\ii(b_\al+b_\be)} U_{\si\al}^{-1} U_{\si\be}^{-1}  U_\be U_\al e^{ 2\pi\ii (\pi_0\be|\al) } e^{ 2\pi\ii \, \pi_0(\al+\be) } \,,
\end{align*}
where in the last line we used \eqref{tig2}.
Then applying \eqref{tig3} and \eqref{twlat3}, we get:
\begin{align*}
U_{\si\al}^{-1} U_{\si\be}^{-1}  U_\be U_\al & = \frac{B_{\si\be,\si\al}}{\ep(\si\be,\si\al)} \frac{\ep(\be,\al)}{B_{\be,\al}} U_{\si(\al+\be)}^{-1} U_{\al+\be} \\
&= \frac{\eta(\al+\be)}{\eta(\al) \eta(\be)} U_{\si(\al+\be)}^{-1} U_{\al+\be} \,,
\end{align*}
because $B_{\si\be,\si\al} = B_{\be,\al}$. The rest of the proof follows from \eqref{bal} and the fact that $(\pi_0\be|\al) = (\pi_0\be|\pi_0\al)$.
\end{proof}

Now we define 
\begin{equation}\label{Nsi}
N_\si=\{C_\al\,|\,\al\in Q\}.
\end{equation}
It follows from \leref{CalCbe} that $N_\si$ is a subgroup of $G$. Moreover, $N_\si \subset Z(G)$.
Let $G_\si$ be the quotient group $G/N_\si$.
%\begin{remark}
%The $\si$-twisted $V_Q$-modules are in correspondence with irreducible restricted $G_{\si}$-modules (see \cite[Propositions 4.2, 4.3, 4.4]{BK}).
%\end{remark}
%In particular, $N_\si$ is a subgroup of $Z(G)$.
Note that $\exp {2\pi\ii(Q\cap\mathfrak{h}_0)}$ is a subgroup of $N_\si$ so that these elements become trivial in $G_\si$.
Also consider the subgroup 
\begin{equation}
T_\si=\exp \lieh_0 \big/ \exp 2\pi\ii(Q\cap\mathfrak{h}_0) \subset G_\si .
\end{equation}
Then $\CC^\times\times T_\si$ is the connected component of the identity in $G_\si$.

Let $G^\perp\subset G$ be the subgroup %of $G$ consisting of elements $c\,U_\al$, where $c\in\CC^\times$ and $\al\in Q\cap\mathfrak{h}_\perp$. %Note the centralizer of $\mathfrak{h}_\si$ in $G$ is the subgroup $\exp\mathfrak{h}_0\times G^\perp$.
\begin{equation}
G^\perp=\{c\,U_\al\,|\, c\in\CC^\times, \; \al\in Q\cap\mathfrak{h}_\perp\},
\end{equation}
and denote by $G_\si^\perp$ the image of $G^\perp$ in $G_\si$. This image is described as the quotient group $G^\perp/N_\si^\perp$, where 
\begin{equation}\label{Nsiperp}
N_\si^\perp=N_\si\cap G^\perp=\{C_\al\,|\,\al\in Q\cap\mathfrak{h}_\perp\}.%\{\eta(\al) U_{\si\al}^{-1} U_\al \,|\,\al\in Q\cap\mathfrak{h}_\perp\}.
\end{equation}
Then it is easy to see that $N_\si^\perp$ is a subgroup of $Z(G^\perp)$ and the centralizer of $\hat{\mathfrak{h}}_\si$ in $G_\si$ is equal to $T_\si\times G_\si^\perp$. The group $G_\si^\perp$ can also be viewed as a central extension by $\CC^\times$ of the finite abelian group 
\begin{align}\label{Gsiperp}
\frac{Q\cap\mathfrak{h}_\perp} {(1-\si)(Q\cap\mathfrak{h}_\perp)} \,.
\end{align}
%Note that $C_\al=\eta(\al) U_{\si\al}^{-1} U_\al\in Z(G^\perp)$ for each $\al\in Q\cap\mathfrak{h}_\perp$ since $U_{(1-\si)\al}\in Z(G^\perp)$. This then implies that $N_\si^\perp$ is a subgroup of $Z(G^\perp)$. %We also have $N_\si^\perp<N_\si<Z(G)$.

We will need the following description of the centers.

\begin{lemma}[\cite{BK}, Lemma 4.5]
The centers of the groups $G_\si$, $G^\perp$, and $G_\si^\perp$ are given by$:$
\begin{align}
Z(G_\si)&\cong Z(G)/N_\si \cong \CC^\times \times (Q^*/Q)^\si, \label{ZGsi}\\
Z(G^\perp)&=\{c \, U_\al\,|\,c\in\CC^\times, \; \al\in Q \cap (1-\si)Q^*\},\label{ZGperp}\\
\label{ZGsiperp}
Z(G_\si^\perp)&\cong Z(G^\perp)/N_\si^\perp \cong \CC^\times \times \frac{Q \cap (1-\si)Q^*} {{(1-\si)(Q\cap\mathfrak{h}_\perp)}} \,.
\end{align}
\end{lemma}

Note that in \eqref{ZGsi}, $(Q^*/Q)^\si$ is the subgroup of $Q^*/Q$ consisting of cosets $\la+Q$, where $\la\in Q^*$, such that $\si(\la+Q)=\la+Q$, i.e., $(1-\si)\la\in Q$.

%%%%%%%%%%%%%%%%%%%%%%%%%%%%%%%%%%%%%
\subsection{The $G_\si$-modules $W(\mu,\ze)$} %{Representations of $G_\si$ and $G_\si^\perp$}
%%%%%%%%%%%%%%%%%%%%%%%%%%%%%%%%%%%%%
We continue to use the notation of the previous subsection.
Let $\zeta\colon Z(G^\perp_\si)\rightarrow\CC^\times$ be a central character of $G^\perp_\si$ such that $\zeta(c)=c$  for every $c\in\CC^\times$. 
Then $\zeta$ can also be viewed as a central character of $G^\perp$ such that $\zeta(c)=c$ and $\zeta(C_\al)=1$ for every $c\in\CC^\times$ and $\al\in Q\cap\mathfrak{h}_\perp$.
Let $\Om(\ze)$ be the unique (up to isomorphism) finite-dimensional irreducible $G_\si^\perp$-module corresponding to the central character $\zeta$.
Recall that $\Om(\ze)$ can be constructed as follows (see e.g.\ \cite{FLM}). Pick a maximal abelain subgroup $A_\si^\perp$ of $G_\si^\perp$, and extend $\zeta$ to $A_\si^\perp$; then $\Om(\ze)$ is the induced module
\begin{equation}\label{A1}
\Om(\ze)=\Ind_{A_\si^\perp}^{G_\si^\perp}\CC_\zeta,
\end{equation}
where $\CC_\zeta$ is the $1$-dimensional $A_\si^\perp$-module $\CC$ with character $\zeta$.

By Lemma \ref{tdual}, we have $\pi_0(Q^*)=(Q\cap\mathfrak{h}_0)^{*_0}$, where $*_0$ represents taking dual in $\mathfrak{h}_0$.
For $\mu\in\pi_0(Q^*)$, consider the $1$-dimensional $T_\si$-module $\CC_\mu$, with the action of $T_\si$ given by 
\begin{equation}
e^{2\pi\ii h}\mapsto e^{2\pi\ii (h|\mu)} \,, \qquad h\in\lieh \mod Q\cap\mathfrak{h}_0.
\end{equation}
Notice that this is independent of the representative of $h$ modulo $Q\cap\mathfrak{h}_0$ because $(\al|\mu)\in\ZZ$ for $\al\in Q\cap\mathfrak{h}_0$.
%Since $\mathfrak{h}_0$ is the Lie algebra of the Lie group $T_\si$, 
We also have an action of $\mathfrak{h}_0$ on $\CC_\mu$ given by $h\mapsto(h|\mu)$ for $h\in\mathfrak{h}_0$. 

Now consider the $T_\si\times G_\si^\perp$-module $\Om(\mu,\zeta) = \CC_\mu \otimes \Om(\ze)$, and induce it to a module for $G_\si$:
\begin{equation}\label{W}
W(\mu,\zeta)=\Ind_{T_\si\times G_\si^\perp}^{G_\si}\Om(\mu,\zeta).
\end{equation}
Alternatively, by \eqref{A1}, we can write
\begin{equation}\label{A}
W(\mu,\zeta)=\Ind_{T_\si\times A_\si^\perp}^{G_\si}\CC_{\mu,\zeta} \,,
\end{equation}
where $\CC_{\mu,\zeta}$ is the $1$-dimensional $T_\si\times A_\si^\perp$-module $\CC$ with action %given by
\begin{equation}\label{Aaction}
(e^{2\pi\ii h},c\,U_\al) 1_{\mu,\ze} = c\,e^{2\pi\ii(h|\mu)}\,\zeta(U_\al) 1_{\mu,\ze}
\end{equation}
for $h\in\mathfrak{h}_0$, $c\in\CC^\times$, and $\al\in Q\cap\mathfrak{h}_\perp$ such that $U_\al\in A_\si^\perp$.

By \cite[Proposition 4.4]{BK} and its proof, $W(\mu,\zeta)$ is an irreducible $G_\si$-module such that the action of $T_\si$ on it is semisimple and every $c\in\CC^\times$ acts as the scalar $c$.
Moreover, any irreducible $G_\si$-module with these properties is isomorphic to $W(\mu,\zeta)$ for some $\mu,\ze$.

\begin{lemma}[\cite{BK}]\label{lmuze}
Two pairs\/ $(\mu,\zeta)$ and\/ $(\mu',\zeta')$ correspond to isomorphic irreducible\/ $G_\si$-modules if and only if they are related by
\begin{equation}\label{equivrel}
\mu'=\mu+\pi_0\al, \qquad \zeta'(U_\be)=C_{\al,\be}^{-1} \, \zeta(U_\be),
\end{equation}
for some fixed\/ $\al\in Q$ and all\/ $\be\in Q \cap (1-\si)Q^*$.
In this case, the isomorphism\/ $W(\mu',\zeta') \to W(\mu,\ze)$ sends\/ $1_{\mu',\ze'}$ to\/ $U_\al 1_{\mu,\ze}$.
\end{lemma}
\begin{proof}
This follows from the discussion in \cite{BK} above Eq.\ (4.57). Notice that \eqref{equivrel} coincides with \cite[(4.57)]{BK}.
\end{proof}

It is easy to check that \eqref{equivrel} defines an equivalence relation on pairs $(\mu,\zeta)$; henceforth, two such related pairs will be called \emph{equivalent}. We emphasize that while $\mu$ can be taken as a representative of the coset $\mu+\pi_0(Q)$, 
a change in the representative $\mu$ also changes the central character $\zeta$.
%\begin{proposition}[BK]
%There is a bijection between central characters $\chi$ of $G$ such that $\chi(c)=c$ for $c\in\CC^\times$ and pairs $(\mu,\zeta)$, up to equivalence.
%\end{proposition}
We summarize the above discussion as follows.

\begin{theorem}[cf.\ \cite{BK}, Proposition 4.4]\label{bij}
There is a bijective correspondence between the following sets of objects$:$
\begin{enumerate}
\item Irreducible representations of\/ $G_\si$, up to isomorphism, such that the action of\/ $T_\si$ is semisimple and every $c\in\CC^\times$ acts as the scalar $c;$ \label{1}
\item Characters $\chi\colon Z(G)\to\CC^\times$ such that\/ $\chi|_{N_\si}=1$ and\/ $\chi(c)=c$ for all\/ $c\in\CC^\times$ {\upshape{(}}cf.\ \eqref{Cal}, \eqref{Nsi}{\upshape{);}} \label{2}
\item Pairs $(\mu,\zeta)$, up to equivalence \eqref{equivrel}, where $\mu\in\pi_0(Q^*)$ and $\zeta$ is a central character of\/ $G_\si^\perp$ such that 
\begin{equation}\label{con}
e^{2\pi\ii(\ga|\mu)} \zeta\bigl(U^{-1}_{\si\ga}U_{\ga}\bigr)=\eta(\ga)e^{-2\pi\ii b_\ga},
%e^{2\pi\ii(\ga|\mu)} \, \zeta\bigl(U_{\ga}U^{-1}_{\si\ga}\bigr)=\eta(\ga)e^{-2\pi\ii b_\ga}(-1)^{(\ga|\si\ga)},
%e^{2\pi\ii(\al|\mu)}\zeta\left(U_{(1-\si)\al}\right)=\eta(\al)e^{-2\pi\ii b_\al}\ep(\si\al,\al-\si\al)B_{\si\al,\al-\si\al}^{-1}.
\end{equation}
for all\/ $\ga\in Q$ {\upshape{(}}cf.\ \eqref{bal}{\upshape{)}}.
\label{3}
\end{enumerate}
\end{theorem}
\begin{proof}
In addition to the above discussion, let us fill the gaps from \cite{BK} that were presented without proof there.
Recall that the center $Z(G)$ consists of elements \eqref{ZG}. Given a pair $(\mu,\zeta)$, we define $\chi\colon Z(G)\to\CC^\times$ by setting
\begin{equation}\label{chizeta}
\chi(c \, e^{2\pi\ii\,\pi_0(\la)}U_{(1-\si)\la})=c \, e^{2\pi\ii(\la|\mu)} \, \zeta(U_{(1-\si)\la}),
\end{equation}
where $c\in\CC^\times$  and $\la\in Q^*$ with $(1-\si)\la\in Q$. 

We check that equivalent pairs correspond to the same character $\chi$. 
Suppose $(\mu,\zeta)$ and $(\mu',\zeta')$ are equivalent, so they satisfy \eqref{equivrel} for some $\al\in Q$.
Let $\be=(1-\si)\la\in Q\cap(1-\si)Q^*$. Then \eqref{Calbe} with $\al,\be$ switched and $\be_*=\pi_\perp(\la)$ implies
\begin{equation}\label{Calbe1}
C_{\al,\be}^{-1} = C_{\be,\al} = e^{2\pi\ii(\pi_\perp\la|\al)} = e^{-2\pi\ii(\pi_0\la|\al)}\,.
\end{equation}
Hence
\begin{align*}
e^{2\pi\ii(\la|\mu')}\zeta'(U_{(1-\si)\la})&=e^{2\pi\ii(\la|\mu+\pi_0\al)} \, C_{\al,\be}^{-1} \, \zeta(U_{(1-\si)\la}) \\
&=e^{2\pi\ii(\la|\mu)}e^{2\pi\ii(\la|\pi_0\al)}e^{2\pi\ii(\pi_\perp\la|\al)} \zeta(U_{(1-\si)\la}) \\
&=e^{2\pi\ii(\la|\mu)}\zeta(U_{(1-\si)\la}),
\end{align*}
using that 
\begin{equation*}
(\la|\pi_0\al) + (\pi_\perp\la|\al) = (\pi_0\la|\al) + (\pi_\perp\la|\al) = (\la|\al)\in\ZZ, 
\end{equation*}
as $\la\in Q^*$ and $\al\in Q$.

Putting $\la=0$ in \eqref{chizeta}, we obviously have $\chi(c)=c$ for $c\in\CC^\times$.
Next, we show that condition \eqref{con} is equivalent to $\chi|_{N_\si}=1$, i.e., to 
$\chi(C_\ga) = 1$ for all $\ga\in Q$.
By \eqref{tig3}, we can write
\begin{equation*}
U^{-1}_{\si\ga}U_{\ga} = x_\ga U_{(1-\si)\ga}
\end{equation*}
for some $x_\ga\in\CC$.
Then from \eqref{Cal} and \eqref{tig2}, we have
\begin{align*}%\label{Cal2}
C_\ga &= x_\ga \eta(\ga) e^{ 2\pi\ii b_\ga} U_{(1-\si)\ga} e^{ 2\pi\ii \, \pi_0(\ga) } \\
&= x_\ga \eta(\ga) e^{ 2\pi\ii b_\ga} e^{ 2\pi\ii \, \pi_0(\ga) } U_{(1-\si)\ga} \,,
\end{align*}
because $(\pi_0\ga|(1-\si)\ga)=0$. Then $\chi(C_\ga) = 1$ is equivalent to
\begin{align*}%\label{Cal2}
1 &= x_\ga \eta(\ga) e^{ 2\pi\ii b_\ga} \chi\bigl( e^{ 2\pi\ii \, \pi_0(\ga) } U_{(1-\si)\ga} \bigr) \\
&= x_\ga \eta(\ga) e^{ 2\pi\ii b_\ga} e^{ 2\pi\ii (\ga|\mu) } \ze(U_{(1-\si)\ga}) \\
&= \eta(\ga) e^{ 2\pi\ii b_\ga} e^{ 2\pi\ii (\ga|\mu) } \ze\bigl( U^{-1}_{\si\ga}U_{\ga} \bigr),
\end{align*}
which is exactly \eqref{con}.
The rest of the proof is in \cite[Proposition 4.4]{BK}.
\end{proof}

\begin{remark}\label{remark1}
Suppose that $(1-\sigma)Q=Q\cap\mathfrak{h}_\perp$. Then $Q \cap (1-\si)Q^* = (1-\sigma)Q$.
Hence, for a given $\mu\in\pi_0(Q^*)$, relation \eqref{con} completely determines the character $\zeta$ of $Z(G_\si^\perp)$.
Some examples that satisfy the condition $(1-\sigma)Q=Q\cap\mathfrak{h}_\perp$ include the root lattice $A_{2n+1}$ $(n\geq1)$ with $\si$ a Dynkin diagram automorphism 
and the class of permutation orbifolds (see Section \ref{permorb} below). 
\end{remark}

\begin{example}\label{remark2}
Another important special case is when $\mathfrak{h}_0=0$. Then $\pi_0=0$ and $\mu=0$.
%and the irreducible representations of $G_\si$ can be described using only central characters $\zeta$ of $G_\si^\perp$ subject to \eqref{con} with $\mu=0$. Note also that \eqref{chizeta} implies that $\chi=\zeta$ so that part 3 of Theorem \ref{bij} becomes vacuous. 
Examples that satisfy this condition include $\si=-1$ for an arbitrary lattice $Q$ (cf.\ \cite{D2} and Section \ref{ex2} below)
and the ADE root lattices with $\si$ a Coxeter element from the Weyl group (cf.\ \cite{KP,KT}).
%In the latter examples, there is a unique $\si$-twisted module described by $\mu=0$ and one character $\zeta$.
%The case of affine orbifolds is studied in \cite[Theorem 4.3]{KT}.
\end{example}

%%%%%%%%%%%%%%%%%%%%%%%%%%%%
\subsection{The $\si$-twisted $V_Q$-modules $M(\mu,\zeta)$}
%%%%%%%%%%%%%%%%%%%%%%%%%%%%
Now we review the construction of irreducible $\si$-twisted $V_Q$-modules from \cite{BK},
where as before $\si$ is an isometry of $Q$ of order $N$.

Starting from an irreducible $G_\si$-module $W(\mu,\zeta)$ corresponding to a pair $(\mu,\ze)$ as in \thref{bij},
we make it an $\mathfrak{h}_0\oplus\hat{\mathfrak{h}}_\si^>$-module by letting $\hat{\mathfrak{h}}_\si^>$ act trivially. 
Note that $\lieh_0$ acts on $W(\mu,\zeta)$ according to
\begin{equation}\label{tig4}
h (U_\al v) = (h|\al+\mu) U_\al v \,,
\end{equation}
for $h\in\lieh_0$, $\al\in Q$ and $v\in\Om(\mu,\zeta)$ (cf.\ \eqref{tig2}, \eqref{W}).
Then by inducing this action to $\hat{\mathfrak{h}}_\si$, we obtain an irreducible $\si$-twisted $V_Q$-module
\begin{align}\label{M}
M(\mu,\zeta)&=\Ind_{\mathfrak{h}_0\oplus\hat{\mathfrak{h}}_\si^>}^{\hat{\mathfrak{h}}_\si}W(\mu,\zeta)\cong \F_\si\otimes W(\mu,\zeta)
\end{align}
with an action defined as follows.

We define $Y(a,z)$ for $a\in\lieh$ as before (see \eqref{twheis3}), 
and for $\al\in Q$ we let
\begin{equation}\label{twlat12}
Y(e^\al,z) = 
E_\al(z)
\otimes U_\al z^{b_\al+\pi_0\al} \,,
\end{equation}
where 
\begin{equation}\label{E}
\begin{aligned}
E_\al(z)&=\exp\Biggl( \sum_{  n\in\frac1N\ZZ_{<0} } \al_n \frac{z^{-n}}{-n} \Biggr)
\exp\Biggl( \sum_{  n\in\frac1N\ZZ_{>0} } \al_n \frac{z^{-n}}{-n} \Biggr).
%&=:\exp\int (Y^{\tw}(\al,z)-(\pi_0\al)z^{-1}):.
\end{aligned}
\end{equation}
%and the normally ordered product $:\cdot:$ is defined as follows. 
Here the action of $z^{\pi_0\al}$ is given by $z^{\pi_0\al} (U_\be v) = z^{(\pi_0\al|\be+\mu)} U_\be v$ for $\be\in Q$ and $v\in\Om(\mu,\zeta)$.
Notice that $(\pi_0\al|\mu) \in \frac1N\ZZ$.
The action of $Y$ on all of $V_Q$ can then be obtained by applying the product
formula \eqref{locpr3}.

\begin{theorem}[\cite{BK}, Theorem 4.2]\label{thm:BK}
Every irreducible\/ $\si$-twisted\/ $V_Q$-module is isomorphic to one of the modules\/ $M(\mu,\zeta)$,
and two such modules are isomorphic if and only if the corresponding pairs\/ $(\mu,\ze)$ are equivalent according to \eqref{equivrel}.
Moreover, every\/ $\si$-twisted\/ $V_Q$-module is a direct sum of irreducible ones.
\end{theorem}

In the special case when $\si=1$, we obtain Dong's Theorem that the irreducible $V_Q$-modules
are classified by $Q^*/Q$ (see \eqref{VFW} and \cite{D1}).

Now we describe a basis for $M(\mu,\zeta)$. We start by choosing a basis $\B_\Om$ for the finite-dimensional $G_\si^\perp$-module $\Om(\mu,\zeta)$.

\begin{remark}\label{defect}
All irreducible $G_\si^\perp$-modules $\Om(\mu,\zeta)$ have the same dimension:
\[
\dim\Om(\mu,\zeta)=d(\si),
\]
where $d(\si)$ is called the \emph{defect} of $\si$ (see \cite{KP,BK}).
It is known that 
\begin{align}\label{defectdef}
d(\si)^2={|G_\si^\perp:Z(G_\si^\perp)|}=|(Q\cap\lieh_\perp)/Q_\si|, 
\end{align}
where $Q_\si=Q\cap(1-\si)Q^*\subset Q\cap\lieh_\perp$ (see \cite[(4.53)]{BK}).
%We note that in the special case that $Q\cap\lieh_\perp=(1-\si)Q$, the defect is always equal to one.
In the setting of \cite{KP}, $Q$ is a root lattice and $\si$ is an element of its Weyl group. In this case, it was shown in \cite{KP} that $(1-\si)Q^*\subset Q$ for such $\si$ and that $d(\si)^2$ can be described as the order of the torsion subgroup of the abelian group $Q/(1-\si)Q^*$.
\end{remark}

Due to \eqref{W}, the set $\{gv \,|\, g\in\C_G, \, v\in\B_\Om\}$ is a basis for the $G_\si$-module $W(\mu, \zeta)$, where $\C_G \subset G_\si$ is a set of representatives of the cosets of $T_\si\times G_\si^\perp$ in $G_\si$.
The next lemma provides a way for constructing $\C_G$.

\begin{lemma}\label{basis}
If\/ $\C_Q \subset Q$ is a set of representatives of the cosets of\/ $Q\cap\mathfrak{h}_\perp$ in\/ $Q$, then\/ $\C_G=\{U_{\ga}N_\si\in G_\si \,|\, \ga\in\C_Q\}$ is a set of rep\-resen\-tatives of the cosets of\/ $T_\si\times G_\si^\perp$ in\/ $G_\si$. 
\end{lemma}
\begin{proof}
By assumption, every $\al\in Q$ can be written as $\al=\ga+\be$ for some $\ga\in\C_Q$ and $\be\in Q\cap\h_\perp$. 
Then for any $h\in\h_0$ and $c\in\CC^\times$, using \eqref{tig2}, \eqref{tig3}, we have
\[
c\,e^h\,U_\al=U_{\ga}\, e^h \, c'\,U_{\be} %\in U_{\ga}(T_\si\times G_\si^\perp)
\]
for some $c'\in\CC^\times$.
The claim of the lemma follows, since $e^h \in T_\si$ and $c'\,U_{\be}\in G_\si^\perp$.
\end{proof}

As a consequence of Lemma \ref{basis}, we have:

\begin{corollary}\label{cbasisW}
With the above notation, the set
\[
\bigl\{U_{\ga}v \,\big|\, \ga\in\C_Q, \, v\in\B_\Om\bigr\}
\]
is a basis for\/ $W(\mu,\zeta)$. 
\end{corollary}

In order to write a basis for the $\si$-twisted Heisenberg algrebra $\hat{\mathfrak{h}}_\si$, we pick a basis $\{a^i\}_{i=1}^r$ for $\mathfrak{h}$ such that 
$a^i\in\mathfrak{h}_{j_i/N}$,
where $0\leq j_i<N$ (cf.\ \eqref{j/N}). A basis for $\hat{\mathfrak{h}}_\si$ is then given by
\[
\Bigl\{a^it^{n}\,\Big|\,1\leq i\leq r, \; n\in \frac{j_i}{N}+\ZZ\Bigr\}.
\]
The following is now immediate.

\begin{lemma}\label{basisM}
Let\/ $\C_Q\subset Q$ be representatives of the cosets\/ $Q/(Q\cap\mathfrak{h}_\perp)$, let\/ $\B_\Om$ be a basis for the irreducible\/ $G_\si^\perp$-module\/ $\Om(\mu,\zeta)$,
and\/ $\{a^i\}_{i=1}^r$ be a basis for\/ $\mathfrak{h}$ as above.
Then a basis for the\/ $V_Q$-module\/ $M(\mu,\zeta)$ consists of elements
\begin{equation}
(a^{i_1}t^{-n_1})\cdots(a^{i_k}t^{-n_k})U_{\ga}v \qquad (\ga\in \C_Q, \; v\in\B_\Om),
\end{equation}
where $k\ge0$ $($the case $k=0$ corresponding to $U_{\ga}v)$,
$1\le i_l \le r$, $n_l\in-(j_l/N)
%\frac{j_l}{N}
+\ZZ$, $n_l>0$ for $1\leq l\leq k$,
and the pairs\/ $(i_l,n_l)$ are ordered lexicographically.
\end{lemma}

%%%%%%%%%%%%%%%%%%%%%%%%%%%%%%%
\subsection{$L_0^{\tw}$-action on the $V_Q$-module $M(\mu,\zeta)$}
%%%%%%%%%%%%%%%%%%%%%%%%%%%%%%%
Recall that the lattice vertex algebra $V_Q$ is conformal with central charge $r=\rank Q$ and a Virasoro vector given by \eqref{om}.
Then for any $\si$-twisted $V_Q$-module $M$, from the action $Y$ of $V_Q$ on $M$,
we obtain a representation on $M$ of the Virasoro Lie algebra by
\[
Y(L,z)=\sum_{n\in\ZZ} L_n^M z^{-n-2}, \qquad L_n^M \in\End M.
\]
To emphasize that it is a twisted module, we will also denote the modes $L_n^M$ as $L_n^{\tw}$.
An explicit formula for $L_n^{\tw}$ in the more general case of twisted logarithmic modules was given in \cite{Ba} for an arbitrary (not necessarily semisimple) automorphism $\si$. 
Here we determine the action of the twisted energy operator 
$L_0^{\tw}$ on a basis of the irreducible $\si$-twisted $V_Q$-module $M(\mu,\zeta)$. 

To give this action explicitly, we define a linear operator $s$ on each eigenspace $\mathfrak{h}_{j/N}$ by (cf.\ \eqref{j/N}):
\begin{equation}\label{s}
sa=-\frac{j}{N}a \quad\text{for}\quad a\in\mathfrak{h}_{j/N},\;\; j=0,\ldots,N-1,
\end{equation}
so that
$\si=e^{2\pi\ii s}$.
Let $\{a^i\}_{i=1}^r$ and $\{b^i\}_{i=1}^r$ be dual bases of $\mathfrak{h}$ consisting of eigenvectors of $\si$.
Suppose that $a^i\in\mathfrak{h}_{j_i/N}$.
Notice that this implies 
\[
sa^i=-\frac{j_i}{N}a^i, \qquad sb^i=\Bigl(\frac{j_i}{N}-1\Bigr)b^i
\] 
if $j_i\neq0$, while $sa^i=sb^i=0$ if $j_i=0$.
Then, by \cite[(6.8)]{Ba}, the action of $L_0^{\tw}$ is given by 
\begin{equation}\label{L0}
L_0^{\tw}=\frac12\sum_{i=1}^r\sum_{n\in \frac{j_i}{N}+\ZZ}{:}(a^it^n)(b^it^{-n}){:}-\frac12\tr_\lieh \binom{s+1}{2}I,
\end{equation}
where the normal ordering $:\,:$ is defined by
\begin{align*}
{:}(a^it^n)(b^it^{-n}){:}=\begin{cases}
(a^it^n)(b^it^{-n}),&\; n<0,\\
(b^it^{-n})(a^it^n),&\; n\geq0.
\end{cases}
\end{align*}
For convenience, we set
\begin{equation}\label{Delta}
\Delta_\si=-\frac12\tr_\lieh \binom{s+1}{2}
=\displaystyle\frac14\sum_{j=1}^{N-1}\frac{j}{N}\left(1-\frac{j}{N}\right)\dim\lieh_{j/N}.
\end{equation}
Note that $\Delta_1=0$ for the identity automorphism $\si=1$, which corresponds to the case of untwisted modules.

\begin{remark}\label{trace}
Using that $\dim\mathfrak{h}_{j/N}=\dim\mathfrak{h}_{1-(j/N)}$, we have
\[
\tr_\lieh s=\frac12\sum_{j=1}^{N-1}\left(\left(\frac{j}{N}-1\right)+\left(-\frac{j}{N}\right)\right)\dim\lieh_{j/N}=-\frac12\dim\mathfrak{h}_\perp.
\]
%and $\Delta_\si=-\displaystyle\frac14\sum_{\substack{i=1\\p_i\neq0}}^r(-p_i+1)(-p_i)$.
\end{remark}

\begin{remark}
Formula (6.8) in \cite{Ba} uses a linear operator $\mathcal{S}$ on $\lieh$ such that $\si=e^{-2\pi\ii \mathcal{S}}$ and the eigenvalues of $\mathcal{S}$ are in the interval $(-1,0]$.
The form \eqref{L0} for $L_0^{\tw}$ is obtained by setting $\mathcal{S}=-s$ and using that $\binom{-s}{2}=\binom{s+1}{2}$. 
%This is merely using the powers of $z$ to define creation and annihilation operators instead of the modes.
Note that the term $-\bar{\omega}t^n$ in \cite[(6.8)]{Ba} vanishes in our case since the Lie algebra $\mathfrak{h}$ is abelian and $\bar{\omega}$ is given in terms of Lie brackets (see \cite[Lemma 6.4]{Ba}).
\end{remark}

\begin{lemma}\label{Deltalemma}
Let\/ $Q$ be an even lattice and\/ $\si$ be an isometry of\/ $Q$ of prime order\/ $p$. Then
\begin{equation}\label{Deltap}
\Delta_\si=\frac{p+1}{24p}\dim\lieh_\perp,
\end{equation}
where\/ $\lieh_\perp$ is given in \eqref{hperp}. In particular, $\Delta_{\si^l}=\Delta_\si$ for all\/ $l=1,\ldots, p-1$.
\end{lemma}
\begin{proof}
By the Cyclic Decomposition Theorem, $\lieh$ has a $\si$-invariant basis on which $\si$ acts as a permutation. The order of each orbit of this action divides $p$.
Suppose that there are $e$ singleton orbits and $d$ orbits of order $p$. Then we have for $1\leq j\leq p-1$:
\begin{equation}\label{dimh}
\dim\lieh=e+dp, \quad \dim\lieh_0=e+d, \quad \dim\lieh_{j/p}=d.
\end{equation}
Hence, \eqref{Delta} becomes
\begin{align*}
\Delta_\si
&=\displaystyle\frac{d}{4}\sum_{\substack{j=1}}^{p-1}\frac{j}{p}\left(1-\frac{j}{p}\right)
=\frac{d(p-1)(p+1)}{24p}\,,
\end{align*}
thus proving \eqref{Deltap} since $\dim\lieh_\perp=d(p-1)$.
The last claim of the lemma follows from the fact that $\lieh_\perp$ remains the same for $\si$ and $\si^l$. %since $\si$ has prime order.
\end{proof}

\begin{remark}\label{rem-Delta}
%When the group $\Gamma_{\si,\mu,\zeta}$ is cyclic with generator $g$ (cf.\ \eqref{smz}), it is difficult to relate \eqref{Delta} with $\si=g^k$, for some $k>0$, to \eqref{Delta} with $\si=g$. This is mainly due to our condition that each eigenvalue be in the interval $[-1,0]$. Hence each time $\si$ acts, we must take the eigenvalue modulo 1. 
In general, 
%$\Delta_\si\neq\Delta_{\si^{k}}$ ($k\neq1$) when $\si$ does not have prime order. 
%However, it is easy to relate $\Delta_\si$ with $\Delta_{\si^{-1}}$. To describe the inverse transformation $\si^{-1}$, we use the dual basis for $\mathfrak{h}$ by replacing $-p_i$ in \eqref{s} with $1-p_i$. The result is that
%\begin{equation}
$\Delta_\si=\Delta_{\si^{l}}$
%\end{equation}
if $l$ is coprime to the order of $\si$, but \eqref{Deltap} may no longer hold. In particular, we have $\Delta_\si=\Delta_{\si^{-1}}$. 
\end{remark}

%\begin{remark}
%It will be useful in the examples to compute $\Delta_\si$ in the case when $\si^3=1$. The values of $p_i$ in \eqref{s} are either $0, \frac13$, or $\frac23$. The product $p_i(1-p_i)$ is then equal to $\frac29$ for each $p_i\neq0$. It follows from \eqref{Delta} that 
%\begin{equation}
%\Delta_\si=\frac14\left(\frac29\dim\mathfrak{h}_\perp\right)=\frac{1}{18}\dim\mathfrak{h}_\perp.
%\end{equation}
%\end{remark}

Our goal is to determine the action of $L_0^{\tw}$ on the basis of the $\si$-twisted $V_Q$-module $M(\mu,\zeta)$. For $a\in V_Q$, we have
\begin{equation}
[L_0^{\tw},Y(a,z)]=z\partial_zY(a,z)+Y(L_0a,z)
\end{equation}
(see e.g.\ \cite[Section 5.4]{Ba}). Recall that in $V_Q$, elements of $\mathfrak{h}$ have \emph{conformal weight} (eigenvalue of $L_0$) $1$, and elements $e^\al$ have conformal weight $|\al|^2/2$
(see e.g.\ \cite{K2}).
Hence, we have:
\begin{align}
[L_0^{\tw},Y(h,z)]&=\left(1+z\partial_z\right)Y(h,z), & h&\in\lieh,
\label{L0al}\\
[L_0^{\tw},Y(e^\al,z)]&=\left(\frac{|\al|^2}{2}+z\partial_z\right)Y(e^\al,z), & \al&\in Q.
\label{L0e^al}
\end{align}
\begin{lemma}
For $\al\in Q$, we have 
\begin{equation}
[L_0^{\tw},U_\al]=U_\al\left(\frac{|\pi_0\al|^2}{2}+\pi_0\al\right).
\end{equation}
\end{lemma}
\begin{proof}
Equations \eqref{L0al} and \eqref{twheis3} imply that $[L_0^{\tw},\al_n]=-n\al_n$, from which we get
$[L_0^{\tw}, E_\al(z)]=z\partial_zE_\al(z)$ (cf.\ \eqref{E}). The result now follows from \eqref{twlat12}, \eqref{L0e^al}, 
and the relation $\frac{|\al|^2}{2}+b_\al=\frac{|\pi_0\al|^2}{2}$.
\end{proof}

Now we determine the action of $L_0^{\tw}$ on $v\in\Om(\mu,\zeta)$:
\begin{align}\label{L0Om}
L_0^{\tw}v&=\frac12\sum_{i=1}^r(a^it^0)(b^it^0)v+\Delta_\si v=\left(\frac{|\mu|^2}{2}+\Delta_\si\right)v,
\end{align}
using that $(ht^0)v=(h|\mu)v$ and $(ht^n)v=0$ for $n>0$. Then for $\ga\in Q$ and $v\in\Om(\mu,\zeta)$, we find
\begin{equation}\label{L0prep}
\begin{aligned}
L_0^{\tw}U_\ga v&=U_\ga L_0^{\tw}v+[L_0^{\tw},U_\ga]v\\
&=\left(\frac{|\mu|^2}{2}+\Delta_\si\right)U_\ga v+U_\ga\left(\frac{|\pi_0\ga|^2}{2}+\pi_0\ga\right)v\\
&=\left(\frac{|\mu+\pi_0\ga|^2}{2}+\Delta_\si\right)U_\ga v,
\end{aligned}
\end{equation}
using that $(\pi_0\ga) v=(\pi_0\ga|\mu)v$. Using \eqref{L0prep}, the action of $L_0^{\tw}$ on $M(\mu,\zeta)$ can now be determined.

\begin{proposition}\label{pL0tw}
Let\/ $w=(a^{i_1}t^{-n_1})\cdots(a^{i_k}t^{-n_k})U_{\ga}v$ be a basis element for\/ $M(\mu,\zeta)$ as in Lemma \ref{basisM}. Then 
\begin{align}\label{h0M}
h_0w&=(h|\mu+\pi_0\ga)w,\qquad h\in\mathfrak{h}_0,\\
L_0^{\tw}w&=\Bigl(n_1+\cdots+n_k+\frac{|\mu+\pi_0\ga|^2}{2}+\Delta_\si\Bigr)w.\label{L0M}
\end{align}
In particular, $L_0^{\tw}$ is diagonalizable on\/ $M(\mu,\zeta)$ with positive eigenvalues,
with the only exception when\/ $M(\mu,\zeta) \cong V_Q$ and\/ $w=\vac$, which has eigenvalue $0$.
\end{proposition}

\begin{proof}
The only thing left to prove is the claim that the eigenvalues of $L_0^{\tw}$ are positive.
From \eqref{Delta}, we see that $\Delta_\si\ge0$ and $\Delta_\si>0$ unless $\si=1$. 
Since $Q$ is positive definite, $|\mu+\pi_0\ga|^2 \ge0$ and $|\mu+\pi_0\ga|^2 >0$ unless $\mu+\pi_0\ga=0$. 
Hence, the eigenvalues of $L_0^{\tw}$ are non-negative, and the only way to obtain $0$ is if all $n_j=0$ and $\si=1$, $\mu=\ga=0$.
\end{proof}

\begin{remark}
We see from \eqref{L0M} that the \emph{conformal weight} (i.e., the minimal eigenvalue of $L_0^{\tw}$) of $M(\mu,\zeta)$ is 
\[
\Delta_\si + \frac12 \min_{\ga\in Q} |\mu+\pi_0\ga|^2.
\]
In particular, if $\mu=0$, then the conformal weight is $\Delta_\si$.
\end{remark}

%%%%%%%%%%%%%%%%%%%%%%%%%%%%%%%
\subsection{The sublattice $\bar{Q}\subset Q$% and subgroup $\Gamma_{\si,\mu,\zeta}$
}\label{subsub}
%%%%%%%%%%%%%%%%%%%%%%%%%%%%%%%

In this subsection and the next one, we consider a more general situation when the isometry $\si$ of a positive-definite even lattice $Q$ is an element of a finite group $\Gamma$ of isometries.
In subsequent sections, we will return to the case of a single element $\si$, or equivalently, a cyclic group $\Ga=\langle\si\rangle$. 

As in \seref{lat}, we fix a bimultiplicative $2$-cocycle $\ep\colon Q \times Q \to \{\pm1\}$.
%which gives rise to an automorphism the associative algebra $\CC_\ep[Q]$ and the lattice vertex algebra $V_Q$. 
As in \seref{gsigsiperp}, for any isometry $\ph\in\Ga$, we pick a function $\eta_\ph\colon Q\to\{\pm1\}$ such that
\begin{equation}\label{etaph}
\eta_\varphi(\al+\be)\ep(\varphi\al,\varphi\be)
=\eta_\varphi(\al)\eta_\varphi(\be)\ep(\al,\be),
\qquad \al,\be\in Q.
\end{equation} 
%Moreover, we will assume that $\eta_\ph(\al)=1$ for all $\al\in Q$ such that $\ph\al=\al$. In particular, 
We will assume that $\eta_1=1$ for the identity element $\ph=1$.

Then $\ph$ induces an automorphism of $\CC_\ep[Q]$ defined by
\begin{equation}\label{pheal}
\ph(e^\al)=\eta_\ph(\al) e^{\ph\al} \,,
\qquad \al\in Q
\end{equation}
(cf.\ \eqref{lat1}, \eqref{twlat4}). 
It also defines an automorphism of $V_Q$ similarly to \eqref{twlat4}, which by abuse of notation we will denote again as $\ph$.
As before, the order of $\ph$ as an automorphism of $\CC_\ep[Q]$ (or $V_Q$) may be double its order as an automorphism of $Q$.
Moreover, this does not define a representation of $\Ga$, but only a projective representation, or a representation of its central extension.
Following an idea of \cite{BE}, we will avoid these complications by restricting to a certain sublattice $\bar{Q}$ of $Q$, for which this is true and which produces the same 
orbifold subalgebra $(V_{\bar Q})^\Ga = (V_Q)^\Ga$ of fixed points under $\Ga$.

\begin{definition}\label{dQbar}
Let $\bar Q$ be the set of all $\al\in Q$ such that, for every $\ph_1,\ph_2\in\Ga$, we have
\begin{equation}\label{Qbar1}
(\ph_1\ph_2)(e^\al) = \ph_1(\ph_2(e^\al))
\end{equation}
in $\CC_\ep[Q]$ (and hence also in $V_Q$).
Equivalently, $\bar Q$ consists of all $\al\in Q$ such that
\begin{equation}\label{Qbar2}
\eta_{\ph_1\ph_2}(\al) = \eta_{\ph_1}(\ph_2\al) \eta_{\ph_2}(\al)
\end{equation}
for every $\ph_1,\ph_2\in\Ga$.
\end{definition}

\begin{lemma}\label{lQbar1}
The subset\/ $\bar{Q}\subset Q$ is a $\Ga$-invariant sublattice of\/ $Q$ containing\/ $2Q$. In particular, $\rank\bar Q=\rank Q$.
\end{lemma}
\begin{proof}
Consider arbitrary $\al,\be\in Q$ and $\ph_1,\ph_2\in\Ga$. Using \eqref{lat1} and that each $\ph_i$ acts as an automorphism the associative algebra $\CC_\ep[Q]$,
we obtain:
\[
\ep(\al,\be) (\ph_1\ph_2)(e^{\al+\be})
= (\ph_1\ph_2)(e^\al e^\be) 
= (\ph_1\ph_2)(e^\al) (\ph_1\ph_2)(e^\be)
\]
and
\[
\ep(\al,\be) \ph_1(\ph_2(e^{\al+\be}))
= \ph_1(\ph_2(e^\al e^\be)) 
= \ph_1(\ph_2(e^\al)) \ph_1(\ph_2(e^\be)).
\]
If $\al,\be\in\bar Q$, then the right-hand sides of the above equations are equal; hence, the left sides are equal and $\al+\be\in\bar Q$.

Similarly, taking $\be=-\al$ and using that $\ph(e^0)=e^0$ for all $\ph\in\Ga$, we also see that $\al\in\bar Q$ implies $-\al\in\bar Q$.
Therefore, $\bar Q$ is a sublattice of $Q$.
Now let $\be=\al\in Q$ be arbitrary. Then, since $(\ph_1\ph_2)(e^\al) = \pm\ph_1(\ph_2(e^\al))$, when we square these we get that
the right-hand sides of the above equations are again equal. Hence, $2\al\in\bar Q$ for all $\al\in Q$.

Finally, to prove the $\Ga$-invariance of $\bar Q$, let $\al\in\bar Q$ and $\ph_1,\ph_2,\ph_3\in\Ga$. Then
\[
\eta_{\ph_3}(\al) (\ph_1\ph_2)(e^{\ph_3\al}) = (\ph_1\ph_2)(\ph_3(e^\al)) = (\ph_1\ph_2\ph_3)(e^\al)
\]
and
\[
\eta_{\ph_3}(\al) \ph_1(\ph_2(e^{\ph_3\al})) = \ph_1(\ph_2(\ph_3(e^\al))) = (\ph_1\ph_2\ph_3)(e^\al).
\]
This implies that $\ph_3\al\in\bar Q$.
\end{proof}

Due to \leref{lQbar1}, we can view $\Ga$ as a group of isometries of the lattice $\bar Q$. Hence, as above, every $\ph\in\Ga$ induces an
automorphism of the lattice vertex algebra $V_{\bar Q}$, which is a subalgebra of $V_Q$. By construction, this does give
a representation of $\Ga$ on $V_{\bar Q}$. In particular, the order of $\ph$ remains the same as an element of $\Ga$ and as an
automorphism of $V_{\bar Q}$. Now we prove that the subalgebra of $\Ga$-fixed points is the same in $V_{\bar Q}$ and in $V_Q$.

\begin{lemma}\label{lQbar2}
We have\/ $(V_Q)^\Gamma=(V_{\bar{Q}})^\Gamma$.
\end{lemma}
\begin{proof}
Since $\bar Q\subset Q$, we have the obvious inclusions $V_{\bar Q} \subset V_Q$ and $(V_{\bar{Q}})^\Ga \subset (V_Q)^\Ga$.
To prove the opposite inclusion, consider an arbitrary element
\[
v = \sum_{\al\in Q} f_\al\otimes e^\al \in V_Q \qquad (f_\al\in\F),
\]
where only finitely many $f_\al\ne0$. 

For $\ph_1,\ph_2\in\Ga$, the automorphism $\psi=(\ph_1\ph_2)^{-1} \circ \ph_1\circ\ph_2$ of $V_Q$
has the property that $\psi(f)=f$ and $\psi(e^\al)=\pm e^\al$ for all $f\in\F$, $\al\in Q$. Hence,
\[
\psi(v) = \sum_{\al\in Q} f_\al\otimes \psi(e^\al).
\]
Suppose that $v\in(V_Q)^\Ga$, i.e., $\ph(v)=v$ for all $\ph\in\Ga$. This implies that $\psi(v)=v$ and $\psi(e^\al)=e^\al$
for all $\al\in Q$ with $f_\al\ne0$. Therefore, $\al\in\bar Q$ and $v\in V_{\bar Q}$.
\end{proof}

If the group $\Ga$ is given in terms of generators and relations, then it is enough to specify $\eta_\ph$ for the generators $\ph$.
Due to \eqref{etaph}, we can let $\eta_{\ph^{-1}}=\eta_\ph$. Then we can use \eqref{Qbar2} to define $\eta_\ph$ in the case when
$\ph$ is a product of generators and their inverses. For this process to be well defined, it needs to respect the relations of $\Ga$; explicitly,
we must have
\begin{equation*}
\ph_1\circ\cdots\circ\ph_k(e^\al) = e^\al,
\end{equation*}
or equivalently 
\begin{equation*}
\eta_{\ph_1}(\ph_2\cdots\ph_k\al) \eta_{\ph_2}(\ph_3\cdots\ph_k\al) \cdots \eta_{\ph_{k-1}}(\ph_k\al) \eta_{\ph_k}(\al) = 1,
\end{equation*}
whenever $\ph_1\cdots\ph_k=1$ in $\Gamma$. The above identities do not hold in general for all $\al\in Q$, 
but by construction they hold for all $\al\in\bar Q$.

In this paper, our main focus is the case when $\Gamma=\langle\si\rangle$ is a cyclic group of order $N$. Then all relations in $\Ga$ are consequences of
the relation $\si^N=1$. We fix the function $\eta=\eta_\si\colon Q\to\{\pm1\}$ satisfying \eqref{twlat3} and \eqref{twlat2}.
Then from the above discussion we can define $\eta_\ph$ for all $\ph\in\Ga$.
Hence, the sublattice $\bar Q\subset Q$ is determined by the equation
\begin{equation}\label{Qbar3a}
\underbrace{\si\circ\cdots\circ\si}_N(e^\al) = e^\al,
\end{equation}
or equivalently 
\begin{equation}\label{Qbar3}
\eta(\si^{N-1}\al) \eta(\si^{N-2}\al) \cdots \eta(\si\al) \eta(\al) = 1.
\end{equation}
The final result of this subsection provides a useful characterization of $\bar Q$ in terms of the bilinear form $(\cdot|\cdot)$ of the lattice $Q$.

\begin{lemma}\label{Qbarlemma}
Suppose that\/ $\Gamma=\langle\si\rangle$ is a cyclic group of order\/ $N$ of isometries of\/ $Q$.
Then\/ $\bar Q$ consists of all\/ $\al\in Q$ such that
\begin{equation}\label{Qbar4}
\sum_{m=1}^{N-1}(\al|\si^m\al)\in2\ZZ,
\end{equation}
or equivalently,
\begin{equation}\label{Qbar41}
 |\pi_0\al|^2 = (\al|\pi_0\al) \in\frac2N\ZZ,
\end{equation}
where\/ $\pi_0$ is the projection onto the space of\/ $\si$-invariant vectors in\/ $\lieh$.
\end{lemma}
\begin{proof}
The equivalence of the two conditions is obvious from $(\al|\al)\in2\ZZ$ and the formula 
\begin{equation}\label{eqpi0}
\pi_0 = \frac1N \sum_{m=0}^{N-1} \si^m.
\end{equation}
Since the element $N\pi_0\al = \al+\si\al+\cdots+\si^{N-1}\al \in Q$ is $\si$-invariant, we have
\begin{equation}\label{etasiinv}
\eta(\al+\si\al+\cdots+\si^{N-1}\al) = 1, \qquad \al\in Q,
\end{equation}
by \eqref{twlat2}.

We want to calculate the left-hand side of \eqref{Qbar3}, which we denote by $L$.
Using \eqref{twlat3}, we have for $0\le j\le N-2$:
\begin{align*}
\eta&(\si^{j+1}\al) \eta(\al+\si\al+\cdots+\si^j\al) \eta(\al+\cdots+\si^{j+1}\al) \\
&= \ep(\al+\si\al+\cdots+\si^j\al,\si^{j+1}\al) \ep(\si\al+\cdots+\si^{j+1}\al,\si^{j+2}\al).
\end{align*}
Then, by \eqref{etasiinv}, $L$ equals the product of these expressions, which by the bimultiplicativity of $\ep$ gives
\begin{equation*}
L=\prod_{j=0}^{N-2}\prod_{i=0}^{j} \ep(\si^i\al,\si^{j+1}\al)\ep(\si^{i+1}\al,\si^{j+2}\al).
\end{equation*}

Consider first the factors of the form $\ep(\al,\si^m\al)$ or $\ep(\si^m\al,\al)$ in the above product, where $1\leq m \leq N-1$ and we
keep in mind that $\si^N\al=\al$. For each $m$, there are exactly two such factors, 
coming from $\ep(\si^i\al,\si^{j+1}\al)$ with $j=m-1$, $i=0$, or from $\ep(\si^{i+1}\al,\si^{j+2}\al)$ with $j=N-2$, $i=m-1$. 
By \eqref{lat22}, the product of these two factors is
\begin{equation*}
\ep(\al,\si^m\al)\ep(\si^m\al,\al)=(-1)^{(\al|\si^m\al)}.
\end{equation*}
Now consider the remaining factors of the form $\ep(\si^k\al,\si^m\al)$ for $1\leq k<m \leq N-1$. 
We claim that every such factor appears exactly twice in the above product, and hence they cancel each other.
Indeed, we can get $\ep(\si^k\al,\si^m\al)$ either from $\ep(\si^i\al,\si^{j+1}\al)$ with $j=m-1$, $i=k$,
or from $\ep(\si^{i+1}\al,\si^{j+2}\al)$ with $j=m-2$, $i=k-1$.
Therefore,
\begin{equation*}
L=\prod_{m=1}^{N-1}(-1)^{(\al|\si^m\al)},
\end{equation*}
which completes the proof of the lemma.
\end{proof}

\begin{corollary}\label{corQbar1}
If\/ $\Gamma=\langle\si\rangle$ is cyclic of odd order, then\/ $\bar Q=Q$.
\end{corollary}
\begin{proof}
When $\si$ has order $N=2n+1$, we have for every $\al\in Q$,
\begin{equation*}
\sum_{m=1}^{N-1}(\al|\si^m\al) = \sum_{m=1}^n (\al|\si^m\al+\si^{N-m}\al) = 2\sum_{m=1}^n (\al|\si^m\al) \in2\ZZ,
\end{equation*}
since
\begin{equation*}
(\al|\si^{N-m}\al) = (\al|\si^{-m}\al) = (\si^m\al|\al) = (\al|\si^m\al)
\end{equation*}
by the $\si$-invariance and symmetry of the bilinear form. Hence, \leref{Qbarlemma} gives $\bar Q=Q$.
\end{proof}

\begin{remark}\label{remQbar}
When $\Gamma=\langle\si\rangle$ is cyclic of even order $N=2n$, the same reasoning as in the proof of \coref{corQbar1} shows that \eqref{Qbar4}
is equivalent to the condition $(\al|\si^n\al)\in2\ZZ$.
\end{remark}

\begin{corollary}\label{corQbar}
In the case when\/ $\Gamma=\langle\si\rangle$ is cyclic, the index of\/ $\bar Q$ in\/ $Q$ is either\/ $1$ or\/ $2$.
\end{corollary}
\begin{proof}
For any $\al,\be\in Q$, formula \eqref{eqpi0} implies
\[
(\pi_0\al|\pi_0\be) = (\al|\pi_0\be) \in\frac1N\ZZ.
\]
Thus,
\[
|\pi_0(\al+\be)|^2 -  |\pi_0\al|^2 -  |\pi_0\be|^2 \in\frac2N\ZZ.
\]
If $\al,\be\not\in\bar Q$, then $N|\pi_0\al|^2$ and $N|\pi_0\be|^2$ are odd integers; hence, their sum is even and $\al+\be\in\bar Q$.
\end{proof}

When $\Gamma=\langle\si\rangle$ is cyclic, the proof of the above corollary also provides an alternative proof of the fact that $\bar Q$ is a sublattice of $Q$,
while the $\si$-invariance of $\bar Q$ is an immediate consequence of \eqref{twlat1a} and \leref{Qbarlemma}.

%%%%%%%%%%%%%%%%%%%%%%%%%%%%%%%
\subsection{The subgroup $\Gamma_{\si,\mu,\zeta}$
}\label{subsubgroup}
%%%%%%%%%%%%%%%%%%%%%%%%%%%%%%%

As in the previous subsection, let $\Gamma$ be a finite group of isometries of a positive-definite even lattice $Q$.
Following \cite{DVVV}, for every pair of commuting elements $\si,\varphi\in\Gamma$, 
we want to define an action of $\ph$ on the set of irreducible $\si$-twisted $V_Q$-modules.

We will continue to use the notation from the previous subsections regarding $\si$.
If we extend $\varphi$ linearly to $\lieh$, then it preserves the $\si$-eigenspaces $\lieh_{j/N}$; in particular,
\begin{equation}\label{phih}
\varphi(\mathfrak{h}_0)\subset\mathfrak{h}_0,\qquad\varphi(\mathfrak{h}_\perp)\subset\mathfrak{h}_\perp.
\end{equation}
Hence, $\varphi$ induces an automorphism of the $\si$-twisted Heisenberg algebra $\hat{\mathfrak{h}}_\si$, given by 
\begin{equation}\label{phhtn}
\varphi(ht^n)=\varphi(h)t^n, \quad \ph(K)=K \qquad
\Bigl(h \in\lieh_{j/N}, \; n\in-\frac{j}{N}+\ZZ\Bigr).
\end{equation}
We also define an action of $\varphi$ on the group $G$ by 
\begin{equation}\label{phactG}
\varphi\bigl(c\,e^h\,U_\al\bigr)=c\,\eta_\varphi(\al)\,e^{\varphi(h)}\,U_{\varphi\al} \qquad
(c\in\CC^\times, \; h\in\lieh_0, \; \al\in Q),
\end{equation}
where $\eta_\varphi$ plays the same role for $\varphi$ as $\eta$ does for $\si$ (see \eqref{etaph}).
It is easy to show that $\varphi$ is an automorphism of $G$.
However, in general, it is not true that $\varphi$ preserves the subgroup $N_\si$ defined by \eqref{Nsi}.
It will be true if we assume that $\bar Q=Q$, which we will do from now on (cf.\ \leref{Qbarlemma} and  \coref{corQbar1}).

\begin{lemma}\label{lQbar3}
If\/ $\ph\si=\si\ph$ and\/ $\al\in\bar Q$, then\/ $\ph(C_\al)=C_{\varphi\al}$, where\/ $C_\al$ is defined by \eqref{Cal}.
\end{lemma}
\begin{proof}
From \eqref{Qbar2}, we have
\begin{equation*}
\eta_{\ph\si}(\al) = \eta_{\ph}(\si\al) \eta_{\si}(\al) = \eta_{\si\ph}(\al) = \eta_{\si}(\ph\al) \eta_{\ph}(\al).
\end{equation*}
We rewrite this as
\begin{equation*}
\eta_\varphi(\si\al)\eta_\varphi(\al)=\eta(\varphi\al)\eta(\al),
\end{equation*}
where as before $\eta=\eta_\si$. 
Furthermore, observe that $\varphi$ commutes with $\pi_0$ and $b_{\varphi\al}=b_\al$ (cf.\ \eqref{eqpi0}, \eqref{bal}).
Using all that, \eqref{Cal}, and \eqref{phactG}, we find
\begin{align*}
\varphi(C_\al)&=\eta(\al)\eta_\varphi(\si\al)\eta_\varphi(\al)U_{\varphi\si\al}^{-1}U_{\varphi\al}e^{2\pi\ii(b_{\al}+\varphi\pi_0\al)}\\
&=\eta(\varphi\al)U_{\si\ph\al}^{-1}U_{\varphi\al}e^{2\pi\ii(b_{\varphi\al}+\pi_0\varphi\al)}\\
&=C_{\varphi\al} \,,
\end{align*}
as claimed.
\end{proof}

Due to the above lemma, we have
\[
\varphi (N_\si)\subset N_\si,\qquad\varphi(N_\si^\perp)\subset N_\si^\perp.
\]
Hence, we can view $\varphi$ as an automorphism of both $G_\si$ and $G_\si^\perp$.
This induces an action of $\ph$ among the irreducible $G_\si$-modules $W(\mu,\ze)$.
Since such modules are classified by pairs $(\mu,\ze)$,
up to equivalence \eqref{equivrel}, we obtain an action of $\ph$ on such pairs (see \thref{bij}).

\begin{lemma}\label{lphWmuze}
We have a linear map\/ $\ph\colon W(\mu,\ze) \to W(\ph(\mu),\ze\circ\ph^{-1})$ such that
\begin{equation}\label{phgw1}
\ph(1_{\mu,\ze}) = 1_{\ph(\mu),\ze\circ\ph^{-1}}
\end{equation}
and 
\begin{equation}\label{phgw}
\ph(gw) = \ph(g) \ph(w), \qquad g\in G_\si, \;\; w\in W(\mu,\ze).
\end{equation}
\end{lemma}
\begin{proof}
This follows from comparing \eqref{A}, \eqref{Aaction}, \eqref{phactG}, and noting that $(\ph h|\ph\mu) = (h|\mu)$ for $h\in\lieh_0$.
\end{proof}

We will be interested in the case when the pairs $(\ph(\mu),\ze\circ\ph^{-1})$ and $(\mu,\ze)$ are equivalent. 
We introduce the following subgroup of $\Ga$:
\begin{equation}\label{smz}
\Gamma_{\si,\mu,\zeta}=\bigl\{\varphi\in\Gamma\,\big|\,\varphi\si=\si\varphi,\; (\ph(\mu),\ze\circ\ph^{-1})\sim(\mu,\zeta)\bigr\} \,,
\end{equation}
where $\sim$ denotes the equivalence \eqref{equivrel}. Then, by \leref{lmuze}, the $G_\si$-modules $W(\ph(\mu),\ze\circ\ph^{-1})$
and $W(\mu,\ze)$ are isomorphic. Composing this isomorphism with the map $\ph\colon W(\mu,\ze) \to W(\ph(\mu),\ze\circ\ph^{-1})$
from \leref{lphWmuze}, we obtain an action of $\ph$ on $W(\mu,\ze)$, which satisfies \eqref{phgw} and
\begin{equation}\label{phgw2}
\ph(1_{\mu,\ze}) = U_\al 1_{\mu,\ze} \,,
\end{equation}
where $\al\in Q$ is such that
\begin{equation}\label{equivrel2}
\ph(\mu)=\mu+\pi_0\al, \qquad \ze(\ph^{-1}(U_\be))=C_{\al,\be}^{-1} \, \zeta(U_\be),
\end{equation}
for all $\be\in Q \cap (1-\si)Q^*$.

Similarly, as we assumed that $\bar Q=Q$, the group $\Ga$ acts on $V_Q$ by automorphisms. Hence, $\ph$ acts on the irreducible $\si$-twisted modules
by composing the $V_Q$-action with $\ph$. More precisely:

\begin{lemma}\label{lphMmuze}
We have a linear map\/ $\ph\colon M(\mu,\ze) \to M(\ph(\mu),\ze\circ\ph^{-1})$ satisfying \eqref{phgw1} and
\begin{equation}\label{phav}
\ph\bigl(Y(a,z)v\bigr) = Y(\ph(a),z) \ph(v), \qquad a\in V_Q, \;\; v\in M(\mu,\ze).
\end{equation}
\end{lemma}
\begin{proof}
This follows immediately from \eqref{M}, \eqref{phhtn}, and \leref{lphWmuze}.
\end{proof}

As above, for $\ph\in\Gamma_{\si,\mu,\zeta}$, we obtain an action of $\ph$ on $M(\mu,\ze)$ satisfying \eqref{phgw2} and \eqref{phav}.
In particular, if we restrict to the orbifold subalgebra $V_Q^\Ga$, then we will have $\ph(a)=a$ in \eqref{phav}; 
hence the map $\ph\colon M(\mu,\ze) \to M(\mu,\ze)$ is an isomorphism of $V_Q^\Ga$-modules. 

\begin{remark}\label{rgasimuze}
Later we will restrict to the case when the group $\Ga=\langle\si\rangle$ is cyclic, so $\ph=\si^k$ for some $k$. Then we have $\ph(\mu)=\mu$
and $\ze\circ\ph^{-1}=\ze$, which imply, in particular, that $\Gamma_{\si,\mu,\zeta} = \Ga$.
Indeed, $\si(\mu)=\mu$ because $\mu\in\pi_0(Q^*)$, and $\ze\circ\si=\ze$ because $\ze(C_\al)=1$ for all $\al\in Q\cap\mathfrak{h}_\perp$
(see \eqref{Cal}, \eqref{Nsiperp}, \eqref{phactG}).
\end{remark}

%%%%%%%%%%%%%%%%%%%%%%%%%%%%%%%
\section{Modified Characters of Twisted $V_Q$-modules}\label{charVQ}
%%%%%%%%%%%%%%%%%%%%%%%%%%%%%%%
Throughout this section, as before, $Q$ will be a positive-definite even integral lattice, and $\ph,\si$ will be two commuting isometries of $Q$.
Our goal is to compute the modified characters of irreducible $\si$-twisted $V_Q$-modules $M(\mu,\zeta)$, which are defined by
\begin{equation}\label{chiM}
\chi_{M(\mu,\zeta)}^{\si,\ph}(\tau,h)=\tr_{M(\mu,\zeta)} \ph e^{2\pi\ii h}q^{L_0^{\tw}-\frac {r}{24}},
\end{equation}
where $r=\rank Q$, $h\in\mathfrak{h}_0$, $\ph\in\Gamma_{\si,\mu,\zeta}$ (cf.\ \eqref{smz}), and $q=e^{2\pi\ii\tau}$ with $\tau\in\CC$, $\text{Im}\,\tau>0$.
Since $M(\mu,\zeta)$ is a tensor product of the $\si$-twisted Fock space $\F_\si$ and the $G_\si$-module $W(\mu,\zeta)$, the trace is a product
of traces over them. First, we compute the trace over $\F_\si$ under only the assumption that $\si\ph=\ph\si$. Then we find the trace over $W(\mu,\zeta)$
under the additional assumption that $\ph$ is a power of $\si$, which is always the case when $\Ga=\langle\si\rangle$. %and $\si$ is of prime order.
 
%In section \ref{Gsi}, we will restrict the order to be prime. 
%Furthermore, we assume that $Q=\bar{Q}$ (see \leref{Qbarlemma} and  \coref{corQbar1}).
%(cf.\ Lemma \ref{Qbarlemma}), i.e., that 
%\begin{equation}\label{Qbar4}
%\sum_{i=1}^{N-1}(\al|\si^i\al)\in2\ZZ\quad \text{for every}\quad \al\in Q.
%\end{equation}

%This trace can be written as the product of traces over the $\si$-twisted Fock space $\F_\si$ and the $G_\si$-module $W(\mu,\zeta)$
%, i.e., 
%\begin{equation}\label{chiMFW}
%\chi_{M(\mu,\zeta)}^{\si,\ph}(\tau,h)=\chi_{\F_\si}(\varphi, \tau)\cdot\chi_{W(\mu,\zeta)}^{\si,\ph}(\tau,h)
%\end{equation}
% (cf.\ \eqref{M}). 

%First we calculate a formula for the trace over the twisted Heisenberg algebra $\F_\si$. Then we use the theory of theta functions to calculate the trace over $W(\mu,\zeta)$ in the case when the group of automorphisms $\Gamma$ is cyclic of prime order.

%%%%%%%%%%%%%%%%%%%%%%%%%%%%%%%
\subsection{Calculating the trace over $\F_\si$}\label{strFsi}
%%%%%%%%%%%%%%%%%%%%%%%%%%%%%%%
 
 In this subsection, we only assume that $\lieh$ is a finite-dimensional complex vector space equipped with a nondegenerate symmetric bilinear form $(\cdot|\cdot)$,
 and $\si,\ph$ are two commuting automorphisms of $\lieh$ of finite order preserving $(\cdot|\cdot)$. 
 
 Any linear operator $\si$, such that $\si^N=1$ for some positive integer $N$, can be diagonalized with eigenvalues $N$-th roots of $1$ (cf.\ \eqref{j/N}).
Since $\si$ and $\ph$ commute, they can be diagonalized simultaneously. We denote by
\begin{equation}\label{hpq}
\mathfrak{h}_{x,y}=\bigl\{h\in\mathfrak{h}\,\big|\,\si h=e^{-2\pi\ii x}h, \;\ph h=e^{-2\pi\ii y}h\bigr\}
\end{equation}
the common eigenspaces for $\si$ and $\ph$.
Similarly to \eqref{s}, we define linear operators $s$ and $f$ on $\mathfrak{h}$ by
\begin{equation}\label{sf}
s(h)=-xh, \quad f(h)=-yh \qquad\text{for}\quad h\in\lieh_{x,y} \,,
\end{equation}
where the eigenvalues of $s$ and $f$ are chosen to be in the interval $(-1,0]$, i.e., $0\leq x,y<1$.
Hence, by construction,
\begin{equation}\label{siph}
\si =e^{2\pi\ii s},\qquad \ph=e^{2\pi\ii f},
\end{equation}
and $s$ coincides with the operator previously defined by \eqref{s}.

Consider the $\si$-twisted Heisenberg algebra $\hat\lieh_\si$ and its irreducible highest-weight representation, the $\si$-twisted Fock space $\F_\si$ (see \seref{twheis}).
Recall that $\varphi$ induces an automorphism of $\hat{\mathfrak{h}}_\si$ given by \eqref{phhtn}. Since $\ph$ preserves the subalgebra $\hat\lieh_\si^<$, it also induces
a linear operator $\ph$ on $\F_\si$ such that (cf.\ \eqref{twheis2}):
\begin{equation}\label{phFsi}
\ph(1)=1, \quad \ph(av) = \ph(a) \ph(v), \qquad a\in \hat\lieh_\si, \;\; v\in \F_\si \,.
\end{equation}
%(cf.\ \eqref{twheis2}).

Now we find the trace over $\F_\si$ of $\ph e^{2\pi\ii h}q^{L_0^{\tw}-\frac {r}{24}}$ for $h\in\lieh_0$. As the action of $\mathfrak{h}_0$ on $\F_\si$ is trivial, this trace simplifies to
\begin{equation}\label{chiFdef}
\chi_{\F_\si}(\ph,\tau)=\tr_{\F_\si}\ph q^{L_0^{\tw}-\frac {r}{24}} \,.
\end{equation}

\begin{theorem}\label{chiF}
For every two commuting automorphisms\/ $\si,\ph$ of\/ $\lieh$ as above, we have 
\[
\chi_{\F_\si}(\ph,\tau)=\frac{1}{P_{\si,\ph}(\tau)} \,,
\]
where
\begin{equation}\label{P}
P_{\si,\ph}(\tau)=q^{-\Delta_\si+\frac {r}{24}}{\displaystyle\prod_{m=1}^\infty}
\det\nolimits_\mathfrak{h}(1-\ph q^{m+s})
\end{equation}
and\/ $\Delta_\si$ is defined by \eqref{Delta}.
\end{theorem}
\begin{proof}
Choose a basis $\{a^i\}_{i=1}^r$ for $\lieh$ consisting of common eigenvectors for $\si$ and $\ph$, so that $a^i\in\mathfrak{h}_{s_i,f_i}$.
Then $\F_\si$ has a basis of monomials of the form
\begin{equation*}%\label{v}
w=(a^{i_1}t^{-n_1})\cdots(a^{i_k}t^{-n_k})1,
\end{equation*}
where $k\ge 0$ (with $k=0$ corresponding to $w=1$), $1\leq i_l\leq r$, $n_l\in-s_{i_l}+\ZZ$, $n_l>0$ for $1\le l\le k$, and the pairs
$(i_l,n_l)$ are ordered lexicographically.

As a special case of \eqref{L0M}, we have that 
\begin{equation*}
L_0^{\tw}w=(n_1+\cdots+n_k+\Delta_\si)w,
\end{equation*} 
while \eqref{phhtn} and \eqref{hpq}--\eqref{phFsi} imply that
\begin{equation*}
\ph w=e^{-2\pi\ii(f_{i_1}+\cdots+f_{i_k})}w.
\end{equation*} 
Therefore,
\begin{align*}
\ph q^{L_0^{\tw}-\frac {r}{24}} w &= q^{\Delta_\si-\frac {r}{24}}e^{-2\pi\ii(f_{i_1}+\cdots+f_{i_k})}q^{n_{1}+\cdots+n_{k}}w \\
&=q^{\Delta_\si-\frac {r}{24}}(e^{-2\pi\ii f_{i_1}}q^{m_1-s_{i_1}})\cdots(e^{-2\pi\ii f_{i_k}}q^{m_k-s_{i_k}})w \,,
\end{align*} 
where $m_l=n_l+s_{i_l}\in\ZZ$. Notice that $n_l>0$ and $s_{i_l}\geq0$ imply that $m_l>0$ for all $1\le l\le k$.

For each positive integer $m$ that appears in the above product, consider the set of indices $l\in\{1,\dots,k\}$ such that $m_l=m$,
and let $u_{j,m}\ge0$ be the number of such $l$ for which $i_l=j$ ($1\le j\le r$). Then we can rewrite
\begin{equation*}
\ph q^{L_0^{\tw}-\frac {r}{24}} w
= q^{\Delta_\si-\frac {r}{24}} \Bigl( \prod_{m=1}^\infty \prod_{j=1}^r (e^{-2\pi\ii f_j} q^{m-s_j})^{u_{j,m}} \Bigr) w,
\end{equation*}
where only finitely many $u_{j,m}\ne0$.
Summing over the eigenvalues of all basis vectors $w\in\F_\si$, we obtain that the trace \eqref{chiFdef} is
\begin{equation*}
\chi_{\F_\si}(\ph,\tau)=q^{\Delta_\si-\frac {r}{24}} \prod_{m=1}^\infty \prod_{j=1}^r \sum_{u_{j,m} = 0}^\infty (e^{-2\pi\ii f_j} q^{m-s_j})^{u_{j,m}} \,.
\end{equation*}
On the other hand, for every fixed $m$, we have
\begin{equation*}
\det\nolimits_\mathfrak{h}(1-\ph q^{m+s}) = \prod_{j=1}^r (1-e^{-2\pi\ii f_j} q^{m-s_j}) \,,
\end{equation*}
as the determinant is the product of eigenvalues. Using a geometric series expansion, we find
\begin{equation*}
\frac{1}{\det\nolimits_\mathfrak{h}(1-\ph q^{m+s})} = \prod_{j=1}^r \sum_{u_{j,m} = 0}^\infty (e^{-2\pi\ii f_j} q^{m-s_j})^{u_{j,m}} \,,
\end{equation*}
and a comparison with the above expression for $\chi_{\F_\si}(\ph,\tau)$ completes the proof of the theorem.
\end{proof}

The products $P_{\si,\ph}(\tau)$ can be computed more explicitly in the case when $\ph=\si^k$ and the order of $\si$ is prime, 
by use of the following lemma.

\begin{lemma}\label{lPsisik}
Suppose that\/ $\si$ has a prime order\/ $p$. Then, for all\/ $k,m\in\ZZ$, we have
\[
\det\nolimits_\mathfrak{h} \bigl(1-\si^k q^{m+s}\bigr) 
= (1-q^m)^{r_0} \prod_{j=1}^{p-1} \Bigl(1 - \om^{-jk} q^{m-\frac{j}p} \Bigr)^d \,,
\]
where\/ $r_0=\dim\lieh_0$, $d=\dim\lieh_{1/p}$, and\/ $\om=e^{2\pi\ii /p}$. %, and\/ $q=e^{2\pi\ii\tau}$.
\end{lemma}
\begin{proof}
This follows directly from the definition of $s$ (see \eqref{s}) and the proof of \leref{Deltalemma} (see \eqref{dimh}).
\end{proof}

As a consequence of \leref{lPsisik} and \reref{rem-Delta},
we see that $P_{\si^l,\si^{kl}}(\tau) = P_{\si,\si^k}(\tau)$ for all $1\le l\le p-1$,
because replacing $\si$ with $\si^l$ amounts to performing the permutation $j\mapsto lj$ mod $p$ on the index set $\{1,\dots,p-1\}$.
More generally, one can prove such an invariance without assuming that the order of $\si$ is prime:
\begin{equation}\label{Psilsikl}
P_{\si^l,\si^{kl}}(\tau) = P_{\si,\si^k}(\tau), \qquad \chi_{\F_{\si^l}}(\si^{kl},\tau) = \chi_{\F_\si}(\si^{k},\tau),
\end{equation}
for every $l$ that is coprime to the order $N$ of $\si$.

%%%%%%%%%%%%%%%%%%%%%%%%%%%%%%%
\subsection{The sublattice $R_\perp$ and a basis for $W(\mu,\zeta)$}\label{sRperp}
%%%%%%%%%%%%%%%%%%%%%%%%%%%%%%%
In this subsection and the next one, our goal is to compute the trace
\begin{equation}\label{chiWph}
\chi_{W(\mu, \zeta)}^{\si,\ph}(\tau, h)=\tr_{W(\mu, \zeta)}\ph e^{2\pi\ii h}q^{L_0^{\tw}-\Delta_\si},\qquad h\in\lieh_0,
\end{equation}
where $W(\mu,\zeta)$ is one of the irreducible $G_\si$-modules (see \thref{bij}).
The factor $q^{\Delta_\si}$ corresponding to the eigenvalue of $L_0^{\tw}$ in \eqref{L0M} is included in the trace over the Fock space \eqref{chiFdef}. We thereby subtract $\Delta_\si$ from the exponent of $q$ in \eqref{chiWph}, 
so that we have (cf.\ \eqref{chiM}, \eqref{chiFdef}):
\begin{equation}\label{chiMFW}
\chi_{M(\mu,\zeta)}^{\si,\ph}(\tau,h) = \chi_{\F_\si}(\ph,\tau) \chi_{W(\mu, \zeta)}^{\si,\ph}(\tau, h).
\end{equation}
From now on, we assume that $\bar Q=Q$ and $\ph\in\Gamma_{\si,\mu,\zeta}$ (cf.\ \leref{Qbarlemma},  \coref{corQbar1}, and \eqref{smz}).
The condition $\ph\in\Gamma_{\si,\mu,\zeta}$ ensures that we have an action of $\ph$ on $W(\mu,\zeta)$
satisfying \eqref{phgw} and \eqref{phgw2}.

Recall that a basis for $W(\mu,\zeta)$ is given in \coref{cbasisW}. If we take one of the basis vectors of the form
$U_\ga v$ ($\ga\in\C_Q$, $v\in\B_\Om$), then by \eqref{phactG}, \eqref{phgw}, we have 
\[
\ph(U_\ga v) = \ph(U_\ga) \ph(v) = \eta_\varphi(\ga)U_{\varphi\ga} \ph(v).
\]
Notice that, as $v\in\Om(\mu,\zeta)$, it has $\lieh_0$-weight $\mu$, i.e., $e^h v = e^{(h|\mu)} v$ for $h\in\lieh_0$.
Then the weight of $U_\ga v$ is $\pi_0\ga+\mu$. On the other hand, by \leref{lphWmuze}, the weight of $\ph(v)$ is $\ph(\mu)$; hence, the weight of
$\ph(U_\ga v)$ is $\pi_0\ph(\ga)+\ph(\mu)$. Thus, the only representatives $\ga\in\C_Q$ contributing to the trace \eqref{chiWph}
are those satisfying
\begin{equation}\label{condition1}
\pi_0\ga+\mu = \pi_0\ph(\ga)+\ph(\mu).
\end{equation}
Notice that this condition trivially holds in the case when $\ph=\si^k$ is a power of $\si$, because
$\si(\mu)=\mu$ and $\pi_0\circ\si=\pi_0$.

In order to provide a more explicit basis for $W(\mu,\zeta)$, we recall its alternate description given by \eqref{A}.
For that, we need to pick a maximal abelain subgroup $A_\si^\perp$ of $G_\si^\perp$, which we can do as follows.
Notice that when $\al,\be\in Q\cap\mathfrak{h}_\perp$, the commutator \eqref{Calbe} reduces to
\[
C_{\al,\be}=e^{2\pi\ii(\al_*|\be)},\qquad\al=(1-\si)\al_* \,.
\]
%where $\al=(1-\si)\al_*$. 
Hence, if we choose a maximal sublattice $R_\perp\subset Q\cap\mathfrak{h}_\perp$ with the pro\-perty that
\begin{equation}\label{Rperp}
(\al_*|\be)\in\ZZ\quad\text{for}\quad\al,\be\in R_\perp \,,
\end{equation}
then we can take
\begin{equation}\label{Aperp}
A_\si^\perp=\bigl\{c\,U_\al N_\si^\perp \,\big|\,c\in\CC^\times, \;\al\in R_\perp\bigr\} \,.
\end{equation}
Thus, a basis for $W(\mu,\zeta)$ consists of elements $U_\de 1_{\mu,\zeta}$, where $\de$ runs over a set of representatives of the cosets $Q/R_\perp$.

We can improve the above description even further by introducing the lattice (cf.\ \cite{BE}):
\begin{equation}\label{L}
L=(Q\cap\mathfrak{h}_\perp)+(Q\cap\mathfrak{h}_0) \,.
\end{equation}
Note that this sum is direct and $Q/L$ is a finite group. We have a chain of sublattices in $Q$:
\[
R_\perp\subset Q\cap\mathfrak{h}_\perp\subset L\subset Q,
\]
and $L/(Q\cap\mathfrak{h}_\perp) \cong Q\cap\mathfrak{h}_0$.
Let $\C_L \subset Q$ be a set of representatives of the cosets $Q/L$,
and $\C_R \subset Q\cap\lieh_\perp$ be a set of representatives of the cosets $(Q\cap\mathfrak{h}_\perp)/R_\perp$.
Then 
\[
\bigl\{ \al+\be+\ga \,\big|\, \al\in Q\cap\mathfrak{h}_0, \; \be\in\C_L, \; \ga\in\C_R \bigr\}
\]
is a set of representatives of the cosets $Q/R_\perp$.
We can summarize the above discussion as follows.

\begin{proposition}\label{Wbasis}
Let\/ $\C_L \subset Q$ be a set of representatives of the cosets\/ $Q/L$,
and\/ $\C_R \subset Q\cap\lieh_\perp$ be a set of representatives of the cosets\/ $(Q\cap\mathfrak{h}_\perp)/R_\perp$.
Then a basis for\/ $W(\mu,\zeta)$ consists of the elements
\[
U_\al U_{\be}U_{\ga}1_{\mu,\zeta} \qquad (\al\in Q\cap\mathfrak{h}_0, \; \be\in\C_L, \; \ga\in\C_R).
\]
\end{proposition}

Similarly, notice that $\{U_\ga 1_{\mu,\zeta} \,|\, \ga\in\C_R\}$
%the cosets $(Q\cap\mathfrak{h}_\perp)/R_\perp$ correspond to 
is a basis of the irreducible $G_\si^\perp$-module $\Om(\mu, \zeta)$, by \eqref{A1}. 
As a consequence, we obtain another formula for the defect:
\begin{equation}\label{dsiCR}
d(\si)=|(Q\cap\mathfrak{h}_\perp)/R_\perp| = |\C_R|\,,
\end{equation}
since $d(\si)=\dim\Om(\mu,\zeta)$ by Remark \ref{defect}.
Finally, we point out that the cosets of $Q/L$ can be replaced with the cosets of $\pi_0(Q)/(Q\cap\lieh_0)$, as we now show.

\begin{lemma}\label{Q/L}
For an integral lattice\/ $Q$ with an isometry\/ $\si$, define\/ $L$ by \eqref{L} and let\/ $M=Q\cap\lieh_0$. Then the map\/
$\be+L\mapsto \pi_0\be+M$ is a group isomorphism $Q/L\cong\pi_0(Q)/M$.
\end{lemma}
\begin{proof}
Consider the map $f\colon Q\to \pi_0(Q)/M$ given by $f(\be)=\pi_0\be+M$. 
Then clearly $f$ is a surjective group homomorphism and $L\subset\Ker f$. 
In order to show that $L=\Ker f$, consider $\be\in\Ker f$ so that $\al:=\pi_0\be\in M$. 
Then $\pi_0(\al-\be)=0$, which implies $\al-\be\in Q\cap\lieh_\perp$. 
Hence $\be\in L$, and the result follows from the First Isomorphism Theorem.
\end{proof}

\begin{remark}
Sections \ref{subsub}, \ref{subsubgroup}, \ref{strFsi}, and \ref{sRperp} were developed in a more general setting than needed for the rest of the paper. We hope that they will be useful for future investigations.
\end{remark}

%%%%%%%%%%%%%%%%%%%%%%%%%%%%%%%
\subsection{Calculating the trace over the $G_\si$-module $W(\mu,\zeta)$}\label{Gsi}
%%%%%%%%%%%%%%%%%%%%%%%%%%%%%%%
From now on, we will assume that $\Gamma=\langle\si\rangle$ is a cyclic group of finite order $N$ and $\ph=\si^k$ is a power of $\si$. 
Then, by \leref{lphWmuze} and \reref{rgasimuze}, we have an action of $\ph$ on $W(\mu, \zeta)$ such that $\ph(1_{\mu,\ze}) = 1_{\mu,\ze}$
and \eqref{phgw} holds. Now \eqref{chiWph} reduces to
%Note that the $G_\si$-module $W(\mu,\zeta)$ is naturally a subspace of the $V_Q$-twisted module $M(\mu,\zeta)$, which is $\lieh_0$-, $\si$-, and $L_0^{\tw}$-invariant. Hence we can define the trace function (cf.\ \eqref{W}, \eqref{L0Om})
\begin{equation}\label{chiW}
\chi_{W(\mu, \zeta)}^{\si,\si^k}(\tau, h)=\tr_{W(\mu, \zeta)}\si^k e^{2\pi\ii h}q^{L_0^{\tw}-\Delta_\si},\qquad h\in\lieh_0.
\end{equation}
In this subsection, we will express \eqref{chiW} in terms of a theta function 
%corresponding to a coset of an even integral lattice 
(recall Section \ref{theta}).

\begin{proposition}\label{chiW1}
Let\/ $Q$ be a positive-definite even integral lattice, and\/ $\si$ be an isometry of\/ $Q$ of finite order.
Then
\begin{equation}\label{chiW1a}
\chi_{W(\mu, \zeta)}^{\si,\si^k}(\tau, h)=d(\si)e^{-\pi\ii k|\mu|^2}\th_{\mu+\pi_0(Q)}(\tau+k,h, 0).
\end{equation}
As a consequence, we have
\begin{equation}\label{chiW1b}
\chi_{W(\mu, \zeta)}^{\si,\si^k}(\tau, h)=e^{-\pi\ii k|\mu|^2}\chi_{W(\mu, \zeta)}^{\si,1}(\tau+k, h).
\end{equation}
\end{proposition}
\begin{proof}
Let $w=U_\al U_{\be}U_{\ga}1_{\mu,\zeta}$ be a basis element of $W(\mu,\zeta)$ as in Proposition \ref{Wbasis}.
Using \eqref{phgw} and $\si(1_{\mu,\ze}) = 1_{\mu,\ze}$, we compute the action of $\si$ on $w$:
\begin{equation*}
\si w = \si(U_\al) \si(U_\be) \si(U_\ga) 1_{\mu,\zeta} \,.
%\eta_\ph(\al) \eta_\ph(\be) \eta_\ph(\ga) U_{\ph\al} U_{\ph\be}U_{\ph\ga}1_{\mu,\zeta} \,,
\end{equation*}
%Applying  \eqref{phactG}, we will determine each of the three factors in the right-hand side. 
First, by \eqref{twlat2} and \eqref{phactG} with $\eta=\eta_\si$, we have
\begin{equation*}
\si(U_\al)=\eta(\al) U_{\si\al} = U_\al \,, \qquad \al\in M:=Q\cap\lieh_0 \,.
\end{equation*}
%hence $\ph(U_\al)=U_\al$.
Next, we find 
\begin{equation*}
\si(U_\be)=\eta(\be) U_{\si\be} = U_ \be e^{2\pi\ii (b_{\be}+\pi_0\be)}  \,, \qquad \be\in Q \,,
\end{equation*}
since $C_\be=1$ in $G_\si$; see \eqref{Cal}. Similarly,
\begin{equation*}
\si(U_\ga) = U_ \ga e^{2\pi\ii (b_{\ga}+\pi_0\ga)} = U_ \ga \,, \qquad \ga\in Q\cap\lieh_\perp \,,
\end{equation*}
because $\pi_0\ga=0$ and $b_\ga = -|\ga|^2/2 \in\ZZ$.
Finally, by \eqref{tig2} and \eqref{Aaction}, we have
\begin{equation*}
e^{2\pi\ii \,\pi_0\be}  U_\ga 1_{\mu,\zeta} = U_\ga e^{2\pi\ii \,\pi_0\be} 1_{\mu,\zeta} 
= U_\ga e^{2\pi\ii (\be|\mu)} 1_{\mu,\zeta} \,,
\end{equation*}
as $(\pi_0\be|\ga)=0$. Putting the above together, we get
\begin{equation*}
\si w = e^{2\pi\ii (b_{\be}+(\be|\mu))} w \,.
\end{equation*}
Then the $k$-th power gives
\begin{equation*}
\si^k w = e^{2\pi\ii k(b_{\be}+(\be|\mu))} w \,.
\end{equation*}

On the other hand, 
\begin{equation*}
e^{2\pi\ii h} w = e^{2\pi\ii (h|\al+\be+\mu)} w \,, \qquad h\in\lieh_0\,,
\end{equation*}
again by \eqref{Aaction}, and
\begin{equation*}
q^{L_0^{\tw}-\Delta_\si} w = q^{\frac12|\mu+\al+\pi_0\be|^2} w,
\end{equation*}
by \eqref{L0M}, where we used $\pi_0\al=\al$, $\pi_0\ga=0$.
Therefore,
\begin{equation*}
\si^k e^{2\pi\ii h} q^{L_0^{\tw}-\Delta_\si} w = e^{2\pi\ii ( kb_{\be} + k(\be|\mu)+(h|\al+\be+\mu) )} q^{\frac12|\mu+\al+\pi_0\be|^2} w \,.
\end{equation*}
Then the trace \eqref{chiW} is the sum of the eigenvalues:
\begin{align*}
\sum_{\al\in M} \sum_{\be\in\C_L} \sum_{\ga\in\C_R}
e^{2\pi\ii ( kb_{\be} + k(\be|\mu)+(h|\al+\be+\mu) )} q^{\frac12|\mu+\al+\pi_0\be|^2}
\end{align*}

The sum over $\ga$ gives $|\C_R|=d(\si)$; see \eqref{dsiCR}. To simplify the rest, we rewrite
\begin{align*}
2b_{\be}+2(\be|\mu)&=|\pi_0\be|^2-|\be|^2+2(\be|\mu)\\
&\equiv |\pi_0\be+\mu|^2-|\mu|^2\;\mod2\ZZ\\
&\equiv |\mu+\al+\pi_0\be|^2-|\mu|^2\;\mod2\ZZ \,,
\end{align*}
where we used that $|\al|^2\in2\ZZ$, $|\be|^2\in2\ZZ$ and $(\al|\pi_0\be+\mu) = (\al|\be)+(\al|\mu) \in\ZZ$.
For the last claim, we have $(\al|\mu)=(\al|\la)\in\ZZ$ if $\mu=\pi_0\la$ for some $\la\in Q^*$.

Hence, we can express the trace as a multiple of a theta function:
\begin{align*}
\chi_{W(\mu, \zeta)}^{\si,\si^k}(\tau, h)&=d(\si)e^{-\pi\ii k|\mu|^2} \sum_{\al\in M} \sum_{\be\in\C_L} 
e^{2\pi\ii(h|\mu+\al+\pi_0\be)} e^{\pi\ii(\tau+k)|\mu+\al+\pi_0\be|^2}\\
&=d(\si)e^{-\pi\ii k|\mu|^2} \sum_{\be\in\C_L} \theta_{\mu+\pi_0\be+M}(\tau+k, h, 0)\\
&=d(\si)e^{-\pi\ii k|\mu|^2}\th_{\mu+\pi_0(Q)}(\tau+k,h, 0),
\end{align*}
where in the last equality we used the  isomorphism $Q/L\cong\pi_0(Q)/M$ from \leref{Q/L}.
\end{proof}

Next, we derive some properties of the defect $d(\si)$ that are useful when describing the characters of $\si^l$-twisted modules (cf.\ Remark \ref{defect}).  

\begin{lemma}\label{defectlemma}
Let\/ $Q$ be a positive-definite even integral lattice, and\/ $\si$ be an isometry of\/ $Q$ of order $N$. %\emph{(}not necessarily prime\emph{)}.
\begin{enumerate}[$(i)$]
\item If\/ $l$ and\/ $N$ are coprime, then\/ $d(\si^l)=d(\si).$
\item If\/ $(1-\si)Q = Q\cap\lieh_\perp$, then\/ $d(\si)=1$.
\end{enumerate}
\end{lemma}

\begin{proof}
Recall that the defect $d(\si)$ is defined by \eqref{defectdef}:
\begin{equation*}
d(\si)^2=\big|(Q\cap\lieh_\perp)/(Q\cap(1-\si)Q^*) \big| \,.
\end{equation*}
As $l$ is coprime to $N$, the space $\lieh_\perp$ is the same for $\si$ and $\si^l$. Since $1-\si$ divides $1-\si^l$, we have 
\begin{align*}
Q\cap(1-\si^l)Q^*\subset Q\cap(1-\si)Q^*.
\end{align*}
It follows that $d(\si)\leq d(\si^l)$. But $\langle\si\rangle=\langle\si^l\rangle$ similarly implies that $d(\si)\geq d(\si^l)$. Hence $d(\si)=d(\si^l).$ Finally, $(ii)$ follows from 
\[
(1-\si)Q\subset Q\cap(1-\si)Q^* \subset Q\cap\lieh_\perp \,.
\]
This completes the proof.  
\end{proof}

It will be convenient to have an expression for $\chi_{W(\mu, \zeta)}^{\si,\si^k}(\tau, h)$ in terms of a theta function corresponding to an integral lattice. 
Notice that the lattice $\pi_0(Q)$ is not integral in general. However, $\sqrt{N}\pi_0(Q)$ is an even integral lattice by Lemma \ref{Qbarlemma},
if we assume $Q=\bar{Q}$.
Using the rescaling property \eqref{thc} of theta functions, we obtain the following.

\begin{proposition}\label{chiW2}
Let\/ $Q$ be a positive-definite even integral lattice, $\si$ be an isometry of\/ $Q$ of finite order\/ $N$, and assume that\/ $Q=\bar{Q}$.
Then\/ $\sqrt{N}\pi_0(Q)$ is an even integral lattice and
\begin{equation}\label{chiW3}
\chi_{W(\mu, \zeta)}^{\si,\si^k}(\tau, h)
=d(\si)e^{-\pi\ii k|\mu|^2}\th_{\sqrt{N}\mu+\sqrt{N}\pi_0(Q)}\Bigl(\frac{\tau+k}{N},\frac{h}{\sqrt{N}}, 0\Bigr).
\end{equation}
\end{proposition}

Notice that in the left-hand side of \eqref{chiW3}, $k$ can be taken mod $N$ because $\si^N=1$. 
Using the transformation \eqref{thtlaws2}, one can check that the same is true for the right-hand side of \eqref{chiW3}.
It is also not hard to see that
\begin{equation}\label{sqrtNmu}
\sqrt{N}\mu\in \bigl(\sqrt{N} \pi_0(Q) \bigr)^* \quad\text{for all}\quad \mu\in\pi_0(Q^*) \,,
\end{equation}
where the dual of $\sqrt{N} \pi_0(Q)$ is taken in $\lieh_0$.
%because, by Lemma \ref{tdual}, 
%\[
%\mu\in\pi_0(Q^*)=M^* \,, \qquad M=Q\cap\mathfrak{h}_0 \subset\lieh_0 \,,
%\]
%and $(\sqrt{N}\mu | \sqrt{N}\pi_0\be) = (\mu| N\pi_0\be) \in\ZZ$ since 
%$N\pi_0\be %= \be+\si\be+\cdots+\si^{N-1}\be\in M$ for any $\be\in Q$.

Finally, we multiply the modified characters from Theorem \ref{chiF} and Proposition \ref{chiW2} as in \eqref{chiMFW}
to yield the modified characters of irreducible $\si$-twisted $V_Q$-modules $M(\mu,\zeta)$,
given by the following theorem.
%in the case that $\Gamma=\langle\si\rangle$ is a cyclic group.

\begin{theorem}\label{chiM2}
Let\/ $Q$ be a positive-definite even integral lattice, $\si$ be an isometry of\/ $Q$ of finite order\/ $N$, and assume that\/ $Q=\bar{Q}$.
Then the modified characters \eqref{chiM} for\/ $\ph=\si^k$ are given by
\begin{equation}\label{chiMk}
\chi_{M(\mu, \zeta)}^{\si,\si^k}(\tau, h)=d(\si)e^{-\pi\ii k|\mu|^2}\frac{\th_{\sqrt{N}\mu+\sqrt{N}\pi_0(Q)}\bigl(\frac{\tau+k}{N},\frac{h}{\sqrt{N}}, 0\bigr)}{P_{\si,\si^k}(\tau)} \,.
\end{equation}
\end{theorem}

\begin{remark}
When $\pi_0(Q)$ is an even integral lattice, we have the simpler formula
\begin{equation}\label{chiMk2}
\chi_{M(\mu, \zeta)}^{\si,\si^k}(\tau, h)=d(\si)\frac{\th_{\mu+\pi_0(Q)}\left({\tau}{},h, 0\right)}{P_{\si,\si^k}(\tau)} \,,
\end{equation}
which follows from \eqref{chiW1a} and \eqref{thtlaws2}.
\end{remark}

\begin{corollary}\label{sisame}
Suppose that\/ $l$ is coprime to the order\/ $N$ of\/ $\si$. Let\/ $\mu\in\pi_0(Q^*)$, let $\ze$ be a central character of\/ $G_\si^\perp$,
and\/ $\ze'$ be a central character of\/ $G_{\si^l}^\perp$. Denote the irreducible $\si^l$-twisted\/ $V_Q$-module corresponding to the pair\/
$(\mu,\ze')$ as $M(\mu, \zeta';\si^l)$.
Then
\begin{equation}\label{Msilsikl}
\chi_{M(\mu, \zeta';\si^l)}^{\si^l,\si^{kl}}(\tau, h)=\chi_{M(\mu, \zeta)}^{\si,\si^k}(\tau, h).
\end{equation}
\end{corollary}
\begin{proof}
First of all, note that the projection $\pi_0$ is the same for both $\si$ and $\si^l$. We have
\begin{equation}\label{Wsilsikl}
\chi_{W(\mu, \zeta')}^{\si^l,\si^{kl}}(\tau, h)=\chi_{W(\mu, \zeta)}^{\si,\si^k}(\tau, h),
\end{equation}
due to \eqref{chiW1a} and \leref{defectlemma}($i$). Then \eqref{Msilsikl} follows from \eqref{Psilsikl}, \eqref{chiMFW}, and \eqref{Wsilsikl}.
\end{proof}

\begin{remark}\label{rem4.12}
If $l$ is coprime to the order $N$ of $\si$, then $l$ has a multiplicative inverse mod $N$. Hence, for fixed $l$ coprime to $N$, all powers of $\si$ can be written
in the form $\si^{kl}$ for some $k$.
\end{remark}

Notice that the characters $\chi_{M(\mu, \zeta)}^{\si,\si^k}(\tau, h)$ are independent of the central character $\zeta$ of $G_\si^\perp$. 
In the special case when $(1-\si)Q=Q\cap\lieh_\perp$, there is a unique $\ze$ for any fixed $\mu\in\pi_0(Q^*)$ (see \reref{remark1}).
Otherwise, we will require a choice of $\ze$ when inverting formula \eqref{chiW3}, which we will do when we compute the modular transformations of characters in \seref{s5.2} below.

%\begin{remark}\label{rzedegen}
%Besides the degeneracy caused by \eqref{Msilsikl}, there is additional degeneracy of the characters $\chi_{M(\mu, \zeta)}^{\si,\si^k}(\tau, h)$, which is due to their independence of the central character $\zeta$ of $G_\si^\perp$. 
%In the special case when $(1-\si)Q=Q\cap\lieh_\perp$, condition \eqref{con} determines a unique $\ze$ for any fixed $\mu\in\pi_0(Q^*)$ (see \reref{remark1}).
%Otherwise, we will require a choice of $\ze$ when inverting formula \eqref{chiW3}, which we will do when we compute the modular transformations of characters in \seref{s5.2} below.
%\end{remark}

%%%%%%%%%%%%%%%%%%%%%%%%%%%%%%%
\subsection{Calculating the trace over the untwisted $V_Q$-module $V_{\la+Q}$}
%%%%%%%%%%%%%%%%%%%%%%%%%%%%%%%

In this subsection, we find the modified characters 
\begin{equation}\label{chiVla}
\chi_{V_{\la+Q}}^{1,\ph}(\tau,h)=\tr_{V_{\la+Q}} \ph e^{2\pi\ii h}q^{L_0-\frac {r}{24}}
\end{equation}
of irreducible untwisted $V_Q$-modules $V_{\la+Q}$, where 
$\ph$ is an isometry of finite order of the positive-definite even lattice $Q$,
$\la\in Q^*$, $h\in\lieh$, and $r=\rank Q$.
While these may be derived as a special case of previous results, 
it is instructive to present a direct calculation here.

Recall from \eqref{VFW} the construction of $V_{\la+Q}$ as $\F\otimes W(\la)$, where $\F$ is the Fock space \eqref{F} and
\begin{equation}\label{Wlambda}
W(\la)= \CC_\ep[Q] e^\la = \Span\bigl\{e^{\la+\al} \,\big|\, \al\in Q\bigr\} \,.
\end{equation}
Clearly, $W(\la)$ only depends on the coset $\la+Q \in Q^*/Q$. As in \leref{basisM},
$V_{\la+Q}$ has a basis consisting of monomials of the form
\begin{equation}\label{Vlabasis}
v = a^{i_1}_{-n_1} \cdots a^{i_k}_{-n_k} e^{\la+\al} \,,
\end{equation}
where $\al\in Q$, $\{a^i\}_{i=1}^r$ is a fixed basis for $\lieh$, $k\ge0$ (with $k=0$ corresponding to $v=e^{\la+\al}$),
$1\le i_l\le r$, $n_l\in\ZZ$, $n_l>0$ for $1\le l\le k$, and the pairs $(i_l,n_l)$ are ordered lexicographically.

Similarly to the action \eqref{twlat4}, we have an action of $\varphi$ on the vectors \eqref{Vlabasis} given by
\begin{equation}\label{phVla}
\ph(v) = \eta_\ph(\la+\al) \, \ph(a^{i_1})_{-n_1} \cdots \ph(a^{i_k})_{-n_k}  e^{\ph(\la+\al)} \,.
\end{equation}
As in \leref{lphMmuze}, we have:

\begin{lemma}\label{lphVla}
The linear map\/ $\ph\colon V_{\la+Q}\to V_{\ph(\la)+Q}$, defined by \eqref{phVla}, satisfies
\begin{equation}\label{phav2}
\ph\bigl(Y(a,z)v\bigr) = Y(\ph(a),z) \ph(v), \qquad a\in V_Q, \;\; v\in V_{\la+Q}.
\end{equation}
\end{lemma}
%\begin{proof}
%This follows immediately from \eqref{M}, \eqref{phhtn}, and \leref{lphWmuze}.
%\end{proof}

Since the action of $\ph$ sends $V_{\la+Q}$ to $V_{\ph(\la)+Q}$, the trace \eqref{chiVla} makes sense only when 
\begin{equation*}%\label{phcoset}
\la+Q = \ph(\la)+Q,  \quad\text{i.e.,}\quad (1-\ph)\la \in Q.
\end{equation*}
Let $v$ be as in \eqref{Vlabasis}. Any $h\in\lieh$ acts on $v$ as the zero mode $h_0$ according to
\begin{equation*}
h_0 v = (h|\la+\al) v
\end{equation*}
(cf.\ \eqref{lat3}, \eqref{h0M}), while the action of $L_0$ is given by
\begin{equation*}
L_0 v = \Bigl(n_1+\cdots+n_k+\frac{|\la+\al|^2}2\Bigr) v
\end{equation*}
(cf.\ \eqref{L0M}). Hence, 
\begin{equation*}
e^{2\pi\ii h}q^{L_0} v = e^{2\pi\ii(h|\la+\al)} q^{n_1+\cdots+n_k+\frac12|\la+\al|^2} v.
\end{equation*}
Therefore, the only basis vectors $v$ that contribute to the trace \eqref{chiVla} are those with
\begin{equation*}%\label{phcoset}
\ph(\la+\al) = \la+\al, \qquad \al\in Q.
\end{equation*}
If one such $\al$ exists, then we can replace the representative $\la\in Q^*$ of the coset $\la+Q$ with $\la+\al$, and assume without loss of generality that $\ph(\la)=\la$.
Then all other $\al$ lie in $Q\cap\lieh^\ph$, where $\lieh^\ph\subset\lieh$ is the subspace of fixed points under $\ph$.
Notice that when $\ph(\la+\al) = \la+\al$, we can assume $\eta_\ph(\la+\al)=1$; cf.\ \eqref{twlat2}.

Under the above conditions, the trace \eqref{chiVla} becomes a product of traces
\begin{equation}\label{chiVla2}
\chi_{V_{\la+Q}}^{1,\ph}(\tau,h) = \Bigl( \tr_{\F} \ph q^{L_0-\frac {r}{24}} \Bigr) \Bigl( \tr_{W(\la)^\ph} e^{2\pi\ii h}q^{L_0} \Bigr),
\end{equation}
where $\la\in Q^*\cap\lieh^\ph$ and
\begin{equation*}
W(\la)^\ph=\Span\bigl\{e^{\la+\al} \,\big|\, \al\in Q\cap\lieh^\ph\bigr\} \subset W(\la)
\end{equation*}
is the subspace of $\ph$-invariants.

The first factor in the right-hand side of \eqref{chiVla2} is the special case of $\chi_{\F_\si}(\ph,\tau)$ for $\si=1$ (see \eqref{chiFdef}).
From Theorem \ref{chiF}, we obtain
\begin{equation*}
\tr_{\F} \ph q^{L_0-\frac {r}{24}} = \chi_{\F}(\ph,\tau) = \frac{1}{P_{1,\ph}(\tau)} \,,
\end{equation*}
where
\begin{equation*}%\label{P1}
P_{1,\ph}(\tau)=q^{\frac {r}{24}}{\displaystyle\prod_{m=1}^\infty}
\det\nolimits_\mathfrak{h}(1-\ph q^{m}) \,.
\end{equation*}
The second factor in the right-hand side of \eqref{chiVla2} is given by
\begin{equation*}
\tr_{W(\la)^\ph} e^{2\pi\ii h}q^{L_0} = \sum_{\ga\in\la+(Q\cap\lieh^\ph)} e^{2\pi\ii(h|\ga)} q^{\frac12|\ga|^2}
= \theta_{\la+(Q\cap\lieh^\ph)}(\tau,h,0).
\end{equation*}
We summarize the answers in the following theorem.

\begin{theorem}\label{chiVlaQ}
Let\/ $Q$ be a positive-definite even integral lattice, $\ph$ be an isometry of\/ $Q$ of finite order, and\/ $\la\in Q^*$ be such that\/ $(1-\ph)\la \in Q$.
If there exists\/ $\al\in Q$ such that\/ $\ph(\la+\al) = \la+\al$, then the modified character \eqref{chiVla} is given by
\begin{equation}\label{chiVla3}
\chi_{V_{\la+Q}}^{1,\ph}(\tau,h) = \frac{\theta_{\la+\al+(Q\cap\lieh^\ph)}(\tau,h,0)}{P_{1,\ph}(\tau)} \,.
\end{equation}
If no such\/ $\al$ exists, then\/ $\chi_{V_{\la+Q}}^{1,\ph}(\tau,h) = 0$.
\end{theorem}

In the rest of the paper, we will restrict to the special case considered earlier, when $\si$ is an isometry of $Q$ of order $N$, and $\ph=\si^k$ with $k$ coprime to $N$.
In this case, since $\langle\si\rangle = \langle\si^k\rangle$, the subspace $\lieh^\ph$ of $\ph$-invariants coincides with $\lieh^\si=\lieh_0$.
Thus, \eqref{chiVla3} can be rewritten as
\begin{equation}\label{chiVM}
\chi_{V_{\la+Q}}^{1,\si^k}(\tau,h) = \frac{\theta_{\la+M}(\tau,h,0)}{P_{1,\si^k}(\tau)} \,, \qquad \la\in Q^*\cap\lieh_0, \; M=Q\cap\lieh_0.
\end{equation}
We also restate the last claim of \thref{chiVlaQ} as
\begin{equation}\label{chiVM2}
\begin{split}
\chi_{V_{\la+Q}}^{1,\si^k}(\tau,h) = 0 \quad &\text{if}\quad \la\in Q^* \,, \;\; (1-\si)\la\in Q, \\
&\text{and}\;\; (\la+Q)\cap\lieh_0 = \{0\}.
\end{split}
\end{equation}
In particular, for $\si=1$ (or directly from \eqref{chiVla3} for $\ph=1$), we recover the well-known result that the character of $V_{\la+Q}$ is
\begin{equation}\label{chiVQ}
\chi_{V_{\la+Q}}^{1,1}(\tau, h)=\frac{\theta_{\la+Q}(\tau,h,0)}{P_{1,1}(\tau)} \,,\qquad\la\in Q^* \,.
\end{equation}

For future reference, we denote the numerators of \eqref{chiVM} and \eqref{chiVQ} as follows:
\begin{align}
\label{chiWuntwk}
\chi_{W(\la)}^{1,\si^{k}}(\tau, h) &= \theta_{\la+M}(\tau,h,0), & \la&\in Q^*\cap\lieh_0 \,, \\
\label{chiWuntw}
\chi_{W(\la)}^{1,1}(\tau, h) &= \theta_{\la+Q}(\tau,h,0), & \la&\in Q^* \,.
\end{align}
We can compute the denominators %of \eqref{chiVM} and \eqref{chiVQ}
by applying the next lemma, which is similar to \leref{lPsisik}.

\begin{lemma}\label{lPsisik2}
Suppose that\/ $\si$ has a prime order\/ $p$. Then, for\/ $m\in\ZZ$ and\/ $1\le k\le p-1$, we have
\begin{align*}
\det\nolimits_\mathfrak{h} \bigl(1- q^{m}\bigr) &= (1-q^m)^r \,, \\
\det\nolimits_\mathfrak{h} \bigl(1-\si^k q^{m}\bigr) &= (1-q^{mp})^d \, (1-q^m)^{r-dp} \,,
\end{align*}
where\/ $r=\dim\lieh$ and\/ $d=\dim\lieh_{1/p}$. %, and\/ $\om=e^{2\pi\ii /p}$. %, and\/ $q=e^{2\pi\ii\tau}$.
\end{lemma}
\begin{proof}
The first equation is obvious. For the second one, we compute using \eqref{dimh}:
\[
\det\nolimits_\mathfrak{h} \bigl(1-\si^k q^{m}\bigr) 
= (1-q^m)^{r_0} \prod_{j=1}^{p-1} \Bigl(1 - \om^{-jk} q^{m} \Bigr)^d \,,
\]
where $r_0=\dim\lieh_0$ and $\om=e^{2\pi\ii /p}$.
Notice that the map $j\mapsto -jk$ mod $p$ is a permutation of the set $\{1,\dots,p-1\}$;
hence, we can replace $\om^{-jk}$ with $\om^j$ in the above product.
Then apply the polynomial identity
\[
\frac{x^p-1}{x-1} = x^{p-1}+\dots+x+1 =  \prod_{j=1}^{p-1} (x-\om^j)
\]
to finish the proof.
\end{proof}

In particular, we see from \leref{lPsisik2} that
\begin{equation}\label{Psik1}
P_{\si^k,1}(\tau) = P_{\si,1}(\tau), \qquad 
\chi_{V_{\la+Q}}^{1,\si^k}(\tau,h) = \chi_{V_{\la+Q}}^{1,\si}(\tau,h),
\end{equation}
for all $1\le k\le p-1$.

%%%%%%%%%%%%%%%%%%%%%%%%%%%%%%%
\subsection{Irreducible orbifold modules and their characters}
%%%%%%%%%%%%%%%%%%%%%%%%%%%%%%
Let, as before, $Q$ be a positive-definite even integral lattice and $\si$ be an isometry of $Q$ such that $Q=\bar{Q}$ (cf.\ \leref{Qbarlemma} and  \coref{corQbar1}).
%From now on, for simplicity, we will assume that the order $N=p$ of $\si$ is prime. 
%This ensures that $\Ga=\langle\si\rangle = \langle\si^l\rangle$ for any $\si^l\ne 1$. 
In this subsection, we derive the classification of irreducible modules over the orbifold subalgebra $V_Q^\si$ from the general results of \cite{DRX},
assuming that the order $N=p$ of $\si$ is prime. Then we obtain their characters from the characters of twisted and untwisted $V_Q$-modules.
In order to state the classification, we need to introduce some notation. 

First, recall that if $\si$ acts as a linear operator of order $N$ on a vector space $W$, then $\si$ is
diagonalizable with eigenvalues $N$-th roots of $1$, and the projections onto the eigenspaces of $\si$ are given by:
\begin{equation}\label{pijproj}
\pi_j = \frac1N \sum_{k=0}^N \om^{jk} \si^k \qquad \bigl(0\le j\le N-1, \;\; \om=e^{2\pi\ii/N}\bigr).
\end{equation}
We denote by $W^j=\pi_j(W)$ the eigenspace with eigenvalue $\om^{-j}$. Compared with our earlier notation, we have
$\lieh^j=\lieh_{j/N}$ (see \eqref{j/N}, \eqref{eqpi0}).

Second, assuming that $N=p$ is prime, if $\si$ acts on a finite set, then every orbit has order either $1$ or $p$.
In particular, this applies to the finite set $Q^*/Q$. The set of singleton orbits is the set of fixed points $(Q^*/Q)^\si$, and it consists of
$\la+Q$ with $\la\in Q^*$ such that $(1-\si)\la\in Q$. 
We choose a set $\mathcal{O} \subset Q^*/Q$ of representatives of the orbits of order $p$.
%will be denoted $\mathcal{O}$, and $[\la+Q]\in\mathcal{O}$ will designate the orbit of $\la+Q$ for $\la\in Q^*$.

Finally, again for $N=p$ prime, every $\si^l$ with $1\le l\le p-1$ is a generator of the cyclic group $\Ga=\langle\si\rangle$. 
Hence, the projection $\pi_0$ is the same both for $\si$ and $\si^l$, and $\lieh^\si=\lieh^{\si^l}=\lieh_0$.
By \thref{thm:BK}, the irreducible $\si^l$-twisted $V_Q$-modules are classified by pairs $(\mu,\ze)$, where
$\mu\in\pi_0(Q^*)$ and $\ze$ is a central character of $G_{\si^l}^\perp$ satisfying the analog of \eqref{con} with
$\si^l$ in place of $\si$:
\begin{equation}\label{con-l}
e^{2\pi\ii(\ga|\mu)} \zeta\bigl(U^{-1}_{\si^l\ga}U_{\ga}\bigr)=\eta_{\si^l}(\ga)e^{-2\pi\ii b_\ga} \,,
\qquad \ga\in Q.
\end{equation}

\begin{definition}\label{dZmu}
For fixed $\mu\in\pi_0(Q)^*$ and $1\le l\le p-1$, 
let $\Z_{\mu,\si^l}$ be the set of all central characters $\zeta$ of $G_{\si^l}^\perp$ that satisfy relation \eqref{con-l}. 
For $l=1$, we will use the shorter notation $\Z_\mu=\Z_{\mu,\si}$.
\end{definition}

\begin{proposition}\label{p-con-l}
Let\/ $\mu\in\pi_0(Q^*)$, and\/ $1\le k,l\le p-1$ be such that\/ $kl\equiv 1 \mod p$.
Then the map\/ $\ze\mapsto\ze^k$ is a bijection\/ $\Z_\mu\to\Z_{\mu,\si^l}$.
Moreover, $\ze^p=1$ for all\/ $\ze\in\Z_\mu$.
\end{proposition}
\begin{proof}
First, note that as in the proof of \leref{defectlemma}, we have 
\[
(1-\si)Q^* = (1-\si^l) Q^* \,, \quad (1-\si)(Q\cap\lieh_\perp) = (1-\si^l)(Q\cap\lieh_\perp) \,.
\]
Hence, we can identify the centers $Z(G_{\si^l}^\perp)$ and $Z(G_\si^\perp)$ due to \eqref{ZGsiperp}.
Furthermore, since 
\[
p\,\pi_\perp = \sum_{j=0}^{p-1} (1-\si^j) = (1-\si) \sum_{i=0}^{p-2} (p-1-i) \si^i \,,
\]
we have that 
\[
p(Q \cap (1-\si)Q^*) \subset p(Q\cap\lieh_\perp) \subset (1-\si)(Q\cap\lieh_\perp) \,.
\]
Therefore, every element of $Z(G_\si^\perp)$
has order $1$ or $p$. This implies $\ze^p=1$ for $\ze\in\Z_\mu$.

It is enough to show that $\ze^k\in\Z_{\mu,\si^l}$ for $\ze\in\Z_\mu$, because then the inverse map
$\Z_{\mu,\si^l} \to\Z_\mu$ will be given by $\ze'\mapsto\ze^{\prime\, l}$.
In order to check that $\ze^k$ satisfies \eqref{con-l},
we find $\zeta(U^{-1}_{\si^l\ga}U_{\ga})$ by writing
\[
U^{-1}_{\si^l\ga}U_{\ga} 
= \bigl( U^{-1}_{\si^l\ga}U_{\si^{l-1}\ga} \bigr) \bigl( U^{-1}_{\si^{l-1}\ga}U_{\si^{l-2}\ga} \bigr) \cdots
\bigl( U^{-1}_{\si\ga}U_{\ga} \bigr)
\]
and applying $\ze$ to the product. We obtain
\begin{align*}
\ze\bigl( U^{-1}_{\si^l\ga}U_{\ga} \bigr) &= \prod_{i=0}^{l-1} \eta(\si^i\ga) e^{-2\pi\ii(\si^i\ga|\mu)} e^{-2\pi\ii b_{\si^i\ga}} \\
&= e^{-2l\pi\ii(\ga|\mu)} e^{-2l\pi\ii b_\ga} \prod_{i=0}^{l-1} \eta(\si^i\ga) \,,
\end{align*}
using \eqref{con}, $\si\mu=\mu$ and $b_{\si\ga}=b_\ga$. On the other hand, since $Q=\bar Q$, we get from \eqref{Qbar2} that
\[
\eta_{\si^l}(\ga) = \prod_{i=0}^{l-1} \eta(\si^i\ga) \,.
\]
Hence,
\[
\ze\bigl( U^{-1}_{\si^l\ga}U_{\ga} \bigr) = e^{-2l\pi\ii(\ga|\mu)} e^{-2l\pi\ii b_\ga} \eta_{\si^l}(\ga) \,.
\]
To finish the proof, we raise this identity to the $k$-th power and use that
\[
e^{2p\pi\ii(\ga|\mu)} = 1 \,, \qquad  \bigl( \eta(\ga) e^{-2\pi\ii b_\ga} \bigr)^p = 1 \,,
\]
which follow from $p\mu\in p\,\pi_0(Q^*) \subset Q^*$ and $\ze^p=1$.
\end{proof}

We will denote by $M(\mu,\zeta; \si^l)$ the irreducible $\si^l$-twisted $V_Q$-module corresponding to the pair $(\mu,\ze)$,
where $\mu\in\pi_0(Q^*)$ and $\ze\in\Z_{\mu,\si^l}$.
Again by \thref{thm:BK}, two such modules are isomorphic if and only if
the corresponding pairs $(\mu,\ze)$ and $(\mu',\ze')$ are related by the following analog of \eqref{equivrel}:
\begin{equation}\label{equivrel-l}
\mu'=\mu+\pi_0\al, \qquad \zeta'(U_\be) = e^{-2\pi\ii (\pi_0\nu|\al)} \zeta(U_\be),
\end{equation}
for some fixed $\al\in Q$ and all $\be=(1-\si^l)\nu\in Q$ with $\nu\in Q^*$.

\begin{remark}\label{rsi-l}
Note that \eqref{equivrel-l} is equivalent to \eqref{equivrel} with $\si^l$ playing the role of $\si$, due to
the alternative expression \eqref{Calbe1} for $C_{\al,\be}$.
%\begin{equation*}%\label{Calbe2}
%C_{\al,\be} = e^{2\pi\ii (\pi_0\la|\al)} \,, \qquad \al\in Q, \; \be=(1-\si)\la\in Q, \; \la\in Q^* \,.
%\end{equation*}
In particular, \eqref{Calbe1} implies that $C_{\al,\be}$ is a $p$-th root of $1$, because $p\,\pi_0\la\in Q^*$.
If $\be=(1-\si^l)\nu$ for some $\nu\in Q^*$, then $\be=(1-\si)\la$ with $\la=(1+\si+\dots+\si^{l-1})\nu\in Q^*$.
%In fact, $(1-\si)Q^* = (1-\si^l) Q^*$, as in the proof of \leref{defectlemma}.
Hence, $\pi_0\la = l\,\pi_0\nu$, and $C_{\al,\be}$ for $\si$ is the $l$-th power of $C_{\al,\be}$ for $\si^l$.
This implies that the bijection $\Z_\mu\to\Z_{\mu,\si^l}$ from \prref{p-con-l} is compatible with the
equivalence relations \eqref{equivrel} and \eqref{equivrel-l}.
\end{remark}

Now we can formulate the classification theorem.

\begin{theorem}\label{class}
Let\/ $Q$ be a positive-definite even integral lattice, $\si$ be an isometry of\/ $Q$ of prime order\/ $p$, and assume that\/ $Q=\bar{Q}$
$($which holds if\/ $p$ is odd, due to \coref{corQbar1}$)$.
Choose a set\/ $\C_\M\subset \pi_0(Q^*)$ of representatives of the cosets\/ $\pi_0(Q^*) / \pi_0(Q)$,
and a set $\mathcal{O} \subset Q^*/Q$ of representatives of the orbits of order\/ $p$ of\/ $\si$ on\/ $Q^*/Q$.
Then the following is a complete list of non-isomorphic irreducible modules over the orbifold algebra\/ $V_Q^\si${\upshape{:}}
\begin{enumerate}[\upshape(1)]
\medskip
\item[{\upshape{(Type 1)}}] 
$V_{\la+Q}^j$ \; $(\la+Q\in (Q^*/Q)^\si, \; 0\le j\le p-1),$

\medskip
\item[{\upshape{(Type 2)}}] 
$V_{\la+Q}$ \; $(\la+ Q\in\mathcal{O}),$

\medskip
\item[{\upshape{(Type 3)}}] 
$M(\mu,\zeta; \si^l)^j$ \; $(\mu\in\C_\M,\, \ze\in\Z_{\mu,\si^l}, \, 0\le j\le p-1, \, 1\leq l\le p-1).$
\end{enumerate}
\medskip
The characters of these modules are given by$:$
%{\upshape{(}}cf.\ \eqref{chiMk}, \eqref{chiVM}, \eqref{chiVQ}{\upshape{):}}
\begin{enumerate}[\upshape(1)]
\medskip
\item[{\upshape{(Type 1)}}] 
$\chi_{V_{\la+Q}}^j(\tau, h) = %\displaystyle\frac{1}{p} \chi_{V_{\la+Q}}^{1,1}(\tau, h) + 
\displaystyle\frac{1}{p} \sum_{k=0}^{p-1}\om^{jk} \chi_{V_{\la+Q}}^{1,\si^k}(\tau, h)$
\qquad $($cf.\ \eqref{chiVM}$),$

\medskip
\item[{\upshape{(Type 2)}}] 
$\chi_{V_{\la+Q}}(\tau, h)=\chi_{V_{\la+Q}}^{1,1}(\tau, h)$
\qquad $($cf.\ \eqref{chiVQ}$),$

\medskip
\item[{\upshape{(Type 3)}}] 
$\chi_{M(\mu,\zeta;\si^l)}^j(\tau, h)=\displaystyle\frac{1}{p}\sum_{k=0}^{p-1}\om^{jk}\chi_{M(\mu,\zeta;\si^l)}^{\si^l,\si^{k}}(\tau, h)$
\qquad $($cf.\ \eqref{Msilsikl}, \eqref{chiMk}, and \reref{rem4.12}$),$
\end{enumerate}
\medskip
where\/ $\om=e^{2\pi\ii/p}$.
\end{theorem}

\begin{proof}
Recall that the lattice vertex algebra $V_Q$ is regular (see Definition \ref{reg} and Theorem \ref{DLM2}). Then the orbifold subalgebra $V_Q^\Gamma=V_Q^\si$ is also regular for $\Gamma=\langle\si\rangle$, by Theorem \ref{MC}.
Due to \cite[Theorem 3.3]{DRX}, every irreducible $V_Q^\si$-module is a submodule of some irreducible $\si^l$-twisted $V_Q$-module $W$ for $0\le l\le p-1$.

The case $l=0$ corresponds to untwisted $V_Q$-modules; hence, $W\cong V_{\la+Q}$ for some $\la+Q\in Q^*/Q$, by Dong's Theorem \cite{D1}.
By \leref{lphVla}, for any $\ph\in\Ga$, we have a linear map $\ph\colon V_{\la+Q}\to V_{\ph(\la)+Q}$ satisfying \eqref{phav2}, which implies that $\ph$ is a homomorphism of $V_Q^\si$-modules.
If $\la+Q\in (Q^*/Q)^\si$, then $\si$ acts on $V_{\la+Q}$, and every eigenspace $V_{\la+Q}^j$ is a $V_Q^\si$-submodule ($0\le j\le p-1$).
As $V_{\la+Q}^j = \pi_j(V_{\la+Q})$, its character is
\begin{align*}
\tr_{V_{\la+Q}^j} e^{2\pi\ii h} q^{L_0-\frac{r}{24}} &= \tr_{V_{\la+Q}} \pi_j e^{2\pi\ii h} q^{L_0-\frac{r}{24}} \\
&=\frac{1}{p}\sum_{k=0}^{p-1}\om^{jk} \tr_{V_{\la+Q}} \si^k e^{2\pi\ii h} q^{L_0-\frac{r}{24}} \,.
\end{align*}
If $\la+Q \not\in (Q^*/Q)^\si$, then $\si\colon V_{\la+Q}\to V_{\si(\la)+Q}$ is an isomorphism of $V_Q^\si$-modules;
hence, we can assume that $\la+Q\in\O$.

When $1\le l\le p-1$, we have an irreducible $\si^l$-twisted $V_Q$-module $W$. By \thref{thm:BK}, $W\cong M(\mu,\zeta; \si^l)$ for some $\mu\in\pi_0(Q^*)$ and $\ze\in\Z_{\mu,\si^l}$.
Since pairs $(\mu,\ze)$ and $(\mu',\ze')$ that are related by \eqref{equivrel-l} correspond to isomorphic modules, we can arrange that $\mu\in\C_\M$.
By \leref{lphMmuze} and \reref{rgasimuze}, any $\ph\in\Ga$ acts on $M(\mu,\zeta; \si^l)$ so that \eqref{phav} holds.
Hence, $\ph$ is a homomorphism of $V_Q^\si$-modules, and the eigenspaces $M(\mu,\zeta; \si^l)^j$ are $V_Q^\si$-submodules ($0\le j\le p-1$).
Their characters are found similarly as the characters of $V_{\la+Q}^j$ above.

We have shown that every irreducible $V_Q^\si$-module is a submodule of one of the listed modules. To finish the proof of the theorem, we need to check that the listed modules are themselves irreducible and non-isomorphic to each other.
We will derive these claims again from \cite[Section 3]{DRX}. %The setting in \cite{DRX} is more general than ours. 
We specialize the setting of \cite{DRX} to the case when their group $G=\Gamma=\langle\si\rangle$ is cyclic of prime order. 

In \cite{DRX}, for any $\psi\in\Ga$ and any irreducible $\psi$-twisted $V_Q$-module $W$, they introduce the subgroup $\Ga_W\subset\Ga$ consisting of $\ph\in\Ga$ such that $W\cong W\circ\ph$ as twisted $V_Q$-modules.
Here $W\circ\ph$ is defined similarly to our actions from Lemmas \ref{lphMmuze} and \ref{lphVla}.
Since the order of $\Ga$ is prime, we have $\Ga_W=\{1\}$ if the automorphism $\si$ does not act on the module $W$, and $\Ga_W=\Ga$ if it does.

Then the module $W$ is decomposed under the action of $\Ga_W$; in fact, more generally, under the group algebra of $\Ga_W$ twisted by a certain $2$-cocycle.
However, when $\Ga_W=\Ga$ is cyclic of order $p$, the second cohomology 
$\mathrm{H}^2(\Ga,\CC^\times)\cong\CC^\times/(\CC^\times)^p$ is trivial. Therefore, the twisted group algebra is the usual group algebra $\CC[\Ga]$, 
whose characters correspond to $p$-th roots of unity. Then the decompositions in \cite{DRX} correspond to our decompositions into eigenspaces of $\si$ in the case when $\si$ acts on the module $W$.
\end{proof}

\begin{corollary}\label{KacProp}
The conformal weights of all irreducible\/ $V_Q^\si$-modules are positive, except for\/ $V_Q^\si$ itself, which has conformal weight\/ $0$.
\end{corollary}
\begin{proof}
By \thref{class}, every irreducible $V_Q^\si$-module is a submodule of some twisted or untwisted irreducible $V_Q$-module. Due to \prref{pL0tw}, all of them have positive conformal weights, 
except that the vacuum vector in $V_Q$ has conformal weight $0$.
\end{proof}

Notice that, although the irreducible $V_Q^\si$-modules listed in \thref{class} are non-isomorphic to each other, many of them have equal characters. 
Indeed, by \eqref{Msilsikl}, we have
\begin{equation}\label{degchar1}
\chi_{M(\mu,\zeta';\si^l)}^j(\tau, h) = \chi_{M(\mu,\zeta)}^{jl}(\tau, h),
\end{equation}
for all 
\[
\mu\in\pi_0(Q^*), \; \ze\in\Z_\mu, \; \ze'\in\Z_{\mu,\si^l}, \; 0\le j\le p-1, \; 1\le l\le p-1.
\]
In particular, all these characters are independent of $\ze$.
Similarly, we can derive from \eqref{Psik1} that
\begin{align}\label{degchar2}
\chi_{V_{\la+Q}}^0(\tau, h) &=  \frac1p \chi_{V_{\la+Q}}^{1,1}(\tau, h) + \frac{p-1}p \chi_{V_{\la+Q}}^{1,\si}(\tau, h), \\
\chi_{V_{\la+Q}}^j(\tau, h) &= \frac1p \chi_{V_{\la+Q}}^{1,1}(\tau, h) - \frac1p \chi_{V_{\la+Q}}^{1,\si}(\tau, h),
\label{degchar3}
\end{align}
for all $\la+Q\in (Q^*/Q)^\si$ and $1\le j\le p-1$.
However, when we find the modular transformations of characters in the next section, we will retain the different labels prescribed by \thref{class}
as a bookkeeping device.

%%%%%%%%%%%%%%%%%%%%%%%%%%%%%%%
\section{Transformation Laws for Modified Characters of Twisted $V_Q$-modules%$\chi_{M(\mu,\zeta)}(\ph,\tau,h)$
}\label{transfVQ}
%%%%%%%%%%%%%%%%%%%%%%%%%%%%%%%
As before, let $Q$ be a positive-definite even integral lattice, and $\si$ be an isometry of $Q$ of prime order $p$ such that $Q=\bar{Q}$
(recall that the last assumption is superfluous for odd $p$, by \coref{corQbar1}). 
In this section, we derive how the modified characters of twisted and untwisted $V_Q$-modules change under the modular transformations $\tau\mapsto\tau+1$
and $\tau\mapsto -1/\tau$. Then from \thref{class}, we obtain the transformation laws for the characters of irreducible modules over the orbifold subalgebra $V_Q^\si$.

%%%%%%%%%%%%%%%%%%%%%%%%%%%%%%%
\subsection{Transformation laws for modified characters of $\F_\si$}\label{chiFtrans}
%%%%%%%%%%%%%%%%%%%%%%%%%%%%%%%
In this subsection,  we calculate the modular transformations $\tau\mapsto\tau+1$ and $\tau\mapsto -1/\tau$ for the modified characters 
\begin{equation}\label{chiFdef2}
\chi_{\F_\si}(\ph,\tau)=\tr_{\F_\si}\ph q^{L_0^{\tw}-\frac {r}{24}} 
\end{equation}
of the twisted Fock space $\F_\si$ (cf.\ \eqref{twheis2}, \eqref{chiFdef}),
where $r=\rank Q$ and $q=e^{2\pi\ii\tau}$ with $\tau\in\CC$, $\text{Im}\,\tau>0$.
As in Section \ref{strFsi},  we assume only that $\lieh$ is a finite-dimensional complex vector space equipped with a nondegenerate symmetric bilinear form $(\cdot|\cdot)$,
 and $\si,\ph$ are two commuting automorphisms of $\lieh$ of finite order preserving $(\cdot|\cdot)$. 
 
At this point, it will be convenient to set some additional notation for subspaces of $\mathfrak{h}$.  Recall the notation \eqref{hpq}--\eqref{siph} for the common eigenspaces of $\si$ and $\varphi$,  where the first subscript corresponds to $\si$ and the second subscript to $\varphi$.  In particular,  $\mathfrak{h}_{0,0}$ is the invariant subspace for both $\si$ and $\varphi$.  
We designate the perpendicular subspaces by using a subscript $\perp$. 
Then we have
\begin{equation}\label{hnotation}
\begin{aligned}
\mathfrak{h}_0&=\mathfrak{h}_{0,0} \oplus \mathfrak{h}_{0,\perp},\qquad\mathfrak{h}_\perp=\mathfrak{h}_{\perp,0} \oplus \mathfrak{h}_{\perp,\perp},
\end{aligned}
\end{equation}
where using one subscript refers to only the automorphism $\si$, and
\begin{equation}\label{hnotation2}
\mathfrak{h}_{0,0}^\perp=\mathfrak{h}_{0,\perp} \oplus \mathfrak{h}_{\perp, 0} \oplus \mathfrak{h}_{\perp,\perp}.
\end{equation}
We note that in the special case when $\varphi$ is a power of $\si$,  the subspaces $\mathfrak{h}_{0,\perp}$ and $\mathfrak{h}_{\perp,0}$ are trivial, and \eqref{hnotation}, \eqref{hnotation2} simplify to $\mathfrak{h}_0=\mathfrak{h}_{0,0}$ and $\mathfrak{h}_0^\perp=\mathfrak{h}_\perp=\mathfrak{h}_{\perp,\perp}$.

Recall that, by \thref{chiF}, $\chi_{\F_\si}(\ph, \tau)$ is related to the products $P_{\si,\ph}(\tau)$ given in \eqref{P}. 
In order to find the transformation laws of $P_{\si,\ph}(\tau)$, we will express $P_{\si,\ph}(\tau)^2$ in terms of another function
\begin{equation}\label{PP}
P(\tau,\zeta)=q^{1/12}\prod_{n=1}^\infty(1-e^{-2\pi\ii\zeta}q^n)(1-e^{2\pi\ii\zeta}q^{n-1}).
\end{equation}
%and then compute the transformation laws for $P(\tau,\zeta)$
We will relate $P(\tau,\zeta)$ to the functions $K_l(\tau,\zeta;m)$
%of the form
%\begin{equation}
%K_l(\tau,\zeta;m)=\frac1{\eta(\tau)}\theta_{\frac lm\al+\ZZ\al} \Bigl(\tau,\frac{\zeta}{2}\al,0\Bigr)
%\end{equation}
%(cf.\ \eqref{1Dth}) 
introduced at the end of Section \ref{theta} as the quotient of a theta function and the Dedekind $\eta$-function (see \eqref{etadef}, \eqref{1Dth}, \eqref{K}).  Then we will use the transformation laws in Proposition \ref{Ktlaws} to obtain the transformation laws for $P_{\si,\ph}(\tau)^2$ and ultimately for $\chi_{\F_\si}(\varphi,\tau)$. 
%After presenting the general results, we work to simplify these results to case when $\varphi$ is a power of $\si$ (or when $\Gamma=\langle\si\rangle$), which will then be used in the following sections.

%We begin by computing $P_{\si,\ph}(\tau)^2$ and writing the result in terms of $P(\tau,\zeta)$ for a specific value of $\zeta$.

\begin{lemma}\label{P3}
Define the maps $s$ and $f$ by \eqref{sf}, and let\/ $P_{\si,\ph}(\tau)$ be the function defined by \eqref{P}. Then
\[
P_{\si,\ph}(\tau)^2=\frac{\eta(\tau)^{2r_{0,0}}}{\det\nolimits_{\mathfrak{h}_{0, \perp}}(1-\ph^{-1})}\det\nolimits_{\mathfrak{h}^\perp_{0,0}}e^{\pi\ii\tau(s^2+s)}P(\tau,-f-\tau s),
\]
where $r_{0,0}=\dim\mathfrak{h}_{0,0}$. %{\upshape{(}}cf.\  \eqref{P}\upshape{)}.
\end{lemma}
\begin{proof}
We first write
\begin{equation}
P_{\si,\ph}(\tau)=q^{r/24}\det\nolimits_\mathfrak{h}\Bigl(e^{\pi\ii\tau(s^2+s)/2}\prod_{n=1}^\infty(1-e^{2\pi\ii(f+\tau s)}q^n)\Bigr),
\end{equation}
using \eqref{Delta}, \eqref{siph}, and \eqref{P}. 
%Recall that $\mathfrak{h}_0$ and $\mathfrak{h}_\perp$ are with respect to $\si$.
%and that $q^{r/24}=\det\nolimits_\mathfrak{h}e^{\pi\ii\tau/12}$. 
Then we have
\begin{align*}
P_{\si,\ph}(\tau)^2&=q^{r/12}\det\nolimits_\mathfrak{h}e^{\pi\ii\tau(s^2+s)}
\prod_{n=1}^\infty\det\nolimits_{\mathfrak{h}_0}(1-e^{2\pi\ii f}q^n)(1-e^{-2\pi\ii f}q^n)\\
&\quad\times \prod_{n=1}^\infty\det\nolimits_{\mathfrak{h}_\perp}(1-e^{2\pi\ii(f+\tau s)}q^n)(1-e^{2\pi\ii(-f+\tau (-s-1))}q^n) \\
%P_{\si,\ph}(\tau)^2 
&= q^{r/12} \det\nolimits_\mathfrak{h}e^{\pi\ii\tau(s^2+s)} \prod_{n=1}^\infty (1-q^n)^{2r_{0,0}} \\
&\quad\times\det\nolimits_{\mathfrak{h}_{0,\perp}}\frac{P(\tau,-f)}{q^{1/12}(1-e^{-2\pi\ii f})}\det\nolimits_{\mathfrak{h}_{\perp}}\frac{P(\tau,-f-\tau s)}{q^{1/12}} \\
%P_{\si,\ph}(\tau)^2
&=\frac{\eta(q)^{2r_{0,0}}\det\nolimits_\mathfrak{h}e^{\pi\ii\tau(s^2+s)}}{\det\nolimits_{\mathfrak{h}_{0,\perp}}(1-\ph^{-1})}\det\nolimits_{\mathfrak{h}_{0,0}^\perp}P(\tau,-f-\tau s),
\end{align*}
using that $r_{0,0}=r-\dim\mathfrak{h}_\perp- \dim\mathfrak{h}_{0,\perp}$ (cf.\ \eqref{hnotation}), and $\dim\mathfrak{h}_{j/N}=\dim\mathfrak{h}_{1-(j/N)}$
for $1\le j\le N-1$.
\end{proof}

Before moving on, we note some properties of the function $P(\tau,\zeta)$,
each of which is easy to verify.
\begin{lemma}
For the function $P(\tau,\zeta)$ given in \eqref{PP}, we have
\begin{align}
P(\tau,\zeta)&=\displaystyle\frac{1-e^{2\pi\ii\zeta}}{1-e^{-2\pi\ii\zeta}}P(\tau,-\zeta),\\
P(\tau,\zeta+1)&=P(\tau,\zeta),\\
P(\tau,\zeta+\tau)&=P(\tau,-\zeta).
\end{align}
\end{lemma}

Next, we 
%compute the transformation laws for the function $P(\tau,\zeta)$ using  Proposition \ref{Ktlaws}.
%In order to first 
express $P(\tau,\zeta)$ in terms of the functions $K_l(\tau,\zeta;m)$, defined by \eqref{K}.
%To do so, it will be useful to write $P(\tau,\zeta)$ in terms of a summation instead of a product.  
To this end, we employ the Jacobi triple product identity,
\[
\prod_{n=1}^\infty(1-q^n)(1-zq^{n-1})(1-z^{-1}q^n)=\sum_{m\in\ZZ}(-z)^mq^{m(m-1)/2},
\]
which is equivalent to
\[
q^{-1/12}P(\tau,\zeta) \prod_{n=1}^\infty(1-q^n)=\sum_{m\in\ZZ}(-z)^mq^{m(m-1)/2},\qquad z=e^{2\pi\ii\zeta}.
\]
It follows that $P(\tau,\zeta)$ can be written in summation form as
\begin{equation}\label{P2}
P(\tau,\zeta)=\frac{q^{1/8}}{\eta(\tau)}\sum_{m\in\ZZ}(-z)^mq^{m(m-1)/2},\qquad z=e^{2\pi\ii\zeta}.
\end{equation}

%The transformation laws for $P(\tau,\zeta)$ are determined by relating it to the theta functions $K_l(\tau,\zeta;m)$ (cf.\  \eqref{K}) and using Proposition \ref{Ktlaws}.
%We are now able to express $P(\tau,\zeta)$ in terms of the functions $K_l(\tau,\zeta;m)$.

\begin{lemma}
The function $P(\tau,\zeta)$ defined in \eqref{PP} can be written as
\[
P(\tau,\zeta)=e^{\pi\ii\zeta}\bigl(K_{-1}(\tau,\zeta;4)-K_1(\tau,\zeta;4)\bigr).
\]
%where
\end{lemma}
\begin{proof}
From the definitions \eqref{1Dth} and \eqref{K}, we have
\[
K_l(\tau,\zeta;4)=\frac{1}{\eta(\tau)}\sum_{n\in\ZZ}
%e^{4\pi\ii\left((n+\frac{l}{4})^2\tau+(n+\frac l4)\zeta\right)}
q^{2(n+l/4)^2}z^{2(n+l/4)},
\]
where $z=e^{2\pi\ii\zeta}$. Therefore
\begin{align*}
e^{\pi\ii\zeta}&\bigl(K_{-1}(\tau,\zeta;4)-K_1(\tau,\zeta;4)\bigr)=\frac{q^{1/8}}{\eta(\tau)}\sum_{n\in\ZZ}\bigl(q^{2n^2-n}z^{2n}-q^{2n^2+n}z^{2n+1}\bigr)\\
&=\frac{q^{1/8}}{\eta(\tau)}\sum_{n\in\ZZ}\bigl((-z)^{2n}q^{2n(2n-1)/2}+(-z)^{2n+1}q^{(2n+1)((2n+1)-1)/2}\bigr)\\
&=P(\tau,\zeta),
\end{align*}
using \eqref{P2}.
\end{proof}

%Now that $P(\tau,\zeta)$ is expressed in terms of the functions $K_l(\tau,\zeta;m)$, 
Now we can use Proposition \ref{Ktlaws} to calculate the following transformation laws for $P(\tau,\zeta)$.

\begin{proposition}\label{tlawsP}
The transformation laws for $P(\tau,\zeta)$ are{\rm:}
%, defined by \eqref{PP}, are
\begin{align}
P\Bigl(-\frac{1}{\tau},\frac{\zeta}{\tau}\Bigr)&=-\ii e^{\pi\ii\zeta(1+\zeta-\tau)/\tau}P(\tau,\zeta),\\
P(\tau+1,\zeta)&=e^{\pi\ii/6}P(\tau,\zeta).
\end{align}
\end{proposition}
\begin{proof}
Using the transformation law \eqref{Ktlaws1}, we find
\begin{align*}
\begin{split}
K_l\Bigl(-\frac{1}{\tau},\frac{\zeta}{\tau};4\Bigr)=\frac12e^{\pi\ii\zeta^2/\tau} \Bigl((-\ii)^{-l}K_{-1}(\tau,\zeta;4)+(-\ii)^lK_1(\tau,\zeta;4)\\
+K_0(\tau,\zeta;4)+(-1)^lK_2(\tau,\zeta;4)\Bigr).
\end{split}
\end{align*}
Therefore
\begin{align*}
P\Bigl(-\frac{1}{\tau},\frac{\zeta}{\tau}\Bigr)&=e^{\pi\ii\zeta/\tau}\Bigl(K_{-1}\Bigl(-\frac{1}{\tau},\frac{\zeta}{\tau};4\Bigr)-K_1\Bigl(-\frac{1}{\tau},\frac{\zeta}{\tau};4\Bigr)\Bigr)\\
&=\frac{-2\ii}{2}e^{\pi\ii(\zeta^2+\zeta)/\tau}\Bigl(K_{-1}(\tau,\zeta;4)-K_1(\tau,\zeta;4)\Bigr)\\
&=-\ii e^{\pi\ii(\zeta^2+\zeta)/\tau}e^{-\pi\ii\zeta}P(\tau,\zeta)\\
&=-\ii e^{\pi\ii\zeta(1+\zeta-\tau)/\tau}P(\tau,\zeta).
\end{align*}
Using the transformation law \eqref{Ktlaws2}, we find
\begin{align*}
P(\tau+1,\zeta)&=e^{\pi\ii\zeta}\bigl(K_{-1}(\tau+1,\zeta;4)-K_1(\tau+1,\zeta;4)\bigr)\\
&=e^{\pi\ii\zeta}e^{\pi\ii(\frac14-\frac{1}{12})}\bigl(K_{-1}(\tau,\zeta;4)-K_1(\tau,\zeta;4)\bigr)\\
&=e^{\pi\ii/6}P(\tau,\zeta).
\end{align*}
This completes the proof.
\end{proof}

%Now we relate the functions $P_{\si,\varphi}(\tau)$ and $P(\tau,\zeta)$.

Before presenting the final transformation laws for the modified characters $\chi_{\F_\si}(\varphi,\tau)$, we first point out a particular calculation that will be needed in the proof.

\begin{lemma}\label{lemma}
Suppose that\/ $\si, \ph$ are two commuting automorphisms of\/ $\mathfrak{h}$, and the maps $s,f$ are defined by \eqref{sf}. The we have
\[
\det\nolimits_\mathfrak{h} e^{-\pi\ii(s+f+2sf)}=\ii^{r-r_{0,0}},
\]
where $r=\dim\mathfrak{h}$ and $r_{0,0}=\dim\mathfrak{h}_{0,0}$.
\end{lemma}
\begin{proof}
Notice that
\[
\det\nolimits_\mathfrak{h} e^{-\pi\ii(s+f+2sf)} = \exp\bigl( -\pi\ii \tr\nolimits_\mathfrak{h}(s+f) \bigr) \exp\bigl( -\pi\ii \tr_\mathfrak{h}2sf \bigr).
\]
It is easy to show that 
%$\det\nolimits_\mathfrak{h}s=-\frac12\left(\dim\mathfrak{h}_{\perp,0}+\dim\mathfrak{h}_{\perp,\perp}\right)$ and $\det\nolimits_\mathfrak{h}f=-\frac12\left(\dim\mathfrak{h}_{0, \perp}+\dim\mathfrak{h}_{\perp,\perp}\right)$
\[
\tr\nolimits_\mathfrak{h}(s+f)=-\frac12\bigl(\dim\mathfrak{h}_{0,0}^\perp+\dim\mathfrak{h}_{\perp,\perp}\bigr)
\]
(cf.\ Remark \ref{trace}) and
\begin{align*}
\tr_\mathfrak{h}2sf&=\sum_{x,y\neq0} \bigl((-x)(-y)+(x-1)(-y) \bigr)\dim\mathfrak{h}_{x,y}\\
&=\frac12\sum_{x,y\neq0}(y+(1-y))\dim\mathfrak{h}_{x,y}=\frac12\dim\mathfrak{h}_{\perp,\perp},
\end{align*}
using that $\dim\mathfrak{h}_{x,y}=\dim\mathfrak{h}_{1-x,y}=\dim\mathfrak{h}_{x,1-y}$. The result follows.
%Hence we find
%$
%\det\nolimits_\mathfrak{h}\;e^{-\pi\ii(s+f+2sf)}=\ii^{\dim\mathfrak{h}_{0.0}^\perp}=\ii^{r-r_{0,0}}.
%$
\end{proof}

We are now ready to present the transformation laws for the modified characters $\chi_{\F_\si}(\varphi,\tau)$ of the $\si$-twisted Fock space $\F_\si$.

%\label{Ptlaws}
\begin{theorem}\label{chiFtran}
Let\/ $\si, \ph$ be two commuting automorphisms of\/ $\mathfrak{h}$. Then
%the transformation laws for the modified characters $\chi_{\F_\si}(\ph, \tau)$ over the $\si$-twisted Fock space $\F_\si$ are given by{\rm:}
\begin{align}
\chi_{\F_\si}\Bigl(\ph, -\frac{1}{\tau}\Bigr)&=\sqrt{\frac{\det\nolimits_{\mathfrak{h}_{0, \perp}}(1-\ph^{-1})}{\det\nolimits_{\mathfrak{h}_{\perp, 0}}(1-\si)}}(-\ii\tau)^{-r_{0,0}/2}\chi_{\F_\ph}(\si^{-1}, \tau),\label{Ftran}\\
\chi_{\F_\si}(\ph,\tau+1)&=e^{-2\pi\ii(-\Delta_\si+\frac{r}{24})}\chi_{\F_\si}(\ph\si,\tau),
\end{align}
where $r=\dim\mathfrak{h}$, $r_{0,0}=\dim\mathfrak{h}_{0,0}$, and $\Delta_\si$ is given by \eqref{Deltap}.
%The transformation laws for $P_{\si,\ph}(\tau)$ given in \eqref{P} are:
%\begin{align}
%P_{\si,\ph}\left(-\frac{1}{\tau}\right)&=\sqrt{\frac{\det\nolimits_{\mathfrak{h}_{\perp, 0}}(1-\si)}{\det\nolimits_{\mathfrak{h}_{0, \perp}}(1-\ph^{-1})}}(-\ii\tau)^{r_{0,0}/2}P_{\ph, \si^{-1}}(\tau),\label{Ph}\\
%P_{\si,\ph}(\tau+1)&=e^{2\pi\ii(-\Delta_\si+\frac{r}{24})}P_{\si,\ph\si}(\tau),\label{Pe}
%\end{align}
%where $r=\dim\mathfrak{h}$ and $r_{0,0}=\dim\mathfrak{h}_{0,0}$. 
%Furhtermore, we must have $P_{\si,\ph}(\tau)=P_{\si^{-1},\ph^{-1}}(\tau)$ for all $\si, \ph\in\Gamma_{\si,\mu,\zeta}$.
\end{theorem}
\begin{proof}
Due to \thref{chiF}, it is sufficient to compute the transformation laws for $P_{\si,\varphi}(\tau)$.
First, we find the transformation $\tau\mapsto-1/\tau$ for $P_{\si,\varphi}(\tau)^2$. Using Lemma \ref{P3}, Proposition \ref{tlawsP} with $\zeta=s-f\tau$, and \eqref{eta},  we calculate:
\begin{align*}
\begin{split}
P_{\si,\ph}\Bigl(-\frac{1}{\tau}\Bigr)^2
&=\frac{\eta\bigl(-\frac{1}{\tau}\bigr)^{2r_{0,0}}\det\nolimits_\mathfrak{h}e^{-\frac{\pi\ii}{\tau}(s^2+s)}}{\det\nolimits_{\mathfrak{h}_{0,\perp}}(1-\ph^{-1})}\det\nolimits_{\mathfrak{h}_{0,0}^\perp}P\Bigl(-\frac{1}{\tau}, \frac{s-f\tau}{\tau}\Bigr)\\
%\end{align*}
%\begin{align*}
%\begin{split}
%P_{\si,\ph}\left(-\frac{1}{\tau}\right)^2
=\;\,&\frac{(-\ii\tau)^{r_{0,0}}\eta(\tau)^{2r_{0,0}}}{\det\nolimits_{\mathfrak{h}_{0, \perp}}(1-\ph^{-1})}\\
&\times\det\nolimits_\mathfrak{h}e^{-\frac{\pi\ii}{\tau}(s^2+s)+\frac{\pi\ii\zeta}{\tau}(1+\zeta-\tau)}\det\nolimits_{\mathfrak{h}_{0,0}^\perp}(-\ii)P(\tau,s-\tau f)\\
%\end{split}\\
%\end{align*}
%\begin{align*}
%\begin{split}
%P_{\si,\ph}\left(-\frac{1}{\tau}\right)^2
=\;\,&\frac{(-\ii\tau)^{r_{0,0}}\eta(\tau)^{2r_{0,0}}}{\det\nolimits_{\mathfrak{h}_{0, \perp}}(1-\ph^{-1})}(-\ii)^{r-r_{0,0}}\\
&\times\det\nolimits_\mathfrak{h}e^{{\pi\ii}(-f-2sf-s+(f^2+f)\tau)}\det\nolimits_{\mathfrak{h}_{0,0}^\perp}P(\tau,s-\tau f).
\end{split}\\
\end{align*}
Now we use Lemma \ref{P3} again but with the replacement $\si\mapsto\varphi$ and $\varphi\mapsto\si^{-1}$ (i.e., $s\mapsto f$ and $f\mapsto-s$):
\begin{equation}\label{5.4}
P_{\ph,\si^{-1}}(\tau)^2=\frac{\eta(\tau)^{2r_{0,0}}}{\det\nolimits_{\mathfrak{h}_{\perp, 0}}(1-\si)}\det\nolimits_{\mathfrak{h}}e^{\pi\ii\tau(f^2+f)}\det\nolimits_{\mathfrak{h}^\perp_{0,0}}P(\tau,s-\tau f).
\end{equation}
Finally, using \eqref{5.4} and Lemma \ref{lemma} we have 
\begin{align*}
\begin{split}
P_{\si,\ph}\Bigl(-\frac{1}{\tau}\Bigr)^2=(-\ii\tau)^{r_{0,0}}&\frac{\det\nolimits_{\mathfrak{h}_{\perp, 0}}(1-\si)}{\det\nolimits_{\mathfrak{h}_{0, \perp}}(1-\ph^{-1})}
\\&\times
(-\ii)^{r-r_{0,0}}\left(\det\nolimits_\mathfrak{h}e^{-\pi\ii(s+f+2sf)}\right)P_{\ph, \si^{-1}}(\tau)^2\\
=(-\ii\tau)^{r_{0,0}}&\frac{\det\nolimits_{\mathfrak{h}_{\perp, 0}}(1-\si)}{\det\nolimits_{\mathfrak{h}_{0, \perp}}(1-\ph^{-1})}P_{\ph, \si^{-1}}(\tau)^2.
\end{split}
\end{align*}
 We then obtain 
\begin{equation*}
P_{\si,\ph}\Bigl(-\frac{1}{\tau}\Bigr)=\sqrt{\frac{\det\nolimits_{\mathfrak{h}_{\perp, 0}}(1-\si)}{\det\nolimits_{\mathfrak{h}_{0, \perp}}(1-\ph^{-1})}}(-\ii\tau)^{r_{0,0}/2}P_{\ph, \si^{-1}}(\tau)
\end{equation*} 
 by choosing the appropriate branch.
 
Next we compute the transformation $\tau\mapsto\tau+1$.  From \eqref{siph} and \eqref{P}, we have
\begin{align*}
P_{\si,\ph}(\tau+1)&=e^{2\pi\ii(-\Delta_\si+\frac{r}{24})}q^{-\Delta_\si+\frac {r}{24}}{\displaystyle\prod_{n=1}^\infty}
\det\nolimits_\mathfrak{h}(1-\ph e^{2\pi\ii(n+s)}q^{n+s})\\
&=e^{2\pi\ii(-\Delta_\si+\frac{r}{24})}q^{-\Delta_\si+\frac {r}{24}}{\displaystyle\prod_{n=1}^\infty}
\det\nolimits_\mathfrak{h}(1-\ph\si q^{n+s})\\
&=e^{2\pi\ii(-\Delta_\si+\frac{r}{24})}P_{\si,\ph\si}(\tau).
\end{align*}
This completes the proof of the theorem.
\end{proof}

\begin{remark}\label{Psym}
If we apply the transformation law for $P_{\si,\si^k}(\tau)$ twice, as in the proof of Theorem \ref{chiFtran}, we obtain:
\begin{align*}
P_{\si,\si^k}(\tau)&=\sqrt{\frac{\det\nolimits_{\mathfrak{h}_{\perp, 0}}(1-\si)}{\det\nolimits_{\mathfrak{h}_{0, \perp}}(1-\si^{-k})}}\frac{1}{(-\ii\tau)^{r_{0,0}/2}}P_{\si^k, \si^{-1}}\Bigl(-\frac{1}{\tau}\Bigr)\\
&=\sqrt{\frac{\det\nolimits_{\mathfrak{h}_{\perp,0 }}(1-\si^k)}{\det\nolimits_{\mathfrak{h}_{0, \perp}}(1-\si^{-k})}}P_{\si^{-1},\si^{-k}}(\tau)\\
&=P_{\si^{-1},\si^{-k}}(\tau),
\end{align*}
where in the last step we used that $\lieh_{0,\perp}$ and $\lieh_{\perp,0}$ are trivial.
This is a special case of \eqref{Psilsikl} for $l=-1$.
\end{remark}

%\begin{corollary}\label{chiFtran}
%Suppose $\si, \ph\in\Gamma_{\si,\ph,\zeta}$ are commuting automorphisms {\upshape(}cf.\ \eqref{smz}{\upshape)}. Then the transformation laws for the characters $\chi_{\F_\si}(\ph, \tau)$ over the $\si$-twisted Fock space $\F_\si$ are given by:
%\begin{align}
%\chi_{\F_\si}\left(\ph, -\frac{1}{\tau}\right)&=\sqrt{\frac{\det\nolimits_{\mathfrak{h}_{0, \perp}}(1-\ph^{-1})}{\det\nolimits_{\mathfrak{h}_{\perp, 0}}(1-\si)}}(-\ii\tau)^{-r_{0,0}/2}\chi_{\F_\ph}(\si^{-1}, \tau),\label{Ftran}\\
%\chi_{\F_\si}(\ph,\tau+1)&=e^{-2\pi\ii(-\Delta_\si+\frac{r}{24})}\chi_{\F_\si}(\ph\si,\tau),
%\end{align}
%where $r=\dim\mathfrak{h}$ and $r_{0,0}=\dim\mathfrak{h}_{0,0}$.
%\end{corollary}

While Theorem \ref{chiFtran} is given for general commuting automorphisms $\varphi$ and $\si$, in the following sections we will make the additional assumption that $\varphi$ is a power of $\si$, 
which is the case when the order of $\si$ is prime and $\ph\in\langle\si\rangle$. We first need the following lemma.

%(or that $\Gamma=\langle\si\rangle$).  We now show how the formulas given in \eqref{Ftran} simplify in this special case.  The proof uses the calculation of the determinant of $1-\si$ over $\lieh_\perp$, which we compute explicitly in the following lemma.

\begin{lemma}\label{det(1-si)}
%Let $Q$ be an integral lattice, and $\lieh=\CC\otimes_\ZZ Q$ be the complexification of $Q$ with $\dim\lieh<\infty$. Let $\si$ be an isometry of $Q$ of prime order $p$ and extend $\si$ to $\lieh$ in the natural way so that $\lieh=\lieh_0\oplus\lieh_\perp$. 
If\/ $\si$ is an automorphism of\/ $\mathfrak{h}$ of prime order $p$, then 
%\;$p-1$\; divides\; $\dim\lieh_\perp$ \;and
\begin{equation}\label{detk}
\det\nolimits_{\lieh_\perp}(1-\si^k)=p^d\,, \qquad d=\dim\lieh_{1/p} = \frac{\dim\lieh_\perp}{p-1} \,,
\end{equation}
for all\/ $1\leq k\leq p-1$.
\end{lemma}

\begin{proof}
Recall from the proof of \leref{Deltalemma}  (cf.\ \eqref{dimh}) that $\dim\lieh_{j/p}=d$ for all $1\leq j\leq p-1$.
In particular, $\dim\lieh_\perp=d(p-1)$, as $\lieh_\perp=\bigoplus_{j=1}^{p-1}\lieh_{j/p}$.
Hence the minimal polynomial of $\si|_{\lieh_\perp}$ is the cyclotomic polynomial
\begin{equation*}
\frac{x^p-1}{x-1}=x^{p-1}+\cdots+x+1,
%=\prod_{j=1}^{p-1}(x-e^{2\pi\ii j/p}).
\end{equation*}
and the characteristic polynomial is
\begin{align*}
\det\nolimits_{\lieh_\perp}(x-\si)%=\prod_{j=1}^{p-1}(x-e^{2\pi\ii j/p})^d
=\left(x^{p-1}+\cdots+x+1\right)^d.
\end{align*}
Setting $x=1$, we get $\det\nolimits_{\lieh_\perp}(1-\si)=p^d$. The result for $\si^k$ follows from the fact that $\lieh_\perp$ is the same for all $1\leq k\leq p-1$, and the order of $\si^k$ is again $p$.
\end{proof}

We are now ready to present how the transformation \eqref{Ftran} simplifies in the case we will use later.

\begin{corollary}\label{chiFtran3}
Let\/ $\si$ be an automorphism of\/ $\mathfrak{h}$ of prime order $p$. %and let $1\leq k\leq p-1$.
Then %the transformation  \eqref{Ftran} for the modified characters $\chi_{\F_\si}(\si^k, \tau)$  are given by
\begin{align}
\chi_{\F_\si}\Bigl(1, -\frac{1}{\tau}\Bigr)&=p^{-d/2}(-\ii\tau)^{-r_{0}/2}\chi_{\F}(\si^{-1}, \tau),\label{chiFtran3a}\\
\chi_{\F_\si}\Bigl(\si^k, -\frac{1}{\tau}\Bigr)&=(-\ii\tau)^{-r_{0}/2}\chi_{\F_{\si^k}}(\si^{-1}, \tau),\label{chiFtran3b}
\end{align}
for\/ $1\leq k\leq p-1$, where $r_{0}=\dim\mathfrak{h}_{0}$ and\/ $d=\dim\lieh_{1/p}$.
%Furthermore, these formulas are the same if $\si$ is replaced with another generator of $\langle\si\rangle$ .
\end{corollary}

\begin{proof}
These formulas follow from Theorem \ref{chiFtran} and Lemma \ref{det(1-si)}, since $\lieh_{0,\perp}=0$, $\lieh_{\perp,0}=\lieh_\perp$ and $\lieh_{0,0}=\lieh_0$ when $\ph=1$.
\end{proof}

We also present how the transformation \eqref{Ftran} simplifies for modified characters of the untwisted Fock space $\F$ (cf.\  Section \ref{lat}).  

\begin{corollary}\label{chiFtran2}
For any automorphism $\si$ of prime order $p$, and\/ $1\leq k\leq p-1$, we have
%Then the transformation $\tau\mapsto-1/\tau$ for the modified characters $\chi_{\F}(\si^k, \tau)$ are given by
\begin{align}
\chi_{\F}\Bigl(1, -\frac{1}{\tau}\Bigr)&=(-\ii\tau)^{-r/2}\chi_{\F}(1, \tau),\label{chiFtran2a}\\
\chi_{\F}\Bigl(\si^k, -\frac{1}{\tau}\Bigr)&=p^{d/2}(-\ii\tau)^{-r_{0}/2}\chi_{\F_{\si^k}}(1, \tau),\label{chiFtran2b}
\end{align}
where $r=\dim\mathfrak{h}$, $r_{0}=\dim\mathfrak{h}_{0}$, and\/ $d=\dim\lieh_{1/p}=(r-r_0)/(p-1)$.
%Furthermore, these formulas are the same if $\si$ is replaced with another generator of $\langle\si\rangle$ .
\end{corollary}

Finally,  we note that the results \eqref{chiFtran3a}--\eqref{chiFtran2b} remain unchanged if $\si$ is replaced with another power of $\si$ not equal to $1$, due to the invariance  \eqref{Psilsikl}.

%%%%%%%%%%%%%%%%%%%%%%%%%%%%%%%
\subsection{Transformation laws for modified characters of $W(\la)$ and $W(\mu,\zeta)$}\label{s5.2}
%%%%%%%%%%%%%%%%%%%%%%%%%%%%%%%
Throughout the rest of this section,  we let $Q$ be a positive-definite even integral lattice, $\si$ be an isometry of $Q$ of prime order $p$, and we assume that $\bar Q=Q$ 
(cf.\ \leref{Qbarlemma} and  \coref{corQbar1}). 
We also fix a set\/ $\C_\M\subset \pi_0(Q^*)$ of representatives of the cosets\/ $\pi_0(Q^*) / \pi_0(Q)$,
%a subset $\C_\M'=\C_\M\cap M^*$ of representatives of the cosets\/ $M^* / \pi_0(Q)$,
and a set $\mathcal{O} \subset Q^*/Q$ of representatives of the orbits of order\/ $p$ of\/ $\si$ on\/ $Q^*/Q$. 

In this subsection,  we calculate the modular transformations $\tau\mapsto\tau+1$ and $\tau\mapsto -1/\tau$ for the modified characters 
\begin{align}
\chi_{W(\mu, \zeta)}^{\si,\si^k}(\tau, h)&=\tr_{W(\mu, \zeta)}\si^k e^{2\pi\ii h}q^{L_0^{\tw}-\Delta_\si},\label{chiWph2}
\end{align}
where $h\in\lieh_0$,  $k\in\ZZ$,  $\mu\in\pi_0(Q^*)$, $\zeta$ is a central character of $G_\si^\perp$,  and $W(\mu,\zeta)$ is one of the irreducible $G_\si$-modules (cf.\  \eqref{W}, \eqref{chiW}). 
We will also do this for the modified characters \eqref{chiWuntwk}, \eqref{chiWuntw} of the untwisted $\CC_\ep[Q]$-module
\begin{equation}\label{Wlambda2}
W(\la)= \CC_\ep[Q] e^\la = \Span\bigl\{e^{\la+\al} \,\big|\, \al\in Q\bigr\} \,
\end{equation}
 (cf.\  \eqref{Wlambda}). 

Before presenting the calculations of these transformation laws, we need to set a ``rule of thumb'' for how to write a theta function in terms of modified characters. %, as this will be needed in the proof.  
In general, this can be tricky due to the degeneracies \eqref{Psilsikl}, \eqref{Psik1}, and Corollary \ref{sisame}. For fixed $\mu\in\pi_0(Q)^*$, we denote by $\Z_\mu$ the set of all central characters $\zeta$ of $G_\si^\perp$ that satisfy relation \eqref{con}. 
Then, for every $\zeta, \zeta'\in\Z_\mu$ and $l\in\ZZ$,  we have the relations \eqref{Wsilsikl},
%\begin{equation}\label{Wsilsikl2}
%\chi_{W(\mu, \zeta')}^{\si^l,\si^{kl}}(\tau, h)=\chi_{W(\mu, \zeta)}^{\si,\si^k}(\tau, h),
%\end{equation}
which imply that each modified character is equal to the average of all modified characters taken over the set $\Z_\mu$.   
The following is our ``rule of thumb.''
If a theta function originates from the modified character $\chi_{V_{\lambda+Q}}^{1,\si^k}(\tau,h)$ of untwisted type (cf.\ \eqref{chiVM}--\eqref{chiVQ}), then we will choose to use the average over $\Z_\mu$ when writing it in terms of modified characters:
\begin{equation}\label{thetainvert}
\th_{\sqrt{N}\mu+\sqrt{N}\pi_0(Q)}\left(\frac{\tau+k}{N},\frac{h}{\sqrt{N}}, 0\right)=\frac{e^{\pi\ii k|\mu|^2}}{d(\si)|\Z_\mu|}\sum_{\zeta\in\Z_\mu}\chi_{W(\mu, \zeta)}^{\si,\si^k}(\tau, h).
\end{equation}
Alternatively,
if a theta function originates from the modified character $\chi_{W(\mu, \zeta)}^{\si,\si^k}(\tau, h)$ of twisted type (cf.\  \eqref{chiW3}), then we will choose the corresponding character with the same parameters when writing it in terms of modified characters:
%of the form $\chi_{W(\nu, \zeta)}^{\psi,\psi^l}(\tau, h)$, then we will choose the corresponding character with the same $\ze$:
\begin{equation}\label{thetainvert2}
\th_{\sqrt{N}\mu+\sqrt{N}\pi_0(Q)}\left(\frac{\tau+k}{N},\frac{h}{\sqrt{N}}, 0\right)=\frac{e^{\pi\ii k|\mu|^2}}{d(\si)}\chi_{W(\mu, \zeta)}^{\si,\si^k}(\tau, h).
\end{equation}

We are now ready to present the transformation laws for modified characters of $W(\la)$.  
%We do this first,  and then present the transformation laws for modified characters of $W(\mu, \zeta)$.  We have organized the results in this way to hopefully make it easier to reference the different cases. 
For convenience, we set $M=Q\cap\lieh_0$ and define
\begin{align}\label{CL}
C_L( \tau, h)=\frac{(-\ii\tau)^{(\rank L)/2}e^{\pi\ii|h|^2/\tau}}{|L^*/L|^{1/2}} \,,
\end{align}
where $h\in\lieh_0$ and $L$ is any positive-definite even integral lattice. 

\begin{proposition}\label{chiWtran}
With the above notation,  the transformation laws for modified characters of\/ $W(\la)$ are$:$
%{\upshape{(}cf.\ \eqref{Wlambda2}\upshape{):}}
\begin{align}
\chi_{W(\la)}^{1,\si^k}(\tau+1, h)&=
\begin{cases}\label{chiWtlaws12}
e^{\pi\ii|\la|^2}\chi_{W(\la)}^{1,1}(\tau, h),
%\label{chiWtlaws11}
\vspace{0.1in}\\
e^{\pi\ii|\la|^2}\chi_{W(\la)}^{1,\si^{k}}(\tau, h),
\end{cases}
\end{align}
\begin{align}
\chi^{1,\si^k}_{W(\la)}\left(-\frac{1}{\tau}, \frac {h}{\tau}\right)&=
\begin{cases}\label{chiWtlaws3}
C_Q(\tau,h)\displaystyle\sum_{\substack{\ga+Q\in \mathcal{O}}}e^{-2\pi\ii(\ga|\la)}\chi_{W(\ga)}^{1,1}(\tau,h),%\label{chiWtlaws2}
\vspace{0.1in}\\
\displaystyle \frac{C_M(\tau,h)}{d(\si)|\Z_\mu|}\sum_{\substack{\mu\in \C_\M}}\sum_{\zeta\in\Z_\mu}e^{-2\pi\ii(\la|\mu)}\chi_{W(\mu, \zeta)}^{\si^k,1}(\tau,h),
\end{cases}
\end{align}
where in each case the top choice applies when $\la\in Q^*$ and $k=0$ and the bottom choice applies when $\la\in Q^*\cap\lieh_0$ and $1\leq k\leq p-1$.
%and the constants $C_Q(\tau,h)$ and $C_M(\tau,h)$ are defined by \eqref{CL}.
\end{proposition}

\begin{proof}
Formulas 
%\eqref{chiWtlaws11} and 
\eqref{chiWtlaws12} follow from \eqref{chiWuntwk}, \eqref{chiWuntw}, and \eqref{thtlaws2}.
The top formula of \eqref{chiWtlaws3} is clear using \eqref{chiWuntw}, \eqref{thtlaws1}, and that
\begin{align*}
\th_{\la+Q}(\tau, h, 0)=e^{\pi\ii|h|^2/\tau}\th_{\la+Q}(\tau, h, -|h|^2/2\tau).
\end{align*}
Next, we prove the bottom formula of \eqref{chiWtlaws3}. Using \eqref{chiWuntwk} and \eqref{thtlaws1} for $k\neq0$, we have
%Using \eqref{chiWtlawtheta} we have for the case $\si=1$, $\varphi=\si^k$ with $k\neq0$,
\begin{align}\label{chiWtlawspf1}
\chi^{1,\si^k}_{W(\la)}\left(-\frac{1}{\tau}, \frac {h}{\tau}\right)&=C_M(\tau,h)\sum_{\ga+M\in M^*/M}e^{-2\pi\ii(\la|\ga)}\th_{\ga+M}(\tau,h,0).
\end{align}
Using Lemma \ref{tdual}, we have the following sequence of sublattices in $M^*$:
\begin{align}
M=Q\cap\mathfrak{h}_0=\pi_0 (L)\subset\pi_0(Q)\subset \pi_0(Q)^*=M^*
\end{align}
(cf.\ \eqref{L}). Hence the coset $\ga+M$ in $M^*/M$ can be written in the form 
$$\mu+\pi_0\be_j+M,$$
where $\be_j$ (indexed by $j\in J$) is a representative of a coset in $Q/L$, and $\mu$ is a representative for a coset in $M^*/\pi_0(Q)$. Since in this case $\la\in Q^*\cap\mathfrak{h}_0$, we have 
$$(\pi_0\be|\la)=(\be|\la)\in\ZZ, \qquad \be\in Q.$$
Therefore, we can rewrite the sum \eqref{chiWtlawspf1} as
\begin{align*}
\chi^{1,\si^k}_{W(\la)} &\left(-\frac{1}{\tau}, \frac {h}{\tau}\right) \\
&=C_M(\tau,h)\sum_{\substack{\mu+\pi_0(Q)\in M^*/\pi_0(Q)}}e^{-2\pi\ii(\la|\mu)}\sum_{j\in J}\th_{\mu+\pi_0\be_j+M}(\tau,h,0)\\
&=C_M(\tau,h)\sum_{\substack{\mu+\pi_0(Q)\in M^*/\pi_0(Q)}}e^{-2\pi\ii(\la|\mu)}\th_{\mu+\pi_0(Q)}(\tau,h,0)\\
%&=\frac{(-\ii\tau)^{r_0/2}e^{\pi\ii|h|^2/\tau}}{|M^*/M|^{1/2}}\sum_{\substack{\de+\pi_0(Q)\\\de\in M^*}}e^{-2\pi\ii(\la|\mu)}\th_{\sqrt{p}\mu+\sqrt{p}\pi_0(Q)}\left(\frac{\tau}{p},\frac{h}{\sqrt{p}},0\right)\\
&=\frac{C_M(\tau,h)}{d(\si)}\sum_{\substack{\mu+\pi_0(Q)\in M^*/\pi_0(Q)}}\sum_{\zeta\in\Z_\mu}\frac{e^{-2\pi\ii(\la|\mu)}}{|\Z_\mu|}\chi_{W(\mu,\zeta)}^{\si^k,1}(\tau,h),
\end{align*}
where we used \eqref{chiW3} with $k=0$, \eqref{thc}, \eqref{thetainvert}, and $d(\si^k)=d(\si)$ (cf.\ Lemma \ref{defectlemma}).
\end{proof}

Before presenting the transformation laws for the modified characters of $W(\mu,\zeta)$, we state a lemma that will be useful in the proof.

\begin{lemma}\label{esum}
Let\/ $L\subset Q$ be a sublattice with $\rank L=\rank Q$, and\/ $\{\be_j+L\,|\,j\in J\}$ be the distinct cosets of\/ $Q/L$. Then for $\la\in M^*=\pi_0(Q^*)$, we have
\[
\sum_{j\in J}e^{2\pi\ii(\la|\pi_0\be_j)}=\begin{cases}
|\pi_0(Q)/M|,&\la\in Q^*\cap\lieh_0,\\
0,&\la\notin Q^*\cap\lieh_0.
\end{cases}
\]
\end{lemma}
%\pi_0\be_j+M\in\pi_0(Q)/M

\begin{proof}
Write $\pi_0(Q)/M=\bigoplus_{j\in J}\langle\pi_0\be_j+M\rangle$ and let 
$$r_j=\min\{n\in\ZZ\,|\,n\be_j\in M\}.$$
If $\ga=\sum_{j=1}^ra_j\pi_0\be_j\in\pi_0(Q)$, where $r=|\pi_0(Q)/M|=|Q/L|$ (cf.\ Lemma \ref{Q/L}), then
\begin{align}\label{esum1}
\sum_{\ga+M\in\pi_0(Q)/M}e^{2\pi\ii(\la|\ga)}=\prod_{j=1}^r\sum_{a_j=0}^{r_j-1}\left(e^{2\pi\ii (\la|\pi_0\be_j)}\right)^{a_j}.
\end{align}
Since $\pi_0(Q)$ is a rational lattice, $e^{2\pi\ii (\la|\pi_0\be_j)}$ is a root of unity for each $j$. Hence, the right hand side of \eqref{esum1} is zero unless $(\la|\pi_0\be_j)\in\ZZ$ for all $j\in J$. 
The result then follows using that $(\pi_0(Q))^*\cong Q^*\cap\lieh_0$ by Lemma \ref{tdual}.% and that $|Q/L|=|\pi_0(Q)/M|$ (cf.\ Lemma \ref{Q/L}).
\end{proof}

To present the transformation laws for the modified characters of $W(\mu,\zeta)$, we set some additional notation that will make the formulas more compact.  For convenience, we set $\mathcal{L}=\sqrt{p}\pi_0(Q)$ and define (recall that $M=Q\cap\mathfrak{h}_0$)
\begin{equation}\label{CC}
C(\tau,h,\si)=d(\si)C_M(\tau,h)|\pi_0(Q)/M|,
\end{equation}
where $C_M(\tau,h)$ is given in \eqref{CL}. 
%It follows that (cf.\ Lemma \ref{esum})
Then, by Lemma \ref{tdual},
\begin{center}
$M^*=\pi_0(Q^*)$ \quad and\quad $\mathcal{L}^*=\displaystyle\frac{1}{\sqrt{p}}Q^*\cap\lieh_0$.
\end{center}
We fix a set\/ $\C_M\subset M^*$ of representatives of the cosets\/ $\la+M\in M^*/M$ such that $\la\in Q^*\cap\lieh_0$.  For fixed $k\in\ZZ$ and $\mu\in\pi_0(Q^*)$ we fix a set\/ $\C_\L(k, \mu)\subset \L^*$ of representatives of the cosets of the form\/ $\sqrt{p}\nu+\L\in \L^*/\L$ such that $\nu\in k\mu + Q^*\cap\lieh_0$.  %Note that $\C_\L$ could be an empty set.

\begin{proposition}\label{chiWtran2}
Suppose that\/ $p$ is an odd prime.
With the above notation, the transformation laws for modified characters of\/ $W(\mu,\zeta)$ are$:$
%{\upshape{(}cf.\ \eqref{chiWph2}\upshape{):}}
\begin{align}
\chi_{W(\mu,\zeta)}^{\si,\si^k}(\tau+1, h)&=e^{\pi\ii|\mu|^2}\chi_{W(\mu,\zeta)}^{\si,\si^{k+1}}(\tau, h),\qquad k\in\ZZ,\label{chiWtlaws1}
\end{align}
\begin{align}
%\chi^{\si,1}_{W(\mu,\zeta)}\left(-\frac{1}{\tau}, \frac {h}{\tau}\right)
%&=\frac{d(\si)C}{|M^*/M|^{1/2}}\sum_{\substack{\la+M\in M^*/M\\j\in J}} e^{-2\pi\ii(\la|\mu+\pi_0\be_j)}\th_{\la+M},\\
%&=\frac{d(\si)C|Q/L|}{|M^*/M|^{1/2}}\\
%\label{chiWtlaws4}
\chi^{\si,\si^k}_{W(\mu, \zeta)}\left(-\frac{1}{\tau}, \frac {h}{\tau}\right)&=
\begin{cases}
%d(\si)C_M(\tau,h)|\pi_0(Q)/M|
C(\tau,h,\si)
\displaystyle\sum_{\substack{\la\in \C_M}} e^{-2\pi\ii (\la|\mu)}\chi^{1,\si^{-1}}_{W(\la)}(\tau, h),
\vspace{0.1in}\\
v_k\tau^{r_0/2}\displaystyle\sum_{\substack{\sqrt{p}\nu\in \C_\L}}e^{-2\pi\ii(\mu|\nu)}\chi_{W(\nu,\zeta)}^{\si^k,\si^{-1}}(\tau, h),
%\label{chiWtlaws5}
\end{cases} \label{chiWtlaws6}
\end{align}
where the top choice applies when $k=0$ and the bottom choice applies when $1\leq k\leq p-1$, and each $v_k$ is a complex number. 
%and $C(\tau,h,\si)$ is given by \eqref{CC}.
\end{proposition}

\begin{proof}%[Proof of Theorem \ref{chiWtran}]
First, we prove formulas \eqref{chiWtlaws1}. Using  \eqref{thetainvert2} and \prref{chiW2}, we find
\begin{align*}
\chi_{W(\mu,\zeta)}^{\si,\si^k}(\tau+1, h)&=d(\si)e^{-\pi\ii k|\mu|^2}\th_{\sqrt{p}\mu+\sqrt{p}\pi_0(Q)}\left(\frac{\tau+k+1}{p},h, 0\right)\\
&=e^{\pi\ii|\mu|^2}\chi_{W(\mu,\zeta)}^{\si,\si^{k+1}}(\tau, h)
\end{align*}
for any  $0\leq k\leq p-1$.
\begin{comment}
Observe that
\begin{align*}
|\mu+\pi_0\be_j|^2&=|\mu|^2+2(\mu|\be_j)+|\pi_0\be_j|^2\\
&=|\mu|^2+2(b_{\be_j}+(\be_j|\mu))+|\be_j|^2
\end{align*}
(cf.\ \eqref{bal}). Hence 
\begin{align}\label{exp}
e^{\pi\ii|\mu+\pi_0\be_j|^2}=e^{\pi\ii|\mu|^2}e^{2\pi\ii(b_{\be_j}+(\be_j|\mu))}
\end{align}
using that $|\be_j|^2\in2\ZZ$. From \prref{chiW2}, \thref{thtlaws}, and \eqref{exp}
\begin{align*}
\chi_{W}^{\si,\si^k}(\tau+1, h)&=d(\si)\sum_{j\in J}e^{2\pi\ii k(b_{\be_j}+(\be_j|\mu))}\theta_{\mu+\pi_0\be_j+M}(\tau+1, h, 0)\\
&=d(\si)\sum_{j\in J}e^{2\pi\ii k(b_{\be_j}+(\be_j|\mu))}e^{\pi\ii|\mu+\pi_0\be_j|^2}\theta_{\mu+\pi_0\be_j+M}(\tau, h, 0)\\
&=d(\si)e^{\pi\ii|\mu|^2}\sum_{j\in J}e^{2\pi\ii (k+1)(b_{\be_j}+(\be_j|\mu))}\theta_{\mu+\pi_0\be_j+M}(\tau, h, 0)\\
&=e^{\pi\ii|\mu|^2}\chi_W^{\si,\si^{k+1}}(\tau, h).
\end{align*}
\end{comment}

Next, we prove the top formula of \eqref{chiWtlaws6}. We have from the proof of Theorem \ref{chiFtran} that $P_{\si, 1}\left(-\frac{1}{\tau}\right)$ can be written as a multiple of $P_{1,\si^{-1}}(\tau)$.
We therefore expect the transformation $\chi_{W(\mu,\zeta)}^{\si,1}\left(-\frac{1}{\tau}, \frac {h}{\tau}\right)$ to be written as a linear combination of the trace functions $\chi_{W(\la)}^{1,\si^{-1}}(\tau)$ for suitable $\la\in Q^*\cap\lieh_0$. 
From \eqref{chiWuntwk}  and \prref{chiW2}, we have:
\begin{align}
\begin{split}
\chi_{W(\mu,\zeta)}^{\si,1}(\tau, h)&=d(\si)\th_{\mu+\pi_0(Q)}\left(\tau,h,0\right),\\
&=d(\si)\sum_{j\in J}\th_{\mu+\pi_0\be_j+M}(\tau,h,0),
\end{split}
\end{align}
where $\be_j+L$ for $j\in J$ are the distinct cosets in $Q/L$, and
\begin{align}
\chi_{W(\la)}^{1,\si^{-1}}(\tau, h)&=\th_{\la+M}\left(\tau,h, 0\right)=\chi_{W(\la)}^{1,\si^{}}(\tau, h), \qquad \la\in Q^*\cap\lieh_0.
\end{align}
%for $\la\in Q^*\cap\lieh_0$.
%We start by rewriting \eqref{chiW4} in a different form:
%\begin{align}
%\begin{aligned}
%%\begin{split}
%\chi_{W(\mu, \zeta)}^{\si,1}(\tau, h)&=d(\si)\sum_{\substack{\al\in Q\cap\mathfrak{h}_0\\j\in J}}  e^{2\pi\ii(h|\mu+\al+\be_j)} q^{\frac12|\mu+\al+\pi_0\be_j|^2}\\
%&=d(\si)\sum_{j\in J}\th_{\mu+\pi_0\be_j+M}(\tau,h,0)
%%\end{split}
%\end{aligned}
%\end{align}
Then the transformation law \eqref{thtlaws1} yields
\begin{align*}
&\chi^{\si,1}_{W(\mu,\zeta)}\left(-\frac{1}{\tau}, \frac {h}{\tau}\right)\\
&=d(\si)C_M(\tau,h)\sum_{\substack{\la+M\in M^*/M\\j\in J}} e^{-2\pi\ii(\la|\mu+\pi_0\be_j)}\th_{\la+M}(\tau,h,0)\\
\begin{split}
&=d(\si)C_M(\tau,h)\sum_{\la+M\in M^*/M}\left(\sum_{j\in J}e^{-2\pi\ii(\la|\pi_0\be_j)}\right)e^{-2\pi\ii(\la|\mu)}\th_{\la+M}(\tau,h,0)
\end{split}\\
&=C(\tau,h,\si)\sum_{\substack{\la+M\in M^*/M\\\la\in Q^*\cap\lieh_0}}e^{-2\pi\ii(\la|\mu)}\chi_{W(\la)}^{1,\si^{-1}}(\tau, h),
\end{align*}
where we used Lemma \ref{esum} in the last step. 

Finally, we prove the bottom formula of \eqref{chiWtlaws6}. We have from the proof of Theorem \ref{chiFtran} and Corollary \ref{detk} that
$P_{\si, \si^k}\left(-\frac{1}{\tau}\right)$ can be written as a multiple of $P_{\si^k,\si^{-1}}(\tau)$.
We therefore expect the transformation $\chi_{W(\mu,\zeta)}^{\si,\si^k}\left(-\frac{1}{\tau}, \frac {h}{\tau}\right)$ to be written as a linear combination of the trace functions $\chi_{W(\mu',\zeta')}^{\si^k,\si^{-1}}(\tau,h)$ for suitable $\mu'\in\pi_0(Q)$ and central character $\zeta'$ of $G_\si^\perp$ (cf.\ Theorem \ref{bij}). To this end,  we set $\si'=\si^k$. Then 
$$\si^{-1}=(\si')^{k'}=\si^{kk'},$$
where $kk'\equiv-1\mod p$, and we write 
\begin{equation}\label{kk'}
kk'+1=pm
\end{equation}  
for a suitable integer $m$. 
%From Theorem \ref{chiW2} we have
%\begin{align}\label{chiWreverse}
%\chi_{W(\mu, \zeta)}^{\si^k,\si^{-1}}(\tau, h)
%&=d(\si')e^{-\pi\ii k'|\mu|^2}\th_{\sqrt{p}\mu+\sqrt{p}\pi_0(Q)}\left(\tau',h', 0\right),
%\end{align}
%where 
It will be convenient to also make the transformation
\begin{equation}\label{tauprime}
\tau'=\frac{\tau+k'}{p}\,, \qquad h'=\frac{h}{\sqrt{p}} \,.
\end{equation}
%and $d(\si')=d(\si)$ by Lemma \ref{defectlemma}. N
Note that from \eqref{kk'} we can write
\begin{align}\label{tauprime2}
\frac{-\frac{1}{\tau}+k}{p}=\frac{k\tau'-m}{p\tau'-k'}=A\cdot\tau',
\end{align}
with 
\begin{equation}\label{Amatrix}
A=\begin{pmatrix}
k&-m\\p&-k'
\end{pmatrix} \in \mathrm{SL}_2(\ZZ),
\end{equation}
where $A\cdot\tau'$ denotes the action described in Theorem \ref{KP}.
%It is easy to show the matrix $A$ is an element of the group SL$_2(\ZZ)$.
%The desired transformation $\chi_{W(\mu,\zeta)}^{\si,\si^k}\left(-\frac{1}{\tau}, \frac {h}{\tau}\right)$ can now be computed using Theorem \ref{KP} with \eqref{Amatrix}.
Since the matrix $A$ depends on the exponent $k$, we relabel the scalar $v(A)$ from Theorem \ref{KP} as $v_k$.

We next employ Theorem \ref{KP} using the matrix $A$, the even lattice $\mathcal{L}=\sqrt{p}\pi_0(Q)$, and with $\tau$ replaced by $\tau'$. 
%For convenience we set $\mathcal{L}_\mu=\sqrt{p}\mu+\sqrt{p}\pi_0(Q)$ for $\mu\in Q^*\cap\lieh_0$, so that $\mathcal{L}_0=\mathcal{L}$. 
First, we describe the set over which the summation in Theorem \ref{KP} occurs.  Using the coset labeling in the theta function in \eqref{chiW3}, the summation index in the current case is over the cosets
%\[
%\sqrt{p}\nu+\sqrt{p}\pi_0(Q)\in \frac{1}{\sqrt{p}}Q^*\cap\lieh_0/\sqrt{p}\pi_0(Q)
%\]
\[
\sqrt{p}\nu+\mathcal{L}\in \frac{1}{\sqrt{p}}Q^*\cap\lieh_0/\mathcal{L} \qquad (\nu\in Q^*\cap\lieh_0),
\]
%with $\nu\in Q^*\cap\lieh_0$,
%$\sqrt{p}\nu\in \sqrt{p}Q^*\cap\lieh_0$,
since
$$\mathcal{L}^*=(\sqrt{p}\pi_0(Q))^*=\displaystyle\frac{1}{\sqrt{p}} Q^*\cap\lieh_0$$
 by Lemma \ref{tdual}. 
 %But this is equivalent to the cosets
%\[
%\nu+\pi_0(Q)\in \frac{1}{p}Q^*\cap\lieh_0/\pi_0(Q)
%\]
%with $\nu\in Q^*\cap\lieh_0$. 
Therefore, we have from Theorem \ref{KP} that

\begin{align}\label{new24}
\begin{split}
&\th_{\sqrt{p}\mu+\mathcal{L}}\left(A\cdot(\tau',h',0)
%,\frac{h}{p\tau'-k'}, -\frac{|h|^2}{2(p\tau'-k')}
\right)
=v_k \, (p\tau'-k')^{r_0/2}\\
&\times\sum_{\substack{\sqrt{p}\nu+\mathcal{L}\in\mathcal{L}^*/\mathcal{L}\\\nu\in Q^*\cap\lieh_0}}e^{\pi\ii(-\frac{k'}{p}|\sqrt{p}\nu|^2-2m(\sqrt{p}\mu|\sqrt{p}\nu)-km|\sqrt{p}\mu|^2)}\th_{\sqrt{p}\nu+k\sqrt{p}\mu+\mathcal{L}}(\tau', h', 0)\\
&=v_k\,\tau^{r_0/2}\sum_{\substack{\sqrt{p}\nu+\mathcal{L}\in\mathcal{L}^*/\mathcal{L}\\\nu\in k\mu+Q^*\cap\lieh_0}}e^{\pi\ii(-k'|\nu|^2-2(\mu|\nu)+k|\mu|^2)}\th_{\sqrt{p}\nu+\mathcal{L}}(\tau', h', 0),
\end{split}
\end{align}
where we used the relation $cd|\mu|^2=\frac{d}{c}|c\mu|^2$ and made the shift $\nu\mapsto\nu-k\mu$.  In the last step, we used \eqref{kk'} to calculate  
\begin{equation*}
-k'|\nu-k\mu|^2-2mp(\mu|\nu-k\mu)-kmp|\mu|^2=-k'|\nu|^2-2(\mu|\nu)+k|\mu|^2.
\end{equation*}
Also recall that the value of $\be_0$ in Theorem \ref{KP} can be set to zero when $p$ is odd (cf.\ Remark \ref{beta0} with $c=p$ in this case).
%where we used that 
%$(\sqrt{p}\pi_0(Q))^*=\displaystyle\frac{1}{\sqrt{p}} Q^*\cap\lieh_0$ (cf.\ Lemma \ref{tdual}) and 
%$\tau=p\tau'-k'$. 

Note that the first equation in \eqref{new24} remains unchanged when replacing $k'$ by $k'+ap$ for some integer $a$. Since \eqref{kk'} implies that $m$ gets replaced with $m+k$, and \eqref{tauprime} implies that $\tau'$ gets replaced with $\tau'+a$, it is clear that the resulting exponentials involving $a$ cancel out (cf.\ Theorem \ref{thtlaws}). Hence, \eqref{new24} only depends on $k'$ modulo $p$.

Now we finish the proof using \eqref{new24},  \eqref{thetainvert2}, and Proposition \ref{chiW2} to get 
\begin{align*}
\chi_{W(\mu,\zeta)}^{\si,\si^k}&\left(-\frac{1}{\tau}, \frac{h}{\tau}\right)=d(\si)e^{-\pi\ii k|\mu|^2}\th_{\sqrt{p}\mu+\sqrt{p}\pi_0(Q)}\left(A\cdot(\tau',h',0)\right)\\
\begin{split}
&=d(\si)v_k\tau^{r_0/2}\sum_{\substack{\sqrt{p}\nu+\mathcal{L}\in\mathcal{L}^*/\mathcal{L}\\\nu\in k\mu+Q^*\cap\lieh_0}}e^{\pi\ii(-k'|\nu|^2-2(\mu|\nu))}\th_{\sqrt{p}\nu+\mathcal{L}}(\tau', h', 0)
\end{split}\\
&=v_k\tau^{r_0/2}\sum_{\substack{\sqrt{p}\nu+\mathcal{L}\in\mathcal{L}^*/\mathcal{L}\\\nu\in k\mu+Q^*\cap\lieh_0}}e^{-2\pi\ii(\mu|\nu)}\chi_{W(\nu,\zeta)}^{\si^k,\si^{-1}}(\tau, h).
\end{align*}
%for some set of central characters $\zeta'$ of $G_\si^\perp$ (cf.\ \eqref{Aaction}).
%In the last line we used \eqref{thetainvert2}, $\sqrt{p}\nu+\mathcal{L}_{k\mu}=\mathcal{L}_{k\mu+\nu}$ and
%\begin{align*}
%e^{- k|\mu|^2}e^{k'|k\mu+\nu|^2}e^{(-k'|\nu|^2-2mp(\mu|\nu)-kmp|\mu|^2)}=e^{-2(\mu|k\mu+\nu)}.
%\end{align*}
%Finally, the sum over the quotient $(Q^*\cap\lieh_0)/p\pi_0(Q)$ can be replaced by the quotient group $\pi_0(Q^*)/\pi_0(Q)$ by Lemma \ref{isoquotients}. 
This completes the proof of \prref{chiWtran2}.
\end{proof}

\subsection{Transformation laws for modified characters of $V_{\la+M}$ and $M(\mu,\zeta)$}\label{s5.3}

In this subsection,  we calculate the modular transformations $\tau\mapsto\tau+1$ and $\tau\mapsto -1/\tau$ for the modified characters 
\begin{equation}\label{chiM3}
\chi_{M(\mu,\zeta)}^{\si,\si^k}(\tau,h)=\tr_{M(\mu,\zeta)} \si^k e^{2\pi\ii h}q^{L_0^{\tw}-\frac {r}{24}},
\end{equation}
where $r=\rank Q$, $h\in\lieh_0$,  $k\in\ZZ$,  $\mu\in\pi_0(Q^*)$, $\zeta$ is a central character of $G_\si^\perp$,  and $M(\mu,\zeta)$ is one of the irreducible $\si$-twisted $V_Q$-modules 
(see \eqref{M}, \eqref{chiM}, \eqref{chiMk}). 
 Since $M(\mu,\zeta)$ is a tensor product of the $\si$-twisted Fock space $\F_\si$ and the $G_\si$-module $W(\mu,\zeta)$, the trace is a product
of traces over them (see \eqref{chiMFW}).
The calculation of modular transformations of $M(\mu,\zeta)$ will hence follow from Corollary \ref{chiFtran2} and Proposition %\ref{chiWtran} and 
\ref{chiWtran2}.

Similarly, we will also do this for the modified characters 
\begin{equation}\label{chiVla4}
\chi_{V_{\la+M}}^{1,\si^k}(\tau,h)=\tr_{V_{\la+M}} \si^k e^{2\pi\ii h}q^{L_0-\frac {r}{24}}
\end{equation}
of untwisted $V_Q$-modules,
where $r=\rank Q$, $k\in\ZZ$,  $M=Q\cap\mathfrak{h}_0$, and $\la\in Q^*$ (see \eqref{chiVla}, \eqref{chiVla2}, \eqref{chiVla3}).
Note that when $k=0$,  the cosets in \eqref{chiVla4} are taken over $Q=M$.
%
%Throughout this section and the next,  we let $Q$ be a positive-definite even integral lattice, $\si$ be an isometry of $Q$ of prime order $p$, and we assume that $\bar Q=Q$ (cf.\  \leref{Qbarlemma}). 
%We also fix a set\/ $\C_\M\subset \pi_0(Q^*)$ of representatives of the cosets\/ $\pi_0(Q^*) / \pi_0(Q)$
%%a subset $\C_\M'\subset\C_\M$ of representatives of the cosets\/ $M^* / \pi_0(Q)$,
%and a set $\mathcal{O} \subset Q^*/Q$ of representatives of the orbits of order\/ $p$ of\/ $\si$ on\/ $Q^*/Q$. We let $\Z_\mu$ be the set of central characters $\zeta$ of $G_\si^\perp$ that satisfy  \eqref{con} for fixed $\mu\in\pi_0(Q^*)$. 
%
%We are now ready to present the transformation laws for the modified characters of $V_{\la+M}$. 
For convenience, we set %$M=Q\cap\mathfrak{h}_0$ and define
\begin{equation}\label{D}
D_{\la}(\tau,h,\si)=
\displaystyle\frac{e^{\pi\ii|h|^2/\tau}p^{\dim\lieh_\perp/2(p-1)}}{|M^*/M|^{1/2}d(\si)} \,,
\end{equation}
when $\la\in Q^*\cap\lieh_0$ and set $D_{\la}(\tau,h,\si)=0$ otherwise.

\begin{proposition}\label{chiMtran}
With the above notation, the transformation laws for modified characters of\/ $V_{\la+Q}$ are given by$:$
%{\upshape{(}}cf.\ \eqref{chiVla4}{\upshape{):}}
\begin{align}
\chi^{1,\si^k}_{V_{\la+M}}(\tau+1,h)&=e^{\pi\ii(|\la|^2-\frac{r}{12})}\chi^{1,\si^k}_{V_{\la+M}}(\tau,h),\label{M3}
%\chi^{1,1}_{V_{\la+Q}}(\tau+1,h)&=e^{\pi\ii(|\la|^2-\frac{r}{12})}\chi^{1,1}_{V_{\la+Q}}(\tau,h),\label{M1}\\
%\chi^{1,1}_{V_{\la+Q}}\left(-\frac{1}{\tau}, \frac {h}{\tau}\right)&=
\end{align}
\begin{align}
\chi^{1, \si^k}_{V_{\la+M}}\left(-\frac{1}{\tau}, \frac {h}{\tau}\right)&=
\begin{cases}
\displaystyle\frac{e^{\pi\ii|h|^2/\tau}}{|Q^*/Q|^{1/2}}\sum_{\substack{\ga+Q\in \mathcal{O}}}e^{-2\pi\ii(\la|\ga)}\chi_{V_{\ga+Q}}^{1,1}(\tau,h),
%\label{M2}
\vspace{0.1in}\\
\displaystyle \frac{D_{\la}(\tau,h,\si)}{|\Z_\mu|}\sum_{\substack{\mu\in \C_\M}}\sum_{\zeta\in\Z_\mu}e^{-2\pi\ii(\la|\mu)}\chi^{\si^k,1}_{M(\mu,\zeta)}(\tau,h),
%\label{M4}
\end{cases}\label{M8}
\end{align}
where the top choice applies when $\la\in Q^*$ and $k=0$, and the bottom choice applies when $\la\in Q^*\cap\lieh_0$ and\/ $1\leq k\leq p-1$.
%and the constant $D_{\la}(\tau,h,\si)$ is defined by \eqref{D}.  
%In particular, when $k\neq0$ and\/ $\la\notin Q^*\cap\lieh_0$,  the right-hand side of \eqref{M8} is zero.
\end{proposition}

\begin{proof}
These results follow easily from Proposition \ref{chiWtran} and Corollary \ref{chiFtran2}, using that $V_{\la+M}$ is a tensor product of the Fock space $\F$ and the $\CC_\ep[Q]$-module $W(\la)$;
see \eqref{chiVM}--\eqref{chiWuntw}.
\end{proof}

%We make two notes on the way we the above relations are written.  First, in the case when $k=0$,  the cosets in \eqref{M3} are taken over $Q=M$, and second when $k\neq0$ and $\la\notin Q^*\cap\lieh_0$,  the right-hand side of \eqref{M8} is zero.

Before presenting the transformation laws for the modified characters of $M(\mu,\zeta)$, we set some additional notation, which will make the formulas more compact.  
We let $\mathcal{L}=\sqrt{p}\pi_0(Q)$ and
\begin{equation}\label{v0}
v_0=d(\si)p^{-\dim\lieh_\perp/2(p-1)}\frac{|\pi_0(Q)/M|}{|M^*/M|^{1/2}},
\end{equation}
where, as before, $M=Q\cap\mathfrak{h}_0$.
Then, by Lemma \ref{tdual}, 
\begin{equation}\label{Mdual}
M^*=\pi_0(Q^*), \qquad \mathcal{L}^*=\displaystyle\frac{1}{\sqrt{p}}Q^*\cap\lieh_0.
\end{equation}
We fix a set\/ $\C_M\subset M^*$ of representatives of the cosets\/ $\la+M\in M^*/M$ such that $\la\in Q^*\cap\lieh_0$.  For fixed $k\in\ZZ$ and $\mu\in\pi_0(Q^*)$, we fix a set $\C_\L(k, \mu)\subset \L^*$ of representatives of the cosets of the form $\sqrt{p}\nu+\L\in \L^*/\L$ such that $\nu\in k\mu + Q^*\cap\lieh_0$.  
Finally,  recall the constant
\begin{equation}\label{Delta0}
\Delta_\si=\frac{p+1}{24p}\dim\lieh_\perp
\end{equation}
from \eqref{Delta} and Lemma \ref{Deltalemma}.

\begin{proposition}\label{chiMtran2}
Suppose that\/ $p$ is an odd prime.
With the above notation,  the following are the transformation laws for modified characters of\/ $M(\mu,\zeta)$$:$ %{\upshape{(}}cf.\ \eqref{chiM3}{\upshape{):}}
\begin{align}
\chi_{M(\mu,\zeta)}^{\si,\si^k}(\tau+1, h)&=e^{2\pi\ii(\Delta_\si-\frac{r}{24})}e^{\pi\ii|\mu|^2}\chi_{M(\mu,\zeta)}^{\si,\si^{k+1}}(\tau, h),\label{M5}
\end{align}
\begin{align}
%\chi^{\si,1}_{M(\mu,\zeta)}\left(-\frac{1}{\tau}, \frac {h}{\tau}\right)
\chi^{\si,\si^k}_{M(\mu, \zeta)}\left(-\frac{1}{\tau}, \frac {h}{\tau}\right)&=
\begin{cases}
\displaystyle v_0 \, e^{\pi\ii|h|^2/\tau}\sum_{\substack{\la\in \C_M}} e^{-2\pi\ii (\la|\mu)}\chi^{1,\si^{-1}}_{V_{\la+M}}(\tau, h),&
%\label{M6}
\vspace*{0.1in}\\
\displaystyle v_k \, \ii^{r_0/2}\sum_{\substack{\sqrt{p}\nu\in\C_\L(k,\mu)}}e^{-2\pi\ii(\mu|\nu)}\chi_{M(\nu,\zeta)}^{\si^k,\si^{-1}}(\tau, h),&%\label{M7}
\end{cases}\label{M9}
\end{align}
where the top choice applies when $k=0$ and $v_0$ is given by \eqref{v0},  while the bottom choice applies when $1\leq k\leq p-1$ and each $v_k$ is a complex number.
\end{proposition}

\begin{proof}
These results follow easily from Proposition \ref{chiWtran2} and Corollary \ref{chiFtran3}, using that $M(\mu,\zeta)$ is a tensor product of the $\si$-twisted Fock space $\F_\si$ and the $G_\si$-module $W(\mu,\zeta)$;
see \eqref{chiMFW}.
\end{proof}

\subsection{Orbifold modules and transformation laws for orbifold characters}

In this subsection, we present the main results of this paper,  the calculation of the transformation laws $\tau\mapsto\tau+1$ and $\tau$ $\mapsto$ $-1/\tau$ 
for the irreducible characters given in Theorem \ref{class} of the orbifold algebra $V_Q^\si$.
As in the previous sections,  we let $Q$ be a positive-definite even integral lattice, $\si$ be an isometry of $Q$ of prime order $p$, and we assume that $\bar Q=Q$ (cf.\ \leref{Qbarlemma} and  \coref{corQbar1}).  
We set $r=\rank Q$ and fix a set\/ $\C_\M\subset \pi_0(Q^*)$ of representatives of the cosets\/ $\pi_0(Q^*) / \pi_0(Q)$,
%a subset $\C_\M'=\C_\M\cap M^*$ of representatives of the cosets\/ $M^* / \pi_0(Q)$,
and a set $\mathcal{O} \subset Q^*/Q$ of representatives of the orbits of order\/ $p$ of\/ $\si$ on\/ $Q^*/Q$.
We also set $M=Q\cap\lieh_0$. %(there should be no confusion between this and the modules labeled $M$). 
%and let $[\ga+Q]\in\mathcal{O}$ represent the equivalence class of $\ga+Q\in Q^*/Q$ . 

Let us recall the complete list of non-isomorphic irreducible modules over the orbifold algebra\/ $V_Q^\si$ given in Theorem \ref{class} (cf.\  Definition \ref{dZmu}):
\begin{enumerate}
\item[{\upshape{(Type 1)}}]
$V_{\la+Q}^j$ \; $(\la+Q\in (Q^*/Q)^\si, \; 0\le j\le p-1),$

\medskip
\item[{\upshape{(Type 2)}}] 
$V_{\la+Q}$ \; $(\la+ Q\in\mathcal{O}),$

\medskip
\item[{\upshape{(Type 3)}}] 
$M(\mu,\zeta; \si^l)^j$ \; $(\mu\in\C_\M,\, \ze\in\Z_{\mu,\si^l}, \, 0\le j\le p-1, \, 1\leq l\le p-1).$
\end{enumerate}
\noindent
The characters of these modules are given by {\upshape{(}}cf.\ \eqref{chiMk}, \eqref{Msilsikl}, \eqref{chiVM}, \eqref{chiVQ}{\upshape{):}}
\begin{align}
\chi_{V_{\la+Q}}^j(\tau, h) &= %\displaystyle\frac{1}{p} \chi_{V_{\la+Q}}^{1,1}(\tau, h) + 
\displaystyle\frac{1}{p} \sum_{k=0}^{p-1}\om^{jk} \chi_{V_{\la+Q}}^{1,\si^k}(\tau, h),\label{chiV1}\\
\medskip
\chi_{V_{\la+Q}}(\tau, h)&=\chi_{V_{\la+Q}}^{1,1}(\tau, h),\label{chiV2}\\
\chi_{M(\mu,\zeta;\si^l)}^j(\tau, h)&=\displaystyle\frac{1}{p}\sum_{k=0}^{p-1}\om^{jk}\chi_{M(\mu,\zeta;\si^l)}^{\si^l,\si^{k}}(\tau, h),\label{chiMj}
\end{align}
%\medskip
where\/ $\om=e^{2\pi\ii/p}$.
In the proof, it will be necessary to invert these formulas as follows:
%The formulas for the orbifold characters given in Theorem \ref{class} can easily be inverted:
\allowdisplaybreaks
\begin{align}
\chi_{V_{\la+Q}}^{1,1}(\tau, h)&=\begin{cases} \displaystyle\sum_{l=0}^{p-1}\chi_{V_{\la+Q}}^l(\tau, h),&(1-\si)\la\in Q, \label{V1invert}\\
\chi_{V_{\la+Q}}(\tau, h),&(1-\si)\la\notin Q,
\end{cases}\\
\chi_{V_{\la+Q}}^{1,\si^k}(\tau, h)&=\sum_{l=0}^{p-1}\om^{-lk}\chi_{V_{\la+Q}}^l(\tau, h),\label{V2invert}\\
\chi_{M(\mu,\zeta)}^{\si^l,\si^{lk}}(\tau, h)&=\sum_{j=0}^{p-1}\om^{-jk}\chi_{M(\mu,\zeta;\si^l)}^j(\tau, h).\label{chiMinvert}
\end{align}

%We are now ready to present the main theorem,  the transformation laws for the irreducible characters of the orbifold algebra $V_Q^\si$. 
%We present these results in three separate theorems below,  each corresponding to the three module types given above.  
First, we present the transformation laws for the characters of $V_{\la+Q}^j$.
For convenience,  we set 
\begin{equation}\label{Dla}
D_{\la}^\si=\begin{cases}\displaystyle\frac{p^{\dim\lieh_\perp/2(p-1)}}{|M^*/M|^{1/2}d(\si)}\,,&\la\in Q^*\cap\lieh_0, \\
0,&\la\notin Q^*\cap\lieh_0,\end{cases}
\end{equation}
where the defect $d(\si)$ is defined in Remark \ref{defect}. 

\begin{theorem}\label{chiorbtran}
The transformation laws for the characters of the\/ $V_Q^\si$-modules\/ $V_{\la+Q}^j$ 
for\/ $\la+Q\in (Q^*/Q)^\si$ and\/ $0\le j\le p-1$ are as follows$:$
%{\upshape{(}}cf.\  \eqref{chiV1}{\upshape{):}}
%The transformation laws for the orbifold characters of $V_{\la+Q}^j$ for $\la\in Q^*$ with $(1-\si)\la\in Q$ and $j=0,\ldots,p-1$ are:
\begin{equation}
\chi^{j}_{V_{\la+Q}}(\tau+1,h)=e^{\pi\ii(|\la|^2-\frac{r}{12})}\chi^{j}_{V_{\la+Q}}(\tau,h),\label{O1}
\end{equation}
\begin{align}
\chi^{j}_{V_{\la+Q}}&\left(-\frac{1}{\tau}, \frac {h}{\tau}\right)=\frac{e^{\pi\ii|h|^2/\tau}}{p|Q^*/Q|^{1/2}}\sum_{\substack{\ga+Q\in Q^*/Q\\(1-\si)\ga\in Q}}\sum_{t=0}^{p-1}e^{-2\pi\ii(\la|\ga)}\chi_{V_{\ga+Q}}^t(\tau, h)
\notag\\
&+\frac{e^{\pi\ii|h|^2/\tau}}{|Q^*/Q|^{1/2}}\sum_{\substack{\ga+Q\in \mathcal{O}}}e^{-2\pi\ii(\la|\ga)}\chi_{V_{\ga+Q}}(\tau, h)
\label{O2}\\
&+\frac{D_{\la}^\si e^{\pi\ii|h|^2/\tau}}{p|\Z_\mu|}\sum_{k=1}^{p-1}\sum_{\substack{\mu\in \C_\M}}\sum_{t=0}^{p-1}\sum_{\zeta\in\Z_{\mu,\si^k}}\om^{jk}e^{-2\pi\ii(\la|\mu)}\chi_{M(\mu,\zeta;\si^k)}^t(\tau, h).
\notag
\end{align}
%where $0\le j\le p-1$ and the constant $D_{\la}^\si$ is given in \eqref{Dla}.
\end{theorem}

\begin{proof}
The transformation laws are obtained using \eqref{chiV1} and Proposition \ref{chiMtran}. While most of the details are straightforward, we emphasize the nontrivial details regarding the transformation $\tau\mapsto-1/\tau$.

The details of the third sum in \eqref{O2} are straightforward using \eqref{chiV1}, \eqref{M8}, and the inversion formulas  \eqref{V1invert}--\eqref{chiMinvert}. For the first and second sums, we separate the sum over cosets $\ga+Q$ in \eqref{M8} into two sums over cosets for which $(1-\si)\ga\in Q$ and $(1-\si)\ga\notin Q$. Then we apply the inversion formula \eqref{V1invert}. The terms in the sum over cosets with $(1-\si)\ga\notin Q$ simplify further as the $\si$-orbit $\ga+Q,\si\ga+Q,\ldots,\si^{p-1}\ga+Q$ of cosets in $Q^*/Q$ correspond to equal characters over the same orbifold module.  Since $(1-\si)\la\in Q$ implies $\si^i\la\in\la+Q$ for each $i$,  we obtain sums of the form 
\begin{equation*}
\sum_{\substack{\de+Q\in[\ga+Q]\\(1-\si)\ga\notin Q}}e^{-2\pi\ii(\la|\de)}\chi_{V_{\ga+Q}}(\tau, h)=\sum_{i=0}^{p-1}e^{-2\pi\ii(\la|\si^i\ga)}\chi_{V_{\ga+Q}}(\tau, h),
\end{equation*}
where for $\ga+Q\in\mathcal{O}$, $[\ga+Q]$ is the set of cosets within the orbit of order $p$.
Using this fact and that $Q$ is integral, we find 
\begin{align*}
\sum_{i=0}^{p-1}e^{-2\pi\ii(\la|\si^i\ga)}=\sum_{i=0}^{p-1}e^{-2\pi\ii(\si^{-i}\la|\ga)}=pe^{-2\pi\ii(\la|\ga)}.
\end{align*}
This yields the second sum in \eqref{O2}. 
%Also recall the choices made in \eqref{thetainvert} and \eqref{thetainvert2}.
\end{proof}

Next, we present the transformation laws for the characters of $V_{\la+Q}$.
For convenience, we set
\begin{equation}\label{Ela}
E_{\la,\ga}=\displaystyle\sum_{i=0}^{p-1}e^{-2\pi\ii(\la|\si^i\ga)},
\qquad \la,\ga\in Q^*.
\end{equation}
%where $\la,\ga\in Q^*$.

\begin{theorem}\label{chiorbtran2}
The transformation laws for the characters of the $V_Q^\si$-modules $V_{\la+Q}$ with $\la+ Q\in\mathcal{O}$ are$:$
%{\upshape{(}}cf.\  \eqref{chiV2}{\upshape{):}}
\begin{align}
\chi_{V_{\la+Q}}(\tau+1,h)&=e^{\pi\ii(|\la|^2-\frac{r}{12})}\chi_{V_{\la+Q}}(\tau,h),\label{O3}
\end{align}
\begin{align}
\begin{split}
\chi_{V_{\la+Q}}\left(-\frac{1}{\tau}, \frac {h}{\tau}\right)&=\frac{e^{\pi\ii|h|^2/\tau}}{|Q^*/Q|^{1/2}}\sum_{\substack{\ga+Q\in Q^*/Q\\(1-\si)\ga\in Q}}\sum_{t=0}^{p-1}e^{-2\pi\ii(\la|\ga)}\chi_{V_{\ga+Q}}^t(\tau, h)\\
&+\frac{e^{\pi\ii|h|^2/\tau}}{|Q^*/Q|^{1/2}}\sum_{\substack{\ga+Q\in \mathcal{O}}}E_{\la,\ga}\chi_{V_{\ga+Q}}(\tau, h).\label{O4}
\end{split}
\end{align}
%where $0\le j\le p-1$ and the constants $E_{\la,\ga}$ are given in \eqref{Ela}.
\end{theorem}

\begin{proof}
The sums and coefficients $E_{\la,\ga}$ in the transformation law \eqref{O4} are obtained in a similar manner as in the proof of Theorem \ref{chiorbtran} using \eqref{M8} and inversion formula \eqref{V1invert}.  
\end{proof}

%Before we present 
In the transformation laws for the characters of $M(\mu,\zeta;\si^l)^j$, 
we will use the notation $v_0$, $\mathcal{L}$, and $\Delta_\si$ from \eqref{v0}--\eqref{Delta0}.
%we set some additional notation that will make the formulas more compact.  For convenience we set (recall $M=Q\cap\mathfrak{h}_0$),
%\begin{equation}\label{v02}
%v_0=d(\si)p^{-\dim\lieh_\perp/2(p-1)}\frac{|\pi_0(Q)/M|}{|M^*/M|^{1/2}},
%\end{equation}
%and fix $\mathcal{L}=\sqrt{p}\pi_0(Q)$.  Then by Lemma \ref{esum} we have
%\begin{center}
%$M^*=\pi_0(Q^*)$ \quad and\quad $\mathcal{L}^*=\displaystyle\frac{1}{\sqrt{p}}Q^*\cap\lieh_0$.
%\end{center}
%We fix a set\/ $\C_M\subset M^*$ of representatives of the cosets\/ $\la+M\in M^*/M$ such that $\la\in Q^*\cap\lieh_0$.  For fixed $k\in\ZZ$ and $\mu\in\pi_0(Q^*)$ we fix a set\/ $\C_\L(k, \mu)\subset \L^*$ of representatives of the cosets of the form\/ $\sqrt{p}\nu+\L\in \L^*/\L$ such that $\nu\in k\mu + Q^*\cap\lieh_0$.  
%Finally,  recall the constant
%\begin{equation}\label{Delta2}
%\Delta_\si=\frac{p+1}{24p}\dim\lieh_\perp
%\end{equation}
%from \eqref{Delta} and Lemma \ref{Deltalemma}.

\begin{theorem}\label{chiorbtran3}
When\/ $p$ is an odd prime,
the transformation laws for the characters of the $V_Q^\si$-modules $M(\mu,\zeta;\si^l)^j$ are$:$
%{\upshape{(}}cf.\  \eqref{chiMj}{\upshape{):}} 
\begin{align}
\chi_{M(\mu,\zeta;\si^s)}^{j}(\tau+1, h)&=\om^{-j}e^{2\pi\ii\Delta_{\si}}e^{\pi\ii(|\mu|^2-\frac{r}{12})}\chi_{M(\mu,\zeta;\si^s)}^j(\tau, h),\label{O5}
\end{align}
%where $\Delta_{\si}$ is given in \eqref{Delta2}, and
\begin{align}
\begin{split}
\chi^{j}_{M(\mu,\zeta;\si^l)}&\left(-\frac{1}{\tau}, \frac {h}{\tau}\right)
=\frac{v_0}{p}e^{\pi\ii|h|^2/\tau}\sum_{\la\in \C_M} \sum_{t=0}^{p-1}e^{-2\pi\ii (\la|\mu)}\om^{lt}\chi^{t}_{V_{\la+Q}}(\tau, h)\\
&+\frac{\ii^{r_0/2}}{p}\sum_{k=1}^{p-1}\sum_{\sqrt{p}\nu\in\C_\L(k,\mu)}\sum_{t=0}^{p-1}v_k\om^{jk+k^{p-2}t}e^{-2\pi\ii(\mu|\nu)}\chi^{t}_{M(\nu, \zeta;\si^{lk})}(\tau, h),
\end{split}\label{O6}
\end{align}
where 
$$
\mu\in\C_\M, \quad \ze\in\Z_{\mu,\si^l}, \quad 0\le j\le p-1, \quad 1\leq l\le p-1,
$$
$r_0=\dim\lieh_0$, $v_k\in\CC$ for each $k$, and $v_0$ is given in \eqref{v0}.
\end{theorem}

\begin{proof}
The transformation laws are obtained using \eqref{chiMj} and Proposition \ref{chiMtran2}. While most of the details are straightforward, we emphasize the nontrivial details regarding the transformation $\tau\mapsto-1/\tau$.
%Finally we prove the transformation law \eqref{O6}. 
This uses \eqref{M9}, \eqref{chiMj}, and inversion formulas \eqref{V2invert}, \eqref{chiMinvert}. We use \eqref{M9} and \eqref{V2invert} to transform the term in \eqref{chiMj} with $k=0$. To transform the terms in \eqref{chiMj} with $k>0$, we use \eqref{M9} and \eqref{chiMinvert}, where in \eqref{chiMinvert}, $k$ is replaced with $k'=-k^{p-2}$ and $l$ is replaced with $lk$.
%When $\si^s$ acts instead for $s>1$, we use \eqref{chiMjs}, inversion formula \eqref{Minverts} with $s$ replaced with $sk$ and $k$ replaced with $-k^{-1}=-k^{p-2}$, and inversion formula \eqref{V2invert} with $k$ replaced with $-s$.
\end{proof}

\begin{remark}
Note that Theorems \ref{chiorbtran} and \ref{chiorbtran2} hold for $p=2$, while in Theorem \ref{chiorbtran3} we assume that $p$ is odd. This is due to the presence of $\be_0$ in Theorem \ref{KP}
for $p=2$ (cf.\ Remark \ref{beta0}). The case $p=2$ is considered in detail in the next section.
\end{remark}

As an immediate corollary of the transformation laws, we obtain the asymptotic and quantum dimensions of irreducible orbifold characters using the coefficients in \eqref{O2} for $j=0$ and $\la=0$ corresponding to the vacuum module.
This corollary also holds for $p=2$, and in this case it agrees with the previous results of \cite{E2}.
%In the special case when $p=2$, the quantum dimensions were determined previously in \cite{E2} and they agree with our results.

\begin{corollary}\label{corqdim}
Let\/ $Q$ be an even integral lattice, $\si$ be an isometry of\/ $Q$ of %odd 
prime order $p$, and assume that\/ $Q=\bar{Q}$. Then the asymptotic dimensions of irreducible\/ 
$V_Q^\si$-modules are determined by type as follows$:$
\begin{enumerate}
\item[{\upshape{(Type 1)}}] 
$\asdim V_{\la+Q}^j=p^{-1}|Q^*/Q|^{-1/2}$, where $(1-\si)\la\in Q,$

\medskip
\item[{\upshape{(Type 2)}}] 
$\asdim V_{\la+Q}=|Q^*/Q|^{-1/2}$, where $(1-\si)\la\notin Q,$

\medskip
\item[{\upshape{(Type 3)}}] 
$\asdim M(\mu,\zeta; \si^s)^j=\displaystyle\frac{p^{\dim\lieh_\perp/2(p-1)}}{p \, |\Z_\mu| \, d(\si)} \, |(Q\cap\mathfrak{h}_0)^*/(Q\cap\mathfrak{h}_0)|^{-1/2},$
\end{enumerate}
\medskip
where\/ $d(\si)^2=|(Q\cap\lieh_\perp)/(Q\cap(1-\si)Q^*)|$ and\/ $j=0,\ldots,p-1$ label the eigenspaces of\/ $\si$.  
In all cases, the quantum dimensions are related to the asymptotic dimensions by$:$
\begin{equation}
\qdim M=(\asdim M)\, p\, |Q^*/Q|^{1/2}.
\end{equation}
\end{corollary}

\begin{proof}
We may apply Corollary \ref{KacCor}, due to \coref{KacProp}.
Let $S_{0,j}$ be the coefficient of the vacuum module $V_Q^0$ in the linear combinations in \eqref{O2}, \eqref{O4}, \eqref{O6}. %where we set $h=0$. 
It is uniquely defined since the character of the vacuum module does not lie in the span of the remaining modules, by \coref{KacProp}.
Then the asymptotic dimensions are $S_{0,j}$,
where $j$ runs through the labeling of inequivalent irreducible orbifold modules (see e.g.\ \cite[Section 4.2]{DJX}).
For orbifold modules of Type 1, we see from \eqref{O2} for $\la=0$ and $j=0$ that 
\begin{align*}
S_{0,0}=\frac{1}{p\, |Q^*/Q|^{1/2}} \,,
\end{align*}
$S_{0,j}=S_{0,0}$ for each term in the first sum, while $S_{0,j}=pS_{0,0}$ for $j$ representing orbifold modules of Type 2, and $S_{0,j}=(D_0^\si/|\Z_\mu|)S_{0,0}$ for $j$ representing orbifold modules of twisted type. The result follows.
\end{proof}

%\begin{remark}
%Corollary \ref{corqdim} is also true when $p=2$. 
%In the special case when $p=2$,
%the quantum dimensions %in the case of an order two automorphism 
%were determined previously in \cite{E2}.
%\end{remark}

\begin{remark}
Since the transformation $\tau\mapsto-{1}/{\tau}$ corresponding to the $S$-matrix is unitary, we must have in particular that
\begin{align}\label{Ssum}
\sum_{j=0}^m|S_{0,j}|^2=1,
\end{align}
where $m$ is the number of irreducible orbifold modules (cf.\ Remark \ref{Sprops}). In general, the number of irreducible orbifold modules of each type are given in the following table.
\vspace{0.2in}

\begin{center}
\begin{table}[H]\label{tab1}
\begin{tabular}{|c|c|}
\hline
Module&Number of irreducible orbifold modules\\
\hline
$V_{\la+Q}^j$&$p\, |(Q^*/Q)^\si|$\\
\hline
$V_{\la+Q}$&$p^{-1} (|Q^*/Q|-|(Q^*/Q)^\si|)$\\
\hline
$M(\mu,\zeta;\si^l)^j$&$p(p-1)\, |\Z_\mu| \, |M^*/\pi_0(Q)|$ \\
\hline
\end{tabular}
\caption{Number of irreducible orbifold modules}
\end{table}
\end{center}

Recall that the set $\Z_\mu$ was introduced in Section \ref{Gsi}.
In the twisted case, there are $|\Z_\mu||M^*/\pi_0(Q)|$ many $(\mu,\zeta)$ pairs describing irreducible twisted $V_Q$-modules, the factor $p-1$ counts the nontrivial powers of $\si$ which can act, and the factor $p$ counts the eigenspaces corresponding to the orbifold modules. Then from \eqref{O2} with $j=0, \la=0$, and $h=0$, we get that \eqref{Ssum} becomes
\begin{align*}
\frac1p+\frac{1-p}{p}\frac{(D_\la^\si)^2}{|\Z_\mu|}|M^*/\pi_0(Q)|=1,
\end{align*}
where $D_\la^\si$ is given in \eqref{Dla}.  It follows that
\begin{align}\label{Ssumcond}
p^{{\dim\lieh_\perp}/(p-1)}=|\Z_\mu| \, d(\si)^2 \, |\pi_0(Q)/M|,
\end{align}
where we used that $|M^*/\pi_0(Q)| \, |\pi_0(Q)/M|=|M^*/M|$.
%While we are unable to prove this identity directly, we verified the validity of this formula for each of the examples that we will present in later sections. 
Note that by Lemma \ref{det(1-si)} the left-hand side of \eqref{Ssumcond} is an integer power of the prime $p$, where the exponent is the number of orbits of $\si$ of order $p$.  Since $p(Q\cap\lieh_\perp)\subset(1-\si)Q\cap\lieh_\perp$, this implies that indeed $d(\si)^2$ and $|\Z_\mu|$ are a power of $p$ (cf.\ \eqref{Gsiperp} and Remark \ref{defect}).  In addition,  $p\pi_0(Q)\subset M$ implies that $|\pi_0(Q)/M|$ is also a power of $p$. 
%We argue to the plausibility of this formula by showing that each factor on the right-hand side of \eqref{Ssumcond} is also a power of the prime $p$.
\end{remark}

\section{Examples in Order $2$}\label{ex2}
%%%%%%%%%%%%%%%%%%%%%%%%%%%%%%%%%%%
In this section, we consider the lattice $Q$ to be even and positive definite with an isometry $\si$ of order 2. 
As before, we will assume that $Q=\bar{Q}$, which means explicitly that $(\al|\si\al)\in2\ZZ$ for all $\al\in Q$ (see  \leref{Qbarlemma}).

\subsection{The irreducible characters of twisted type in the general setting}\label{Irrgen2}

First, we demonstrate how the transformation law \eqref{chiWtlaws6} can be calculated by using repeatedly Theorem \ref{thtlaws} rather than Theorem \ref{KP},
which will allow us to determine explicitly the unknown constants $v_k$ in \eqref{chiWtlaws6}.

It follows from Theorem \ref{chiFtran} and Corollary \ref{detk} that
\begin{align}\label{P62}
P_{\si, \si}\left(-\frac{1}{\tau}\right)=(-\ii\tau)^{r_{0}/2}P_{\si,\si}(\tau)
\end{align}
when $\si$ has order $2$.
Therefore, as in the general case, we expect the transformation $\chi_{W(\mu,\zeta)}^{\si,\si}\left(-\frac{1}{\tau}, \frac {h}{\tau}\right)$ to be written as a linear combination of the trace functions $\chi_{W(\mu',\zeta')}^{\si,\si}(\tau,h)$ for suitable $\mu'\in\pi_0(Q)$ and central character $\zeta'$ of $G_\si^\perp$ (cf.\ Theorem \ref{bij}). 
%We write $\si=\si^{-1}=\si^{kk'}$, where $kk'-1=2m$ for some $m\in\ZZ$. This implies that we can take $k=k'=1$, and $m=0$.

From Theorem \ref{chiM2}, we have % and Lemma \ref{defectlemma},
\begin{align}
\chi_{W(\mu, \zeta)}^{\si,\si}(\tau, h)
&=d(\si)e^{-\pi\ii |\mu|^2}\th_{\sqrt{2}\mu+\sqrt{2}\pi_0(Q)}\left(\frac{\tau+1}{2},\frac{h}{\sqrt{2}}, 0\right).
\end{align}
Similarly to \eqref{tauprime} and \eqref{tauprime2}, we set 
\begin{equation}\label{tau'2}
\tau'=\frac{\tau-1}{2}
\end{equation} 
and note that
\begin{align}
\frac{-\frac{1}{\tau}+1}{2}=\frac{\tau'}{2\tau'+1}=A\cdot\tau',
\end{align}
where
$A=
\begin{pmatrix}
1&0\\2&1
\end{pmatrix}
\in\SL_2(\ZZ)$. The group $\SL_2(\ZZ)$ has generators
\begin{equation}
S=\begin{pmatrix}
0&-1\\1&0
\end{pmatrix}\qquad\text{and}\qquad T=\begin{pmatrix}
1&1\\0&1
\end{pmatrix},
\end{equation}
%Using $S$ and $T$ as row operations on $A$, we find that 
and we can write $A$ as
\begin{equation}\label{ST2}
A=\begin{pmatrix}
1&0\\2&1
\end{pmatrix}=S^3T^{-2}S=-ST^{-2}S.
\end{equation}

As before, consider the lattices
\begin{equation}
\mathcal{L}=\sqrt{2}\pi_0(Q),\qquad\mathcal{L}^*=\frac{1}{\sqrt{2}}Q^*\cap\lieh_0
\end{equation}
(cf.\ Lemma \ref{tdual}).
As in Theorem \ref{KP}, fix
$\be_0\in\CC\otimes_\ZZ\mathcal{L}=\lieh_0$ such that
\begin{equation}\label{be0p2}
2|\nu|^2\equiv 2(\nu|\be_0)\mod2\ZZ\quad\text{for all}\quad\nu\in\mathcal{L}^*\;\;\text{with}\;\; 2\nu\in\mathcal{L}.
\end{equation}
Then we can calculate the constant $v_1$ from Theorem \ref{KP} using Proposition \ref{vTSprop} and Remarks \ref{vprops} and \ref{vSTremark}:
\begin{align}\label{v1gen2}
\begin{split}
v_1&=v(-ST^{-2}S)=\ii^rv(ST^{-2}S)\\
&=\ii^rv(ST^{-2})v(S)\sum_{\mu+\mathcal{L}\in\mathcal{L}^*/\mathcal{L}}e^{-2\pi\ii(|\mu|^2+(\mu|\be_0))}\\
&=\ii^rv(S)^2\sum_{\mu+\mathcal{L}\in\mathcal{L}^*/\mathcal{L}}e^{-2\pi\ii(|\mu|^2+(\mu|\be_0))}\\
&=\displaystyle\frac{c_{\be_0}}{|\mathcal{L}^*/\mathcal{L}|} \,,
%\sum_{\mu+\mathcal{L}\in\mathcal{L}^*/\mathcal{L}}e^{-2\pi\ii(|\mu|^2+(\mu|\be_0))},
\end{split}
\end{align}
where 
\begin{align}\label{cbe0}
c_{\be_0}=\sum_{\mu+\mathcal{L}\in\mathcal{L}^*/\mathcal{L}}e^{-2\pi\ii(|\mu|^2+(\mu|\be_0))}.
\end{align}
%Using that (cf.\ Theorem \ref{KP})
%\[
%2|\mu|^2\equiv 2(\mu|\be_0)\mod2\ZZ\quad\text{for all}\quad\mu\in\mathcal{L}^*\;\;\text{such that}\;\; 2\mu\in\mathcal{L},
%\]
%we can write \eqref{cbe0} as
%\begin{equation}
%c_{\be_0}=
%\begin{cases}
%\sum_{\mu+\mathcal{L}\in\mathcal{L}^*/\mathcal{L}}e^{-2\pi\ii|\mu|^2},&\be_0=0\\
%\sum_{\mu+\mathcal{L}\in\mathcal{L}^*/\mathcal{L}}e^{-4\pi\ii|\mu|^2},&\be_0\neq0
%\end{cases}.
%\end{equation}

%Since this calculation does not depend on the exponent of the generator $T$, this proof is essentially the same with $k=1$ and by replacing 2 with $p$.

Now we are ready to present the analog of Theorems \ref{chiorbtran}--\ref{chiorbtran3} in the case when the order of the lattice isometry is $2$. 
%The main difference is the addition of the ingredient $\be_0$ in the transformations of orbifold characters of twisted type. Note also that we can 
In this case, we simplify the notation for orbifold characters of twisted type as
\[
\chi^j_{M(\mu,\zeta;\si)}(\tau,h)=\chi^j_{M(\mu,\zeta)}(\tau,h),\qquad j=0,1.
\]

\begin{theorem}\label{chiorbtranrank2}
Let\/ $Q$ be an even integral lattice and\/ $\si$ be an isometry of\/ $Q$ of order $2$ such that\/ $Q=\bar{Q}$, i.e., $(\al|\si\al)\in2\ZZ$ for all\/ $\al\in Q$.
Denote by superscript $j$ the eigenspaces of\/ $\si$, where $j=0, 1$.
Let\/ $\mathcal{O}$ be the set of orbits of\/ $\si$ in $Q^*/Q$ of order $2$ and set\/ $[\ga+Q]\in\mathcal{O}$, for $\ga\in Q^*$,  to designate the orbits. %Also set %$\om=e^{2\pi\ii /p}$, 
%Let $\si$ be an isometry of an integral lattice $Q$ with prime order $p$ and suppose $\ph\in\Gamma=\langle\si\rangle$. Set $\om=e^{2\pi\ii/p}$, 
Consider the lattices 
\[
M=Q\cap\lieh_0, \quad M^*=\pi_0(Q^*), \quad \mathcal{L}=\sqrt{2}\pi_0(Q), \quad \mathcal{L}^*=\frac{1}{\sqrt{2}}Q^*\cap\lieh_0,
\]
and fix\/ $\be_0\in\lieh_0$ satisfying \eqref{be0p2}.

\begin{enumerate}[$(i)$]
\item The transformation laws for the orbifold characters of\/ $V_{\la+Q}^j$ for $\la\in Q^*$ with $(1-\si)\la\in Q$ and\/ $j=0, 1$ are$:$
\begin{align*}
\chi^{j}_{V_{\la+Q}}(\tau+1,h)&=e^{\pi\ii(|\la|^2-\frac{r}{12})}\chi^{j}_{V_{\la+Q}}(\tau,h),%\label{O12}
\end{align*}
\begin{align*}
&\chi^{j}_{V_{\la+Q}}\left(-\frac{1}{\tau}, \frac {h}{\tau}\right)=\frac{e^{\pi\ii|h|^2/\tau}}{2|Q^*/Q|^{1/2}}\sum_{\substack{\ga+Q\in Q^*/Q\\(1-\si)\ga\in Q}}e^{-2\pi\ii(\la|\ga)}(\chi_{V_{\ga+Q}}^0+\chi_{V_{\ga+Q}}^1)(\tau, h)\\
&+\frac{e^{\pi\ii|h|^2/\tau}}{|Q^*/Q|^{1/2}}\sum_{\substack{[\ga+Q]\in \mathcal{O}}}e^{-2\pi\ii(\la|\ga)}\chi_{V_{\ga+Q}}(\tau, h)\\
&+\frac{D_{\la}^\si e^{\pi\ii|h|^2/\tau}}{2|\Z_\mu|}\sum_{\substack{\mu+\pi_0(Q)\in M^*/\pi_0(Q) \\ \zeta\in\Z_\mu}} (-1)^{j}e^{-2\pi\ii(\la|\mu)}(\chi_{M(\mu,\zeta)}^0+\chi_{M(\mu,\zeta)}^1)(\tau, h),%\label{O22}
\end{align*}
where $D_{\la}^\si=\begin{cases}\displaystyle\frac{2^{\dim\lieh_\perp/2}}{|M^*/M|^{1/2}d(\si)},&\la\in Q^*\cap\lieh_0,\\0,&\la\notin Q^*\cap\lieh_0.\end{cases}$

\item The transformation laws for the orbifold characters of\/ $V_{\la+Q}$ for $\la\in Q^*$ with $(1-\si)\la\notin Q$ are$:$
\begin{align*}
\chi_{V_{\la+Q}}(\tau+1,h)&=e^{\pi\ii(|\la|^2-\frac{r}{12})}\chi_{V_{\la+Q}}(\tau,h),%\label{O32}
\end{align*}
\begin{align*}
\chi_{V_{\la+Q}} &\left(-\frac{1}{\tau}, \frac {h}{\tau}\right)=\frac{e^{\pi\ii|h|^2/\tau}}{|Q^*/Q|^{1/2}}\sum_{\substack{\ga+Q\in Q^*/Q\\(1-\si)\ga\in Q}}e^{-2\pi\ii(\la|\ga)}(\chi_{V_{\ga+Q}}^0+\chi_{V_{\ga+Q}}^1)(\tau, h)\\
&+\frac{e^{\pi\ii|h|^2/\tau}}{|Q^*/Q|^{1/2}}\sum_{[\ga+Q]\in \mathcal{O}}(e^{-2\pi\ii(\la|\ga)}+e^{-2\pi\ii(\la|\si\ga)})\chi_{V_{\ga+Q}}(\tau, h).\\%\label{O42}
\end{align*}
%where $E_{\la,\ga}=e^{-2\pi\ii(\la|\ga)}+e^{-2\pi\ii(\la|\si\ga)}$.

\item The transformation laws for modified characters of\/ $M(\mu,\zeta)^j$, with $\mu\in\pi_0(Q^*)$, $\zeta$ is a central character of\/ $G_\si^\perp$, and\/ $j=0, 1$ are$:$
\end{enumerate}
\begin{align*}
\chi_{M(\mu,\zeta)}^{j}(\tau+1, h)&=\om^{-j}e^{2\pi\ii\Delta_{\si}}e^{\pi\ii(|\mu|^2-\frac{r}{12})}\chi_{M(\mu,\zeta)}^j(\tau, h),%\label{O52}
\end{align*}
\begin{align*}
&\chi^{j}_{M(\mu,\zeta)}\left(-\frac{1}{\tau}, \frac {h}{\tau}\right)
=\frac{v_0}{2}e^{\pi\ii|h|^2/\tau}  \!\!\! \sum_{\substack{\la+M\in M^*/M\\\la\in Q^*\cap\lieh_0}}  \!\!\! 
e^{-2\pi\ii (\la|\mu)}(\chi^{0}_{V_{\la+Q}}-\chi^{1}_{V_{\la+Q}})(\tau, h)\\
&+\frac{\ii^{r_0/2}v_1}{2}  \!\!\!\!\! \sum_{\substack{\nu+\pi_0(Q)\in\frac12Q^*\cap\lieh_0/\pi_0(Q)\\\nu\in \mu+Q^*\cap\lieh_0}}  \!\!\!\!\!\!\!\!\!\!\!\!\!\!\! 
(-1)^{j}e^{\pi\ii(-2(\nu|\mu)+\sqrt{2}(\nu-\mu|\be_0))}(\chi^{0}_{M(\nu, \zeta)}-\chi^{1}_{M(\nu, \zeta)})(\tau, h),
%\label{O62}
\end{align*}
where $r_0=\dim\lieh_0$, $v_1$ is given in \eqref{v1gen2}, \eqref{cbe0}, and
\begin{equation*}
v_0=\frac{d(\si)|\pi_0(Q)/M|}{2^{\dim\lieh_\perp/2}|M^*/M|^{1/2}},\quad \Delta_\si=\frac{\dim\lieh_\perp}{16}.
\end{equation*}
%For each $k\neq0$ in \eqref{O6}, each $v_k$ is a complex number.
% $k>0$, are given in \eqref{vk} {\upshape{(}}see also the proof of \eqref{chiWtlaws5}{\upshape{)}}. 
\end{theorem}

\begin{proof}
Recall that the transformation laws for orbifold characters of Types 1 and 2 remain valid for $p=2$ (see Theorems \ref{chiorbtran} and \ref{chiorbtran2}).
The case of Type 3 is similar to the proof of \eqref{O1}--\eqref{O6}. Here we prove the last transformation formula. 
For convenience we set:
$$\la=\sqrt{2}\mu\in\sqrt{2}\pi_0(Q^*), \qquad\de=\sqrt{2}\nu\in\sqrt{2}(Q^*\cap\lieh_0)\subset\sqrt{2}\pi_0(Q^*).$$ 
Using Proposition \ref{chiW2} and equations \eqref{P62}, \eqref{v1gen2},  we calculate:
%Using this substitution, we find

%\frac{\tau^{r_0/2}}{|\mathcal{L}^*/\mathcal{L}|}\sum_{\de+\mathcal{L}\in \mathcal{L}^*/\mathcal{L}}\bar{c}_{\de-\la}e^{-\pi\ii|\de|^2}\th_{\de+\mathcal{L}}\left(\frac{\tau+1}{2},\frac{h}{\sqrt{2}},0\right)
\allowdisplaybreaks
\begin{equation}\label{chisisi}
\begin{split}
&\chi_{M(\mu,\zeta)}^{\si,\si}\left(-\frac{1}{\tau}, \frac{h}{\tau}\right)\\
&=d(\si)e^{-\pi\ii |\mu|^2}\frac{\th_{\sqrt{2}\mu+\sqrt{2}\pi_0(Q)}\left(\frac{-\frac{1}{\tau}+1}{2},\frac{h}{\tau\sqrt{2}}, 0\right)}{P_{\si,\si}\left(-\frac{1}{\tau}\right)}\\
&=d(\si)e^{-\pi\ii|\mu|^2}\frac{\th_{\la+\mathcal{L}}\left(A\cdot(\tau',h',0)\right)}{P_{\si,\si}\left(-\frac{1}{\tau}\right)}\\
&=\frac{\tau^{r_0/2}d(\si)c_{\be_0}e^{-\pi\ii|\mu|^2}}{(-\ii\tau)^{r_{0}/2}|\mathcal{L}^*/\mathcal{L}|}\sum_{\substack{\de+\mathcal{L}\in \mathcal{L}^*/\mathcal{L}\\\de\in2\mathcal{L}^*}}e^{\pi\ii(|\nu|^2+\sqrt{2}(\nu|\be_0))}\frac{\th_{\la+\de+\mathcal{L}}\left(\frac{\tau-1}{2},\frac{h}{\sqrt{2}},0\right)}{P_{\si,\si}(\tau)} \\
&=\frac{\ii^{r_0/2}d(\si)c_{\be_0}}{|\mathcal{L}^*/\mathcal{L}|}\sum_{\substack{\de+\mathcal{L}\in \mathcal{L}^*/\mathcal{L}\\\de\in\la+2\mathcal{L}^*}}e^{\pi\ii(|\nu-\mu|^2-|\mu|^2+\sqrt{2}(\nu-\mu|\be_0))}\frac{\th_{\de+\mathcal{L}}\left(\frac{\tau-1}{2},\frac{h}{\sqrt{2}},0\right)}{P_{\si,\si}(\tau)}\\
&=\frac{\ii^{r_0/2}d(\si)c_{\be_0}}{|\mathcal{L}^*/\mathcal{L}|} \!\!\! \sum_{\substack{\de+\mathcal{L}\in \mathcal{L}^*/\mathcal{L}\\\de\in\la+2\mathcal{L}^*}}  \!\!\! e^{\pi\ii(|\nu|^2-2(\nu|\mu)+\sqrt{2}(\nu-\mu|\be_0))}e^{-2\pi\ii|\nu|^2}\frac{\th_{\de+\mathcal{L}}\left(\frac{\tau+1}{2},\frac{h}{\sqrt{2}},0\right)}{P_{\si,\si}(\tau)}\\
&=\frac{\ii^{r_0/2}c_{\be_0}}{|\mathcal{L}^*/\mathcal{L}|}\sum_{\substack{\nu+\pi_0(Q)\in\frac12Q^*\cap\lieh_0/\pi_0(Q)\\\nu\in \mu+Q^*\cap\lieh_0}}e^{\pi\ii(-2(\nu|\mu)+\sqrt{2}(\nu-\mu|\be_0))}\chi_{M(\nu,\zeta)}^{\si,\si}\left(\tau, h\right).\\
%&=\frac{\ii^{r_0/2}c_{\be_0}}{|\mathcal{L}^*/\mathcal{L}|}\sum_{\substack{\de+\mathcal{L}\in \mathcal{L}^*/\mathcal{L}\\\de\in\la+2\mathcal{L}^*}}e^{\pi\ii(-(\de|\la)+(\de-\la|\be_0))}\chi_{M(\nu,\zeta)}^{\si,\si}\left(\tau, h\right),
\end{split}
\end{equation}
Note that we used \eqref{thtlaws2} to change the input of the theta function from $\frac{\tau-1}{2}$ to $\frac{\tau+1}{2}$.
% are both elements in $\sqrt{2}\pi_0(Q^*)$.
%$\mu=\displaystyle\frac{\la}{\sqrt{2}}\in\pi_0(Q^*)$ and $\de'\in\mathcal{L}^*$ is such that $\de-\la+\mathcal{L}=2\de'+\mathcal{L}$.
Now recall that  (cf.\ \eqref{chiMj})

\begin{align*}%\label{gen2pf}
\chi_{M(\mu,\zeta)}^j(\tau)&=\frac{1}{2}(\chi_{M(\mu,\zeta)}^{\si,1}(\tau)+(-1)^j\chi_{M(\mu,\zeta)}^{\si,\si}(\tau)),\quad j=0,1.
%\chi_{M(\mu,\zeta)}^{\si,\si^{-1}}(\tau, h)&=\sum_{l=0}^{p-1}\om^{l}\chi_{M(\mu,\zeta;\si)}^l(\tau, h).
\end{align*}
The transformation of the trace functions $\chi_{M(\mu,\zeta)}^{\si,1}(\tau,h)$ follow similarly to the proof of \eqref{chiWtlaws6} and \eqref{M9}. 
Hence, the last transformation formula now follows from %\eqref{M9}, %\eqref{chisisi}, \eqref{gen2pf}, 
\eqref{chisisi} and \eqref{chiMinvert}.
\end{proof}

\begin{remark}
All formulas in Theorem \ref{chiorbtranrank2} except the last one agree with \eqref{O1}--\eqref{O6} for $p=2$. The only difference in the last formula is the addition of a possible ingredient $\be_0$ (cf.\ Remark \ref{beta0}). In the case when $\be_0=0$, this last transformation formula in Theorem \ref{chiorbtranrank2} agrees with \eqref{O6}.
\end{remark}

%%%%%%%%%%%%%%%%%%%%%%%%%%%%%%%%%%%
\subsection{The irreducible characters and $S$-matrix of a $\ZZ_2$-orbifold using the root lattice $A_2$}\label{A2}
%%%%%%%%%%%%%%%%%%%%%%%%%%%%%%%%%%%

Consider the simple roots $\{\alpha_1, \alpha_2\}$ associated to the root lattice $A_2=\mathbb{Z}\alpha_1+\mathbb{Z}\alpha_2$, and the Dynkin diagram automorphism $\sigma\colon \alpha_1 \leftrightarrow \alpha_2$. Recall that $|\alpha_1 |^2=|\alpha_2 |^2=2$ and $ (\alpha_1 |\al_2)=-1.$
This example will emphasize the roles of the sublattice $\bar{Q}$ (cf.\ Section \ref{subsub}) as well as the set $\Z_\mu$ (cf.\ \eqref{thetainvert} and the discussion before it). The classification of irreducible orbifold representations for the root lattice $A_n$ for even $n$ is treated in \cite{E}. Another construction of the modules presented here is discussed in \cite{BE}, which uses a previously studied orbifold in \cite{DN}.

For convenience, we set $\alpha=\alpha_1+\alpha_2$ and $\beta=\alpha_1-\alpha_2$. Then
\[
|\alpha |^2=2 \,, \qquad |\beta |^2=6 \,, \qquad (\alpha |\beta)=0 \,.
\]
Using Lemma \ref{Qbarlemma}, we find that  $\al_1, \al_2\notin \bar{A_2}$ and $\bar{A_2}= \ZZ\al \oplus\ZZ\be$.
For the rest of this subsection, we set 
\[
Q=\ZZ\al \oplus\ZZ\be.
\] 
Then for $\lieh=\CC\otimes_\ZZ Q$, we find that 
\[
M=Q\cap\lieh_0=\ZZ\alpha=\pi_0(Q)
\] 
and $Q\cap\lieh_\perp=\ZZ\be$. Since $(1-\si)Q=\ZZ2\be$ is not equal to $Q\cap\lieh_\perp$, the central characters of $G_\si$ are needed in the description of twisted $V_Q$-modules (cf.\ \reref{remark1}).

Next we describe the irreducible orbifold modules. We first find that 
\[
Q^*=\ZZ\frac{\al}{2} \oplus\ZZ\frac{\be}{6} \,,
\]
and the set of $\si$-invariants in $Q^*/Q$ is
\begin{align}\label{A2Type1}
(Q^*/Q)^\si=\ZZ\frac{\al}{2} \oplus\ZZ\frac{\be}{2}\cong\ZZ_2\times\ZZ_2.
\end{align}
We also note that
\[
M^*=\pi_0(Q^*)=\ZZ\frac{\al}{2} \,.
\]
We then have $\pi_0(Q^*)/\pi_0(Q)\cong\ZZ_2$, and that the group $G_\si^\perp$ is an abelian group generated by $U_\be$ such that $U_{-\be}=U_{\be}$. Therefore, there are four $(\mu,\zeta)$ pairs which describe the irreducible twisted $V_Q$-modules (cf.\ \thref{bij}), and this agrees with using the $\si$-invariant elements in $Q^*/Q$ to describe the twisted modules (cf.\ \cite{BK}). Since the automorphism $\si$ acts on each of these modules, each of them will decompose into two eigenspaces of $\si$, resulting in 8 irreducible $V_Q^\si$-modules.

Since $\si$ acts on any untwisted $V_Q$-module $V_{\la+Q}$ if the coset $\la+Q$ is fixed under $\si$, we obtain from \eqref{A2Type1} 8 irreducible $V_Q^\si$-modules of Type 1. The $V_Q$-modules on which $\si$ does not act form orbits comprised of two modules, each isomorphic as orbifold modules under $\si$. Hence there are $\frac12(12-4)=4$ irreducible $V_Q^\si$-modules of Type 2. All together that makes 20 irreducible $V_Q^\si$-modules given in the following list:
\[
V_{i,j}^l=V_{i\frac{\al}{2}+j\frac{\be}{2}+Q}^l,\qquad i, j=0, 1,
\]
\[
V_{i,j}=V_{i\frac{\al}{2}+j\frac{\be}{6}+Q},\qquad i=0, 1, \;\; j=1, 2,
\]
\[
M(0, \zeta_k)^l,\quad M\left(\frac{\al}{2}, \zeta_k\right)^l, \qquad k=1,2,
\]
where $l=0,1$ denotes the eigenspaces of $\si$ and $\zeta_1, \zeta_2$ are the characters of $G_\si^\perp$. Note that 
\[
\si\colon V_{i\frac{\al}{2}+j\frac{\be}{6}+Q}\rightarrow V_{i\frac{\al}{2}-j\frac{\be}{6}+Q}.
\]
The characters of the above modules are given by:
%\vspace{0.1in}
\allowdisplaybreaks
\begin{align}
\chi_{i,j}^l(\tau)&=\frac12\left(\frac{\th_{i\frac{\al}{2}+j\frac{\be}{2}+Q}\left({\tau}{}\right)}{P_{1,1}(\tau)} + (-1)^l \frac{\th_{i\frac{\al}{2}+\ZZ\al}(\tau)}{P_{1,\si}(\tau)}\right),\qquad i, j=0, 1,\\
\chi_{i,j}(\tau)&=\frac{\th_{i\frac{\al}{2}+j\frac{\be}{6}+Q}(\tau)}{P_{1,1}(\tau)},\qquad i=0, 1, \;\; j=1, 2,\\
\chi_{M(0,\zeta_k)}^l(\tau)&=\frac12\left(\frac{\th_{\ZZ\al}\left({\tau}{}\right)}{P_{\si,1}(\tau)} + (-1)^l \frac{\th_{\ZZ\al}\left({\tau}{}\right)}{P_{\si,\si}(\tau)}\right),\qquad k=1,2,\\
\chi_{M(\frac{\al}{2},\zeta_k)}^l(\tau)&=\frac12\left(\frac{\th_{\frac{\al}{2}+\ZZ\al}\left({\tau}{}\right)}{P_{\si,1}(\tau)} + (-1)^{l+1} \ii\frac{\th_{\frac{\al}{2}+\ZZ\al}\left({\tau}{}\right)}{P_{\si,\si}(\tau)}\right),\qquad k=1,2,
\end{align}
%\vspace{0.1in}\noindent
where we used \eqref{thtlaws2} and \eqref{chiMk2}, since $\pi_0(Q)=\ZZ\al$ is an even integral lattice.

In this case we find that $\be_0$ can be set to zero in Theorem \ref{KP}. We also calculate that for $\mathcal{L}=\ZZ\sqrt{2}\al$, \eqref{cbe0} becomes $c_{0}=2-2\ii$, and 
\begin{equation}\label{v1A22}
v_1=\frac12(1-\ii)
\end{equation} 
(cf.\ \eqref{v1gen2}). Using Theorem \ref{chiorbtran} and \eqref{v1A22}, we obtain the following transformation laws of orbifold characters:

\begin{align}
\chi_{i,j}^l(\tau+1)=e^{\frac{\pi\ii}{6}(3i+9j-1)}\chi_{i,j}^l(\tau),
\end{align}

\begin{align}\label{A2S1}
\begin{split}
\chi_{i,j}^l\left(-\frac{1}{\tau}\right)=&\frac{1}{2\sqrt{12}}\sum_{m,n=0}^1(-1)^{im+jn}(\chi_{m,n}^0+\chi_{m,n}^1)(\tau)\\
&+\frac{1}{\sqrt{12}}\sum_{m=0}^1\sum_{n=1}^2(-1)^{im+jn}\chi_{m,n}(\tau)\\
&+\frac14\sum_{r=0}^1\sum_{k=1}^2(-1)^{l+ir}(\chi_{M(r\frac{\al}{2},\zeta_k)}^0+\chi_{M(r\frac{\al}{2},\zeta_k)}^1)(\tau),
\end{split}
\end{align}
where $l=0,1$ and $i, j=0, 1$,

\begin{align}
\chi_{i,j}(\tau+1)=e^{\frac{\pi\ii}{6}(3i+j^2-1)}\chi_{i,j}(\tau),
\end{align}

\begin{align}
\begin{split}
\chi_{i,j}\left(-\frac{1}{\tau}\right)=&\frac{1}{\sqrt{12}}\sum_{m,n=0}^1(-1)^{im+jn}(\chi_{m,n}^0+\chi_{m,n}^1)(\tau)\\
&+\frac{1}{\sqrt{12}}\sum_{m=0}^1\sum_{n=1}^2(-1)^{im}\cosh\left(\frac{\pi\ii}{3}jn\right)\chi_{m,n}(\tau),
\end{split}
\end{align}
where $i=0, 1$ and $j=1, 2$,

\begin{align}
\chi_{M(0,\zeta_k)}^l(\tau+1)=(-1)^le^{-\pi\ii/24}\chi_{M(0,\zeta_k)}^l(\tau),
\end{align}

\begin{align}
\chi_{M(\frac{\al}{2},\zeta_k)}^l(\tau+1)=(-1)^le^{7\pi\ii/24}\chi_{M(\frac{\al}{2},\zeta_k)}^l(\tau),
\end{align}

\begin{align}
\begin{split}
\chi_{M(0,\zeta_k)}^l\left(-\frac{1}{\tau}\right)=&\frac{1}{4}(\chi_{0,0}^0-\chi_{0,0}^1+\chi_{1,0}^0-\chi_{1,0}^1)(\tau)\\
&+\frac{1}{\sqrt{2}}(-1)^l\sum_{r=0}^1(\chi_{M(r\frac{\al}{2},\zeta_k)}^0-\chi_{M(r\frac{\al}{2},\zeta_k)}^1)(\tau),
\end{split}
\end{align}

\begin{align}
\begin{split}
\chi_{M(\frac{\al}{2},\zeta_k)}^l\left(-\frac{1}{\tau}\right)=&\frac{1}{4}(\chi_{0,0}^0-\chi_{0,0}^1-\chi_{1,0}^0+\chi_{1,0}^1)(\tau)\\
&+\frac{1}{2\sqrt{2}}\sum_{r=0}^1(-1)^{l+r}(\chi_{M(r\frac{\al}{2},\zeta_k)}^0-\chi_{M(r\frac{\al}{2},\zeta_k)}^1)(\tau),
\end{split}
\end{align}
where $l=0,1$ and $k=0, 1$.

In the following table, we present the asymptotic and quantum dimensions using the coefficients of the linear combination of characters in \eqref{A2S1}; see also Corollary \ref{KacCor}.

\begin{center}
\begin{table}[H]\label{tab2}
\begin{tabular}{|c|c|c|c|}
\hline
$M$&$V_{i,j}^l$&$V_{i,j}$&$M(\mu,\zeta_k)^l$\\
\hline
$\asdim M$&$\frac{1}{2\sqrt{12}}$&$\frac{1}{\sqrt{12}}$&$\frac14$\\
\hline
$\qdim M$&$1$&$2$&$\sqrt{3}$\\
\hline
\end{tabular}
\caption{Asymptotic and quantum dimensions}
\end{table}
\end{center}

%%%%%%%%%%%%%%%%%%%%%%%%%%%%%%%%%%%
\subsection{The irreducible characters and $S$-matrix of a $\ZZ_2$-orbifold using the root lattice $A_3$}\label{A3}
%%%%%%%%%%%%%%%%%%%%%%%%%%%%%%%%%%%

Consider the simple roots $\{\alpha_1, \alpha_2, \al_3\}$ associated to the root lattice $A_3=\mathbb{Z}\alpha_1+\mathbb{Z}\alpha_2+\ZZ\al_3$, and the Dynkin diagram automorphism 
\begin{align}\label{siA3}
\si(\al_2)=\al_2 \quad\text{ and }\quad \si\colon\al_1\leftrightarrow\al_3.
\end{align} 
Recall that $|\alpha_i |^2=2$ for all $i$, $ (\alpha_1 |\al_2)=(\alpha_2 |\al_3)=-1$, and $(\al_1|\al_3)=0$. 

In this section, we provide an alternative way to do the modular transformations of characters by computing them more directly. This example is special in that we are able to also obtain the $S$-matrix and fusion rules among irreducible orbifold characters. The classification of irreducible orbifold representations for the root lattice $A_n$ for odd $n$ is treated in \cite{E}. 

Using Lemma \ref{Qbarlemma}, we find that $Q=\bar{Q}$. Hence we set $Q=A_3$ and denote the eigenvectors of $\si$ by
\begin{align*}
\al=\al_1+\al_3,\quad
\be=\al_1-\al_3.
\end{align*}
Then $|\al|^2=4=|\be|^2$, and $(\al|\be)=0$. We also find that 
\begin{equation}\label{Z2nondeg}
M=Q\cap\lieh_0=\ZZ\al_2\oplus\ZZ\al,\qquad (1-\si)Q=Q\cap\lieh_\perp=\ZZ\be. 
\end{equation}
It follows from Lemma \ref{defectlemma} that $d(\si)=1$. Therefore, the irreducible twisted $V_Q$-modules can be classified using only $\mu\in\pi_0(Q^*)$ (cf.\ Theorem \ref{bij} and \reref{remark1}).
%Using the relations $\al_3=\al_1-\al$, and $2\al_1=\al+\be$, we obtain $Q/L=\langle\al_1+L\rangle\cong\ZZ_2$ (cf.\ \eqref{L}).
The dual lattice $Q^*$ is spanned by the weights
\begin{align*}
\La_1&=\frac14(3\al_1+2\al_2+\al_3),\quad\La_3=\frac14(\al_1+2\al_2+3\al_3),\\
\La_2&=\frac14(2\al_1+4\al_2+2\al_3)=\frac{\al}{2}+\al_2,
\end{align*}
and the fundamental group is cyclic of order 4:
\begin{align}
Q^*/Q=\langle\La_1+Q\rangle\cong\ZZ_4.
\end{align}
Note that $\La_2+M=\frac{\al}{2}+M=\pi_0\al_1+M$, and 
\[
M^*=\pi_0(Q^*)=\left\langle\frac{\al_2}{2}, \frac{\al}{2}\right\rangle. 
\]
Now $\pi_0$ fixes $\al_2$ and $\La_2$, and $\pi_0\al_1=\pi_0\al_3=\frac{\al}{2}$, $\pi_0\La_1=\pi_0\La_3=\frac{\al}{2}+\frac{\al_2}{2}$, so that
\begin{align}
\pi_0(Q)&=\left\langle\frac{\al}{2},\al_2\right\rangle,\\
\pi_0(Q^*)/\pi_0(Q)&=\displaystyle\left\langle\frac{\al_2}{2}+\pi_0(Q)\right\rangle.
\end{align}
Hence we may take $\mu=0,\frac{\al_2}{2}$ in describing distinct irreducible twisted $V_Q$-modules.
Note that $\pi_0(Q)$ is a self-dual integral lattice. 

%\textcolor{red}{$(Q\cap\lieh_0|\mu)\in\ZZ$ so that we may also take $h=0$ in \eqref{chiW2}.}

It turns out that the trace functions \eqref{chiM} with $h\in\mathfrak{h}_0$ set to zero are sufficient to describe all irreducible $V_Q^\si$-modules. Since the only $\si$-invariant cosets in $Q^*/Q$ are $\La_2+Q$ and the trivial coset, there are 4 irreducible orbifold modules of Type 1: $V_Q^\pm$ and $V_{\La_2+Q}^\pm$ corresponding to the eigenspaces of $\si$ on $V_Q$ and $V_{\La_2+Q}$. The generator $\La_1+Q$ of $Q^*/Q$ forms an orbit of order 2, and there is only one irreducible orbifold module of Type 2. In addition, there are 2 irreducible orbifold modules of twisted type given by $\mu=0,\frac{\al_2}{2}$, and each of them breaks into two eigenspaces of $\si$.
All together there are 9 irreducible $V_Q^\si$-modules.

For convenience we set $\th_{\la+Q}(\tau)=\th_{\la+Q}(\tau,0,0)$ (cf.\ \eqref{th}). Since the central characters of $G_\si^\perp$ are not needed in this example, we also set for convenience 
\begin{equation}
\chi_{M(\mu)}^{\si,\si^k}(\tau)=\chi_{M(\mu,\zeta)}^{\si, \si^k}(\tau,0),
\end{equation}
where $k=0, 1$ (cf.\ \eqref{chiW}). Using \thref{chiF} and \prref{chiW2}, we obtain the following trace functions on irreducible untwisted $V_Q$-modules (the characters corresponding to superscript $1, 1$):

\begin{align*}
\chi_Q^{1,1} (\tau) &=\frac{\th_{Q}(\tau) }{ P_{1,1}(\tau) } \,, \qquad &\chi_Q^{1,\si} (\tau) = 
\frac{\th_{M}(\tau) }{ P_{1,\si}(\tau) } 
\,, 
\\
\chi_{\Lambda_1+Q}^{1,1} (\tau) &= 
\frac{\th_{\La_1+Q}(\tau) }{ P_{1,1}(\tau) } \,,\qquad &\chi_{\La_2+Q}^{1,\si} (\tau) = 
\frac{\th_{\La_2+M}(\tau) }{ P_{1,\si}(\tau) }\,,
\\
\chi_{\La_2+Q}^{1,1} (\tau) &= 
\frac{\th_{\La_2+Q}(\tau) }{ P_{1,1}(\tau) } \,.
\end{align*}
We also get the following trace functions on irreducible twisted $V_Q$-modules (the characters corresponding to $\si, 1$):

%\begin{align*}
%\chi_{M(0)}^{\si,1} (\tau) &=\frac{\th_{\pi_0(Q)}(\tau) }{ P_{\si,1}(\tau) } \,, \qquad &\chi_{M(0)}^{\si,\si} (\tau) = 
%\frac{\th_{\pi_0(Q)}(\tau+1) }{ P_{\si,\si}(\tau)}= 
%\frac{(\th_{M}-\th_{\frac{\al}{2}+M})(\tau) }{ P_{\si,\si}(\tau) } 
%\,, 
%\\
%\chi_{M(\frac{\al_2}{2})}^{\si,1} (\tau) &= 
%\frac{\th_{\frac{\al_2}{2}+\pi_0(Q)}(\tau) }{ P_{\si,1}(\tau) } \,,\qquad &\chi_{M(\frac{\al_2}{2})}^{\si,\si} (\tau) = \frac{\th_{\frac{\al_2}{2}+\pi_0(Q)}(\tau+1) }{ P_{\si,\si}(\tau)}=
%\frac{(\th_{\frac{\al_2}{2}+M}-\th_{\frac{\al_2}{2}+\frac{\al}{2}+M})(\tau) }{ P_{\si,\si}(\tau) },
%\end{align*}

\begin{align*}
\chi_{M(0)}^{\si,1} (\tau) &=\frac{\th_{\pi_0(Q)}(\tau) }{ P_{\si,1}(\tau) } \,, \qquad&\chi_{M(0)}^{\si,\si} (\tau) &= \frac{\th_{\sqrt{2}\pi_0(Q)}(\frac{\tau+1}{2}) }{ P_{\si,\si}(\tau)} \,,\\
\chi_{M(\frac{\al_2}{2})}^{\si,1} (\tau) &= 
\frac{\th_{\frac{\al_2}{2}+\pi_0(Q)}(\tau) }{ P_{\si,1}(\tau) } \,,\qquad&\chi_{M(\frac{\al_2}{2})}^{\si,\si} (\tau) &= \frac{\th_{\frac{\al_2}{\sqrt{2}}+\sqrt{2}\pi_0(Q)}(\frac{\tau+1}{2}) }{ P_{\si,\si}(\tau)} \,,
\end{align*}
where $\th_{\pi_0(Q)}(\tau)=\th_M(\tau)+\th_{\frac{\al}{2}+M}(\tau)$. 

From these trace functions, we compute the irreducible charaters for the orbifold by taking trace of projections of $\si$. We label the orbifold characters as follows:
\begin{align*}
\chi_Q^\pm(\tau)&=\frac12(\chi_Q^{1,1}\pm\chi_Q^{1,\si})(\tau),\\
\chi_{\La_2+Q}^\pm(\tau)&=\frac12(\chi_{\La_2+Q}^{1,1}\pm\chi_{\La_2+Q}^{1,\si})(\tau),\\
\chi_{\La_1+Q}(\tau)&=\chi_{\La_1+Q}^{1,1}(\tau),\\
\chi_{M(0)}^{\si,\pm}(\tau)&=\frac12(\chi_{M(0)}^{\si,1}\pm\chi_{M(0)}^{\si,\si})(\tau),\\
\chi_{M(\frac{\al_2}{2})}^{\si,\pm}(\tau)&=\frac12(\chi_{M(\frac{\al_2}{2})}^{\si,1}\pm\chi_{M(\frac{\al_2}{2})}^{\si,\si})(\tau).
\end{align*}

We compute the transformation laws under $\tau\mapsto-{1}/{\tau}$ directly,
by invoking Theorem \ref{thtlaws} (with $z=u=0$) to get:
\begin{align*}
\th_Q\Bigl(-\frac1\tau\Bigr) 
%&=\frac12(-\ii\tau)^\frac32\sum_{\mu+Q\in Q^*/Q}\th_{\mu+Q}\\
&=\frac12(-\ii\tau)^\frac32(\th_Q+2\th_{\La_1+Q}+\th_{\La_2+Q})(\tau),\\
\th_M\Bigl(-\frac1\tau\Bigr)&=\frac12(-\ii\tau)(\th_M+\th_{\frac{\al}{2}+M}+\th_{\frac{\al_2}{2}+M}+\th_{\frac{\al+\al_2}{2}+M})(\tau),\\
\th_{\frac{\al}{2}+M}\Bigl(-\frac1\tau\Bigr)&=\frac12(-\ii\tau)(\th_M+\th_{\frac{\al}{2}+M}-\th_{\frac{\al_2}{2}+M}-\th_{\frac{\al+\al_2}{2}+M})(\tau),
\end{align*}
and Theorem \ref{chiFtran} to get:
\begin{align*}
P_{1,1}\Bigl(-\frac1\tau\Bigr)&=(-\ii\tau)^\frac32P_{1,1}(\tau),\quad
P_{1,\si}\Bigl(-\frac1\tau\Bigr)=\frac{1}{\sqrt{2}}(-\ii\tau)P_{\si,1}(\tau).
\end{align*}
Using these, we compute the transformations for the trace functions. For example, 
\begin{align*}
\chi_0^{1,1} \Bigl(-\frac1\tau\Bigr) &=\frac12 \chi_0^{1,1} (\tau) + \frac12 \chi_{\La_2}^{1,1} (\tau)
+ \chi_{\La_1}^{1,1} (\tau),\\
\chi_0^{1,\si} \Bigl(-\frac1\tau\Bigr) &= 
\frac1{\sqrt2} \chi_0^{\si,1} (\tau)
+ \frac1{\sqrt2} \chi_{\frac{\al_2}{2}}^{\si,1} (\tau).
\end{align*}
We find the transformation laws for the other trace functions on irreducible $V_Q$-modules in a similar way
and state the answer in the following lemma.

\begin{lemma}\label{A3VQtlaws}
The transformation laws under $\tau\mapsto-{1}/{\tau}$ of the trace functions on irreducible $V_Q$-modules are$:$
\begin{align*}
\chi_{\la}^{1,1} \Bigl(-\frac1\tau\Bigr) &= 
\frac12 \chi_0^{1,1} (\tau) + (-1)^{|\la|^2}\chi_{\La_1}^{1,1} (\tau)+ \frac12 \chi_{\La_2}^{1,1} (\tau),\quad \la=0, \La_2%\frac{\al_2}{2}
 ,
\\
\chi_{\La_1}^{1,1} \Bigl(-\frac1\tau\Bigr) &= 
\frac12 \chi_0^{1,1} (\tau) - \frac12 \chi_{\La_2}^{1,1} (\tau) \,,
\\
%\chi_{\La_2}^{1,1} \Bigl(-\frac1\tau\Bigr) &= 
%\frac12 \chi_0^{1,1} (\tau) - \chi_{\La_1}^{1,1} (\tau)+ \frac12 \chi_{\La_2}^{1,1} (\tau)
% \,,
%\\
\chi_{\la}^{1,\si} \Bigl(-\frac1\tau\Bigr) &= 
\frac1{\sqrt2} \left(\chi_0^{\si,1} (\tau)
+ (-1)^{|\la|^2}\chi_{\frac{\al_2}{2}}^{\si,1} (\tau)\right),\quad \la=0, \La_2\,,
\\
%\chi_{\La_2}^{1,\si} \Bigl(-\frac1\tau\Bigr) &= 
%\frac1{\sqrt2} \chi_0^{\si,1} (\tau)
%- \frac1{\sqrt2} \chi_{\frac{\al_2}{2}}^{\si,1} (\tau) \,,
%\\
\chi_{\mu}^{\si,1} \Bigl(-\frac1\tau\Bigr) &=
\frac1{\sqrt2} \left(\chi_0^{1,\si} (\tau) 
+ (-1)^{(\mu|\al_2)}\chi_{\La_2}^{1,\si} (\tau)\right),\quad \mu=0, \frac{\al_2}{2}\,\\
\chi_0^{\si,\si} \Bigl(-\frac1\tau\Bigr) &=\chi_{\frac{\al_2}{2}}^{\si,\si}(\tau) \,,\qquad
\chi_{\frac{\al_2}{2}}^{\si,\si} \Bigl(-\frac1\tau\Bigr) =\chi_0^{\si,\si}(\tau).
%\chi_{\frac{\al_2}{2}}^{\si,1} \Bigl(-\frac1\tau\Bigr) &=\frac1{\sqrt2} \chi_0^{1,\si} (\tau)-\frac1{\sqrt2} \chi_{\La_2}^{1,\si} (\tau),\\
\end{align*}
\end{lemma}

\begin{proof}
We prove the transformations for $\chi_0^{\si,\si}(\tau)$ and $\chi_{\frac{\al_2}{2}}^{\si,\si}(\tau)$. We have that 
\[
\mathcal{L}=\sqrt{2}\pi_0(Q)=\sqrt{2}\left\langle\frac{\al}{2}, \al_2\right\rangle,\quad \mathcal{L}^*=\frac{1}{\sqrt{2}}\left\langle\frac{\al}{2}, \al_2\right\rangle=\frac{1}{\sqrt{2}}\pi_0(Q)=\frac12\mathcal{L},
\]
and $|\mathcal{L}^*/\mathcal{L}|=4$. In this case we find that 
\begin{equation*}
\be_0=\frac{\al_2}{\sqrt{2}}
\end{equation*}
is nonzero and \eqref{cbe0} becomes $c_{\be_0}=|\mathcal{L}^*/\mathcal{L}|$.
%For each coset $\ga+\mathcal{L}\in\mathcal{L}^*/\mathcal{L}$, we calculate the constants \eqref{cgagen2}. We find that $c_{\ga}=4$ for $\ga=\frac{\al_2}{\sqrt{2}}$, and zero otherwise. 
The transformations now follow from \eqref{chisisi} and that  $r_0=2$.
\end{proof}

Using Lemma \ref{A3VQtlaws}, we can now write the transformation laws for irreducible orbifold characters, and determine the $S$-matrix using that the orbifold characters are linear combinations of the trace functions defined above. For example, we show the calculation for the vacuum module $\chi_0^+(\tau)$ and the character $\chi_{\frac{\al_2}{2}}^{\si, +}(\tau)$:
\begin{align*}
\chi_0^+\Bigl(-\frac1\tau\Bigr) &=\frac14 \chi_0^{1,1}  + \frac12\chi_{\La_1} + \frac14 \chi_{\La_2}^{1,1} +\frac1{2\sqrt2} \chi_0^{\si,1} 
+ \frac1{2\sqrt2} \chi_{\frac{\al_2}{2}}^{\si,1} \\
&=\frac14\chi_0^+ +\frac14\chi_0^-+\frac12\chi_{\La_1}+\frac14\chi_{\La_2}^++\frac14\chi_{\La_2}^-\\&+\frac1{2\sqrt2}(\chi_0^{\si,+}+\chi_0^{\si,-}+\chi_{\frac{\al_2}{2}}^{\si,+}+\chi_{\frac{\al_2}{2}}^{\si,-}), \\
\chi_{\frac{\al_2}{2}}^{\si, +}\Bigl(-\frac1\tau\Bigr) &=\frac{1}{\sqrt{2}}\chi_0^{1, \si}-\frac{1}{\sqrt{2}}\chi_{\La_2}^{1, \si}+\chi_0^{\si, \si}\\
&=\frac{1}{2\sqrt{2}}(\chi_0^+-\chi_0^--\chi_{\La_2}^++\chi_{\La_2}^-)+\frac12\chi_0^{\si, +}-\frac12\chi_0^{\si, -}.
\end{align*}

We compute the transformations for the other orbifold characters in a similar way. Using the coefficients in these linear combinations, we then obtain the $S$-matrix. 
This is because the the irreducible characters of this $\ZZ_2$-orbifold are linearly independent, due to the fact that $(1-\si)Q=Q\cap\lieh_\perp$ (see \eqref{Z2nondeg} and \reref{remark1}).

%(cf.\ Remark \ref{Psym}). 

\begin{theorem}\label{thm6.4}
Let\/ $Q$ be the $A_3$ root lattice and\/ $\si$ be the Dynkin diagram automorphism given by\/ \eqref{siA3}. Consider the vector space spanned by the irreducible characters of\/ $V_Q^\si$ with ordered basis
\[
\bigl\{\chi_0^+,\chi_0^-, \chi_{\La_1}^{1,1}, \chi_{\La_2}^+, \chi_{\La_2}^-, \chi_{0}^{\si,+}, \chi_{0}^{\si,-}, \chi_{\frac{\al_2}2}^{\si,+}, \chi_{\frac{\al_2}2}^{\si,-}\bigr\}.
\]
Then the modular transformation $\tau\mapsto-1/{\tau}$ for characters of irreducible\/ $V_Q^\si$-modules is given by the following $S$-matrix$:$
\begin{center}
$S=\displaystyle\frac14
\left(
\begin{array}{rrrrrrrrr}
1&1&2&1&1&\sqrt{2}&\sqrt{2}&\sqrt{2}&\sqrt{2}\\
1&1&2&1&1&-\sqrt{2}&-\sqrt{2}&-\sqrt{2}&-\sqrt{2}\\
2&2&0&-2&-2&0&0&0&0\\
1&1&-2&1&1&\sqrt{2}&\sqrt{2}&-\sqrt{2}&-\sqrt{2}\\
1&1&-2&1&1&-\sqrt{2}&-\sqrt{2}&\sqrt{2}&\sqrt{2}\\
\sqrt{2}&-\sqrt{2}&0&\sqrt{2}&-\sqrt{2}&0&0&2&-2\\
\sqrt{2}&-\sqrt{2}&0&\sqrt{2}&-\sqrt{2}&0&0&-2&2\\
\sqrt{2}&-\sqrt{2}&0&-\sqrt{2}&\sqrt{2}&2&-2&0&0\\
\sqrt{2}&-\sqrt{2}&0&-\sqrt{2}&\sqrt{2}&-2&2&0&0\\
\end{array}
\right)$.
\end{center}
\end{theorem}

%Using the the Verlinde formula \cite{}, we can compute the quantum dimensions and fusion rules for the irreducible orbifold modules.
%Looking at the first row of the $S$-matrix we get that the quantum dimensions of the irreducible orbifold modules of untwisted type are either 1 or 2, and the quantum dimension of each irreducible orbifold module of twisted type is $\sqrt{2}$. 

We also present the asymptotic and quantum dimensions using the coefficients of the linear combination of characters in \eqref{A2S1}; see also Corollary \ref{KacCor}.

\begin{center}
\begin{table}[H]\label{tab3}
\begin{tabular}{|c|c|c|c|}
\hline
$M$&$V_Q^\pm, V_{\La_2+Q}^\pm$&$V_{\La_1+Q}$&$M(0)^\pm, M(\frac{\al_2}{2})^\pm$\\
\hline
$\asdim M$&$\frac{1}{4}$&$\frac{1}{2}$&$\frac{1}{2\sqrt{2}}$\\
\hline
$\qdim M$&$1$&$2$&$\sqrt{2}$\\
\hline
\end{tabular}
\caption{Asymptotic and quantum dimensions}
\end{table}
\end{center}

\begin{remark}
These results agree with the quantum dimensions for irreducible representations of orbifold modules computed in \cite{E2} in the general case of an order $2$ automorphism. 
In general, the quantum dimensions of the irreducible orbifold modules of untwisted type corresponding to an involution are either $1$ or $2$, and the quantum dimensions of the irreducible orbifold modules of twisted type are all equal.
\end{remark}

%%%%%%%%%%%%%%%%%%%%%%%%%%%%%%%
%\subsubsection{The T matrix}
Next we compute the transformation laws for the transformation $\tau\mapsto\tau+1$, which are given in Theorem \ref{thtlaws} with $z=u=0$. For example, 
\begin{align*}
\th_{\La_1+Q}(\tau+1)&=e^{\frac34\pi\ii}\th_{\La_1+Q}(\tau),\\
\th_{\pi_0(Q)}(\tau+1)&=\th_M(\tau+1)+\th_{\frac{\al}{2}+M}(\tau+1)=\th_M(\tau)-\th_{\frac{\al}{2}+M}(\tau),\\
\th_{\frac{\al_2}{2}+\pi_0(Q)}(\tau+1)&=\th_{\frac{\al_2}{2}+M}(\tau+1)+\th_{\frac{\al_2}{2}+\frac{\al}{2}+M}(\tau+1)=\ii\th_{\frac{\al_2}{2}+\pi_0(Q)}(\tau).\\
\end{align*}
To use the results of Theorem \ref{chiFtran}, we first find that $\Delta_1=0$ and $\Delta_\si=\frac{1}{16}$. Then we obtain
\begin{align*}
%\Delta_1&=0,\quad \Delta_\si=-\frac14\sum_{\substack{i=1\\p_i\neq0}}^2(1-p_i)(-p_i)=\frac{1}{16},\\
P_{1,\ph}(\tau+1)&=e^{\pi\ii/4}P_{1,\ph}(\tau),\quad P_{\si,\ph}(\tau+1)=e^{\pi\ii/8}P_{\si,\ph\si}(\tau),
\end{align*}
for $\ph=1,\si$.
We can now compute the transformations $\tau\mapsto\tau+1$ for the trace functions. For example, 
\begin{align*}
\chi_{\La_1}(\tau+1)&=e^{-\frac{1}{4}\pi\ii}e^{\frac34\pi\ii}\chi_{\La_1}(\tau)=\ii\chi_{\La_1}(\tau).
\end{align*}
The other transformations of the trace functions are found in a similar way and the results are stated in the next lemma.

\begin{lemma}\label{A3T}
The transformation laws under $\tau\mapsto\tau+1$ of the trace functions on irreducible\/ $V_Q$-modules are$:$
%\begin{multicols}{2}
\begin{align*}
\chi_0^{1,1} (\tau+1) &= 
e^{-\pi\ii/4}\chi_0^{1,1} (\tau),\quad\quad\chi_0^{\si,1} (\tau+1) =e^{-\pi\ii/8}\chi_0^{\si,\si} (\tau),
\\
\chi_{\La_1}^{1,1} (\tau+1) &= \ii\chi_{\La_1}^{1,1} (\tau) \quad\quad\quad\;\;\quad\chi_{\frac{\al_2}{2}}^{\si,1} (\tau+1) =\ii e^{-\pi\ii/8}\chi_{\frac{\al_2}{2}}^{\si,\si} (\tau),
\\
\chi_{\La_2}^{1,1} (\tau+1) &= -e^{-\pi\ii/4}\chi_{\La_2}^{1,1} (\tau),
 \quad\;\chi_0^{\si,\si} (\tau+1) =e^{-\pi\ii/8}\chi_0^{\si,1} (\tau),
\\
\chi_0^{1,\si} (\tau+1) &= e^{-\pi\ii/4}\chi_0^{1,\si} (\tau),\quad\quad\chi_{\frac{\al_2}{2}}^{\si,\si} (\tau+1) =\ii e^{-\pi\ii/8}\chi_{\frac{\al_2}{2}}^{\si,1} (\tau),
\\
\chi_{\La_2}^{1,\si} (\tau+1) &= -e^{-\pi\ii/4}\chi_{\La_2}^{1,\si} (\tau).
\end{align*}
%\end{multicols}
\end{lemma}

Using Lemma \ref{A3T}, we can now write the transformation laws for irreducible orbifold characters and determine the $T$-matrix.
%using that the orbifold characters are linear combinations of the trace functions defined above.

\begin{theorem}
%Let\/ $Q$ be the $A_3$ root lattice and\/ $\si$ be the Dynkin diagram automorphism given by\/ \eqref{siA3}. Consider the vector space spanned by the irreducible characters of\/ $V_Q^\si$ with ordered basis
%\[
%\bigl\{\chi_0^+,\chi_0^-, \chi_{\La_1}, \chi_{\La_2}^+, \chi_{\La_2}^-, \chi_{0}^{\si,+}, \chi_{0}^{\si,-}, \chi_{\frac{\al_2}2}^{\si,+}, \chi_{\frac{\al_2}2}^{\si,-}\bigr\}.
%\]
In the setting of \thref{thm6.4}, the modular transformation $\tau\mapsto\tau+1$ for irreducible characters of the orbifold vertex algebra\/ $V_Q^\si$ is given by the following diagonal matrix$:$
\begin{align*}
T
%&=\text{\upshape diag}(e^{-2\pi\ii/8},e^{-2\pi\ii/8},e^{3\pi\ii/4-2\pi\ii/8},-e^{-2\pi\ii/8},-e^{-2\pi\ii/8},e^{-2\pi\ii/16},-e^{-2\pi\ii/16},\ii e^{-2\pi\ii/16},-\ii e^{-2\pi\ii/16})\\
&=e^{-\pi\ii/4}\text{\upshape diag}(1,1,e^{3\pi\ii/4},-1,-1,e^{\pi\ii/8},-e^{\pi\ii/8},\ii e^{\pi\ii/8},-\ii e^{\pi\ii/8}).
\end{align*}
%\begin{center}
%$T=
%\left(
%\begin{array}{rrrrrrrrr}
%e^{-2\pi\ii/8}&&&&&&&&\\
%&e^{-2\pi\ii/8}&&&&&&&\\
%&&\ii&&&&&&\\
%&&&-e^{-2\pi\ii/8}&&&&&\\
%&&&&-e^{-2\pi\ii/8}&&&&\\
%&&&&&e^{-2\pi\ii/16}&&&\\
%&&&&&&-e^{-2\pi\ii/16}&&\\
%&&&&&&&\ii e^{-2\pi\ii/16}&\\
%\end{array}
%\right)$
%\end{center}
\end{theorem}

%%%%%%%%%%%%%%%%%%%%%%%%%%%%%
%\subsubsection{The fusion matrix}
Next we describe the fusion matrix corresponding to the fusion product among irreducible orbifold modules. The 9 irreducible orbifold modules for the $A_3$ root lattice and Dynkin diagram automorphism are given in \cite{BE} (where $L_+=Q\cap\mathfrak{h}_0$):
\begin{align}
V_Q^\pm \cong& \bigl(V_{L_+}\otimes V_{\mathbb{Z}\beta}^\pm\bigr)\oplus\bigl(V_{\frac{\alpha}{2}+{L_+}}\otimes V_{\frac{\beta}{2}+\mathbb{Z}\beta}^\pm\bigr)\,, \\
V_{\frac{\al}2+Q}^\pm \cong& \bigl(V_{L_+}\otimes V_{\frac{\beta}{2}+\mathbb{Z}\beta}^\pm\bigr)\oplus\bigl(V_{\frac{\al}{2}+L_+}\otimes V_{\mathbb{Z}\beta}^\pm\bigr) \,\\
V_{\lambda_1+Q} \cong& \bigl(V_{\frac{\alpha_2}{2}+{L_+}}\otimes V_{\frac{\beta}{4}+\mathbb{Z}\beta}\bigr)\oplus\bigl(V_{\frac{\alpha_2}{2}+\frac{\alpha}{2}+{L_+}}\otimes V_{\frac{\beta}{4}+\mathbb{Z}\beta}\bigr)\,,\\
&\bigl(V_{L_+}\otimes V_{\mathbb{Z}\beta}^{T_1,\pm}\bigr)\oplus\bigl(V_{\frac{\alpha}{2}+{L_+}}\otimes V_{\mathbb{Z}\beta}^{T_1,\pm}\bigr) \,,\label{A3mod4} \\
 &\bigl(V_{L_+}\otimes V_{\mathbb{Z}\beta}^{T_2,\pm}\bigr)\oplus\bigl(V_{\frac{\alpha}{2}+{L_+}}\otimes V_{\mathbb{Z}\beta}^{T_2,\mp}\bigr).\label{A3mod5}
\end{align}
Their corresponding characters are $\chi_0^\pm, \chi_{\La_2}^\pm, \chi_{\La_1}^{1,1}, \chi_{0}^{\si,\pm}, \chi_{\frac{\al_2}2}^{\si,\pm}$, respectively. For convenience in writing the matrix describing the fusion algebra, we make the labelling:
 $$
% 0^\pm:=V_Q^\pm,\quad 
\frac{\alpha}{2}^\pm:=V_{\frac{\al}2+Q}^\pm, \quad
 \lambda_1:=V_{\lambda_1+Q},$$ and we denote the modules \eqref{A3mod4} by $T_1^\pm$ and \eqref{A3mod5} by $T_2^\pm$. We also set 
 \begin{align*}
 T_i&=T_i^++T_i^-,\qquad i=1,2,\\ 
 U^{m, n}&=V_Q^m+V_{\frac{\al}2+Q}^n,\qquad m,n\in\{\pm\}.
 \end{align*}
 %We let $n$-tuples of these symbols denote summation of the corresponding modules.

%  that the quantum dimension is 1 for the irreducible orbifold modules with characters $\chi_0^\pm, \chi_{\La_1}^\pm$, 2 for the irreducible orbifold module with character $\chi_{\La_1}^{1,1}$, and $\sqrt{2}$ for each orbifold module of twisted type.
%\begin{Remark}
%These quantum dimensions have also been computed in \cite{E} using known quantum dimensions and fusion rules for $\si=-1$.
%\end{Remark}

%Then the transformation law for \eqref{chiK1} is given by
%
%\begin{align*}
%\chi_0^{1,1} \Bigl(-\frac1\tau\Bigr) &= 
%K_0\Bigl(-\frac1\tau;2\Bigr)^2 K_0\Bigl(-\frac1\tau;4\Bigr) + K_1\Bigl(-\frac1\tau;2\Bigr)^2 K_2\Bigl(-\frac1\tau;4\Bigr) \\
%%&=\Bigl(\frac1{\sqrt2} (K_0 + K_1)(\tau;2)\Bigr)^2\\
%&= \frac12 \chi_0^{1,1} (\tau) + \frac12 \chi_{\La_2}^{1,1} (\tau)
%+ \chi_{\La_1}^{1,1} (\tau), 
%\end{align*}
%where we have used \eqref{chiK1}-\eqref{chiK3}.
\begin{theorem} 
%Let $Q$ be the $A_3$ root lattice and $\si$ be the Dynkin diagram automorphism given by \eqref{siA3}. Consider the vector space spanned by the irreducible characters for $V_Q^\si$ with ordered basis
%\[
%\{\chi_0^+,\chi_0^-, \chi_{\La_1}, \chi_{\La_2}^+, \chi_{\La_2}^-, \chi_{0}^{\si,+}, \chi_{0}^{\si,-}, \chi_{\frac{\al_2}2}^{\si,+}, \chi_{\frac{\al_2}2}^{\si,-}\},
%\]
With the above notation, consider the fusion matrix whose $(i,j)$-entry is the fusion product \;$V_i\,\boxtimes \,V_j$, where $V_i, V_j$ are the irreducible\/ $V_Q^\si$-modules in position $i$ and $j$, respectively. 
Then fusion matrix is given by$:$
\medskip
\begin{center}
$\left(
\begin{array}{ccccccccc}
V_Q^+&V_Q^-&\la_1&\frac{\al}{2}^+&\frac{\al}{2}^-&T_1^+&T_1^-&T_2^+&T_2^-\\
V_Q^-&V_Q^+&\la_1&\frac{\al}{2}^-&\frac{\al}{2}^+&T_1^-&T_1^+&T_2^-&T_2^+\\
\la_1&\la_1&U^{+,+}+U^{-,-}&\la_1&\la_1&T_2&T_2&T_1&T_1\\
\frac{\al}{2}^+&\frac{\al}{2}^-&\la_1&V_Q^+&V_Q^-&T_1^-&T_1^+&T_2^+&T_2^-\\
\frac{\al}{2}^-&\frac{\al}{2}^+&\la_1&V_Q^-&V_Q^+&T_1^-&T_1^+&T_2^-&T_2^+\\
T_1^+&T_1^-&T_2&T_1^-&T_1^+&U^{+,-}&U^{-,+}&\la_1&\la_1\\
T_1^-&T_1^+&T_2&T_1^+&T_1^-&U^{-,+}&U^{+,-}&\la_1&\la_1\\
T_2^+&T_2^-&T_1&T_2^+&T_2^-&\la_1&\la_1&U^{+,+}&U^{-,-}\\
T_2^-&T_2^+&T_1&T_2^-&T_2^+&\la_1&\la_1&U^{-,-}&U^{+,+}\\
\end{array}
\right)$
\end{center}
\medskip
%\begin{center}
%$\left(
%\begin{array}{ccccccccc}
%0^+&0^-&\la_1&\frac{\al}{2}^+&\frac{\al}{2}^-&T_1^+&T_1^-&T_2^+&T_2^-\\
%0^-&0^+&\la_1&\frac{\al}{2}^-&\frac{\al}{2}^+&T_1^-&T_1^+&T_2^-&T_2^+\\
%\la_1&\la_1&X&\la_1&\la_1&T_2^++T_2^-&T_2^++T_2^-&T_1^++T_1^-&T_1^++T_1^-\\
%\frac{\al}{2}^+&\frac{\al}{2}^-&\la_1&0^+&0^-&T_1^-&T_1^+&T_2^+&T_2^-\\
%\frac{\al}{2}^-&\frac{\al}{2}^+&\la_1&0^-&0^+&T_1^-&T_1^+&T_2^-&T_2^+\\
%T_1^+&T_1^-&T_2^++T_2^-&T_1^-&T_1^+&0^++\frac{\al}{2}^-&0^-+\frac{\al}{2}^+&\la_1&\la_1\\
%T_1^-&T_1^+&T_2^++T_2^-&T_1^+&T_1^-&0^-+\frac{\al}{2}^+&0^++\frac{\al}{2}^-&\la_1&\la_1\\
%T_2^+&T_2^-&T_1^++T_1^-&T_2^+&T_2^-&\la_1&\la_1&0^++\frac{\al}{2}^+&0^-+\frac{\al}{2}^-\\
%T_2^-&T_2^+&T_1^++T_1^-&T_2^-&T_2^+&\la_1&\la_1&0^-+\frac{\al}{2}^-&0^++\frac{\al}{2}^+\\
%\end{array}
%\right)$,
%\end{center}
%where $X=U^{+,+}+U^{-,-}=0^++0^-+\frac{\al}{2}^++\frac{\al}{2}^-$.
\end{theorem}

%%%%%%%%%%%%%%%%%%%%%%%%%%%%%%%%%%%
\section{Examples in Order $3$}\label{ex3}
%%%%%%%%%%%%%%%%%%%%%%%%%%%%%%%%%%%
In this section, we consider the lattice $Q$ to be even and positive definite with an isometry $\si$ of order 3. Then $\bar Q=Q$ by  \coref{corQbar1}.

\subsection{The irreducible characters of twisted type in the general setting}\label{Irrgen3}

First, we will determine the complex numbers $v_1$ and $v_{-1}$ from \eqref{O6}.
%We demonstrate in this case that the transformation law \eqref{chiWtlaws5} can be calculated using a repeated use of Theorem \ref{thtlaws} rather than using Theorem \ref{KP}. 
Recall from Theorem \ref{chiFtran} and Corollary \ref{detk} that
\begin{align}\label{P6}
P_{\si, \si^k}\left(-\frac{1}{\tau}\right)=(-\ii\tau)^{r_{0}/2}P_{\si^k,\si^{-1}}(\tau).
\end{align}
As in the general case, we expect the transformation $\chi_{W(\mu,\zeta)}^{\si,\si^k}\left(-\frac{1}{\tau}, \frac {h}{\tau}\right)$ to be written as a linear combination of the trace functions $\chi_{W(\mu',\zeta')}^{\si^k,\si^{-1}}(\tau,h)$ for suitable $\mu'\in\pi_0(Q)$ and central character $\zeta'$ of $G_\si^\perp$ (see Theorem \ref{bij}). 

We write $\si^{-1}=\si^{kk'}$, where $kk'+1=3m$ for some $m\in\ZZ$. This implies that we can take $k=\pm1$, $k'=-k$, and $m=0$.
From \prref{chiW2}, we have % and Lemma \ref{defectlemma},
\begin{align*}
\chi_{W(\mu, \zeta)}^{\si,\si}(\tau, h)
&=d(\si)e^{-\pi\ii |\mu|^2}\th_{\sqrt{3}\mu+\sqrt{3}\pi_0(Q)}\left(\frac{\tau+1}{3},\frac{h}{\sqrt{3}}, 0\right),\\
\chi_{W(\mu, \zeta)}^{\si,\si^{-1}}(\tau, h)
&=d(\si)e^{\pi\ii |\mu|^2}\th_{\sqrt{3}\mu+\sqrt{3}\pi_0(Q)}\left(\frac{\tau-1}{3},\frac{h}{\sqrt{3}}, 0\right).
\end{align*}
Similarly to \eqref{tauprime}, \eqref{tauprime2}, we set 
\begin{equation}\label{tau'3}
\tau'=\frac{\tau-k}{3},\qquad h'=\frac{h}{\sqrt{3}},
\end{equation} 
and note that
\begin{align*}
\frac{-\frac{1}{\tau}+k}{3}=\frac{k\tau'}{3\tau'+k}=A\cdot\tau',
\end{align*}
where
$A=
\begin{pmatrix}
k&0\\3&k
\end{pmatrix}
$ and $k=\pm1$. Then $A$ is an element of the group $\text{SL}_2(\ZZ)$ with generators
\begin{equation*}
S=\begin{pmatrix}
0&-1\\1&0
\end{pmatrix}\qquad\text{and}\qquad T=\begin{pmatrix}
1&1\\0&1
\end{pmatrix}.
\end{equation*}
%Using $S$ and $T$ as row operations on $A$, 
We find that 
\begin{equation}\label{ST3}
\begin{pmatrix}
1&0\\3&1
\end{pmatrix}=-ST^{-3}S\qquad\text{and}\qquad
\begin{pmatrix}
-1&0\\3&-1
\end{pmatrix}=ST^{3}S.
\end{equation}
Note that these matricies do not depend on the lattice $Q$. 
%By viewing $A\in\text{PSL}_2(\ZZ)$, we can ignore the minus sign in $-ST^{-3}S$
%\eqref{A3} 
%and use the decompositions $ST^{\pm3}S$. 
Our goal will be to use these decompositions to transform the theta function in steps, using one generator at a time.

%For convenience, we set $\la=\sqrt{3}\mu$ for $\mu\in\pi_0(Q^*)$, and
We have that
\begin{equation*}
\mathcal{L}=\sqrt{3}\pi_0(Q),\qquad\mathcal{L}^*=\frac{1}{\sqrt{3}}Q^*\cap\lieh_0
\end{equation*}
(cf.\ Lemma \ref{tdual}).
In this case we find that $\be_0$ in Theorem \ref{KP} can be set to zero. Therefore we set (cf.\ \eqref{cbe0})
\begin{align}\label{cL3}
{c}_0=\sum_{\mu+\mathcal{L}\in\mathcal{L}^*/\mathcal{L}}e^{3\pi\ii|\mu|^2}.
\end{align}

Then the complex numbers $v_1$ and $v_{-1}$ in \eqref{O6} can be determined using Proposition \ref{vTSprop} and Remarks \ref{vprops} and \ref{vSTremark} as follows:
\begin{align}\label{v1}
\begin{split}
v_1=v(-ST^{-3}S)&=\ii^rv(ST^{-3}S)\\
&=\ii^rv(ST^{-3})v(S)\sum_{\mu+\mathcal{L}\in\mathcal{L}^*/\mathcal{L}}e^{-3\pi\ii|\mu|^2}\\
&=\ii^rv(S)^2\bar{c}_0\\
&=\displaystyle\frac{\bar{c}_0}{|\mathcal{L}^*/\mathcal{L}|}\,,
\end{split}
\end{align}
and
\begin{align}\label{v2}
\begin{split}
v_{-1}=v(ST^{3}S)&=v(ST^{3}S)\\
&=v(ST^{3})v(S)\sum_{\mu+\mathcal{L}\in\mathcal{L}^*/\mathcal{L}}e^{3\pi\ii|\mu|^2}\\
&=v(S)^2{c}_0\\
&=\displaystyle\frac{(-1)^{r_0/2}{c}_0}{|\mathcal{L}^*/\mathcal{L}|}\,.
\end{split}
\end{align}

\begin{remark}
Since the above calculations do not depend on the exponent of the generator $T$, formulas \eqref{v1}, \eqref{v2} hold for any odd prime $p$ in place of $3$.
%this proof is essentially the same with $k=\pm1$ and by replacing 3 with $p$.
\end{remark}

We set $\om=e^{2\pi\ii/3}$ and recall that  (cf.\ \eqref{chiMj}, \eqref{chiMinvert}):

\begin{align*}%\label{gen3}
\begin{split}
\chi_{M(\mu,\zeta;\si^s)}^j(\tau, h)&=\frac{1}{3}\sum_{k=0}^{2}\om^{jk}\chi_{M(\mu,\zeta)}^{\si^s,\si^{sk}}(\tau, h),\\
\chi_{M(\mu,\zeta)}^{\si,\si^{-1}}(\tau, h)&=\sum_{l=0}^{2}\om^{l}\chi_{M(\mu,\zeta;\si)}^l(\tau, h),\\
\chi_{M(\mu,\zeta)}^{\si^{-1},\si^{-1}}(\tau, h)&=\sum_{l=0}^{2}\om^{-l}\chi_{M(\mu,\zeta;\si^{-1})}^l(\tau, h),
\end{split}
\end{align*}
where $j=0,1,2$ and $s=1,2$.
For convenience, we also set $\la=\sqrt{3}\mu\in\sqrt{3}\pi_0(Q^*)$ and $\de=\sqrt{3}\nu\in\sqrt{3}(Q^*\cap\lieh_0)\subset\sqrt{3}\pi_0(Q^*)$. Using \eqref{v1}, %and \eqref{gen3}, 
we specialize the transformation \eqref{O6} in this case:
%$$\chi_{M(\mu,\zeta;\si)}^j\left(-\frac{1}{\tau}, \frac{h}{\tau}\right)$$ 
%that contributes to orbifold characters of twisted type:

\begin{align*}
\begin{split}
\chi^{j}&_{M(\mu,\zeta;\si^s)}\left(-\frac{1}{\tau}, \frac {h}{\tau}\right)
=\frac{v_0}{3}e^{\pi\ii|h|^2/\tau}\sum_{\substack{\la+M\in M^*/M\\\la\in Q^*\cap\lieh_0}} \sum_{l=0}^{2}e^{-2\pi\ii (\la|\mu)}\om^{sl}\chi^{l}_{V_{\la+Q}}(\tau, h)\\
&+\frac{\ii^{r_0/2}\bar{c}_0}{3|\mathcal{L}^*/\mathcal{L}|}\sum_{l=0}^{2}\sum_{\substack{\de+\mathcal{L}\in\mathcal{L}^*/\mathcal{L}\\\de\in \la+3\mathcal{L}^*}}\om^{j+l}e^{-2\pi\ii(\mu|\nu)}\chi_{M(\nu,\zeta;\si)}^l(\tau, h)\\
&+\frac{\ii^{-r_0/2}{c}_0}{3|\mathcal{L}^*/\mathcal{L}|}\sum_{l=0}^{2}\sum_{\substack{\de+\mathcal{L}\in\mathcal{L}^*/\mathcal{L}\\\de\in -\la+3\mathcal{L}^*}}\om^{-j-l}e^{-2\pi\ii(\mu|\nu)}\chi_{M(\nu,\zeta;\si^{-1})}^l(\tau, h),
\end{split}
\end{align*}
where $j=0, 1, 2$, $r_0=\dim\lieh_0$, $s=1,2$, and
\begin{equation*}
v_0=d(\si)3^{-\dim\lieh_\perp/4}\frac{|\pi_0(Q)/M|}{|M^*/M|^{1/2}} \,.
\end{equation*}

\subsection{The $\ZZ_3$-orbifold using the root lattice $D_4$}
In this subsection, we consider the lattice $Q=\bigoplus_{i=1}^4\ZZ\al_i$ to be the $D_4$ root lattice and $\si$ to be the Dynkin diagram automorphism given by a $3$-cycle on the outer nodes of the diagram.  
We fix the labelling so that $\al_4$ is the center node: 
$$\si(\al_4)=\al_4 \qquad\text{and}\qquad\; \si\colon\al_1\mapsto\al_2\mapsto\al_3\mapsto\al_1.$$
%In this case, we have $Q=\bar{Q}$, since 
%$$\sum_{i=0}^2(\de|\si^i\de)\in2\ZZ\quad\text{ for all }\quad\de\in Q$$ 
%(cf.\ Lemma \ref{Qbarlemma}). 
Denote the eigenvectors of $\si$ by 
\begin{align*}
\al&=\al_1+\al_2+\al_3,\quad  \ga_1=\al_1+\omega\al_2+\omega^2\al_3,\quad \ga_2=\al_1+\omega^2\al_2+\omega\al_3,
\end{align*}
where $\om=e^{2\pi\ii/3}$. The basis vectors of $Q\cap\lieh_\perp$ are
\begin{align*}
\be_1=\frac{\ga_1-\om\ga_2}{1-\omega}=\al_1-\al_2,\quad\quad
\be_2=\frac{\ga_1-\ga_2}{\omega-\omega^2}=\al_2-\al_3.
\end{align*}
Note that $|\al|^2=6$ and $(\be_i|\al)=0=(\be_i|\al_4)$, $i=1,2$. We also find that
$$M=Q\cap\lieh_0=\ZZ\al_4\oplus\ZZ\al,\qquad (1-\si)Q=Q\cap\lieh_\perp=\ZZ\be_1\oplus\ZZ\be_2,$$ and $L=\langle\al, \al_4\rangle\oplus\langle\be_1, \be_2\rangle$.
It follows that the defect of $\si$ is $d(\si)=1$ (cf.\ Lemma \ref{defectlemma}) and that the irreducible twisted $V_Q$-modules are determined only from elements $\mu\in\pi_0(Q^*)$ (cf.\ \reref{remark1} and \eqref{equivrel}).
Using the relations
\begin{align*}
\al_2=\al_3+\be_2,\quad\al_1=\al_2+\be_1,\quad3\al_1=\al+2\be_1+\be_2,
\end{align*}
we obtain $Q/L=\langle\al_1+L\rangle\cong\ZZ_3.$

The dual lattice $D_4^*$ is spanned by the weights
\begin{align*}
\la_1&=\al_1+\al_4+\frac12(\al_{2}+\al_{3}),\quad
\la_2=\al_2+\al_4+\frac12(\al_{1}+\al_{3}),\\
\la_3&=\al_3+\al_4+\frac12(\al_{1}+\al_{2}),\quad
\la_4=\al+2\al_4,
\end{align*}
and the fundamental group of $D_4$ is
\begin{align*}
Q^*/Q&=\left\langle\la_1+Q, \la_2+Q\right\rangle\\
&=\left\langle\frac{\al_2+\al_3}{2}+Q, \frac{\al_1+\al_3}{2}+Q\right\rangle\cong\ZZ_2\times\ZZ_2.
\end{align*}
We see immediately that $(Q^*/Q)^\si=\{Q\}$ and the other cosets form an orbit of $\si$ of order 3. It follows that there is a unique irreducible untwisted $V_Q$-module of Type 1 (the vacuum module) that breaks into eigenspaces of $\si$ (cf.\ \eqref{chiVQ}). This provides 3 irreducible $V_Q^\si$-modules of Type 1. Since the weights $\la_1, \la_2$, and $\la_3$ form an orbit of $\si$, the corresponding $V_Q$-modules $V_{\la_1+Q}$, $V_{\la_2+Q}$, $V_{\la_3+Q}$ become isomorphic as $V_Q^\si$-modules; we let $V_{\la_1+Q}$ represent the unique irreducible $V_Q^\si$-module of Type 2.

Next, we evaluate the projection $\pi_0=\frac{1}{3}(1+\si+\si^2)$ on the lattices $Q$ and $Q^*$. It is easy to show that (cf.\ Lemma \ref{tdual})
\begin{align*}
\pi_0(Q)&=\left\langle\frac{\al}{3},\al_4\right\rangle,\quad Q^*\cap\lieh_0=(\pi_0(Q))^*=\langle\al, \al_4\rangle=M,
\end{align*}
and that $M\cong\sqrt{3}\pi_0(Q)$ as lattices. Therefore, $M^*=\pi_0(Q)$ and $M^*/M\cong\ZZ_3$.
It is also clear that 
\begin{equation*}
\pi_0(Q^*)=\pi_0(Q)
\end{equation*}
since $\pi_0\la_i\in\pi_0(Q)$ for each $i$. It follows that there is a unique irreducible twisted $V_Q$-module corresponding to $\mu=0$ (cf.\ Theorem \ref{bij}). We denote this module by $M(0)$ (cf.\ \reref{remark1}).

From Corollary \ref{chiM2} and equations \eqref{chiVM}, \eqref{chiVQ}, \eqref{chiV1}--\eqref{chiMj}, we obtain the following characters of the 10 irreducible orbifold modules described above ($j=0,1,2$):
\begin{align*}
\chi^j_{V_{Q}}(\tau)&=\frac13\left(\frac{\th_{Q}(\tau)}{P_{1,1}(\tau)}+\om^{j}\frac{\th_{M}(\tau)}{P_{1,\si}(\tau)}+\om^{2j}\frac{\th_{M}(\tau)}{P_{1,\si^{2}}(\tau)}\right),\\
\chi_{V_{\la_1+Q}}(\tau)&=\frac{\th_{\la_1+Q}(\tau)}{P_{1,1}(\tau)},\\
\chi^j_{M(0;\si)}(\tau)&=\frac13\left(\frac{\th_{\sqrt{3}\pi_0(Q)}(\frac{\tau}{3})}{P_{\si,1}(\tau)}+\om^{j}\frac{\th_{\sqrt{3}\pi_0(Q)}(\frac{\tau+1}{3})}{P_{\si,\si}(\tau)}+\om^{2j}\frac{\th_{\sqrt{3}\pi_0(Q)}(\frac{\tau-1}{3})}{P_{\si,\si^2}(\tau)}\right),\\
\chi^j_{M(0;\si^2)}(\tau)&=\frac13\left(\frac{\th_{\sqrt{3}\pi_0(Q)}(\frac{\tau}{3})}{P_{\si^2,1}(\tau)}+\om^{j}\frac{\th_{\sqrt{3}\pi_0(Q)}(\frac{\tau+1}{3})}{P_{\si^2,\si^2}(\tau)}+\om^{2j}\frac{\th_{\sqrt{3}\pi_0(Q)}(\frac{\tau-1}{3})}{P_{\si^2,\si}(\tau)}\right).
\end{align*}
%where $j=0,1,2$.

%\begin{theorem}
%Let $Q$ be the $D_4$ root lattice and $\si$ be the Dynkin diagram automorphism $\si(\al_4)=\al_4$ \;and\; $\si:\al_1\rightarrow\al_2\rightarrow\al_3\rightarrow\al_1$. Consider the vector space spanned by the irreducible characters for $V_Q^\si$ with ordered basis
%\[
%\{\chi^0_{V_{Q}},\chi^1_{V_{Q}},\chi^2_{V_{Q}}, \chi_{V_{\la_1+Q}}, \chi^0_{M(0;\si)},\chi^1_{M(0;\si)},\chi^2_{M(0;\si)}, \chi^0_{M(0;\si^2)},\chi^1_{M(0;\si^2)},\chi^2_{M(0;\si^2)}\}.
%\]
%Then the modular transformation $\tau\mapsto-1/{\tau}$ for characters of irreducible $V_Q^\si$-modules is given by the $S$-matrix:
%\begin{center}
%$\displaystyle\frac16
%\left(
%\begin{array}{cccccccccc}
%1&1&1&3&2&2&2&2&2&2\\
%1&1&1&3&2\om&2\om&2\om&2\om^2&2\om^2&2\om^2\\
%1&1&1&3&2\om^2&2\om^2&2\om^2&2\om&2\om&2\om\\
%3&3&3&-3&0&0&0&0&0&0\\
%2&2\om&2\om^2&0&2\ii&2\ii\om&2\ii\om^2&-2\ii&-2\ii\om^2&-2\ii\om\\
%2&2\om&2\om^2&0&2\ii\om &2\ii\om^2&2\ii&-2\ii\om^2&-2\ii\om&-2\ii\\
%2&2\om&2\om^2&0&2\ii\om^2&2\ii&2\ii\om&-2\ii\om&-2\ii&-2\ii\om^2\\
%2&2\om^2&2\om&0&-2\ii&-2\ii\om^2&-2\ii\om&2\ii&2\ii\om&2\ii\om^2\\
%2&2\om^2&2\om&0&-2\ii\om^2&-2\ii\om&-2\ii&2\ii\om&2\ii\om^2&2\ii\\
%2&2\om^2&2\om&0&-2\ii\om&-2\ii&-2\ii\om^2&2\ii\om^2&2\ii&2\ii\om\\
%\end{array}
%\right)$
%\end{center}
%where $\om=e^{2\pi\ii/3}$. The modular transformation $\tau\mapsto\tau+1$ for characters of irreducible $V_Q^\si$-modules is given by the $T$-matrix:
%\begin{align*}
%T&=e^{-\pi\ii/9}\text{\upshape diag}(e^{-2\pi\ii/9},e^{-2\pi\ii/9},e^{-2\pi\ii/9}, -e^{-2\pi\ii/9},1,\om^2,\om,1,\om^2,\om)
%\end{align*}
%\end{theorem}

In this case $\be_0$ can be set to zero in Theorem \ref{KP}. We also calculate that for $\mathcal{L}=\sqrt{3}\pi_0(Q)$, \eqref{cbe0} becomes $c_{0}=3$, and hence $v_1=1$, $v_{-1}=-1$ (cf.\ \eqref{v1gen2}). Using Theorem \ref{chiorbtran} and \eqref{v1A22}, we obtain the following transformation laws of orbifold characters, where $j=0,1,2$ and $s=1,2$:
\allowdisplaybreaks
\begin{align*}
&\chi^{j}_{V_{Q}}(\tau+1)=-\om\chi^{j}_{V_{Q}}(\tau), \\%\label{O13}
&\chi^{j}_{V_{Q}}\left(-\frac{1}{\tau}\right)=\frac{1}{6}\sum_{l=0}^{2}\chi_{V_{Q}}^l(\tau)+\frac{1}{2}\chi_{V_{\la_1+Q}}(\tau)+\frac{1}{3}\sum_{k=1}^{2}\sum_{l=0}^{2}\om^{jk}\chi_{M(0;\si^k)}^l(\tau),%\label{O23}
\\[12pt]
%\end{align*}
%where $j=0,1,2$,
%\begin{align*}
&\chi_{V_{\la_1+Q}}(\tau+1)=\om\chi_{V_{\la_1+Q}}(\tau), \\ %\label{O33} \\
&\chi_{V_{\la_1+Q}}\left(-\frac{1}{\tau}\right)=\frac12\sum_{l=0}^{2}\chi_{V_{Q}}^l(\tau)-\frac12\chi_{V_{\la_1+Q}}(\tau), \label{O43}\\[12pt]
&\chi_{M(0;\si^s)}^{j}(\tau+1)=-\om^{1-j}e^{\frac{2\pi\ii}{9}}\chi_{M(0;\si^s)}^j(\tau), \\ %\label{O53} \\
&\chi^{j}_{M(0;\si^s)} \left(-\frac{1}{\tau}\right)=\frac{1}{3} \sum_{l=0}^{2}\om^{sl}\chi^{l}_{V_{Q}}(\tau)+\frac{\ii}{3}\sum_{l=0}^{2}\om^{j+l}\chi^{l}_{M(0;\si^{s})}(\tau) \\
&\qquad\qquad-\frac{\ii}{3}\sum_{l=0}^{2}\om^{-(j+l)}\chi^{l}_{M(0;\si^{-s})}(\tau).
%\label{O63}
\end{align*}
%where $j=0,1,2$ and $s=1,2$.

In the following table, we present the asymptotic and quantum dimensions, which are a special case of \coref{corqdim}.
%using the coefficients of the linear combination of characters in \eqref{A2S1}; see also Corollary \ref{KacCor}.

\begin{center}
\begin{table}[H]\label{tab4}
\begin{tabular}{|c|c|c|c|}
\hline
$M$&$V_Q^j$&$V_{\la_1+Q}$&$M(0, \si^s)^j$\\
\hline
$\asdim M$&$\frac{1}{6}$&$\frac{1}{2}$&$\frac{1}{3}$\\
\hline
$\qdim M$&$1$&$3$&$2$\\
\hline
\end{tabular}
\caption{Asymptotic and quantum dimensions}
\end{table}
\end{center}

Finally, we remark that, even though we have found the modular transformations of the characters of irreducible orbifold modules,
it remains a challenge to determine the $S$-matrix from Theorem \ref{ZABD}. This is due to the fact that the characters are linearly 
dependent and in fact many are just equal; see \eqref{degchar1}, \eqref{degchar3}. 
If we naively pretend that the characters are linearly independent and construct an $S$-matrix from the coefficients in the above transformation formulas,
we obtain an $S$-matrix that may produce non-integral fusion rules when used in Verlinde's formula \eqref{Verlinde} (see \thref{H}).

%%%%%%%%%%%%%%%%%%%%%%%%%%%%%%%%%%
\section{Permutation Orbifolds}\label{permorb}
%%%%%%%%%%%%%%%%%%%%%%%%%%%%%%%%%%
In this section, we study the representation theory of a permutation orbifold of a lattice vertex algebra $V_Q$.
We present the irreducible $V_Q^\si$-modules and their characters in the case when $Q$ is a direct sum of a prime number of copies of an arbitrary positive-definite even lattice
and the automorphism $\sigma$ acts as a cyclic shift of the summands. 
Then we derive the modular transformations of characters, using Theorem \ref{KP} for the twisted type modules.
Permutation orbifolds of lattice vertex algebras for $\sigma$ of order $2$ and $3$
were studied previously by Dong--Xu--Yu in \cite{DXY1,DXY2,DXY3}.
Their recent paper \cite{DXY4} gives a (rather complicated) formula for the $S$-matrix for the permutation orbifold of any regular vertex algebra,
but it is unclear how to derive from it the results of this section.

%A disadvantage to Theorem \ref{KP} is the dependence on some unknown constants $v_k$ for the transformations of twisted type characters. In the next subsection, we present an alternative way to compute these transformations, without using Theorem \ref{KP}, in the case of an order 3 automorphism acting on 3 copies of an arbitrary even lattice in order to avoid needing the constants $v_1$ and $v_{-1}$. Instead, the transformations will only depend on a constant that solely depends on the lattice (see \eqref{cQ03} and \eqref{v1v2perm3} below). In the last subsection, we compute the constant $c_0$ in the case of an order 3 automorphism acting on 3 copies of a rank one lattice.

\subsection{The irreducible characters in the general setting}\label{PermGen}
In this section, we consider the lattice $Q$ to be an orthogonal direct sum $Q_0^{\oplus p}$ of $p$ copies of an even lattice $Q_0$, where $p$ is prime. We represent the elements of $Q$ as $\al=(\al_1,\ldots,\al_p)$ with $\al_i\in Q_0$ for $1\leq i\leq p$. The bilinear form for $Q$ is given by 
\begin{equation}
(\al|\be)=\sum_{i=1}^p(\al_i|\be_i),
\end{equation}
using the bilinear form for $Q_0$.
The dual lattice is $Q^*\cong(Q_0^*)^{\oplus p}$, and we have
\begin{equation}
Q^*/Q\cong\left(Q_0^*/Q_0\right)^{\oplus p} \,.
\end{equation}
At this point, the rank of $Q_0$ is arbitrary. The automorphism $\si$ of $Q$ will be the $p$-cycle that permutes the entries:
\begin{equation}\label{permsi}
\si(\al_1,\ldots, \al_p)=(\al_2,\ldots, \al_p, \al_1).
%=(\al_p,\al_1,\al_2,\ldots, \al_{p-1}).
\end{equation}

Let $\ep_0\colon Q_0\times Q_0\rightarrow\{\pm1\}$ be a 2-cocycle for $Q_0$ satisfying \eqref{lat2} and \eqref{lat22}. As a 2-cocycle for $Q$, we define 
\begin{equation}
\ep(\al,\be)=\prod_{i=1}^p\ep_0(\al_i, \be_i),
\end{equation}
where $\al=(\al_1,\ldots,\al_p)$ and $\be=(\be_1,\ldots,\be_p)$ with $\al_i, \be_i\in Q_0$ for $1\leq i\leq p$. Notice that $\ep(\si\al, \si\be)=\ep(\al, \be)$ for all $\al, \be\in Q$, which implies that we can choose $\eta=1$ in \eqref{twlat3}. 
Hence,  $\bar{Q}=Q$ (this also follows easily from Lemma \ref{Qbarlemma}). %using that $(\al|\si^i\al)=0$ for $\al\in Q$ and $i>0$. 
Now for $\lieh=\CC\otimes_\ZZ Q$, we find:
\begin{align}
\lieh_0&=\text{span}_\CC\bigl\{(\al_0,\ldots,\al_0)\,\big|\,\al_0\in Q_0\bigr\}, \quad \dim\lieh_0=\rank Q_0=r_0, \\
\lieh_\perp&=\Bigl\{(\al_1,\ldots,\al_p)\in Q\,\Big|\,\sum_{i=1}^p\al_i= 0 \Bigr\}, \quad \dim\lieh_\perp=r_0(p-1).\label{permhperp}
\end{align}
%Note that 
%$\dim\lieh_0=\text{rank}\; Q_0=r_0$,  $\dim\lieh_\perp=r_0(p-1)$,
%and that 
We have
$Q\cap\lieh_0\cong\sqrt{p}Q_0$ as lattices under the correspondence 
$$(\al_0,\ldots,\al_0)\leftrightarrow\sqrt{p}\al_0, \qquad \al_0\in Q_0.$$ 
%where $\al_0\in Q_0$.

\begin{lemma}\label{lpnondeg}
With the above setting, we have that\/ $(1-\si)Q=Q\cap\lieh_\perp$.
\end{lemma}
\begin{proof}
It is sufficient to show $Q\cap\lieh_\perp\subset(1-\si)Q$. Indeed, an element in $Q\cap\lieh_\perp$ has the form
\[
(\al_1,\ldots,\al_{p-1},-\al_1-\cdots-\al_{p-1}),
\]
and such elements can be spanned by elements of the form
\[
(1-\si)(0,\ldots,0,\al_0,0,\ldots,0)=(0,\ldots,0,-\al_0,\al_0,0,\ldots,0).
\]
\end{proof}

The above lemma ensures that all twisted $V_Q$-modules can be determined using only $\mu\in\pi_0(Q^*)$:
relation \eqref{con} determines a unique character $\zeta=\ze_\mu$ of $Z(G_\si^\perp)$ for every $\mu\in\pi_0(Q^*)$;
see \reref{remark1}.
Moreover, the defect is $d(\si)=1$, by Lemma \ref{defectlemma}. Recall the lattice 
$$L=(Q\cap\lieh_0) \oplus (Q\cap\lieh_\perp).$$
Then the quotient group $Q/L$ can be described in terms of the lattice $Q_0$ as the next lemma shows.

\begin{lemma}
For\/ $Q=Q_0^{\oplus p}$ and\/ $L$ as above, we have\/ $Q/L\cong Q_0/pQ_0$ as abelian groups. In particular, there are\/ $p^{r_0}$ cosets of\/ $L$ in\/ $Q$.
\end{lemma}

\begin{proof}
Consider the composition of maps $f\colon Q\rightarrow Q_0\rightarrow Q_0/pQ_0$ given by
\begin{align*}
(\al_1,\ldots,\al_p)\mapsto\sum_{i=1}^p\al_i\mapsto\sum_{i=1}^p\al_i+pQ_0.
\end{align*}
Clearly, $f$ is surjective and $L\subset \Ker f$. Now suppose $(\al_1,\ldots,\al_p)\in\Ker f$ so that $\sum_{i=1}^p\al_i=p\be_0$ for some $\be_0\in Q_0$. Then 
\[
(\al_1,\ldots,\al_p)-(\be_0,\ldots,\be_0)\in Q\cap\lieh_\perp,
\] 
and $(\be_0,\ldots,\be_0)\in Q\cap\lieh_0$; hence $L= \Ker f$.
The result follows from the First Isomorphism Theorem.
\end{proof}

Next, we describe the irreducible modules of $V_Q^\si$. We first calculate the $\si$-invariants in $Q^*/Q$ to be
\begin{align}
(Q^*/Q)^\si&=\bigl\{(\la_0+Q_0,\ldots, \la_0+Q_0)\,\big|\,\la_0\in Q_0^*\bigr\}\cong Q_0^*/Q_0.
\end{align}
We also calculate $\pi_0(Q^*)$ and $\pi_0(Q)$:
\begin{align*}
\pi_0(Q^*)&=\Bigl\{\frac1p(\la_0,\ldots,\la_0)\;\Big|\;\la_0\in Q_0^*\Bigr\}\cong\frac{1}{\sqrt{p}}Q_0^*,\\
\pi_0(Q)&=\Bigl\{\frac1p(\al_0,\ldots,\al_0)\;\Big|\;\al_0\in Q_0\Bigr\}\cong\frac{1}{\sqrt{p}}Q_0,
\end{align*}
where the isomorphisms are as lattices. We obtain 
$$\pi_0(Q^*)/\pi_0(Q)\cong Q_0^*/Q_0, \quad \sqrt{p}\pi_0(Q)\cong Q_0, \quad \sqrt{p}\pi_0(Q^*)\cong Q_0^*$$
%and that $\sqrt{p}\pi_0(Q)\cong Q_0$ and $\sqrt{p}\pi_0(Q^*)\cong Q_0^*$ 
 as lattices.
 
Hence, there are $|Q_0^*/Q_0|$ many irreducible $\si$-twisted $V_Q$-modules. 
Since the automorphism $\si$ acts on these modules, each of them will decompose into $p$ eigenspaces for $\si$. The same is true for each $\si^s$, $s=1,\ldots,p-1$, since the order of $\si$ is prime. This yields 
\begin{equation}\label{type3num}
p(p-1)|Q_0^*/Q_0|
\end{equation} 
many irreducible orbifold modules of twisted type. We denote the eigenspaces using $j=0,\ldots,p-1$ and label the irreducible Type 3 orbifold modules by 
\begin{equation}\label{Mla0sis}
M(\la_0;\si^s)^j := M(\mu,\ze_\mu;\si^s)^j \,,
\end{equation} 
for $\la_0\in Q_0^*$ and $s=1,\ldots,p-1$, where
\begin{equation}\label{mula0}
\mu=\pi_0(\la_0,0,\ldots,0)=\frac1p(\la_0,\ldots,\la_0)
\end{equation} 
and $\zeta=\ze_\mu$ is uniquely determined by $\mu$ from \eqref{con}, due to \leref{lpnondeg} and \reref{remark1}.
Notice that the coset $\la_0+Q_0$ is in one-to-one correspondence with the parameter $\mu\in\pi_0(Q^*)$ given by \eqref{mula0}; cf.\ Theorem \ref{bij}.
In particular, $|\la_0|^2=p|\mu|^2$.

Since the lattice $Q$ is a direct sum of copies of $Q_0$, the vertex algebra $V_Q$ can be written as a corresponding tensor product of vertex algebras $V_{Q_0}$ (cf.\ \cite{FHL}). 
The same is true for the irreducible $V_Q$-modules:
%explicitly this correspondence is given by
\[
V_{(\la_1,\ldots,\la_p)+Q}\cong V_{\la_1+Q_0}\otimes\cdots\otimes V_{\la_p+Q_0} \qquad (\la_i\in Q_0^*).
\]
%where each $\la_i\in Q_0^*$.
Now $\si$ acts on the $V_Q$-module $\left(V_{\la_0+Q_0}\right)^{\oplus p}$ for each $\la_0\in Q_0^*$, and each of these will decompose into $p$ eigenspaces for $\si$, yielding $p$ irreducible orbifold modules for each $\la_0\in Q_0^*$. 
Hence, there are
\begin{equation}\label{type1num}
p|Q_0^*/Q_0|
\end{equation}
many irreducible orbifold modules of Type 1. 

We label the eigenspaces by $V^i_{\la+Q}$, where $i=0, 1,\ldots,p-1$ and $\la\in Q^*$ such that $\la=(\la_0,\ldots,\la_0)$ for $\la_0\in Q_0^*$. The action of $\si$ on the other $V_Q$-modules is given by
\[
\si\colon 
V_{\la_1+Q_0}\otimes\cdots\otimes V_{\la_p+Q_0}\rightarrow V_{\la_2+Q_0}\otimes\cdots\otimes V_{\la_p+Q_0}\otimes V_{\la_1+Q_0},
\]
and these distinct $V_Q$-modules become isomorphic as $V_Q^\si$-modules. So the irreducible modules where $\si$ does not act are in orbits of size $p$, and each module in the same orbit corresponds to the same orbifold module. 
We obtain
\begin{equation}\label{type2num}
\frac1p\left(|Q_0^*/Q_0|^p-|Q_0^*/Q_0|\right)
\end{equation}
many irreducible orbifold modules of Type 2. To label these modules, we use notation to describe the orbits of $\si$ in $Q^*/Q$ of order $p$. Let $\mathcal{O}$ be the set of orbits of order $p$, and denote by $[\ga+Q]\in\mathcal{O}$ the $\si$-orbit of the coset $\ga+Q$ with $(1-\si)\ga\notin Q$. Then the order of $\mathcal{O}$ is given by \eqref{type2num}, and we label the irreducible orbifold modules of Type 2 by $V_{\ga+Q}$ for each $[\ga+Q]\in \mathcal{O}$. %with $(1-\si)\ga\notin Q$. 

All together, there are
\begin{equation}\label{totalnum}
\frac1p|Q_0^*/Q_0|^p+\frac{p^3-1}{p}|Q_0^*/Q_0|
\end{equation}
many irreducible $V_Q^\si$-modules.
We summarize this discussion into a theorem, which describes the irreducible orbifold modules and their characters explicitly.

\begin{theorem}\label{PermMod}
Consider a lattice\/ $Q=Q_0^{\oplus p}$, where $p$ is prime and\/ $Q_0$ is a positive-definite even lattice. Let\/ $\si$ be the automorphism of\/ $Q$ that permutes the summands cyclically. 
Denote by superscript\/ $j$ the eigenspace of\/ $\si$ with eigenvalue $\om^{-j}$, where $\om=e^{2\pi\ii /p}$. 
Let\/ $\mathcal{O}$ be the set of orbits of\/ $\si$ in $Q^*/Q$ of order $p$, and denote the orbit of\/ $\ga+Q\in Q^*/Q$ as\/ $[\ga+Q]$. Also choose a set\/ $\C\subset \pi_0(Q_0^*)$ of representatives of the cosets\/ $\pi_0(Q_0^*) / \pi_0(Q_0)$.

Then the following is a complete list of non-isomorphic irreducible modules over the orbifold algebra\/ $V_Q^\si${\upshape{:}}
\begin{enumerate}
\item[{\upshape{(Type 1)}}] 
$V_{\la+Q}^j$\, with \,$\la=(\la_0,\ldots,\la_0)$, \,$\la_0\in Q_0^*,$ \, $j=0,\ldots,p-1;$

\medskip
\item[{\upshape{(Type 2)}}] 
$V_{\ga+Q}$ \, with \, $[\ga+ Q]\in\mathcal{O},$ \, $\ga\in Q^*;$ %with $(1-\si)\ga\notin Q$,

\medskip
\item[{\upshape{(Type 3)}}] 
$M(\la_0;\si^s)^j$ \, with \, $\la_0+Q_0\in\C,$ \, $\la_0\in Q_0^*,$ \, $j=0,\ldots,p-1,$ \, $s=1,\ldots,p-1$ $($cf.\ \eqref{Mla0sis}, \eqref{mula0}$)$.
\end{enumerate}
%where in Type 1 $\la=(\la_0,\ldots,\la_0)\in Q^*$ for any $\la_0\in Q_0^*$. 
The characters of these modules are given by$:$
\begin{enumerate}
\item[{\upshape{(Type 1)}}] 
$\displaystyle\chi^j_{V_{\la+Q}}(\tau)=\frac1p\frac{\th_{\la+Q}(\tau)}{P_{1,1}(\tau)}+\frac1p\sum_{k=1}^{p-1}\om^{jk}\frac{\th_{\la_0+Q_0}(p\tau)}{P_{1,\si^k}(\tau)}\,,$

\medskip
\item[{\upshape{(Type 2)}}] 
$\displaystyle\chi_{V_{\ga+Q}}(\tau)=\frac{\th_{\ga+Q}(\tau)}{P_{1,1}(\tau)}\,,$

\medskip
\item[{\upshape{(Type 3)}}] 
$\displaystyle\chi^j_{M(\la_0;\si^s)}(\tau)=\frac1p\frac{\th_{\la_0+Q_0}(\frac{\tau}{p})}{P_{\si^s,1}(\tau)}+\frac1p\sum_{k=1}^{p-1}\om^{jk}e^{- \frac{\pi\ii}{p}k|\la_0|^2}\frac{\th_{\la_0+Q_0}(\frac{\tau+k}{p})}{P_{\si^s,\si^{sk}}(\tau)}\,.$
\end{enumerate}
\end{theorem}

\begin{proof}
The classification of irreducible $V_Q^\si$-modules is a special case of \thref{class}.
The formulas for the characters follow immediately from Theorem \ref{chiM2} and equations \eqref{chiVM}, \eqref{chiVQ},  \eqref{chiV1}--\eqref{chiMj}. 
For the Type 1 characters, we used that $M=\sqrt{p}Q_0$ and employed identity \eqref{thc}.
%$\displaystyle\chi^j_{V_{\la+Q}}=\frac1p\frac{\th_{\la+Q}(\tau)}{P_{1,1}(\tau)}+\frac1p\sum_{k=1}^{p-1}\om^{jk}\frac{\th_{\la+M}(\tau)}{P_{1,\si^k}(\tau)},\quad(1-\si)\la\in Q,$
\end{proof}

In order to derive the modular transformations of the irreducible orbifold characters,
an important step in the process is to invert the equations representing these characters, 
so that each quotient on the right side is written as a linear combination of the orbifold characters. We can do this using 
\eqref{V1invert}--\eqref{chiMinvert}, as follows.

\begin{lemma}\label{perminvert} 
Let\/ $\la_0\in Q_0^*$, $\la=(\la_0,\ldots,\la_0)\in Q^*$, and\/ $1\leq s, k<p$. Then the formulas for the irreducible orbifold characters of Types 1 and 3 given in Theorem \ref{PermMod} can be inverted as follows$:$
\begin{align*}
\frac{\th_{\la+Q}(\tau)}{P_{1,1}(\tau)} &=\sum_{j=0}^{p-1}\chi_{V_{\la+Q}}^j(\tau), \\
e^{- \frac{\pi\ii}{p}k|\la_0|^2}\frac{\th_{\la_0+Q_0}(\frac{\tau+k}{p})}{P_{\si^s,\si^{sk}}(\tau)} &=\sum_{j=0}^{p-1}\om^{-jk}\chi_{M(\la_0;\si^s)}^j(\tau), \\
\frac{\th_{\la_0+Q_0}(\frac{\tau}{p})}{P_{\si^s,1}(\tau)} &=\sum_{j=0}^{p-1}\chi_{M(\la_0;\si^s)}^j(\tau), \\
\frac{\th_{\la_0+Q_0}(p\tau)}{P_{1,\si^k}(\tau)} &=\sum_{j=0}^{p-1}\om^{-jk}\chi_{V_{\la+Q}}^j(\tau).
\end{align*}
\end{lemma}

\begin{proof}
The results follow easily from the identity $1+\om+\dots+\om^{p-1}=0$.
%$\om$ being a root of the cyclotomic polynomial of degree $p-1$. 
To obtain the two formulas involving theta functions with argument different from $\tau$, we multiply the corresponding characters by $\om^{-jk}$ and then sum over $j$ to yield the $k$-th term in the sum.
\end{proof}

%The formula for the irreducible orbifold characters of Type 2 is clearly invertible, so together with Lemma \ref{perminvert} we obtain an invertible linear transformation among the irreducible characters and the quotients involving theta functions. In order to obtain the transformations of irreducible orbifold characters, it is therefore sufficient to transform the quotients involving theta functions. 

Due to Lemma \ref{perminvert}, to obtain the transformations of irreducible orbifold characters, it is sufficient to transform the quotients involving theta functions. 
We first describe intuitively how these quotients should transform among themselves. Orbifold characters of Type 1 are written in terms of two types of theta functions, with arguments $\tau$ and $p\tau$, respectively, while characters of Type 2 are just one such quotient. The characters of twisted type are also in terms of two types of theta functions, with arguments $\frac{\tau}{p}$ and $\frac{\tau+k}{p}$, respectively, where $1\leq k<p$. As we will show, the theta functions with argument $\tau$ will transform among themselves. Those with argument $p\tau$ will transform to theta functions with argument $\frac{\tau}{p}$ and vice-versa, because $-\frac{p}{\tau}=-\frac{1}{\frac{\tau}{p}}$.
It follows that the theta functions with argument $\frac{\tau+k}{p}$ must transform among themselves. We now prove these transformations for the separate quotients, and as a result obtain the transformations of the irreducible permutation orbifold characters.

\begin{theorem}\label{thm8.5}
%Let us use the same notation as in Theorem \ref{PermMod}, and also set 
%$$E_{\la,\ga}=\sum_{i=0}^{p-1}e^{-2\pi\ii(\la|\si^i\ga)},\qquad\la, \ga\in Q^* \,.$$
%Then 
With the notation of Theorem \ref{PermMod},
the modular transformations of the irreducible permutation orbifold characters are as follows$:$
\begin{align}
\chi^{j}_{V_{\la+Q}}(\tau+1,h)&=e^{\pi\ii(|\la|^2-\frac{r}{12})}\chi^{j}_{V_{\la+Q}}(\tau,h), \\ %\label{O1}
\chi_{V_{\la+Q}}(\tau+1,h)&=e^{\pi\ii(|\la|^2-\frac{r}{12})}\chi_{V_{\la+Q}}(\tau,h), \\ %\label{O3}
\chi_{M(\la_0;\si^s)}^{j}(\tau+1, h)&=\om^{-j}e^{2\pi\ii(\Delta_{\si}-\frac{r}{24})}e^{\frac{\pi\ii}{p}|\la_0|^2}\chi_{M(\la_0;\si^s)}^j(\tau, h),\label{O5perm}
\end{align}
\begin{align}\label{chiorbtrans1}
\begin{split}
\displaystyle\chi^j_{V_{\la+Q}}&\left(-\frac1\tau\right)=\frac{1}{p\sqrt{|Q^*/Q|}}\sum_{\substack{\de+Q\in Q^*/Q\\(1-\si)\de\in Q}}\sum_{l=0}^{p-1}e^{-2\pi\ii(\la|\de)}\chi_{V_{\de+Q}}^l(\tau)\\
&+\frac{1}{\sqrt{|Q^*/Q|}}\sum_{[\ga+Q]\in \mathcal{O}}e^{-2\pi\ii(\la|\ga)}\chi_{V_{\ga+Q}}(\tau)\\
&+\frac{1}{p\sqrt{|Q_0^*/Q_0|}}\sum_{k=1}^{p-1}\sum_{\de_0+Q_0\in Q_0^*/Q_0}\sum_{l=0}^{p-1}\om^{jk}e^{-2\pi\ii(\la_0|\de_0)}\chi_{M(\de_0;\si^k)}^l(\tau),
\end{split}
\end{align}
\begin{align}\label{chiorbtrans2}
\begin{split}
\displaystyle\chi_{V_{\la+Q}}\left(-\frac{1}{\tau}\right)&=\frac{1}{\sqrt{|Q^*/Q|}}\sum_{\substack{\de+Q\in Q^*/Q\\(1-\si)\de\in Q}}\sum_{l=0}^{p-1}e^{-2\pi\ii(\la|\de)}\chi_{V_{\de+Q}}^l(\tau)\\
&\quad+\frac{1}{\sqrt{|Q^*/Q|}}\sum_{[\ga+Q]\in \mathcal{O}}E_{\la,\ga}\chi_{V_{\ga+Q}}(\tau),
\end{split}
\end{align}
\begin{align}\label{chiorbtrans3}
\begin{split}
\displaystyle\chi_{M(\la_0;\si^s)}^j&\left(-\frac1\tau\right)=\frac{1}{p\sqrt{|Q_0^*/Q_0|}}\sum_{\substack{\de+Q\in Q^*/Q\\\de=(\de_0,\ldots,\de_0)}}\sum_{l=0}^{p-1}\om^{ls}e^{-2\pi\ii(\la_0|\de_0)}\chi_{V_{\de+Q}}^l(\tau)\\
&+\frac{\ii^{r_0/2}}{p}\sum_{k=1}^{p-1}\sum_{\substack{\ga_0+Q_0\in Q_0^*/Q_0\\\ga_0\in k\la_0+pQ_0^*}}\sum_{l=0}^{p-1}v_k\om^{jk+lk^{p-2}}e^{-\frac{2\pi\ii}{p}(\la_0|\ga_0)}\chi_{M(\ga_0;\si^{ks})}^l(\tau),
\end{split}
\end{align}
where\/ $\Delta_\si=\frac{p^2-1}{24p}|\mathcal{O}|,$ %\emph{(}cf.\ Lemma \ref{Deltalemma}\emph{)} 
$r=\rank Q$, $E_{\la,\ga}$ is given by \eqref{Ela}, and\/ $v_k$ is a complex number.
\end{theorem} 

\begin{proof}
The transformation laws for $\tau\mapsto\tau+1$ are essentially the same as in Theorem \ref{chiorbtran}. The expression for $\Delta_\si$ follows from writing $\dim\lieh_\perp=(p-1)|\mathcal{O}|$ in Lemma \ref{Deltalemma}. The rest of the proof will describe the transformation $\tau\mapsto-{1}/{\tau}$.

%It is clear from the transformation law \eqref{thtlaws1} that the theta functions with argument $\tau$ will transform as a linear combination of theta functions with the same argument. It follows that the transformation of Type 2 characters can only be a linear combination of Type 1 or Type 2 characters. When transforming the theta functions with argument $p\tau$, we note that since
%\begin{align*}
%-\frac{p}{\tau}=-\frac{1}{\frac{\tau}{p}} \,,
%\end{align*}
%the result will be a linear combination of theta functions with argument $\frac{\tau}{p}$, and vice-versa. From the transformation law of Corollary \ref{chiFtran2} and Theorem \ref{chiF}, the denominator $P_{1,\si^k}(\tau)$ will transform as a multiple of $P_{\si^k,1}(\tau)$. Therefore, the transformation of the quotients in the second sum of characters of Type 1 will result in a linear combination of quotients resembling the first term in the characters of twisted type. It follows that the theta functions with argument $\frac{\tau+k}{p}$, where $k\neq0$, must transform among themselves. 
%%We will show this is the case when $p=3$. We end this section with the modular transformations of permutation orbifold characters of untwisted type (Type 1 and 2). 

First, we transform characters of Type 1. In terms of theta functions, this transformation is
\begin{align}\label{thm1}
\displaystyle\chi^j_{V_{\la+Q}}\left(-\frac1\tau\right)=\frac1p\frac{\th_{\la+Q}(-\frac1\tau)}{P_{1,1}(-\frac1\tau)}+\frac1p\sum_{k=1}^{p-1}\om^{jk}\frac{\th_{\la_0+Q_0}(-\frac p\tau)}{P_{1,\si^k}(-\frac1\tau)} \,.
\end{align}
By Theorem \ref{thtlaws} and Corollary \ref{chiFtran2}, we obtain 
\allowdisplaybreaks
\begin{align*}
\frac{\th_{\la+Q}(-\frac1\tau)}{P_{1,1}(-\frac1\tau)}&=\frac{(-\ii\tau)^{-r_0/2}(-\ii\tau)^{r_0/2}}{\sqrt{|Q^*/Q|}}\sum_{\de+Q\in Q^*/Q}e^{-2\pi\ii(\la|\de)}\frac{\th_{\de+Q}(\tau)}{P_{1,1}(\tau)}\\
&=\frac{1}{\sqrt{|Q^*/Q|}}\sum_{\substack{\de+Q\in Q^*/Q\\(1-\si)\de\in Q}}e^{-2\pi\ii(\la|\de)}\frac{\th_{\de+Q}(\tau)}{P_{1,1}(\tau)}\\
&\hspace*{0.5in}+\frac{1}{\sqrt{|Q^*/Q|}}\sum_{\substack{\ga+Q\in Q^*/Q\\(1-\si)\ga\notin Q}}e^{-2\pi\ii(\la|\ga)}\frac{\th_{\ga+Q}(\tau)}{P_{1,1}(\tau)} \,.
\end{align*}
Using Lemma \ref{perminvert}, we write the transformation as a linear combination of irreducible orbifold characters:

\allowdisplaybreaks
\begin{align*}
\frac{\th_{\la+Q}(-\frac1\tau)}{P_{1,1}(-\frac1\tau)}&=\frac{1}{\sqrt{|Q^*/Q|}}\sum_{\substack{\de+Q\in Q^*/Q\\(1-\si)\de\in Q}}e^{-2\pi\ii(\la|\de)}\sum_{l=0}^{p-1}\chi_{V_{\de+Q}}^l(\tau)\\
&\hspace*{0.5in}+\frac{1}{\sqrt{|Q^*/Q|}}\sum_{\substack{\ga+Q\in Q^*/Q\\(1-\si)\ga\notin Q}}e^{-2\pi\ii(\la|\ga)}\chi_{V_{\ga+Q}}\\
&=\frac{1}{\sqrt{|Q^*/Q|}}\sum_{\substack{\de+Q\in Q^*/Q\\(1-\si)\de\in Q}}e^{-2\pi\ii(\la|\de)}\sum_{l=0}^{p-1}\chi_{V_{\de+Q}}^l(\tau)\\
&\hspace*{0.5in}+\frac{p}{\sqrt{|Q^*/Q|}}\sum_{\substack{[\ga+Q]\in \mathcal{O}}}e^{-2\pi\ii(\la|\ga)}\chi_{V_{\ga+Q}},
\end{align*}
and
\allowdisplaybreaks
\begin{align*}
\frac{\th_{\la_0+Q_0}(-\frac{1}{\frac{\tau}{p}})}{P_{1,\si^k}(-\frac1\tau)}
&=\frac{p^{r_0/2}(-\ii\tau)^{-r_0/2}(-\ii\frac{\tau}{p})^{r_0/2}}{\sqrt{|Q_0^*/Q_0|}} \!\!\! \sum_{\de_0+Q_0\in Q_0^*/Q_0} \!\!\! e^{-2\pi\ii(\la_0|\de_0)}\frac{\th_{\de_0+Q_0}(\frac{\tau}{p})}{P_{\si^k,1}(\tau)}\\
&=\frac{1}{\sqrt{|Q_0^*/Q_0|}}\sum_{\de_0+Q_0\in Q_0^*/Q_0}e^{-2\pi\ii(\la_0|\de_0)}\sum_{l=0}^{p-1}\chi_{M(\de_0;\si^k)}^l(\tau),
\end{align*}
using that $\dim\lieh_\perp=r_0(p-1)$. The transformation \eqref{chiorbtrans1} now follows from \eqref{thm1} and the above calculations.

Next we prove \eqref{chiorbtrans2}. Again from Theorem \ref{thtlaws} and Corollary \ref{chiFtran2}, we obtain
\allowdisplaybreaks
\begin{align*}
\chi_{V_{\la+Q}}\left(-\frac{1}{\tau}\right)&=\frac{\th_{\ga+Q}(-\frac1\tau)}{P_{1,1}(-\frac1\tau)}\\
&=\frac{(-\ii\tau)^{-r_0/2}(-\ii\tau)^{r_0/2}}{\sqrt{|Q^*/Q|}}\sum_{\de+Q\in Q^*/Q}e^{-2\pi\ii(\la|\de)}\frac{\th_{\de+Q}(\tau)}{P_{1,1}(\tau)}\\
&=\frac{1}{\sqrt{|Q^*/Q|}}\sum_{\substack{\de+Q\in Q^*/Q\\(1-\si)\de\in Q}}e^{-2\pi\ii(\la|\de)}\frac{\th_{\de+Q}(\tau)}{P_{1,1}(\tau)}\\
&\hspace*{0.5in}+\frac{1}{\sqrt{|Q^*/Q|}}\sum_{\substack{\ga+Q\in Q^*/Q\\(1-\si)\ga\notin Q}}e^{-2\pi\ii(\la|\ga)}\frac{\th_{\ga+Q}(\tau)}{P_{1,1}(\tau)}\\
&=\frac{1}{\sqrt{|Q^*/Q|}}\sum_{\substack{\de+Q\in Q^*/Q\\(1-\si)\de\in Q}}e^{-2\pi\ii(\la|\de)}\sum_{l=0}^{p-1}\chi_{V_{\de+Q}}^l(\tau)\\
&\hspace*{0.5in}+\frac{1}{\sqrt{|Q^*/Q|}}\sum_{\substack{\ga+Q\in Q^*/Q\\(1-\si)\ga\notin Q}}e^{-2\pi\ii(\la|\ga)}\chi_{V_{\ga+Q}}\\
&=\frac{1}{\sqrt{|Q^*/Q|}}\sum_{\substack{\de+Q\in Q^*/Q\\(1-\si)\de\in Q}}e^{-2\pi\ii(\la|\de)}\sum_{l=0}^{p-1}\chi_{V_{\de+Q}}^l(\tau)\\
&\hspace*{0.5in}+\frac{1}{\sqrt{|Q^*/Q|}}\sum_{\substack{[\ga+Q]\in \mathcal{O}}}E_{\la,\ga}\chi_{V_{\ga+Q}} \,.
\end{align*}
%where $E_{\la,\ga}=\sum_{i=0}^{p-1}e^{-2\pi\ii(\la|\si^i\ga)}$.

Now we prove \eqref{chiorbtrans3}. Similarly to transforming theta functions with argument $p\tau$, the transformation of the first term of the irreducible orbifold character of twisted type in Theorem \ref{PermMod}(3) is
\allowdisplaybreaks
\begin{align}\label{PermMtran1}
\begin{split}
&\frac{\th_{\la_0+Q_0}(-\frac{1}{p\tau})}{P_{\si^s,1}(-\frac1\tau)}=\frac{(-\ii\tau)^{-r_0/2}(-\ii p\tau)^{r_0/2}}{p^{r_0/2}\sqrt{|Q_0^*/Q_0|}}\sum_{\de_0+Q_0\in Q_0^*/Q_0}e^{-2\pi\ii(\la_0|\de_0)}\frac{\th_{\de_0+Q_0}(p\tau)}{P_{1, \si^{-s}}(\tau)}\\
&\quad=\frac{1}{\sqrt{|Q_0^*/Q_0|}}\sum_{\substack{\de+Q\in Q^*/Q\\\de=(\de_0,\ldots,\de_0)}}e^{-2\pi\ii(\la_0|\de_0)}\sum_{l=0}^{p-1}\om^{ls}\chi_{V_{\de+Q}}^l(\tau).
\end{split}
\end{align}
%using that $\dim\lieh_\perp=r_0(p-1)$. 
We use the same method of transforming the term
\begin{align*}
e^{-\pi\ii k\frac{|\la_0|^2}{p}}\frac{\th_{\la_0+Q_0}(\frac{\tau+k}{p})}{P_{\si^s,\si^{sk}}(\tau)}
\end{align*}
as in the proof of \eqref{chiWtlaws6}.
As in \eqref{tauprime} and \eqref{tauprime2}, we set 
\begin{equation}\label{tauprime3}
\tau'=\frac{\tau+k'}{p},
\end{equation}
where $kk'+1=mp$,
and note that
\begin{align}\label{tauprime4}
\frac{-\frac{1}{\tau}+k}{p}=\frac{k\tau'-m}{p\tau'-k'}=A\cdot\tau',
\quad A=
\begin{pmatrix}
k&-m\\p&-k'
\end{pmatrix}
\in\SL_2(\ZZ).
\end{align}
%Then the matrix $A$ is an element of the modular group $\SL_2(\ZZ)$. 
We next use Theorem \ref{KP} with the matrix $A$ and $\tau$ replaced by $\tau'$:
\allowdisplaybreaks
\begin{align}\label{thtwp}
\begin{split}
&\th_{\la_0+Q_0}\left(\frac{-\frac{1}{\tau}+k}{p}\right)=\th_{\la_0+Q_0}\left(A\cdot\tau'\right)\\
&=(p\tau'-k')^{r_0/2}v_k \!\!\! \sum_{\substack{\ga_0+Q_0\in Q_0^*/Q_0\\\ga_0\in pQ_0^*}} \!\!\! e^{\pi\ii( -\frac{k'}{p}|\ga_0|^2-2m(\la_0|\ga_0)-km|\la_0|^2)}\th_{k\la_0+\ga_0+Q_0}(\tau')\\
&=\tau^{r_0/2}v_k \!\!\! \sum_{\substack{\ga_0+Q_0\in Q_0^*/Q_0\\\ga_0\in pQ_0^*}} \!\!\! e^{\pi\ii( -\frac{k'}{p}|\ga_0|^2-2m(\la_0|\ga_0)-km|\la_0|^2)}\th_{k\la_0+\ga_0+Q_0}\left(\frac{\tau+k'}{p}\right)\\
&=\tau^{r_0/2}v_k\sum_{\substack{\ga_0+Q_0\in Q_0^*/Q_0\\\ga_0\in k\la_0+pQ_0^*}}e^{\frac{\pi\ii}{p}( -k'|\ga_0|^2-2(\la_0|\ga_0)+k|\la_0|^2)}\th_{\ga_0+Q_0}\left(\frac{\tau+k'}{p}\right),
\end{split}
\end{align}
using in the last step that
\allowdisplaybreaks
\begin{align*}
-k'|\ga_0&-k\la_0|^2-2mp(\la_0|\ga_0-k\la_0)-kmp|\la_0|^2 \\
&=-k'|\ga_0|^2-2(\la_0|\ga_0)+k|\la_0|^2.
\end{align*}
It now follows that
\allowdisplaybreaks
\begin{align}\label{PermMtran2}
\begin{split}
&e^{- \frac{\pi\ii}{p}k|\la_0|^2}\frac{\th_{\la_0+Q_0}\left(\frac{-\frac{1}{\tau}+k}{p}\right)}{P_{\si^s,\si^{sk}}(-\frac{1}{\tau})}\\
&=\tau^{r_0/2}(-\ii\tau)^{-r_0/2}v_k \!\!\! \sum_{\substack{\ga_0+Q_0\in Q_0^*/Q_0\\\ga_0\in k\la_0+pQ_0^*}} \!\!\! e^{\frac{\pi\ii}{p}( -k'|\ga_0|^2-2(\la_0|\ga_0))}\frac{\th_{\ga_0+Q_0}\left(\frac{\tau+k'}{p}\right)}{P_{\si^{ks},\si^{-s}}(\tau)}\\
%&=\ii^{r_0/2}v_k\sum_{\substack{\ga_0+Q_0\in Q_0^*/Q_0\\\ga_0\in pQ_0^*}}e^{-\frac{2\pi\ii}{p}(k|\la_0|^2+(\la_0|\ga_0))}e^{-\frac{\pi\ii}{p}k'|k\la_0+\ga_0|^2}\frac{\th_{k\la_0+\ga_0+Q_0}\left(\frac{\tau+k'}{p}\right)}{P_{\si^{ks},\si^{-s}}(\tau)}\\
&=\ii^{r_0/2}v_k\sum_{\substack{\ga_0+Q_0\in Q_0^*/Q_0\\\ga_0\in k\la_0+pQ_0^*}}e^{-\frac{2\pi\ii}{p}(\la_0|\ga_0)}e^{-\frac{\pi\ii}{p}k'|\ga_0|^2}\frac{\th_{\ga_0+Q_0}\left(\frac{\tau+k'}{p}\right)}{P_{\si^{ks},\si^{-s}}(\tau)}.\\
\end{split}
\end{align}
%$-\frac{k}{p}|\la_0|^2-\frac{k'}{p}|\ga_0|^2-2m(\la_0|\ga_0)-km|\la_0|^2+\frac{k'}{p}|k\la_0+\ga_0|^2=k|\la_0|^2+(\la_0|\ga_0)+|k\la_0+\ga_0|^2$
Using \eqref{PermMtran1}, \eqref{PermMtran2}, and Lemma \ref{perminvert}, we finish the calculation:
\begin{align*}%\label{PermMtran3}
\begin{split}
&\frac1p\sum_{k=1}^{p-1}\om^{jk}e^{- \frac{\pi\ii}{p}k|\la_0|^2}\frac{\th_{\la_0+Q_0}\left(\frac{-\frac{1}{\tau}+k}{p}\right)}{P_{\si^s,\si^{sk}}\left(-\frac{1}{\tau}\right)}\\
&=\frac{\ii^{r_0/2}}{p}\sum_{k=1}^{p-1}\om^{jk}v_k\sum_{\substack{\ga_0+Q_0\in Q_0^*/Q_0\\\ga_0\in k\la_0+pQ_0^*}}e^{-\frac{2\pi\ii}{p}(\la_0|\ga_0)}e^{-\frac{\pi\ii}{p}k'|\ga_0|^2}\frac{\th_{\ga_0+Q_0}\left(\frac{\tau+k'}{p}\right)}{P_{\si^{ks},(\si^{ks})^{k'}}(\tau)}\\
&=\frac{\ii^{r_0/2}}{p}\sum_{k=1}^{p-1}\om^{jk}v_k\sum_{\substack{\ga_0+Q_0\in Q_0^*/Q_0\\\ga_0\in k\la_0+pQ_0^*}}e^{-\frac{2\pi\ii}{p}(\la_0|\ga_0)}\sum_{l=0}^{p-1}\om^{-lk'}\chi_{M(\ga_0;\si^{ks})}^l(\tau).
\end{split}
\end{align*}
We now obtain \eqref{chiorbtrans3} from \eqref{PermMtran1}, %and \eqref{PermMtran3}, 
using that 
$$k'\equiv -k^{p-2}\mod p.$$ 
This completes the proof of \thref{thm8.5}.
\end{proof}

In the following table, we present the asymptotic and quantum dimensions of permutation orbifold modules, which are a special case of \coref{corqdim}.

\begin{center}
\begin{table}[H]\label{tab5}
\begin{tabular}{|c|c|c|c|}
\hline
$M$&$V_{(\la_0,\ldots,\la_0)+Q}^j$&$V_{\la+Q}$&$M(\la_0, \si^s)^j$\\
\hline
$\asdim M$&$p^{-1} |Q_0^*/Q_0|^{-p/2}$&$|Q_0^*/Q_0|^{-p/2}$&$p^{-1} |Q_0^*/Q_0|^{-1/2}$\\
\hline
$\qdim M$&$1$&$p$&$|Q_0^*/Q_0|^{(p-1)/2}$\\
\hline
\end{tabular}
\caption{Asymptotic and quantum dimensions}
\end{table}
\end{center}

\subsection{The case when $p=3$} %and $Q_0$ is arbitrary}
In this subsection, we let $p=3$ and use the alternative strategy outlined in Sections \ref{Irrgen2} and \ref{Irrgen3} to calculate the part of the transformation $\chi_{M(\mu,\zeta;\si)}^j\left(-\frac{1}{\tau}, \frac{h}{\tau}\right)$ that contributes to orbifold characters of twisted type without reference to Theorem \ref{KP}.

%compute the modular transformation $\tau\mapsto-1/\tau$ of the quotients
%\begin{align}\label{lastterm}
%e^{-\frac{\pi\ii}{3}k|\la_0|^2}\frac{\th_{\la_0+Q_0}\left(\frac{\tau+k}{3}\right)}{P_{\si^s,\si^{sk}}(\tau)},\qquad k=\pm1,
%\end{align}
%for irreducible orbifold characters of twisted type without reference to Theorem \ref{KP}. 
%The advantage will be that all the coefficients will be determined explicitly. 
%This, in turn, will provide a way for us to determine the constants $v_k$ in \eqref{chiorbtrans3} by comparison.

In this case, \eqref{cL3} becomes
\begin{align}\label{cQ03}
c_{Q_0}=\sum_{\nu+Q_0\in Q_0^*/Q_0}e^{3\pi\ii|\nu|^2},
\end{align}
and the complex numbers $v_1$ and $v_{-1}$ in \eqref{thtwp} can be determined explicitly using \eqref{v1} and \eqref{v2}:
\begin{align}\label{v1v2perm3}
v_1=\displaystyle\frac{\bar{c}_{Q_0}}{|Q_0^*/Q_0|}\,,\qquad v_{-1}=\displaystyle\frac{(-1)^{r_0/2}c_{Q_0}}{|Q_0^*/Q_0|}\,.
\end{align}

\noindent
The transformations of the irreducible permutation orbifold characters of twisted type now follow from \eqref{chiorbtrans3}, \eqref{cQ03} and \eqref{v1v2perm3}:
\begin{align*}%\label{chipermorbtrans3}
\displaystyle\chi_{M(\la_0;\si^s)}^j &\left(-\frac1\tau\right)
%&=\frac{p^{\frac{r_0-3}{2}}}{\sqrt{|Q_0^*/Q_0|}}\sum_{\substack{\de+Q\in Q^*/Q\\\de=(\de_0,\ldots,\de_0)}}e^{-2\pi\ii(\la_0|\de_0)}\sum_{l=1}^{p-1}\om^{ls}\chi_{V_{\de+Q}}^l(\tau)\\
%&+\frac{\ii^{-r_0/2}c_{Q_0}}{p|Q_0^*/Q_0|}\sum_{\ga+Q_0\in Q_0^*/Q_0}\om^{-(\la_0|\ga)}e^{-\frac{\pi\ii}{3} |\ga|^2}\frac{\th_{\ga+Q_0}\left(\frac{\tau+1}{3}\right)}{P_{\si^{s},\si^{s}}(\tau)}\\
%&+\frac{\ii^{r_0/2}\bar{c}_{Q_0}}{p|Q_0^*/Q_0|}\sum_{\ga+Q_0\in Q_0^*/Q_0}\om^{(\la_0|\ga)}e^{\frac{\pi\ii}{3} |\ga|^2}\frac{\th_{\ga+Q_0}\left(\frac{\tau-1}{3}\right)}{P_{\si^{s},\si^{-s}}(\tau)}\\
=\frac{1}{3\sqrt{|Q_0^*/Q_0|}}\sum_{\substack{\de+Q\in Q^*/Q\\\de=(\de_0,\de_0,\de_0)}}\sum_{l=1}^{2}\om^{ls}e^{-2\pi\ii(\la_0|\de_0)}\chi_{V_{\de+Q}}^l(\tau)\\
&+\frac{\ii^{r_0/2}\bar{c}_{Q_0}}{3|Q_0^*/Q_0|}\sum_{\substack{\ga+Q_0\in Q_0^*/Q_0\\\ga\in\la_0+3Q_0^*}}\sum_{l=1}^{2}\om^{-(\la_0|\ga)+j+l}\chi_{M(\ga;\si^s)}^l(\tau)\\
&+\frac{\ii^{-r_0/2}c_{Q_0}}{3|Q_0^*/Q_0|}\sum_{\substack{\ga+Q_0\in Q_0^*/Q_0\\\ga\in-\la_0+3Q_0^*}}\sum_{l=1}^{2}\om^{-(\la_0|\ga)-j-l}\chi_{M(\ga;\si^{-s})}^l(\tau),
\end{align*}
where $\om=e^{2\pi\ii/3}$.

%\subsection{The case when $p=3$ and $Q_0$ is rank 1}
%In this section, we assume $Q_0$ is any even rank one lattice and the isometry is of order $p=3$.

For an even more explicit demonstration, we now consider the case $\rank Q_0=1$,
which has been investigated previously in \cite{DXY3}.
Then we can write 
$$Q=Q_0^{\oplus3} = \ZZ\al_1\oplus\ZZ\al_2\oplus\ZZ\al_3$$
with bilinear form given by 
$$(\al_i|\al_j)=2t\de_{ij}$$ 
for some positive integer $t$. 
%We will calculate the constant $c_{Q_0}$, given by \eqref{cQ03}, for $t=1,2,3$. %, and \eqref{cga} for $t=3$. 
%The isometry $\si$ permutes the entries as in \eqref{permsi}, i.e., $\si: \al_1\rightarrow\al_2\rightarrow\al_3\rightarrow\al_1$. 
From the general results of Section \ref{PermGen}, we have that $d(\si)=1$
%, $\bar{Q}=Q$ 
and all $\si^k$-twisted $V_Q$-modules %, for $k=1,\ldots,p-1$, 
can be described using the cosets of $Q_0^*/Q_0$.
%$(1-\si)Q=Q\cap\lieh_\perp$. 

For convenience, we set $\al=\al_1+\al_2+\al_3$. Then $|\al|^2=6t$ and 
\[
M=Q\cap\lieh_0=\ZZ\al, \qquad Q\cap\lieh_\perp=\ZZ(\al_1-\al_2)+\ZZ(\al_2-\al_3). 
\]
We can also write \eqref{L} as
$$L=\{x_1\al_1+x_2\al_2+x_3\al_3\;|\;x_1+x_2+x_3\equiv0\mod 3\},$$
and we find that $Q/L\cong\ZZ_3$ is generated by the coset $\al_1+L$. The dual lattice of $Q$ and its fundamental group are
\begin{align*}
Q^*&=\ZZ\frac{\al_1}{2t}\oplus\ZZ\frac{\al_2}{2t}\oplus\ZZ\frac{\al_3}{2t} \,,\qquad
Q^*/Q\cong\ZZ_{2t}\oplus\ZZ_{2t}\oplus\ZZ_{2t}.
\end{align*}
It follows that $(Q^*/Q)^\si\cong\ZZ_{2t}$, and it consists of the elements 
\begin{align*}
\frac{m}{2t}\al+Q \quad\text{for}\quad 0\leq m\leq 2t-1. 
\end{align*}
We also find that 
$$Q^*\cap\lieh_0=\ZZ\frac{\al}{2t}\,.$$ 
For each $m$, the automorphism $\si$ acts on the irreducible $V_Q$-module $V_{\frac{m}{2t}\al+Q}$, and so it decomposes into $3$ irreducible orbifold modules given by the eigenspaces of $\si$.
%To describe the irreducible twisted $V_Q$-modules, we calculate that 
%\begin{align*}
%\pi_0(Q)=\ZZ\frac{\al}{3},\quad\text{and}\quad\pi_0(Q^*)=\ZZ\frac{\al}{6t}.
%\end{align*}
%Hence these modules are described using the cosets 
%\begin{align*}
%\mu+\pi_0(Q)=\frac{m}{6t}\al+\pi_0(Q)\quad\text{for}\quad0\leq m\leq 2t-1.
%\end{align*}

From \eqref{type3num}--\eqref{type2num}, there are $12t$ irreducible $V_Q^\si$-modules of twisted type, $6t$ irreducible $V_Q^\si$-modules of Type 1, and $\frac23(4t^3-t)$ irreducible $V_Q^\si$-modules of Type 2 (also see the discussion in Section \ref{PermGen}). Therefore, we have a total of
$$18t+\frac23(4t^3-t)$$
irreducible $V_Q^\si$-modules. In particular, we obtain $20$ irreducible permutation orbifold modules for $t=1$, which agrees with \cite{DXY3}.
The orbits of $\si$ in $Q^*/Q$ that are not singletons can be described in four class types, which we describe in Table 6 below. We use 
\[
\ga_i=d_i\frac{\al_1}{2t}+e_i\frac{\al_2}{2t}+f_i\frac{\al_3}{2t} \qquad (0\leq d_i,e_i,f_i\leq2t-1, \;\; i=1, 2, 3, 4)
\]
to represent a generic element in $Q^*$ from which to describe the class types. Notice that the sum of all class sizes is indeed $\frac23(4t^3-t)$.

\begin{center}
\begin{table}[H]\label{tab6}
\begin{tabular}{|c|c|c|}
\hline
Representative&Relation among coefficients&Class size\\
\hline
$\ga_1$&$d_1=e_1<f_1$&$2t^2-t$\\
\hline
$\ga_2$&$d_2<e_2=f_2$&$2t^2-t$\\
\hline
$\ga_3$&$d_3<e_3<f_3$&$\frac{1}{6}(2t-2)(2t-1)(2t)$\\
\hline
$\ga_4$&$d_4<f_4<e_4$&$\frac{1}{6}(2t-2)(2t-1)(2t)$\\
\hline
\end{tabular}
\caption{Cosets $\ga+Q$ with $\ga\in Q^*$ and $(1-\si)\ga\notin Q$}
\end{table}
\end{center}

%Set $\th_1^u(\tau)=\th_{\frac{u}{2t}\al+Q}(\tau)$ and $\th_{tw}^m(\tau)=\th_{\frac{m}{2\sqrt{3}t}\al+\ZZ\frac{\al}{\sqrt{3}}}(\tau)=\th_{\frac{m}{2t}\al_1+\ZZ\al_1}(\tau)$.

%Using Theorem \ref{PermMod} we write the irreducible characters as
%
%\begin{align}
%\chi^j_{V({\frac{m}{2t}\al+Q})}(\tau)&=\frac13\left(\frac{\th_{\frac{m}{2t}\al+Q}(\tau)}{P_{1,1}(\tau)}+\om^{j}\frac{\th_{\frac{m}{2t}\al_1+\ZZ\al_1}(3\tau)}{P_{1,\si}(\tau)}+\om^{-j}\frac{\th_{\frac{m}{2t}\al_1+\ZZ\al_1}(3\tau)}{P_{1,\si^{-1}}(\tau)}\right),\\
%\chi_{V_{\la+Q}}(\tau)&=\frac{\th_{\la+Q}(\tau)}{P_{1,1}(\tau)},\qquad(1-\si)\la\notin Q,\\
%\chi^j_{M(\frac{m}{2t}\al_1;\si)}(\tau)&=\frac13\left(\frac{\th_{\frac{m}{2t}\al_1+\ZZ\al_1}(\frac{\tau}{3})}{P_{\si,1}(\tau)}+\om^je^{-\pi\ii m^2/6t}\frac{\th_{\frac{m}{2t}\al_1+\ZZ\al_1}(\frac{\tau+1}{3})}{P_{\si,\si}(\tau)}+\om^{-j}e^{-\pi\ii m^2/3t}\frac{\th_{\frac{m}{2t}\al_1+\ZZ\al_1}(\frac{\tau-1}{3})}{P_{\si,\si^{-1}}(\tau)}\right),\\
%\chi^j_{M(\frac{m}{2t}\al_1;\si^{-1})}(\tau)&=\frac13\left(\frac{\th_{\frac{m}{2t}\al_1+\ZZ\al_1}(\frac{\tau}{3})}{P_{\si^{-1},1}(\tau)}+\om^je^{-\pi\ii m^2/6t}\frac{\th_{\frac{m}{2t}\al_1+\ZZ\al_1}(\frac{\tau+1}{3})}{P_{\si^{-1},\si^{-1}}(\tau)}+\om^{-j}e^{-\pi\ii m^2/3t}\frac{\th_{\frac{m}{2t}\al_1+\ZZ\al_1}(\frac{\tau-1}{3})}{P_{\si^{-1},\si}(\tau)}\right).
%\end{align}
%where $\phi=e^{\pi\ii/6}$.

Finally, we compute the constant $c_{Q_0}$, given by \eqref{cQ03}, for $t=1,2,3$. 
%not divisible by 3, and the constant $c_\ga$ (cf.\ \eqref{cga}) for $\ga\in Q^*$ and $t=3$.
%Consider  the following elements of $Q_0^*$
%\begin{equation*}
%\frac{n}{2t}\al_1 \qquad (0\leq n\leq 2t-1).
%%\text{and}\quad \ga=\frac{m}{2t}\al_1
%\end{equation*} 
%%to be elements of $Q_0^*$, where $0\leq m,n<2t$.  
Recall from Section \ref{PermGen} that $\sqrt{3}\pi_0(Q)\cong Q_0$ and $\sqrt{3}\pi_0(Q^*)\cong Q_0^*$. Then \eqref{cL3} becomes
\begin{align}\label{c0t}
c_0=c_{Q_0}=\sum_{n=0}^{2t-1}e^{3\pi\ii n^2/2t},
\end{align}
and we calculate this explicitly for some values of $t$ in Table 7.

\begin{table}[H]\label{tab7}
\begin{tabular}{|c|c|}
\hline
$t=1$&$c_0=1-\ii$\\
\hline
$t=2$&$c_0=2e^{\frac{3\pi\ii}{4}}$\\
\hline
$t=3$&$c_0=3+3\ii$\\
\hline
\end{tabular}
\caption{Values of $c_0$}
\end{table}

\section*{Acknowledgments}

B.B. was supported in part by a Simons Foundation grant 584741. 
B.B. is grateful to the Institute for Nuclear Research and Nuclear Energy of the Bulgarian Academy of Sciences for the hospitality during June 2022 when this work was being completed.
V.K. was supported in part by the Simons Collaborative grant and the Berenson mathematical exploration fund.

\bibliographystyle{amsalpha}

\end{document}